\documentclass[12pt]{article}

\usepackage{epsfig,epsfig}
\usepackage{epstopdf}
\usepackage{amsmath}
\usepackage{mathrsfs}
\usepackage{amsfonts}
\usepackage{bm}
\usepackage{amssymb}
\usepackage{booktabs}
\usepackage{color}
\usepackage{multirow}
\usepackage{graphicx}
\usepackage{subfigure}
\usepackage{float}
\usepackage{appendix}
\usepackage{tikz}
\usetikzlibrary{calc}
\usetikzlibrary{matrix}
\usepackage{mathtools}
\usepackage{algorithm}

\usepackage{lineno}
\usepackage{srcltx,graphicx}

\newtheorem{thm}{Theorem}[section]
\newtheorem{prop}[thm]{Proposition}

\newtheorem{lemma}[thm]{Lemma}

\newtheorem{proof}[thm]{Proof}

\newtheorem{rem}[thm]{Remark}
\newtheorem{example}[thm]{Example}

\newcommand{\beq}{\begin{equation}}
\newcommand{\eeq}{\end{equation}}

\usepackage{hyperref}

\numberwithin{equation}{section} \topmargin=-2.0cm \oddsidemargin=1cm
\begin{document}

\title{\textbf{A well-balanced positivity-preserving quasi-Lagrange moving mesh DG method for the shallow water equations}\footnote{M. Zhang and J. Qiu were supported partly by Science Challenge Project (China), No. TZ 2016002 and
National Natural Science Foundation--Joint Fund (China) grant U1630247.
This work was carried out while M. Zhang was visiting the Department of Mathematics, the University of Kansas
under the support by the China Scholarship Council (CSC: 201806310065).
}}
\author{Min Zhang\footnote{School of Mathematical Sciences, Xiamen University,
Xiamen, Fujian 361005, China.
E-mail: minzhang2015@stu.xmu.edu.cn.
},
~Weizhang Huang\footnote{Department of Mathematics, University of Kansas, Lawrence, Kansas 66045, USA. E-mail: whuang@ku.edu.
},
~and Jianxian Qiu\footnote{School of Mathematical Sciences and Fujian Provincial Key Laboratory of Mathematical Modeling and High-Performance Scientific Computing, Xiamen University, Xiamen, Fujian 361005, China.
E-mail: jxqiu@xmu.edu.cn.
}
}
\date{}
\maketitle
\noindent\textbf{Abstract:}
A high-order, well-balanced, positivity-preserving quasi-Lagrange moving mesh DG method is presented for the shallow water equations with non-flat bottom topography. The well-balance property is crucial to the ability of a scheme to simulate perturbation waves over the lake-at-rest steady state such as waves on a lake or tsunami waves in the deep ocean. The method combines a quasi-Lagrange moving mesh DG method, a hydrostatic reconstruction technique, and a change of unknown variables. The strategies in the use of slope limiting, positivity-preservation limiting, and change of variables to ensure the well-balance and positivity-preserving properties are discussed. Compared to rezoning-type methods, the current method treats mesh movement continuously in time and has the advantages that it does not need to interpolate flow variables from the old mesh to the new one and places no constraint for the choice of an update scheme for the bottom topography on the new mesh.
A selection of one- and two-dimensional examples are presented to demonstrate the well-balance property, positivity preservation, and high-order accuracy of the method and its ability to adapt the mesh according to features in the flow and bottom topography.

\vspace{5pt}

\noindent\textbf{The 2020 Mathematics Subject Classification:} 65M50, 65M60, 76B15, 35Q35

\vspace{5pt}


\newcommand{\h}{\hspace{1.cm}}
\newcommand{\hh}{\hspace{2.cm}}
\newtheorem{yl}{\hspace{1.cm}Lemma}
\newtheorem{dl}{\hspace{1.cm}Theorem}
\renewcommand{\sec}{\section*}
\renewcommand{\l}{\langle}
\renewcommand{\r}{\rangle}
\newcommand{\be}{\begin{eqnarray}}
\newcommand{\ee}{\end{eqnarray}}

\normalsize \vskip 0.2in
\newpage

\section{Introduction}

The shallow water equations (SWEs) model the water flow over a surface
such as hydraulic jumps/shocks and open-channel flows
in the ocean/hydraulic engineering.
They can be derived by integrating the Navier-Stokes equations in depth
under the hydrostatic assumption when the depth of the flow is small compared
to its horizontal dimensions.
The two-dimensional SWEs can be cast in conservative form as
\begin{equation}
\label{cswe-2d}
V_t +  \nabla \cdot \mathcal{F}(V)  = \mathcal{S}(h,B),
\end{equation}
where $h$ is the depth of water, $V = (h, m, w)^T$ denote the conservative variables,
$(m, w) = (hu, hv)$ are the discharges,
$(u,v)$ are the velocities,
$B = B(x,y)$ is the bottom topography assumed to be a given time-independent function,
$g$ is the gravitation acceleration,
and the flux $\mathcal{F}(V)$ and the source $\mathcal{S}(h,B)$ are given by
\begin{equation}\label{cswe-2d-flux}
\mathcal{F}(V) =
\begin{bmatrix*}[l]
  m&w\\
  \frac{m^2}{h}+\frac{1}{2}gh^2 &\frac{mw}{h}\\
  \frac{mw}{h}&\frac{w^2}{h}+\frac{1}{2}gh^2\\
\end{bmatrix*} ,
\quad
\mathcal{S}(h,B) = \begin{bmatrix*}[l] 0\\ -hgB_x\\ -hgB_y\end{bmatrix*}.
\end{equation}
An illustration of $h$, $B$, and the free water surface level $\eta = h + B$
is given in Fig.~\ref{Fig:swe-model}.
\begin{figure}[htbp]
  \centering
  \includegraphics[width=0.5\textwidth]{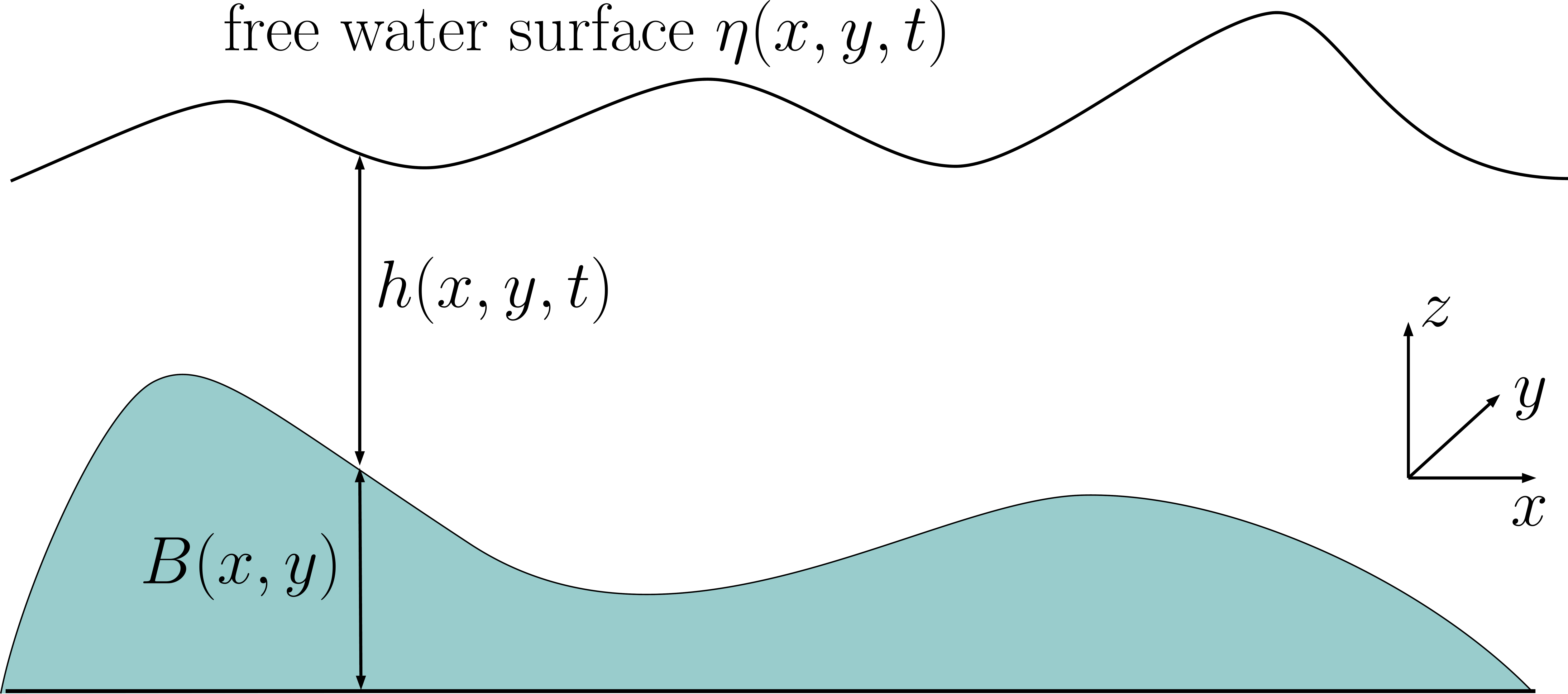}
  \caption{An illustration of the water depth $h$, the bottom topography $B$, and the free water surface level $\eta = h + B$.}
  \label{Fig:swe-model}
\end{figure}

We are interested in the preservation of the ``lake-at-rest'' steady state solution
\begin{equation}
\label{still-water-2d}
m=hu=0, \quad w=hv=0,
\quad
\eta = h+B = C
\end{equation}
where $C$ is a constant.
Many physical phenomena can be described as small perturbations of this steady-state solution,
including waves on a lake or tsunami waves in the deep ocean.
They are difficult, if not impossible, to capture by a numerical method that does not preserve (\ref{still-water-2d}),
on an unrefined mesh. Thus, for the numerical simulation of perturbation waves over
the lake-at-rest steady state, it is important to develop schemes that preserve (\ref{still-water-2d}).
These schemes are said in literature to be well-balanced or have the well-balance property or the C-property. Bermudez and Vazquez \cite{Bermudez-Vazquez-1994} first introduced a concept
of the ``exact C-property". Since then, a number of well-balanced numerical methods
have been developed for the SWEs, e.g., finite volume methods
\cite{Audusse-etal-2004Siam, Bermudez-Vazquez-1994, LeVeque-1998JCP, Zhou-etal-2001JCP},
finite difference/volume WENO methods \cite{Lu-Qiu-Wang-2010JCM,Xing-Shu-2005JCP,Xing-Shu-2006JCP, Xing-Shu-2006CiCP},
and discontinuous Galerkin (DG) methods \cite{Eskilsson-Sherwin-2004,Ern-etal-2008,Li-etal-2018JCAM,
Tumolo-2013JCP, Xing-Shu-2006JCP, Xing-Shu-2006CiCP,Xing-Zhang-Shu-2010,Xing-Zhang-2013JSC}.
DG methods have the advantages of high-order accuracy, high parallel efficiency, and flexibility
for $hp$-adaptivity and arbitrary geometry and meshes.

The SWEs exhibit interesting structures including hydraulic jumps/shocks, rarefaction waves,
and stationary state transitions. Resolving them in the numerical solution
requires fine spatial spacings and thus mesh adaptation becomes a useful tool in improving
the computational accuracy and efficiency. Studies have been made in this direction in the past.
For example, Tang \cite{Tang-2004} developed an adaptive moving structured mesh
kinetic flux-vector splitting (KFVS) scheme for the SWEs and showed that the method
leads to more accurate solutions than methods based on fixed meshes although the well-balance
property was not addressed specifically in the work.
Lamby et al. \cite{Lamby-etal-2005} proposed an adaptive multi-scale finite volume method for the SWEs
with source terms, combining a B-spline based quadtree grid generation
strategy and a fully adaptive multi-resolution method.
Remacle et al. \cite{Remacle-etal-2006} studied an $h$-adaptive meshing procedure
for the transient computation of the SWEs.
Zhou et al. \cite{Zhou-etal-2013} proposed a well-balanced adaptive moving mesh generalized Riemann problem (GRP)-based finite volume scheme for the SWEs with irregular bottom topography.
Donat et al. \cite{Donat-etal-2014JCP} developed a well-balanced shock capturing adaptive
mesh refinement (AMR) scheme for shallow water flows.
Arpaia and Ricchiuto \cite{Arpaia-Ricchiuto-2018} considered several arbitrary Lagrangian-Eulerian (ALE) formulations of the SWEs on moving meshes and provided a discrete analog in the well-balanced finite volume and residual distribution framework.
Most recently, a high-order, well-balanced, positivity-preserving and rezoning-type adaptive moving mesh DG method
was proposed in \cite{Zhang-Huang-Qiu-2020arXiv} for the SWEs. It requires that
both the flow variables and bottom topography be updated from the old mesh to the new one
at each time step using the same interpolation scheme. A positivity-preserving DG-interpolation scheme
\cite{Zhang-Huang-Qiu-2019arXiv} has been used in \cite{Zhang-Huang-Qiu-2020arXiv}
for the purpose.

We consider here a quasi-Lagrange approach of adaptive moving mesh methods where the mesh is considered to move continuously between time steps and interpolation of the flow variables between the old mesh and the new one is unnecessary.
The quasi-Lagrange moving mesh DG (QLMM-DG) method has been  used successfully
for solving hyperbolic conservation laws \cite{Luo-Huang-Qiu-2019JCP}
and the radiative transfer equation \cite{Zhang-Cheng-Huang-Qiu-2020CiCP}.
Our focus here is on its application to the SWEs and the well-balance property,
and we shall use a change of unknown variables.
More specifically, we use the new variables $(\eta= h + B,\, hu,\, hv)$ instead of the original ones
$(h,\, hu,\, hv)$ and rewrite the flux \eqref{cswe-2d-flux} into a special form
(cf. \eqref{swe-2d-FUh}) by replacing only some of $h$'s with $\eta$.
In the construction of the DG numerical flux, we modify the value of $h$ using
the hydrostatic reconstruction technique of
\cite{Audusse-etal-2004Siam,Xing-Shu-2006CiCP,Xing-Zhang-Shu-2010,Xing-Zhang-2013JSC}
but keep $\eta$ unmodified.
We will show that the new QLMM-DG method, in both semi-discrete and fully discrete forms, preserves
\eqref{still-water-2d} while maintaining the high-order accuracy of DG methods.
We will also show that a QLMM-DG scheme can be developed based on the SWEs
in the original variables but the resulting scheme is well-balanced only in semi-discrete form.

It is worth pointing out that the bottom topography $B$ needs to be updated on the new mesh
at each time step in the current method. Nevertheless, unlike the rezoning moving mesh DG method
in \cite{Zhang-Huang-Qiu-2020arXiv}, it places no constraint on the choice of the scheme
for updating $B$ to attain the well-balance property.
We use $L^2$-projection for this purpose in our computation
since it is straightforward and economic to implement.

Another challenge in the numerical solution of the SWEs is to preserve the nonnegativity of the water depth in the computation.
Following \cite{Xing-Zhang-Shu-2010,Xing-Zhang-2013JSC}, we apply a linear scaling positivity-preserving (PP) limiter \cite{Liu-Osher1996,ZhangShu2010,ZhangXiaShu2012} to the water depth.
However, the PP limiter destroys the well-balance property.
To recover the property, we propose to make a high-order correction to the approximation of the bottom topography
according to the modifications in the water depth due to the PP limiting.
Numerical examples show that this strategy works out well.

For the mesh adaptation, we use a moving mesh PDE (MMPDE) method \cite{Huang-Sun-2003JCP,Huang-Russell-2011,Huang-Kamenski-2015JCP}
which is known to produce meshes free of tangling \cite{Huang-Kamenski-2018MC}.
The MMPDE method uses a metric tensor to determine the size, shape, and orientation of the mesh elements
throughout the physical domain.
Following \cite{Zhang-Huang-Qiu-2020arXiv}, we compute the metric tensor
based on the equilibrium variable $\mathcal{E}=\frac{1}{2}(u^2+v^2)+g(h+B)$ and the water depth $h$
in the hope that the mesh adapts to the features in the water flow and bottom topography.

An outline of the paper is organized as follows.
\S\ref{sec:WB-MMDG} is devoted to the description of the QLMM-DG method based on the new variables $(\eta,hu,hv)$ and its well-balance property.
The discussion of a moving mesh method based on the formulation in the original variables $(h,hu,hv)$ is presented
in \S\ref{sec:cswe-QLMMDG}.
The generation of an adaptive moving mesh using the MMPDE moving mesh method is
described in Appendix \ref{sec:mmpde}.
In \S\ref{sec:numerical-results}, a selection
of one- and two-dimensional examples are presented and analyzed.
Finally, \S\ref{sec:conclusions} contains the conclusions.

\section{The well-balanced QLMM-DG method}
\label{sec:WB-MMDG}

In this section we describe the high-order well-balanced positivity-preserving QLMM-DG method for the numerical solution of the SWEs with non-flat bottom topography. This method combines
the quasi-Lagrange moving mesh DG method of \cite{Luo-Huang-Qiu-2019JCP,Zhang-Cheng-Huang-Qiu-2020CiCP}
with the hydrostatic reconstruction technique \cite{Audusse-etal-2004Siam,Xing-Shu-2006CiCP,Xing-Zhang-2013JSC} and
a change of unknown variables to attain the well-balance property.
The method is described here only in two dimensions. It has a similar form in one dimension.

We use here the new variables $ U = (\eta, m,w)^T$ instead of the original ones
$ V = (h,m,w)^T$, where $\eta = h + B$. We rewrite the SWEs \eqref{cswe-2d} and \eqref{cswe-2d-flux} into
\begin{equation}\label{swe-2d-another}
U_t +  \nabla \cdot \mathbf{F}(U,h)  = S(\eta,B),
\end{equation}
where
\begin{equation}
\label{swe-2d-FUh}
\begin{cases}
\begin{split}
&S(\eta,B) = \big(0,-g\eta B_x,-g\eta B_y \big)^T,
\\&\mathbf{F}(U,h) =
 \begin{bmatrix*}[l]
  m&w\\
  \frac{m^2}{h}+\frac{1}{2}g(2 h\eta-\eta^2 ) &\frac{mw}{h}\\
  \frac{mw}{h}&\frac{w^2}{h}+\frac{1}{2}g(2 h\eta-\eta^2 )\\
 \end{bmatrix*}
=
 \begin{bmatrix*}[l]
 \bm{f}^1(U,h)\\
 \bm{f}^2(U,h)\\
 \bm{f}^3(U,h)\\
\end{bmatrix*} .
\end{split}
\end{cases}
\end{equation}
Here, the dependence of $\mathbf{F}$ on $h$ is expressed explicitly although
$h$ is a linear function of $\eta$. This is because the value of $h$ will be modified
but $\eta$ is kept unmodified in the computation of the numerical flux to attain the well-balance property.
Moreover, not all of $h$'s in the flux have been replaced by $\eta - B$, i.e.,
some are replaced with the new variable $\eta$ and some remain the same.
Obviously, there are many of these combinations and thus many forms of the flux; for example, see (\ref{swe-2d-FUh}
and (\ref{swe-2d-etamw-flux}). These forms are equal to each other mathematically but can be different numerically.
Indeed, we will show that {\it the form (\ref{swe-2d-FUh})
leads to a well-balanced scheme in both semi-discrete and fully discrete forms
while it is unclear if a well-balanced
scheme can be obtained from (\ref{swe-2d-etamw-flux})} (cf. Remark~\ref{another-form}).
We will also show in \S\ref{sec:cswe-QLMMDG} that a QLMM-DG scheme can be developed based
on \eqref{cswe-2d} and \eqref{cswe-2d-flux} using the original variables
but the resulting scheme is well-balanced only in semi-discrete form.

For the moment we assume that a sequence of simplicial meshes having the same number of elements and vertices
and the same connectivity, $\mathcal{T}^0_h,\,\mathcal{T}^1_h,\, ...$, have been obtained for time instants
$t_0,\, t_1,\, ...$. The generation of these meshes is discussed in Appendix~\ref{sec:mmpde}.
Recall that we consider here the quasi-Lagrange approach of moving mesh methods where
the mesh is considered to move continuously in time. To this end,
for any $n \ge 0$, we define $\mathcal{T}_h(t),~ t\in(t_n, t_{n+1}]$, as a mesh having the same
number of elements and vertices and the same connectivity as $\mathcal{T}_h^n$ and $\mathcal{T}_h^{n+1}$
and having the vertices and nodal velocities given by
\begin{equation}\label{location+speed}
\begin{split}
&\bm{x}_i(t)= \frac{t-t_n}{\Delta t_n}\bm{x}_i^{n}+\frac{t_{n+1}-t}{\Delta t_n}\bm{x}_i^{n+1},
\quad i = 1,...,N_v
\\&\dot{\bm{x}}_i = \frac{\bm{x}_i^{n+1}-\bm{x}_i^{n}}{\Delta t_n}, \quad i = 1,...,N_v.
\end{split}
\end{equation}
We also define the piecewise linear mesh velocity function $\dot{\bm{X}}(\bm{x},t) =\big(\dot{X}, \dot{Y}\big)$ as
\begin{equation}\label{Xdot-1}
\dot{\bm{X}}(\bm{x},t)= \sum_{i=1}^{N_v} \dot{\bm{x}}_i \phi_i(\bm{x},t),
\end{equation}
where $ \phi_i(\bm{x},t)$ is the linear basis function associated with the vertex $\bm{x}_i$.
For any $K \in \mathcal{T}_h(t)$, let $\phi_{K}^j(\bm{x},t),~ j=1,\cdots, n_b\equiv (k+1)(k+2)/2$
be the basis functions of the set
of polynomials of degree at most $k\ge 1$ on $K$, $P^{k}(K)$.
The DG finite element space is defined as
\begin{equation}\label{Vh}
\mathcal{V}^{k}_h(t)= \{q\in L^2(\Omega):\; q|_{K}\in P^{k}(K),
\; \forall K\in \mathcal{T}_h(t) \}.
\end{equation}

\subsection{The semi-discrete well-balanced QLMM-DG scheme}

Multiplying \eqref{swe-2d-another} with a test function $\phi\in \mathcal{V}^{k}_h(t)$,
integrating the resulting equation over $K$, and using the Reynolds transport theorem, we have
\begin{equation}
\label{re-test-1}
\frac{d}{d t}\int_{K}U \phi d\bm{x}+ \int_{K}\phi\nabla \cdot \big{(}- U \dot{\bm{X}}\big{)}d\bm{x}
+ \int_{K}\phi\nabla \cdot \mathbf{F}(U,h) d\bm{x}= \int_{K}S(\eta,B)\phi d\bm{x}.
\end{equation}
Recall that we have $h = \eta - B$.
Denote
\begin{equation}
\label{H-def}
\begin{split}
\mathbf{H}(U,h)
&= \mathbf{F}(U,h) - U \dot{\bm{X}}
=
 \begin{bmatrix*}[l]
 \bm{f}^1(U,h) -\eta\dot{\bm{X}} \\
 \bm{f}^2(U,h) -m\dot{\bm{X}}\\
 \bm{f}^3(U,h)-w\dot{\bm{X}}\\
 \end{bmatrix*} .
\end{split}
\end{equation}
From the divergence theorem, we can rewrite (\ref{re-test-1}) as
\begin{equation}
\label{re-test}
\frac{d}{d t}\int_{K}U \phi d\bm{x} -  \int_{K} \mathbf{H}(U,h) \cdot \nabla \phi d\bm{x}
+ \sum_{e_K \in \partial K}\int_{e_K}\phi \mathbf{H}(U,h) \cdot \bm{n}d s= \int_{K}S(\eta,B)\phi d\bm{x},
\end{equation}
where $\bm{n} = (n_x,n_y)^T$ is the outward unit normal to the boundary $ \partial K$.

The Jacobian matrix of the vector-valued function $\mathbf{H}\cdot \bm{n}$ with respect to $U$
(with $h$ being considered as a linear function of $\eta$) reads as
\[
\label{Jmat-H}
 \begin{bmatrix*}[l]
    -\dot{X}n_x - \dot{Y}n_y & n_x & n_y \\
    (c^2-u^2)n_x-uvn_y & (2u-\dot{X})n_x+(v-\dot{Y})n_y & un_y\\
    (c^2-v^2)n_y-uvn_x & vn_x & (u-\dot{X})n_x+(2v-\dot{Y})n_y \\
 \end{bmatrix*},
\]
where $c = \sqrt{gh}$ is the sound speed.
The eigenvalues of this matrix can be found as
\begin{equation}\label{eigen-H}
\begin{cases}
\begin{split}
&\lambda^{1}(U,h,\dot{\bm{X}})=(u-\dot{X})n_x +(v-\dot{Y})n_y -c,
\\
&\lambda^{2}(U,h,\dot{\bm{X}})=(u-\dot{X})n_x +(v-\dot{Y})n_y,
\\
&\lambda^{3}(U,h,\dot{\bm{X}})=(u-\dot{X})n_x +(v-\dot{Y})n_y +c.
\end{split}
\end{cases}
\end{equation}

For any variable $q_h$ on the boundary $\partial K$, we denote
by $q_{h,K}^{int}$ and $q_{h,K}^{ext}$ as the values of $q_h$ on $\partial K$ from the interior and exterior of $K$, respectively.
We also note that the bottom topography function $B = B(\bm{x})$ needs to be projected
into the finite element space $\mathcal{V}_h^{k}(t)$ and denote it by $B_h$.
We use the global Lax-Friedrichs numerical flux to approximate
$\mathbf{H}(U,h)\cdot \bm{n}^e_K$ for $\bm{x} \in e_K \subset \partial K$, i.e.,
\begin{equation}\label{lf-flux}
\begin{split}
\hat{ \mathbf{H}}|_{e_K}
&=\hat{ \mathbf{H}}(U_{h,K}^{int},h_{h,K}^{int};U_{h,K}^{ext},h_{h,K}^{ext}; \bm{n}^e_K)
\\&=
\frac{1}{2}\Big{(}
\big{(}\mathbf{H}(U_{h,K}^{int},h_{h,K}^{int})+\mathbf{H}(U_{h,K}^{ext},h_{h,K}^{ext}) \big{)} \cdot \bm{n}_K
-\alpha_h (U_{h,K}^{ext}-U_{h,K}^{int})\Big{)},
\end{split}
\end{equation}
where $h_{h,K}^{int}= \eta_{h,K}^{int} - B_{h,K}^{int}$, $h_{h,K}^{ext}= \eta_{h,K}^{ext} - B_{h,K}^{ext}$, and
\[
\alpha_h = \max\limits_{K, e_K, m}\left ( \max\limits \Big (
|\lambda^{m}(U_{h,K}^{int},h_{h,K}^{int},\dot{\bm{X}}|_{e_K})|,
\; |\lambda^{m}(U_{h,K}^{ext},h_{h,K}^{ext},\dot{\bm{X}}|_{e_K})|\Big{)} \right ).
\]
We can then define a semi-discrete DG approximation $U_h \in \mathcal{V}^{k}_h(t)$
for \eqref{swe-2d-another} such that
\begin{equation}\label{semi-QLMMDG-2d}
\begin{split}
&\frac{d}{d t}\int_{K}U_h \phi d\bm{x}
 -\int_{K} \mathbf{H}(U_h, h_h)\cdot\nabla \phi d\bm{x}
 +\sum_{e_K \in \partial K}\int_{e_K} \phi \hat{ \mathbf{H}}|_{e_K} ds
\\&~~~~~~~~~~~~~~~~~~
 = \int_{K}S(\eta_h,B_h)\phi d\bm{x}, \quad \forall \phi \in \mathcal{V}^{k}_h(t), \quad \forall  K\in \mathcal{T}_h(t) ,
\end{split}
\end{equation}
where $h_h = \eta_h - B_h$.
In actual computation, the area and line integrals in the above equation
are calculated using Gaussian quadrature rules.
Define the residual associated with this scheme as
\begin{equation}\label{residual-2d}
\begin{split}
R_{h,K}(t)
=&\frac{d}{d t}\int_{K}U_h \phi d\bm{x}  - \int_{K}\phi\nabla \cdot( U_h \dot{\bm{X}} ) d\bm{x}
\\=& \int_{K}S(\eta_h,B_h)\phi d\bm{x} +\int_{K} \mathbf{H}(U_h,h_h)\cdot\nabla \phi d\bm{x}
\\&
- \sum_{e_K\in \partial K}\int_{e_K} \phi \hat{ \mathbf{H}}|_{e_K} ds
- \int_{K}\phi\nabla \cdot( U_h \dot{\bm{X}} ) d\bm{x},\quad \forall  K\in \mathcal{T}_h(t) .
\end{split}
\end{equation}
Note that a term is added in the definition to account for the effects of mesh movement.
It is not difficult to show that (\ref{still-water-2d}) is preserved if this residual vanishes
for the lake-at-rest steady state. Unfortunately, it can be verified that the latter does not hold in general
and thus the scheme \eqref{semi-QLMMDG-2d} is not well-balanced.
The main issue is that the numerical flux $\hat{ \mathbf{H}}|_{e_K}$ does not reduce
to $\mathbf{H}\big( U_{h,K}^{int},h_{h,K}^{int}\big)\cdot \bm{n}_K^e$
for the lake-at-rest steady state and the line integrals cannot be converted back into an area
integral involving $\mathbf{H}$ under the divergence theorem.

To attain the well-balance property, we use the hydrostatic reconstruction technique
of \cite{Audusse-etal-2004Siam,Xing-Shu-2006CiCP,Xing-Zhang-2013JSC}
to construct a new numerical flux $\hat{ \mathbf{H}}^{*}|_{e_K}$ from $\hat{ \mathbf{H}}|_{e_K}$.
To this end, we first compute
\begin{equation}
\label{reconst-2d-h}
\begin{cases}
\begin{split}
&h_{h,K}^{*,int}|_{e_K} =\max \Big( 0, \eta_{h,K}^{int}|_{e_K}-\max\big(B_{h,K}^{int}|_{e_K},B_{h,K}^{ext}|_{e_K}\big)\Big),
\\&h_{h,K}^{*,ext}|_{e_K} =\max \Big( 0, \eta_{h,K}^{ext}|_{e_K}-\max\big(B_{h,K}^{int}|_{e_K},B_{h,K}^{ext}|_{e_K}\big)\Big).
\end{split}
\end{cases}
\end{equation}
Notice that the value of $B_h$ on $e_K$ is taken as $\max\big(B_{h,K}^{int}|_{e_K},B_{h,K}^{ext}|_{e_K}\big)$.
Moreover, $h_{h,K}^{*,int}|_{e_K}$ and $h_{h,K}^{*,ext}|_{e_K}$ are chosen
to guarantee $h_{h,K}^{*,int}|_{e_K} \ge 0$ and $h_{h,K}^{*,ext}|_{e_K}\ge 0$ while trying to satisfy
\begin{equation*}
\begin{cases}
\begin{split}
&h_{h,K}^{*,int} + \max\big(B_{h,K}^{int}|_{e_K},B_{h,K}^{ext}|_{e_K}\big) = \eta_{h,K}^{int}|_{e_K},
\\&
h_{h,K}^{*,ext} + \max\big(B_{h,K}^{int}|_{e_K},B_{h,K}^{ext}|_{e_K}\big)=\eta_{h,K}^{ext}|_{e_K}.
\end{split}
\end{cases}
\end{equation*}
Then, the interior and exterior values of $U=(\eta,m,w)^T$ are modified as
\begin{equation}
\label{reconst-2d-U}
\begin{cases}
\begin{split}
 &  \eta_{h,K}^{*,int} =   \eta_{h,K}^{int} ,\qquad\qquad \eta_{h,K}^{*,ext} =   \eta_{h,K}^{ext} ,
\\&  m_{h,K}^{*,int} =\frac{h_{h,K}^{*,int}}{h_{h,K}^{int}}m_{h,K}^{int}, \quad  m_{h,K}^{*,ext} =\frac{h_{h,K}^{*,ext}}{h_{h,K}^{ext}}m_{h,K}^{ext},
\\&  w_{h,K}^{*,int} =\frac{h_{h,K}^{*,int}}{h_{h,K}^{int}}w_{h,K}^{int}, \quad  w_{h,K}^{*,ext} =\frac{h_{h,K}^{*,ext}}{h_{h,K}^{ext}}w_{h,K}^{ext}.
\end{split}
\end{cases}
\end{equation}
It is worth pointing out that {\em the value of $h$ has been modified but that of $\eta$ remains unmodified.}
Finally, the new flux $\hat{ \mathbf{H}}^{*}$ on the edge $e_K \in \partial K$ is given by
\begin{align}
\label{eK-flux}
&\hat{ \mathbf{H}}^{*}|_{e_K} =
\hat{ \mathbf{H}}(U_{h,K}^{*,int},h_{h,K}^{*,int};U_{h,K}^{*,ext},h_{h,K}^{*,ext};\bm{n}^e_K)+
\Delta^{*}_{e_K}\cdot \bm{n}_K^e,
\end{align}
where
\begin{equation}
\label{correct-2d}
\Delta^{*}_{e_K}= \begin{bmatrix*}[l]
 0&0\\
  g \eta_{h,K}^{int}|_{e_K}\big(h_{h,K}^{int}|_{e_K}-h_{h,K}^{*,int}|_{e_K}\big) &0\\
   0& g \eta_{h,K}^{int}|_{e_K}\big(h_{h,K}^{int}|_{e_K}-h_{h,K}^{*,int}|_{e_K}\big)\\
 \end{bmatrix*}.
\end{equation}
The correction term is chosen to satisfy (\ref{upwind-prpty}) below.

\begin{lemma}\label{Hstar2Hint}
For the lake-at-rest steady state, the numerical flux $\hat{ \mathbf{H}}^{*}|_{e_K}$ defined in \eqref{eK-flux} and \eqref{correct-2d} satisfies
\begin{equation}
\label{upwind-prpty}
\hat{ \mathbf{H}}^{*}|_{e_K} = \mathbf{H}\big( U_{h,K}^{int},h_{h,K}^{int}\big)\cdot \bm{n}_K^e.
\end{equation}
\end{lemma}

\begin{proof}
For the lake-at-rest steady state $(\eta_h,m_h,w_h) =(C,0,0)$, where $C$ is a constant,
we have
\begin{align*}
&\eta_{h,K}^{int}|_{e_K}=\eta_{h,K}^{ext}|_{e_K}=C, \quad
m_{h,K}^{int}|_{e_K}= m_{h,K}^{ext}|_{e_K}=0,
\quad w_{h,K}^{int}|_{e_K}= w_{h,K}^{ext}|_{e_K}=0.
\end{align*}
From the definition \eqref{reconst-2d-h} and \eqref{reconst-2d-U}, we have
\begin{align*}
\label{pre-consi}
& h_{h,K}^{*,int}|_{e_K}=h_{h,K}^{*,ext}|_{e_K},\quad U_{h,K}^{*,int}|_{e_K}=U_{h,K}^{*,ext}|_{e_K},
\\
& \eta_{h,K}^{*,int}|_{e_K}=\eta_{h,K}^{*,ext}|_{e_K}=C, \quad m_{h,K}^{*, int}|_{e_K}=m_{h,K}^{*, ext}|_{e_K}=0, \quad w_{h,K}^{*, int}|_{e_K} = w_{h,K}^{*, ext}|_{e_K} = 0.
\end{align*}
Using these and the consistency of the numerical flux, from (\ref{swe-2d-FUh}) and \eqref{eK-flux} we have
\begin{equation*}
\begin{split}
\hat{ \mathbf{H}}^{*}|_{e_K}
&=
\hat{ \mathbf{H}}(U_{h,K}^{*,int},h_{h,K}^{*,int};U_{h,K}^{*,ext},h_{h,K}^{*,ext};\bm{n}^e_K)+
\Delta^{*}_{e_K}\cdot \bm{n}_K^e
\\&=\mathbf{H}\big(U_{h,K}^{*,int},h_{h,K}^{*,int}\big)\cdot \bm{n}_K^e+
\Delta^{*}_{e_K}\cdot \bm{n}_K^e
\\&= \begin{bmatrix*}[l]
   -\big(\dot{X}\eta_{h,K}^{*,int}\big)\big|_{e_K} &-\big(\dot{Y}\eta_{h,K}^{*,int}\big)\big|_{e_K}  \\
  \frac{g}{2}\Big (2\eta_{h,K}^{*,int}h_{h,K}^{*,int}-\big(\eta_{h,K}^{*,int}\big)^2\Big)\Big|_{e_K}&0\\
  0&  \frac{g}{2}\Big (2\eta_{h,K}^{*,int}h_{h,K}^{*,int}-\big(\eta_{h,K}^{*,int}\big)^2\Big)\Big|_{e_K}
 \end{bmatrix*}\cdot \bm{n}_K^e
\\&+
 \begin{bmatrix*}[l]
 0&0\\
  \big(g \eta_{h,K}^{int}(h_{h,K}^{int}-h_{h,K}^{*,int})\big)\big|_{e_K} &0\\
   0& \big(g \eta_{h,K}^{int}(h_{h,K}^{int}-h_{h,K}^{*,int})\big)\big|_{e_K}\\
 \end{bmatrix*}\cdot \bm{n}_K^e
\\&=
 \begin{bmatrix*}[l]
   -\big(\dot{X}\eta_{h,K}^{int}\big)\big|_{e_K} &-\big(\dot{Y}\eta_{h,K}^{int}\big)\big|_{e_K}  \\
  \frac{g}{2}\Big(2\eta_{h,K}^{int}h_{h,K}^{int}-\big(\eta_{h,K}^{int}\big)^2\Big)\Big|_{e_K}&0\\
  0&  \frac{g}{2}\Big (2\eta_{h,K}^{int}h_{h,K}^{int}-\big(\eta_{h,K}^{int}\big)^2\Big)\Big|_{e_K}
 \end{bmatrix*}
 \cdot \bm{n}_K^e
\\&
=
\mathbf{H}\big(U_{h,K}^{int},h_{h,K}^{int}\big)\cdot\bm{n}_K^e.
\end{split}
\end{equation*}
\hfill \end{proof}

\vspace{10pt}

Replacing $\hat{ \mathbf{H}}|_{e_K}$ by $\hat{ \mathbf{H}}^{*}|_{e_K}$ in \eqref{semi-QLMMDG-2d},
we obtain the semi-discrete QLMM-DG scheme, i.e., to find $U_h \in \mathcal{V}^{k}_h(t)$ such that
\begin{equation}
\label{wb-semi-QLMMDG-2d}
\begin{split}
&\frac{d}{d t}\int_{K}U_h \phi d\bm{x}
-\int_{K} \mathbf{H}(U_h, h_h)\cdot\nabla \phi d\bm{x}
+\sum_{e_K\in \partial K }\int_{e_K} \phi \hat{ \mathbf{H}}^{*}|_{e_K} ds
\\&~~~~~~~~~~~~~~~~~~~~~~~~
 = \int_{K}S(\eta_h,B_h)\phi d\bm{x}, \quad \forall \phi \in \mathcal{V}^{k}_h(t),\quad \forall  K\in \mathcal{T}_h(t)
\end{split}
\end{equation}
where $h_h = \eta_h - B_h$.
Denote the residual for this scheme as
\begin{equation}\label{wb-residual-2d}
\begin{split}
R^{*}_{h,K}(t)=&R^{*}(U_h,h_h,\phi,B_h)|_{K}
\\=& \int_{K}S(\eta_h,B_h)\phi d\bm{x}
+\int_{K} \mathbf{H}(U_h,h_h)\nabla \phi d\bm{x}
\\&- \sum_{e_K\in \partial K}\int_{e_K} \phi \hat{ \mathbf{H}}^{*}|_{e_K} ds
- \int_{K}\phi\nabla \cdot( U_h \dot{\bm{X}} ) d\bm{x}.
\end{split}
\end{equation}

\begin{prop}
\label{Rk0-prop}
If all integrals in \eqref{wb-residual-2d} are computed exactly (with suitable Gaussian
quadrature rules), the residual of the semi-discrete QLMM-DG scheme \eqref{wb-semi-QLMMDG-2d}
with the numerical flux \eqref{eK-flux} and \eqref{correct-2d} vanishes
for the lake-at-rest steady state and thus the scheme \eqref{wb-semi-QLMMDG-2d} is well-balanced.
\end{prop}

\begin{proof}
For the lake-at-rest steady state $(\eta_h,m_h,w_h) =(C,0,0)$, using the Lemma \ref{Hstar2Hint},
the divergence theorem, and the definition \eqref{H-def}, we have, for any $K\in \mathcal{T}_h(t)$
and any $\phi \in \mathcal{V}^{k}_h(t)$,
\[
\label{we-residue-zero}
\begin{split}
R^{*}_{h,K}(t)
= &\int_{K}S(\eta_h,B_h)\phi d\bm{x}- \int_{K}\phi\nabla \cdot( U_h \dot{\bm{X}} ) d\bm{x}
\\&
- \sum_{e_K\in \partial K}\int_{e_K} \phi \hat{ \mathbf{H}}^{*}|_{e_K} ds+\int_{K} \mathbf{H}(U_h,h_h)\cdot\nabla \phi d\bm{x}
\\
=&\int_{K}S(\eta_h,B_h)\phi d\bm{x}- \int_{K}\phi\nabla \cdot( U_h \dot{\bm{X}} ) d\bm{x}
\\&- \sum_{e_K\in \partial K}\int_{e_K} \phi \mathbf{H}\big(U_{h,K}^{int},h_{h,K}^{int}\big)\cdot\bm{n}_K^e ds
+\int_{K} \mathbf{H}(U_h,h_h)\cdot\nabla \phi d\bm{x}
\\
=&\int_{K}S(\eta_h,B_h)\phi d\bm{x}
-\int_{K}\phi \nabla \cdot \big(\mathbf{H}(U_h,h_h)+ U_h \dot{\bm{X}}\big)d\bm{x}
\\
=&\int_{K} \big(S(\eta_h,B_h)-\nabla \cdot \mathbf{F}(U_h,h_h)\big)\phi d\bm{x}
=0 .
\end{split}
\]
\hfill \end{proof}

\vspace{10pt}

To see the convergence order of the scheme, we rewrite \eqref{wb-semi-QLMMDG-2d} into
\begin{align}
& \frac{d}{d t}\int_{K}U_h \phi d\bm{x}
 +\sum_{e_K \in \partial K}\int_{e_K} \phi \hat{ \mathbf{H}}|_{e_K} ds
-\int_{K} \mathbf{H}(U_h,h_h)\cdot \nabla \phi d\bm{x}
\notag
\\
& \qquad\qquad = \int_{K}S(h_h,B_h)\phi d\bm{x}
+\sum_{e_K \in \partial K}\int_{e_K} \phi \big(\hat{ \mathbf{H}}^{*}- \hat{ \mathbf{H}}\big)|_{e_K} ds, \quad \forall \phi \in \mathcal{V}^{k}_h(t).
 \label{well-semi-QLMMDG-2d-2}
\end{align}
This is a standard DG scheme for \eqref{swe-2d-another} on a moving mesh with a correction term (the last term). Notice that
\begin{equation*}
\begin{split}
\big|\eta_{h,K}^{int}|_{e_K} \big(h_{h,K}^{int}|_{e_K}-h_{h,K}^{*,int}|_{e_K}\big)\big|
&\leq\big|\eta_{h,K}^{int}|_{e_K}\big|\cdot\big|\big(h_{h,K}^{int}-h_{h,K}^{*,int}\big)|_{e_K}\big|
=\mathcal{O}(a^{k+1}_{max}),
\end{split}
\end{equation*}
where $a_{max}$ denotes the maximum element diameter of the mesh. This gives
$\hat{ \mathbf{H}}^{*}- \hat{ \mathbf{H}}=\mathcal{O}(a^{k+1}_{max})$.
Thus, the scheme \eqref{wb-semi-QLMMDG-2d} is $(k+1)$-th-order in space.

\begin{rem}\label{another-form}
{\em
As mentioned earlier, we can replace all of $h$'s in the flux by $\eta-B$. This gives the system \cite{Kurganov-Levy-2002}
\begin{equation}\label{swe-2d-another-2}
U_t +  \nabla \cdot \mathbb{F}(U,B)  = S(\eta,B),
\end{equation}
where
\begin{equation}\label{swe-2d-etamw-flux}
\begin{split}
&\mathbb{F}(U,B) =
 \begin{bmatrix*}[l]
  m&w\\
  \frac{m^2}{\eta-B}+\frac{1}{2}g(\eta^2-2 \eta B ) &\frac{mw}{\eta-B}\\
  \frac{mw}{\eta-B}                                 &\frac{w^2}{\eta-B}+\frac{1}{2}g(\eta^2-2 \eta B )\\
 \end{bmatrix*}.
\end{split}
\end{equation}
Applying the same procedure to this system, we can obtain a QLMM-DG scheme.
For the lake-at-rest steady state, for this scheme we have
\begin{align*}
&\eta_{h,K}^{int}|_{e_K}=\eta_{h,K}^{ext}|_{e_K}= C,
\quad m_{h,K}^{int}|_{e_K}= m_{h,K}^{ext}|_{e_K}=0,
\quad w_{h,K}^{int}|_{e_K}= w_{h,K}^{ext}|_{e_K}=0,
\end{align*}
and $U_{h,K}^{int}|_{e_K}=U_{h,K}^{ext}|_{e_K}$.
However, since $\mathbb{F}(U,B)$ involves $B$ explicitly and since
$B_{h,K}^{int}|_{e_K}$ is not equal to $B_{h,K}^{ext}|_{e_K}$ in general, it is unclear
how the numerical flux can be modified to ensure (\ref{upwind-prpty}).
Hence, it remains unknown if a well-balanced QLMM-DG scheme can be developed based on (\ref{swe-2d-another-2}).
}
\end{rem}

\subsection{The fully discrete well-balanced QLMM-DG scheme}

We consider the third-order explicit total variation diminishing (TVD) Runge-Kutta scheme to discretize \eqref{wb-semi-QLMMDG-2d} in time.
For notational simplicity, we denote
\begin{equation}\label{L-Uh}
\begin{split}
\mathcal{L}(U_h,h_h,\phi,B_h)|_{K} & =
-\sum_{e_K\in \partial K}\int_{e_K} \phi\hat{ \mathbf{H}}^{*}|_{e_K} ds
+\int_{K} \mathbf{H}(U_h,h_h)\cdot\nabla \phi d\bm{x} \\
& \qquad
 +\int_{K}S(\eta_h,B_h)\phi d\bm{x}.
\end{split}
\end{equation}
We can rewrite \eqref{wb-semi-QLMMDG-2d} into a compact form as
\begin{equation}
\label{QLMMDG-ode}
\frac{d}{d t}\int_{K}U_h \phi d\bm{x}=\mathcal{L}(U_h,h_h,\phi,B_h)|_{K}, \quad \forall \phi \in \mathcal{V}_h^k(t).
\end{equation}
Applying the third-order explicit TVD Runge-Kutta scheme to the above equation, we obtain
the fully-discrete QLMM-DG scheme as
\begin{equation}
\label{RK3-QLMMDG}
\begin{cases}
\begin{split}
\int_{K^{n,(1)}} U_h^{n,(1)}\phi^{n,(1)}d\bm{x}
&=\int_{K^{n}} U_h^n\phi^n d\bm{x}
+\Delta t_n \mathcal{L}(U_h^n,h_h^n,\phi^n,B_h^n)|_{K^n},
\\
\int_{K^{n,(2)}} U_h^{n,(2)}\phi^{n,(2)}d\bm{x}
&=\frac{3}{4}\int_{K^{n}} U_h^n\phi^{n}d\bm{x}+\frac{1}{4}\int_{K^{n,(1)}}U_h^{n,(1)}\phi^{n,(1)} d\bm{x}
\\&~~~~~~~~~~~~~
+\frac{1}{4}\Delta t_n \mathcal{L}(U_h^{n,(1)},h_h^{n,(1)},\phi^{n,(1)},B^{n,(1)}_h)|_{K^{n,(1)}},
\\
\int_{K^{n+1}}U_h^{n+1}\phi^{n+1}d\bm{x}
&= \frac{1}{3}\int_{K^{n}} U_h^n\phi^n d\bm{x}+\frac{2}{3}\int_{K^{n,(2)}} U_h^{n,(2)}\phi^{n,(2)}d\bm{x}
\\&~~~~~~~~~~~~~
+\frac{2}{3}\Delta t_n \mathcal{L}(U_h^{(2)},h_h^{n,(2)},\phi^{n,(2)},B^{n,(2)}_h)|_{K^{n,(2)}},
\end{split}
\end{cases}
\end{equation}
where
$U_h^{n,(1)}$, $h_h^{n,(1)}$, $\phi^{n,(1)}$, $B_h^{n,(1)}$, $K^{n,(1)}$ are stage values
at $t = t_n + \Delta t_n$, $U_h^{n,(2)}$, $h_h^{n,(2)}$, $\phi^{n,(2)}$, $B_h^{n,(2)}$, $K^{n,(2)}$ are the values
at $t = t_n + \frac{1}{2} \Delta t_n$, and
$U_h^{n+1}$, $h_h^{n+1}$, $\phi^{n+1}$, $K^{n+1}$ are at $t = t_n + \Delta t_n$.

We now make a few remarks on the above scheme.
We first note that $h_h$ is updated by
\[
h_h^{n,(1)} = \eta_h^{n,(1)} - B_h^{n,(1)},
\quad h_h^{n,(2)} = \eta_h^{n,(2)} - B_h^{n,(2)},  \quad h_h^{n+1} = \eta_h^{n+1} - B_h^{n+1}.
\]

Second, let $\hat{K}$ be the reference element and
$\hat{\phi} = \hat{\phi}(\bm{\xi})$ be an arbitrary basis function.
Then, the test functions in (\ref{RK3-QLMMDG})
corresponding to this basis function are related by
\begin{align*}
& \phi^{n}(\bm{x}) = \hat{\phi}(F^{-1}_{K^{n}}(\bm{x})),\quad
\phi^{n,(1)}(\bm{x}) = \hat{\phi}(F^{-1}_{K^{n,(1)}}(\bm{x})), \\
& \phi^{n,(2)}(\bm{x}) = \hat{\phi}(F^{-1}_{K^{n,(2)}}(\bm{x})),\quad
\phi^{n+1}(\bm{x}) = \hat{\phi}(F^{-1}_{K^{n+1}}(\bm{x})),
\end{align*}
where $F^{-1}_K$ is the inverse of the affine mapping $F_K: \hat{K} \to K$
with $K$ being $K^{n}$, $K^{n,(1)}$, $K^{n,(2)}$, or $K^{n+1}$.

Third, the area of $K^{n,(1)}$, $K^{n,(2)}$ and $K^{n+1}$ is needed in the computation
of the integrals in (\ref{RK3-QLMMDG}). It can be calculated using
the coordinates of the vertices of the elements. But this does not preserve the so-called
geometric conservation law (GCL) that is a geometric identity in the continuous setting.
Using the Reynolds transport theorem and the divergence theorem,
we can find the GCL as
\begin{equation}
\label{GCL-0}
\frac{d }{d t} \int_{K}  d \bm{x} = \int_{\partial K} \dot{\bm{X}} \cdot \bm{n} d s
\quad \mbox{ or } \quad
\frac{d }{d t} \int_{K}  d \bm{x} = \int_{K} \nabla \cdot \dot{\bm{X}} d \bm{x}.
\end{equation}
Since $\dot{\bm{X}}$ is a linear function in $K$ and $\nabla \cdot \dot{\bm{X}}$ is constant, we get
\begin{equation}
\label{GCL-c}
\frac{d}{d t}|K| =  |K| \nabla \cdot \dot{\bm{X}} |_{K}.
\end{equation}
Applying the third-order Runge-Kutta scheme to the above equation, we have
\begin{equation}
\label{GCL-d}
\begin{cases}
\begin{split}
|K^{n,(1)}|
&= |K^{n}|+\Delta t_n |K^{n}|\nabla\cdot\dot{\bm{X}}^{n}_{K^{n}},
\\
|K^{n,(2)}|
&= \frac{3}{4}|K^{n}|
+\frac{1}{4}\big{(}|K^{n,(1)}|
+\Delta t_n |K^{n,(1)}|\nabla\cdot\dot{\bm{X}}^{n,(1)}_{K^{n,(1)}}\big{)},
\\
|K^{n+1}|
&= \frac{1}{3}|K^{n}|
+\frac{2}{3}\big{ (} |K^{n,(2)}|
+\Delta t_n |K^{n,(2)}|\nabla\cdot\dot{\bm{X}}^{n,(2)}_{K^{n,(2)}}\big{)} .
\end{split}
\end{cases}
\end{equation}
Thus, the area of $K^{n,(1)}$, $K^{n,(2)}$ and $K^{n+1}$ can be updated using this equation.
We note that a factor involving the area of the element appears in the computation of
$\nabla\cdot\dot{\bm{X}}^{n,(1)}_{K^{n,(1)}}$ and $\nabla\cdot\dot{\bm{X}}^{n,(2)}_{K^{n,(2)}}$
and this factor should be computed using $|K^{n,(1)}|$ and $|K^{n,(2)}|$
(i.e., the values obtained through the above equation instead of those directly computed using
the coordinates of the element vertices), respectively.
The preservation of GCL has been studied extensively in the context of moving mesh computation; e.g.,
see Trulio and Trigger \cite{Trulio-Trigger-1961} and Thomas and Lombard \cite{Thomas-Lombard-1979}.
As will be seen in Proposition~\ref{wb-RK3-QLMMDG-prop}, updating the area of elements
using (\ref{GCL-d}) is an important step for the QLMM-DG scheme (\ref{RK3-QLMMDG}) to be well-balanced.

It is interesting to point out that $|K^{n+1}|$ calculated through (\ref{GCL-d})
is the same as that directly computed using the vertex coordinates of $K^{n+1}$.
(The other stage values $|K^{n,(1)}|$ and $|K^{n,(2)}|$ are different from their counterparts in general.)
Indeed, this property holds for the third-order Runge-Kutta scheme in one, two, and three dimensions.
The validity of this property depends on the time integration scheme used and the dimensionality of the space.
For example, it holds only in one dimension when the forward Euler scheme is used.
In case when $|K^{n+1}|$ calculated through the GCL update is not equal to that
computed using the vertex coordinates, we suggest to start the GCL update with
$|K^n|$ calculated from the vertex coordinates. This does not affect the satisfaction of GCL.

Fourth, we emphasize that the update of $B$ does not affect the well-balance property
of the QLMM-DG scheme. For this reason,
in our computation we use $L^2$-projection to compute $B_h^{n,(1)}$, $B_h^{n,(2)}$, and
$B_h^{n+1}$. Since $B$ is a given function, $L^2$-projection is straightforward and economic
to implement.

It is interesting to point out that the requirements for how $B$ is updated are different
in the current QLMM-DG method and the rezoning-type moving mesh DG method of
\cite{Zhang-Huang-Qiu-2020arXiv}. The latter requires that the same scheme be used
to update both the bottom topography and flow variables. It is shown
there that a DG-interpolation scheme
works out well for this purpose but $L^2$-projection may be difficult to use. This is because, when
it is used for $B$, then $L^2$-projection needs to be used for the flow variables as well.
Since only numerical approximations are available for the flow variables, their $L^2$-projection from the old
mesh to the new one requires finding the intersection between elements in the new and old meshes and
performing numerical integration thereon, which is known to be a difficult, if not impossible, task in programming.

Fifth, the DG solution of the SWEs may contain spurious oscillations and even nonlinear instability. We need to apply a nonlinear limiter after each Runge-Kutta stage to aviod those spurious oscillations.
However, caution must be taken since this limiting procedure can destroy the well-balance property.
Following \cite{Audusse-etal-2004Siam,Xing-Shu-2006CiCP,Zhou-etal-2001JCP},
we use the TVB limiter \cite{DG-series2,DG-series3,DG-series5} for the local characteristic variables based on
the variables $\big(\eta,m, w\big)$.
This procedure is known to preserve the lake-at-rest steady state and conserve the cell averages.

Sixth, another challenge in the numerical solution of the SWEs is to preserve the nonnegativity of
the water depth $h$ in the computation.
Following \cite{Xing-Zhang-Shu-2010,Xing-Zhang-2013JSC}, we can show that, after each Runge-Kutta
stage of the scheme (\ref{RK3-QLMMDG}), the cell averages of the current approximation of $h$
are nonnegative if the cell averages and the function values of the previous approximation of $h$
at a set of special quadrature points (Gauss-Lobatto quadrature points in one dimension) \cite{Xing-Zhang-2013JSC}
for each mesh element are nonnegative. Since the TVB limiter preserves the cell averages,
we can use the linear scaling PP limiter \cite{Liu-Osher1996,ZhangShu2010,ZhangXiaShu2012}
to ensure the nonnegativity of $h$ after each application of the TVB limiter.

However, the PP limiter destroys the well-balance property.
To restore the property, we make a high-order correction to the current approximation of the bottom topography
according to the modifications in the water depth due to the PP limiting, i.e.,
\begin{equation}
\label{B-update-2}
\hat{B}_h = B_h  - (\hat{h}_h - h_h),
\end{equation}
where $\hat{h}_h$ denotes the modification of $h_h$ by the PP limiter. It is known that \cite{Liu-Osher1996,ZhangShu2010,ZhangXiaShu2012} this PP limiter maintains the cell averages
and high-order accuracy, i.e., $\overline{\hat{h}_K} = \overline{h_K}$, for all elements $K$
and $\hat{h}_h - h_h =\mathcal{O}(a_{max}^{k+1})$, where $a_{max}$ is the maximum element diameter
of the mesh. Thus, $\hat{B}_h$ has the same cell averages as $B_h$.

It is worth pointing out that the above trick has been used successfully in the rezoning-type
moving mesh DG method \cite{Zhang-Huang-Qiu-2020arXiv} to restore the well-balance property.

Finally, to ensure the stability of the method, the time step for (\ref{RK3-QLMMDG}) is chosen subject to
the Courant-Friedrichs-Lewy (CFL) condition \cite{DG-review}.
For a fixed mesh, the time step is taken as
\begin{equation}\label{FMDG-cfl}
\Delta t_{n,1} \leq \frac{C_{cfl}}{\max\limits_{K, e}\Big{(} \max\limits_m |\tilde{\lambda}^{m}(U_h^n,h_h^n)|\Big{)}}\cdot a^{n}_{min} ,
\end{equation}
where $C_{cfl}$ is a constant typically chosen to be less than $1/(2k+1)$,
$a^{n}_{min}$ is the minimum height of the elements of $\mathcal{T}_h^n$, and
$\tilde{\lambda}^{m}(U,h)$, $m = 1,\, 2,\, 3$ denote the eigenvalues of $\mathbf{F}'(U,h)\cdot \bm{n}_K$, i.e.,
\[
\label{eigen_F}
\tilde{\lambda}^{1}(U,h)=un_x +vn_y -c,\quad
\tilde{\lambda}^{2}(U,h)=un_x +vn_y ,
\quad
\tilde{\lambda}^{3}(U,h)=un_x +vn_y +c .
\]
For a moving mesh, we need to consider the extra convection term caused by mesh movement
and thus take the time step as
\begin{equation}\label{MMDG-cfl}
\Delta t_{n,2} \leq \frac{C_{cfl}}{\max\limits_{K, e}\Big{(} \max\limits_m |\lambda^{m}(U_h^n,h_h^n,\dot{\bm{X}}^n)|\Big{)}}\cdot \min\big{(} a^{n}_{min},a^{n+1}_{min}) ,
\end{equation}
where $\lambda^{m},~m=1,2,3$ are defined in \eqref{eigen-H}.
Finally, we take $\Delta t_n = \min\big( \Delta t_{n,1}, \Delta t_{n,2} \big)$.

Next, we show that the fully discrete QLMM-DG scheme \eqref{RK3-QLMMDG}
is well-balanced in the following proposition.

\begin{prop}
\label{wb-RK3-QLMMDG-prop}
If the area of mesh elements is updated according to \eqref{GCL-d} and
all integrals in \eqref{L-Uh} are computed exactly,
then the fully discrete QLMM-DG scheme \eqref{RK3-QLMMDG} preserves
the lake-at-rest steady-state solutions,
i.e., $\eta^n_h= C$, $m^n_h=0$, and $w^n_h=0$ imply $\eta^{n+1}_h = C$, $m^{n+1}_h=0$, and $w^{n+1}_h=0$,
where $C$ is a constant.
\end{prop}

\begin{proof}
Recall that $R^{*}(U_h,h_h,\phi,B_h)|_{K}$ vanishes for the lake-at-rest steady state.
Comparing the expressions of $\mathcal{L}(U_h,h_h,\phi,B_h)|_{K}$ in \eqref{L-Uh} and $R^{*}(U_h,h_h,\phi,B_h)|_{K}$ in \eqref{wb-residual-2d},
we have
\[
\label{relation-1}
\mathcal{L}(U_h,h_h,\phi,B_h)|_{K} = R^{*}(U_h,h_h,\phi,B_h)|_{K}
+ \int_{K}\phi\nabla \cdot( U_h \dot{\bm{X}} ) d\bm{x}.
\]
Then \eqref{RK3-QLMMDG} becomes
\[
\label{RK3-QLMMDG-s}
\begin{split}
\int_{K^{n,(1)}} U_h^{n,(1)}\phi^{n,(1)}d\bm{x}
&=\int_{K^{n}} U_h^n\phi^n d\bm{x}
+\Delta t_n \int_{K^n}\phi^{n}\nabla \cdot\big( U^n_h \dot{\bm{X}}^n \big) d\bm{x},
\\
\int_{K^{n,(2)}} U_h^{n,(2)}\phi^{n,(2)}d\bm{x}
&=\frac{3}{4}\int_{K^{n}} U_h^n\phi^{n}d\bm{x}+\frac{1}{4}\int_{K^{n,(1)}}U_h^{n,(1)}\phi^{n,(1)} d\bm{x}
\\&~~~~~~~~~~~~~~~~~~~~~~~
+\frac{1}{4}\Delta t_n\int_{K^{n,(1)}}\phi^{n,(1)}\nabla \cdot\big( U^{n,(1)}_h \dot{\bm{X}}^{n,(1)} \big) d\bm{x},
\\
\int_{K^{n+1}}U_h^{n+1}\phi^{n+1}d\bm{x}
&= \frac{1}{3}\int_{K^{n}} U_h^n\phi^n d\bm{x}+\frac{2}{3}\int_{K^{n,(2)}} U_h^{n,(2)}\phi^{n,(2)}d\bm{x}
\\&~~~~~~~~~~~~~~~~~~~~~~~
+\frac{2}{3}\Delta t_n\int_{K^{n,(2)}}\phi^{n,(2)}\nabla \cdot\big( U^{n,(2)}_h \dot{\bm{X}}^{n,(2)} \big) d\bm{x}.
\end{split}
\]
It is not difficult to show from the above equations that $m^{n+1}_h=0$ and $w^{n+1}_h=0$ if $m^n_h=0$ and $w^n_h=0$.
We rewrite the first component of the above equations as
\begin{equation}
\label{RK3-proof-eta}
\begin{split}
\int_{K^{n,(1)}} \eta_h^{n,(1)}\phi^{n,(1)}d\bm{x}
=&~~\int_{K^{n}} \eta_h^n\phi^n d\bm{x}
+\Delta t_n \int_{K^n}\phi^{n}\nabla \cdot\big( \eta^n_h \dot{\bm{X}}^n \big) d\bm{x},
\\
\int_{K^{n,(2)}} \eta_h^{n,(2)}\phi^{n,(2)}d\bm{x}
=& \frac{3}{4}\int_{K^{n}} \eta_h^n\phi^n d\bm{x}
+\frac{1}{4}\int_{K^{n,(1)}}\eta_h^{n,(1)}\phi^{n,(1)} d\bm{x}
\\&~~~~~~~~~~~~~~~~~~
+\frac{1}{4}\Delta t_n \int_{K^{n,(1)}}\phi^{n,(1)}\nabla \cdot\big( \eta^{n,(1)}_h \dot{\bm{X}}^{n,(1)} \big) d\bm{x},
\\
\int_{K^{n+1}} \eta_h^{n+1}\phi^{n+1}d\bm{x}
=& \frac{1}{3}\int_{K^{n}} \eta_h^n\phi^n d\bm{x}
+\frac{2}{3}\int_{K^{n,(2)}}\eta_h^{n,(2)}\phi^{n,(2)} d\bm{x}
\\&~~~~~~~~~~~~~~~~~~
+\frac{2}{3}\Delta t_n \int_{K^{n,(2)}}\phi^{n,(2)}\nabla \cdot \big( \eta^{n,(2)}_h \dot{\bm{X}}^{n,(2)} \big) d\bm{x}.
\end{split}
\end{equation}
Taking $\eta^n_h = C$ in the first equation of \eqref{RK3-proof-eta}, changing independent variables, we get
\[
\label{WB-eta-1}
|K^{n,(1)}|\int_{\hat{K}}\eta^{n,(1)}\hat{\phi} d\bm{\xi}
=|K^{n}|\int_{\hat{K}}C\hat{\phi} d\bm{\xi}
+\Delta t_n|K^{n}|\nabla\cdot \dot{\bm{X}}|_{K^{n}} \int_{\hat{K}}C\hat{\phi}d\bm{\xi}.
\]
From the first equation of the discrete GCL \eqref{GCL-d}, we have
\begin{equation*}
|K^{n,(1)}|\int_{\hat{K}}\big{(}\eta^{n,(1)}-C\big{)} \hat{\phi}d\bm{\xi} = 0
\quad \hbox{or}\quad
\int_{K^{n,(1)}}\big{(}\eta^{n,(1)}-C\big{)} \phi^{n,(1)}d\bm{x}=0.
\end{equation*}
From the arbitrariness of $\phi$ and $K^{n,(1)}$, this implies $\eta^{n,(1)} \equiv C$
on $\mathcal{D}$.

Similarly, we can show $\eta_h^{n,(2)}\equiv C$ and $\eta_h^{n+1}\equiv C$ on $\mathcal{D}$.
\hfill \end{proof}

\vspace{10pt}

To conclude this section, we summarize the procedure of the well-balanced QLMM-DG method in Algorithm~\ref{QLMM-DG}.
\begin{algorithm}
\caption{The well-balanced QLMM-DG method for the SWEs on moving meshes.}\label{QLMM-DG}
\begin{itemize}
\item[0.] {\bf Initialization.}
 Project the initial physical variables and bottom topography into the DG space $\mathcal{V}_h^{k,0}$ to obtain $U^0_h =  (\eta^0_h,m^0_h,w^0_h)^T$ and $B_h^0$. For $n = 0, 1, ...$, do

\item[1.] {\bf Mesh adaptation.}
Generate the new mesh $\mathcal{T}_h^{n+1}$ using the MMPDE moving mesh method (cf. Appendix \ref{sec:mmpde}).

\item[2.] {\bf Solution of the SWEs on the moving mesh.}
Integrate the SWEs from $t_n$ to $t_{n+1}$ using the QLMM-DG scheme \eqref{RK3-QLMMDG} to obtain $U^{n+1}_h = (\eta_h^{n+1},m^{n+1}_h,w^{n+1}_h)^T$.
\begin{itemize}
\item[2(a).] At each of the Runge-Kutta stage, we update $B$.
Compute $B_h^{n,(1)}$, $B_h^{n,(2)}$, and $B_h^{n+1}$ using $L^2$-projection on the corresponding meshes.
\item[2(b).] After each of the Runge-Kutta stage, we apply the TVB limiter for the local characteristic variables based on
the variables $\big(\eta,m, w\big)$.
\item[2(c).] After the TVB limiter, we apply the linear scaling PP limiter to $h_h$, followed by
the correction (\ref{B-update-2}) to $B_h$.
\end{itemize}
\end{itemize}
\end{algorithm}

\section{The well-balance property of the QLMM-DG\\ scheme in the original variables}
\label{sec:cswe-QLMMDG}

For comparison purpose, in this section we discuss the well-balance property of
the QLMM-DG scheme developed based on the SWEs (\ref{cswe-2d}) in the original variables.
The same moving mesh DG procedure and hydrostatic reconstruction technique described
in the previous section can be applied to \eqref{cswe-2d}.
This leads to a semi-discrete well-balanced QLMM-DG scheme but unfortunately, its fully
discrete version is not well-balanced.

Specifically, the semi-discrete QLMM-DG scheme based on \eqref{cswe-2d} is to find the solution $V_h \in \mathcal{V}^{k}_h(t)$ such that
\begin{equation}
\label{WB-semi-QLMMDG-cswe}
\begin{split}
&\frac{d}{d t}\int_{K}V_h \phi d\bm{x}
 +\sum_{e_K\in \partial K }\int_{e_K} \phi \hat{ \mathcal{H}}^{*}|_{e_K} ds
 -\int_{K} \mathcal{H}(V_h)\cdot\nabla \phi d\bm{x}
\\&~~~~~~~~~~~~~~~~~
 = \int_{K}\mathcal{S}(h_h,B_h)\phi d\bm{x}, \quad \forall \phi \in \mathcal{V}^{k}_h(t),\quad \forall  K\in \mathcal{T}_h(t),
\end{split}
\end{equation}
where $\mathcal{H}(V):=\mathcal{F}(V)- V \dot{\bm{X}}$.  The modified numerical
flux $\hat{ \mathcal{H}}^{*}$ for the edge $e_K$ is defined as
\begin{align}
\label{eK-flux-cswe}
&\hat{ \mathcal{H}}^{*}|_{e_K} =
\hat{ \mathcal{H}}(V_{h,K}^{*,int}, V_{h,K}^{*,ext},\bm{n}^e_K)+
\Delta^{*}_{e_K}\cdot \bm{n}_K^e,
\end{align}
where
\begin{equation}\label{reconst-h-cswe}
\begin{cases}
\begin{split}
&h_{h,K}^{*,int}|_{e_K} =
\max \Big( 0, h_{h,K}^{int}|_{e_K}+B_{h,K}^{int}|_{e_K}
             -\max\big(B_{h,K}^{int}|_{e_K},B_{h,K}^{ext}|_{e_K}\big)
             \Big),
\\&
h_{h,K}^{*,ext}|_{e_K} =
\max \Big( 0, h_{h,K}^{ext}|_{e_K}+B_{h,K}^{ext}|_{e_K}
             -\max\big(B_{h,K}^{int}|_{e_K},B_{h,K}^{ext}|_{e_K}\big)
             \Big),
\end{split}
\end{cases}
\end{equation}
\begin{equation}
\begin{split}
\label{red-V-2d}
&V_{h,K}^{*,int}|_{e_K}= \begin{bmatrix*}[l]
     ~~~~~~~~h_{h,K}^{*,int}|_{e_K}    \\
     \Big(\frac{h_{h,K}^{*,int}}{h_{h,K}^{int}}m_{h,K}^{int}\Big)\Big|_{e_K}\\
     \Big(\frac{h_{h,K}^{*,int}}{h_{h,K}^{int}}w_{h,K}^{int}\Big)\Big|_{e_K}\\
 \end{bmatrix*},
\quad \quad
V_{h,K}^{*,ext}|_{e_K}= \begin{bmatrix*}[l]
     ~~~~~~~~h_{h,K}^{*,ext}|_{e_K}    \\
     \Big(\frac{h_{h,K}^{*,ext}}{h_{h,K}^{ext}}m_{h,K}^{int}\Big)\Big|_{e_K}\\
     \Big(\frac{h_{h,K}^{*,ext}}{h_{h,K}^{ext}}w_{h,K}^{int}\Big)\Big|_{e_K}\\
 \end{bmatrix*},
\end{split}
\end{equation}
\begin{equation}
\label{correct-2d-cswe}
\Delta^{*}_{e_K}= \begin{bmatrix*}[l]
 \dot{X}|_{e_K}h_{h,K}^{*,int}|_{e_K} -\dot{X}|_{e_K}h_{h,K}^{int}|_{e_K} &
 \dot{Y}|_{e_K}h_{h,K}^{*,int}|_{e_K} -\dot{Y}|_{e_K}h_{h,K}^{int}|_{e_K}  \\
   \frac{g}{2} (h_{h,K}^{int}|_{e_K})^2-\frac{g}{2} (h_{h,K}^{*,int}|_{e_K})^2 &0\\
   0&\frac{g}{2} (h_{h,K}^{int}|_{e_K})^2-\frac{g}{2} (h_{h,K}^{*,int}|_{e_K})^2\\
 \end{bmatrix*} .
\end{equation}
It can be verified that
\begin{equation}
\label{wb-flux-cswe-2}
\hat{ \mathcal{H}}^{*}|_{e_K} = \mathcal{H}\big( V_{h,K}^{int}\big)\cdot \bm{n}_K^e
\end{equation}
holds when the lake-at-rest steady state is reached.
It is also not difficult to show that the residual of \eqref{WB-semi-QLMMDG-cswe}
\begin{equation}\label{residue-cswe}
\begin{split}
\mathcal{R}_{h,K}(V_h,\phi,B_h) &
= \int_{K}\mathcal{S}(h_h,B_h)\phi d\bm{x} +\int_{K} \mathcal{H}(V_h)\nabla \phi d\bm{x}
\\&~~~~~~~~~~
- \sum_{e_K}\int_{e_K\in \partial K} \phi \hat{ \mathcal{H}}^{*}|_{e_K} ds
- \int_{K}\phi\nabla \cdot( V_h \dot{\bm{X}} ) d\bm{x}
\end{split}
\end{equation}
vanishes for the lake-at-rest steady state if all integrals involved in the \eqref{residue-cswe} are computed
exactly. Thus, the semi-discrete scheme \eqref{WB-semi-QLMMDG-cswe} is well-balanced.

To study the well-balance property of a fully discrete scheme, we notice that \eqref{WB-semi-QLMMDG-cswe}
can be rewritten into a more compact form as
\begin{equation*}\label{ode}
\frac{d}{d t}\int_{K}V_h \phi d\bm{x}=\mathcal{R}_{h,K}(V_h,\phi,B_h)+\int_{K}\phi\nabla \cdot( V_h \dot{\bm{X}} ) d\bm{x}, \quad \forall \phi \in \mathcal{V}_h^k(t).
\end{equation*}

For simplicity, we consider the first-order forward Euler scheme here. Other explicit Runge-Kutta schemes can be considered similarly. The fully discrete QLMM-DG scheme reads as
\begin{equation*}
\label{Euler-U}
\begin{split}
\int_{K^{n+1}} V_h^{n+1}\phi^{n+1}d\bm{x}
= &\int_{K^{n}} V_h^n\phi^n d\bm{x}
\\&+\Delta t_n \Big(\mathcal{R}_{h,K}(V_h^n,\phi^n,B_h^n)+\int_{K^n}\phi^n\nabla \cdot( V^n_h \dot{\bm{X}}^n ) d\bm{x}\Big).
\end{split}
\end{equation*}
The corresponding GCL preserving update of the element area is given by
\begin{equation}
\label{Euler-Karea-1}
|K^{n+1}|
= |K^{n}|+\Delta t_n |K^{n}|\nabla\cdot\dot{\bm{X}}^n|_{K^{n}}.
\end{equation}
For the lake-at-rest steady state, recalling that $\mathcal{R}_{h,K}(V_h^n,\phi^n,B_h^n)=0$, we have
\begin{equation}
\label{Euler-U-1}
\int_{K^{n+1}} V_h^{n+1}\phi^{n+1}d\bm{x}
= \int_{K^{n}} V_h^n\phi^n d\bm{x}
+\Delta t_n \int_{K^n}\phi^n\nabla \cdot( V^n_h \dot{\bm{X}}^n ) d\bm{x},
\end{equation}
or in a component-wise format,
\begin{align}
&\int_{K^{n+1}} h_h^{n+1}\phi^{n+1}d\bm{x}
= \int_{K^{n}} h_h^n\phi^n d\bm{x}
+\Delta t_n \int_{K^n}\phi^n\nabla \cdot( h^n_h \dot{\bm{X}}^n ) d\bm{x},
\label{Euler-h}
\\&
\int_{K^{n+1}} m_h^{n+1}\phi^{n+1}d\bm{x}
= \int_{K^{n}} m_h^n\phi^n d\bm{x}
+\Delta t_n \int_{K^n}\phi^n\nabla \cdot( m^n_h \dot{\bm{X}}^n ) d\bm{x},
\label{Euler-m}
\\&
\int_{K^{n+1}} w_h^{n+1}\phi^{n+1}d\bm{x}
= \int_{K^{n}} w_h^n\phi^n d\bm{x}
+\Delta t_n \int_{K^n}\phi^n\nabla \cdot( w^n_h \dot{\bm{X}}^n ) d\bm{x}.
\label{Euler-w}
\end{align}
From \eqref{Euler-m} and \eqref{Euler-w}, we have $m^{n+1}_h=0$ and $w^{n+1}_h=0$ if
$m^n_h=0$ and $w^n_h=0$. For the water surface level, we can rewrite (\ref{Euler-h}) into
\begin{align*}
&\int_{K^{n+1}} (h_h^{n+1}+B_h^{n+1}) \phi^{n+1}d\bm{x}
 = \int_{K^{n}} (h_h^n+B_h^n) \phi^n d\bm{x}
+\Delta t_n \int_{K^n}\phi^n\nabla \cdot( (h^n_h+B_h^n) \dot{\bm{X}}^n ) d\bm{x}
\\
& \qquad \qquad + \int_{K^{n+1}} B_h^{n+1} \phi^{n+1}d\bm{x} - \int_{K^{n}} B_h^n \phi^n d\bm{x}
-\Delta t_n \int_{K^n}\phi^n\nabla \cdot( B_h^n \dot{\bm{X}}^n ) d\bm{x} .
\end{align*}
Multiplying (\ref{Euler-Karea-1}) with $\hat{\phi}$, integrating the resulting equation over $\hat{K}$,
changing the independent variables, and noticing that $\nabla \cdot \dot{\bm{X}}^n|_{K^{n}}$ is constant,
we get
\[
\int_{K^{n+1}}  \phi^{n+1}d\bm{x}
= \int_{K^{n}}  \phi^{n}d\bm{x} +\Delta t_n \int_{K^n}\phi^n\nabla \cdot \dot{\bm{X}}^n  d\bm{x} .
\]
Combining the above two equations and assuming that $h^n_h+B^n_h=C$, we get
\begin{equation}
\label{Euler-B-1}
\begin{split}
&\int_{K^{n+1}} (h_h^{n+1}+B_h^{n+1}-C) \phi^{n+1}d\bm{x}
\\
& \qquad \qquad = \int_{K^{n+1}} B_h^{n+1} \phi^{n+1}d\bm{x} - \int_{K^{n}} B_h^n \phi^n d\bm{x}
-\Delta t_n \int_{K^n}\phi^n\nabla \cdot( B_h^n \dot{\bm{X}}^n ) d\bm{x} .
\end{split}
\end{equation}
Since the right-hand side does not vanish in general for non-flat $B$, we do not have
$h^{n+1}_h+B^{n+1}_h=C$ and thus the scheme is not well-balanced.

Interestingly, (\ref{Euler-B-1}) suggests that if we update $B$ according to
\begin{equation}
\label{Euler-B-2}
\int_{K^{n+1}} B_h^{n+1} \phi^{n+1}d\bm{x} = \int_{K^{n}} B_h^n \phi^n d\bm{x}
+ \Delta t_n \int_{K^n}\phi^n\nabla \cdot( B_h^n \dot{\bm{X}}^n ) d\bm{x} ,
\end{equation}
then the scheme will be well-balanced. However, the above equation is actually the forward Euler
discretization of the semi-discrete problem
\[
\frac{d}{d t} \int_K B_h \phi d \bm{x} = \int_K \phi  \nabla \cdot (B_h \dot{\bm{X}}) d \bm{x},
\]
which in turn is a ``central" Galerkin approximation to the equation
\[
\frac{\partial B}{\partial t} = 0 .
\]
Thus, (\ref{Euler-B-2}) is unconditionally unstable and cannot be used for updating $B$.

\section{Numerical results}
\label{sec:numerical-results}

In this section we present numerical results obtained with the well-balanced QLMM-DG method described in \S\ref{sec:WB-MMDG} for a selection of one- and two-dimensional examples for the SWEs.

In the computation we take the CFL number in (\ref{FMDG-cfl}) and (\ref{MMDG-cfl}) as
$0.3$ for $P^1$-DG and $0.18$ for $P^2$-DG in one dimension, and $0.2$ for $P^1$-DG
and $0.1$ for $P^2$-DG in two dimensions, unless otherwise ststed.
For the TVB limiter implemented in the RKDG scheme, the TVB constant $M_{tvb}$
is taken as zero except for the accuracy test Example \ref{test5-1d} to avoid the accuracy
order reduction near the extrema.
The gravitation constant $g$ is taken as $9.812$.
Since analytical exact solutions are not available for all of the examples,
unless otherwise stated, we take the numerical solution obtained with the $P^2$-DG method
with a fixed mesh of $N=10,000$ as a reference solution.
The errors are computed based on the values of the numerical solution at $21$ points on each element.
Except for the accuracy test (Example \ref{test5-1d}) and the lake-at-rest steady-state flow tests
(Example \ref{test1+2-1d} and Example \ref{test1-2d}), to save space we omit the results for $P^1$-DG
since they are similar to those for $P^2$-DG.

We generate the adaptive moving mesh using the MMPDE moving mesh method; see Appendix \ref{sec:mmpde}.
A key to the method is that a metric tensor $\mathbb{M}=\mathbb{M}(\bm{x})$
is used to control the size, shape, and orientation of mesh elements
throughout the domain. Roughly speaking, mesh elements are moved toward the goal
that the circumscribed ellipse of each element $K$ is similar to the ellipse whose
principal axes coincide with the eigen-directions of an average of $\mathbb{M}$ on $K$,
$\mathbb{M}_K$, and semi-lengths of the principal axes are inversely proportional to
the square root of the corresponding eigenvalues of $\mathbb{M}_K$ (cf. Appendix \ref{sec:mmpde}).
Various metric tensors have been proposed; e.g., see \cite{Huang-Sun-2003JCP, Huang-Russell-2011}.
We use here an optimal metric tensor based on the $L^2$-norm of piece linear interpolation error.
Let $q$ be a physical variable and $q_h$ be its finite element approximation.
Denote by $H_K$ a recovered Hessian of $q_h$ on $K\in \mathcal{T}_{h}$
obtained using the least squares fitting Hessian recovery technique \cite{Zhang-Naga-2005Siam}.
Assuming that the eigen-decomposition of $H_K$ is given by
\[
H_K = Q\hbox{diag}(\lambda_1,\cdots,\lambda_d)Q^T,
\]
where $Q$ is an orthogonal matrix, we define
\[
|H_K| = Q\hbox{diag}(|\lambda_1|,...,|\lambda_d|)Q^T.
\]
Then, the metric tensor is defined as
\begin{equation}\label{mer}
\mathbb{M}_{K} =\det \big{(}\alpha_{h}\mathbb{I}+|H_K|\big{)}^{-\frac{1}{d+4}}
\big{(}\alpha_{h}\mathbb{I}+|H_K|\big{)},
\quad \forall K \in \mathcal{T}_h
\end{equation}
where $\mathbb{I}$ is the identity matrix, $\det(\cdot)$ is the determinant of a matrix,
and $\alpha_{h}$ is a regularization parameter defined implicitly through the algebraic equation
\[
\sum_{K\in\mathcal{T}_h}|K|\, \hbox{det}(\alpha_{h}\mathbb{I}+|H_K|)^{\frac{2}{d+4}}
=2\sum_{K\in\mathcal{T}_h}|K|\,
\hbox{det}(|H_K|)^{\frac{2}{d+4}}.
\]

Following \cite{Zhang-Huang-Qiu-2020arXiv}, we compute the metric tensor based on the equilibrium variable $\mathcal{E}=\frac{1}{2}(u^2+v^2)+g\eta $ and
the water depth $h$.
A motivation for this is that the mesh is adapted to both the perturbations of the lake-at-rest steady state
through $\mathcal{E}$ and the water depth distribution through $h$.
To be specific, we first compute $\mathbb{M}^{\mathcal{E}}_{K} $ and $\mathbb{M}^{h}_{K} $ using
(\ref{mer}) with $q = \mathcal{E}$ and $h$, respectively.
Then, a new metric tensor is obtained through matrix intersection as
\begin{equation}\label{mer-Eh}
\tilde{\mathbb{M}}_{K}
=\frac{{\mathbb{M}}^{\mathcal{E}}_{K} }{|||{\mathbb{M}}^{\mathcal{E}}_{K}|||}
\cap\frac{\delta\cdot{\mathbb{M}}^{h}_{K}}{|||{\mathbb{M}}^{h}_{K}|||} ,
\end{equation}
where $|||\cdot |||$ denotes the maximum absolute value
of the entries of a matrix and ``$\cap$" stands for matrix intersection (cf. \cite{Zhang-Cheng-Huang-Qiu-2020CiCP}).
Notice that both ${\mathbb{M}}^{\mathcal{E}}_{K}$ and ${\mathbb{M}}^{h}_{K}$ have been normalized
before taking matrix intersection. We take $\delta=0.1$ in our numerical computation.

The next step is to make sure that the metric tensor is bounded above since it is known \cite{Huang-Kamenski-2018MC}
that the moving mesh generated by the MMPDE method stays nonsingular (i.e., free from tangling)
if the metric tensor is bounded and the initial mesh is nonsingular. We define
\begin{equation}\label{mer-ceil}
  \hat{\mathbb{M}}_K = \frac{\tilde{\mathbb{M}}_K}{\sqrt{1+\Big(\frac{\hbox{tr}(\tilde{\mathbb{M}}_K)}{\beta}\Big)^2}},
  \end{equation}
where $\beta$ is a positive number. It can be shown that $\lambda_{\max}(\hat{\mathbb{M}}_K) \leq \beta$.
In our computation, we take $\beta=1000$.

The metric tensor is generally non-smooth since the recovered Hessian can be very rough.
A common practice in the moving mesh context is to smooth the metric tensor for smoother meshes.
Here we simply average the metric tensor at a vertex over its neighboring vertices, i.e.,
\begin{equation}
\mathbb{M}_i \longleftarrow \frac{1}{|\mathcal{N}_i|}\sum_{j\in\mathcal{N}_i}\hat{\mathbb{M}}_j, \quad \forall
\bm{x}_i\in \mathcal{T}_h
\end{equation}
where $\hat{\mathbb{M}}_i$ is the nodal value at vertex $\bm{x}_i$ obtained by area-averaging the values
of the metric tensor on the neighboring elements,
$\mathcal{N}_i$ denotes the set of the immediate neighboring vertices (including itself) of $\bm{x}_i$,
and $|\mathcal{N}_i|$ is the length of $\mathcal{N}_i$.
This process can be repeated several times every time after the metric tensor is  computed.


\begin{example}\label{test5-1d}
(The accuracy test for the 1D SWEs over a sinusoidal hump.)
\end{example}

In this example we verify the high-order accuracy of the well-balanced QLMM-DG method.
The bottom topography is a sinusoidal hump
\begin{equation*}
B(x)=\sin^2(\pi x),\quad x \in (0,1).
\end{equation*}
Periodic boundary conditions are used for all unknown variables.
The initial conditions are
\begin{equation*}
\eta(x,0)=5+e^{\cos(2\pi x)} + B(x), \quad
hu(x,0)=\sin\big(\cos(2\pi x)\big).
\end{equation*}
This example has been used as an accuracy test by a number of researchers; e.g., see
\cite{Li-etal-2018JCAM,Xing-Shu-2006JCP,Xing-Shu-2006CiCP}.
The final simulation time is $T=0.1$ when the solutions remain smooth.
A reference solution is obtained using the $P^2$-DG method with a fixed mesh of $N=20,000$.
The TVB minmod constant $M_{tvb}$ is taken as $40$ in this example to avoid the accuracy order reduction near the extrema.
The $L^1$ and $L^\infty$ norm of the error for $h=\eta-B$ and $hu$ is plotted as a function of $N$
in Fig.~\ref{Fig:test5-1d-h-hu} for fixed and moving meshes.
One can see that the QLMM-DG method is second-order for $P^1$-DG and
third-order for $P^2$-DG in both $L^1$ and $L^\infty$ norm.
Moreover, the error for moving meshes is a slightly smaller than but otherwise comparable to the error
for fixed meshes. The error is much smaller for $P^2$-DG than $P^1$-DG,
as expected for smooth problems.

\begin{figure}[h]
\centering
\subfigure[$L^1$-error: $h$]{
\includegraphics[width=0.35\textwidth,trim=40 0 40 10,clip]{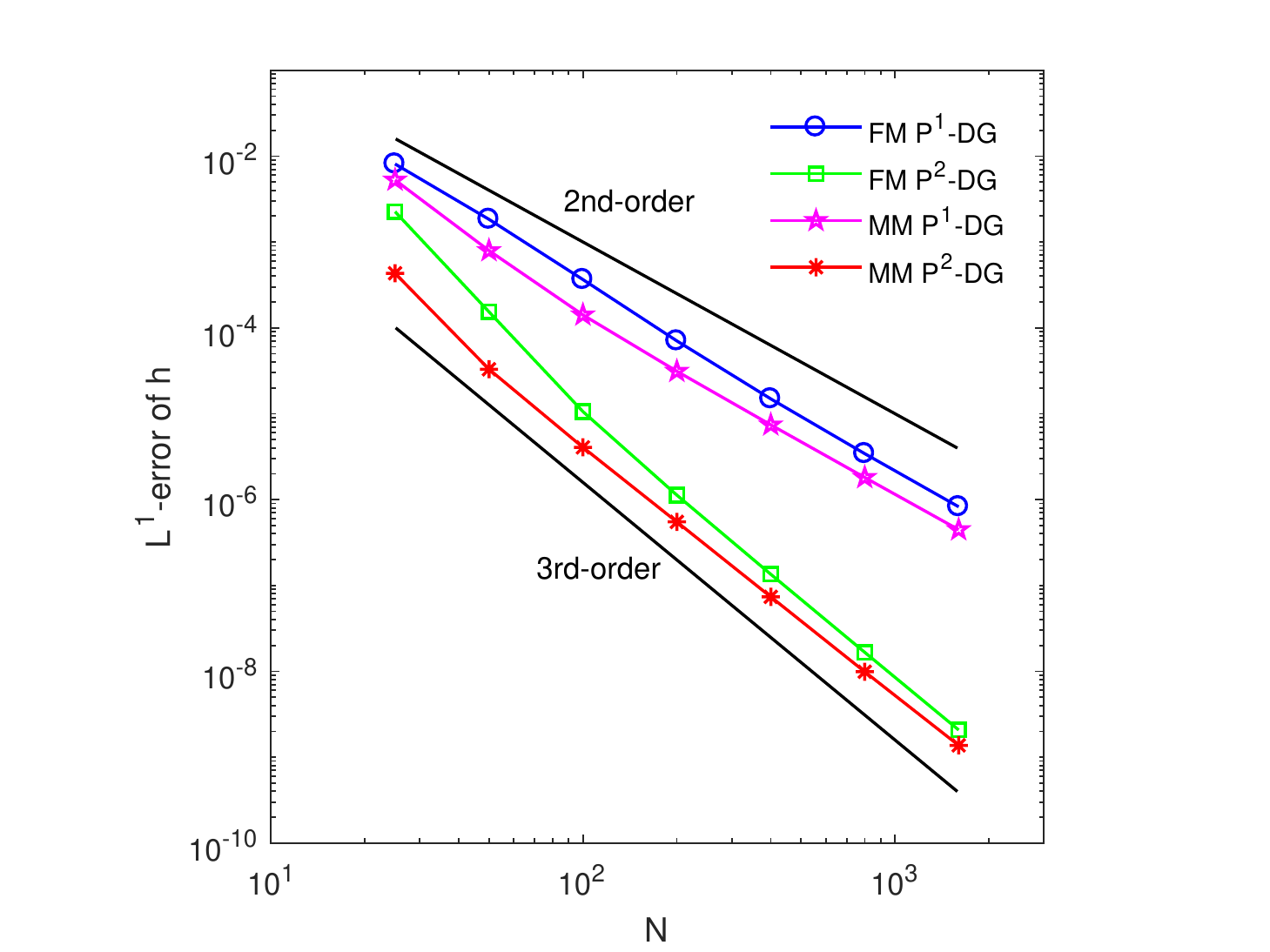}}
\subfigure[$L^\infty$-error: $h$]{
\includegraphics[width=0.35\textwidth,trim=40 0 40 10,clip]{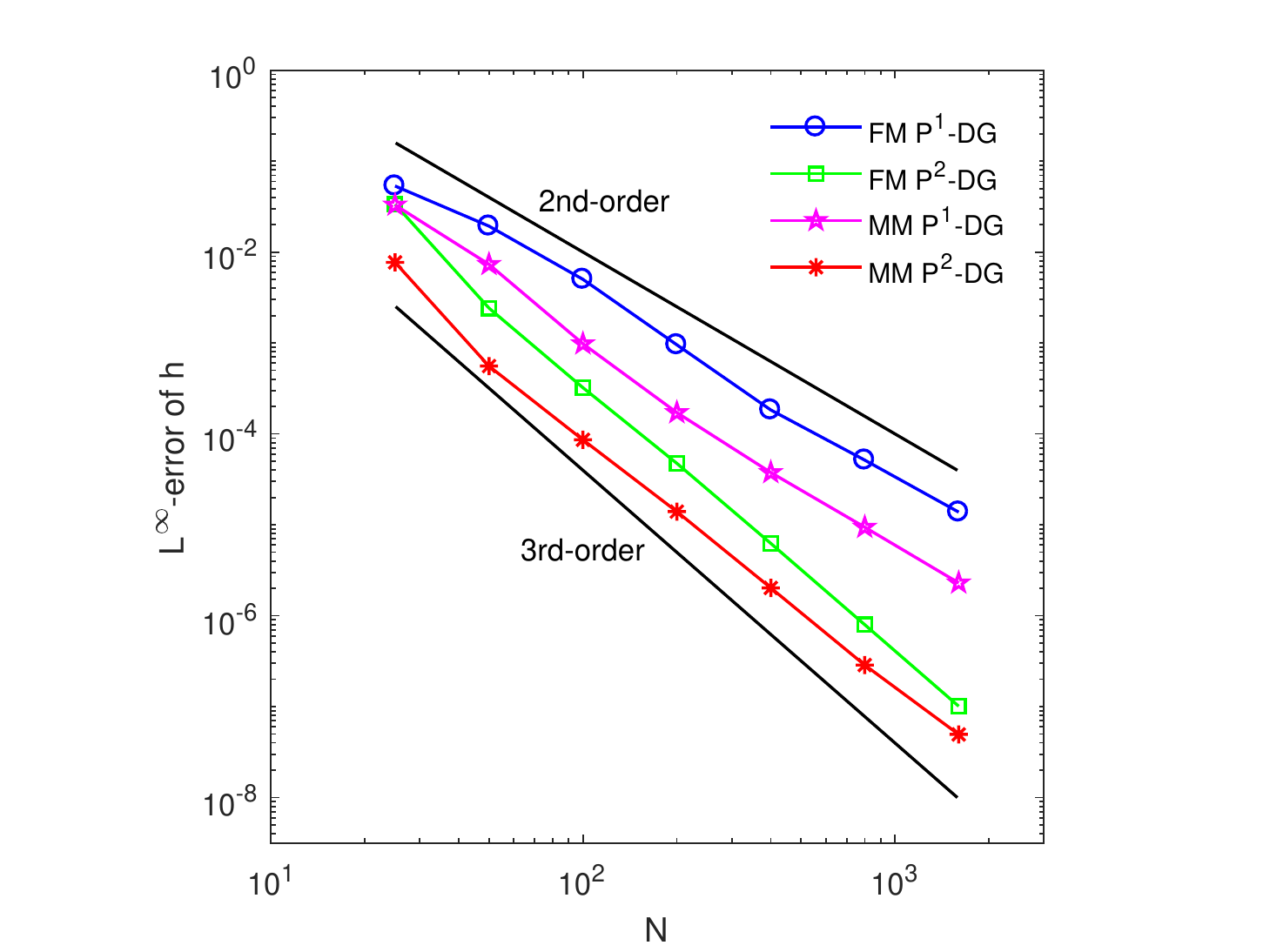}}
\subfigure[$L^1$-error: $hu$]{
\includegraphics[width=0.35\textwidth,trim=40 0 40 10,clip]{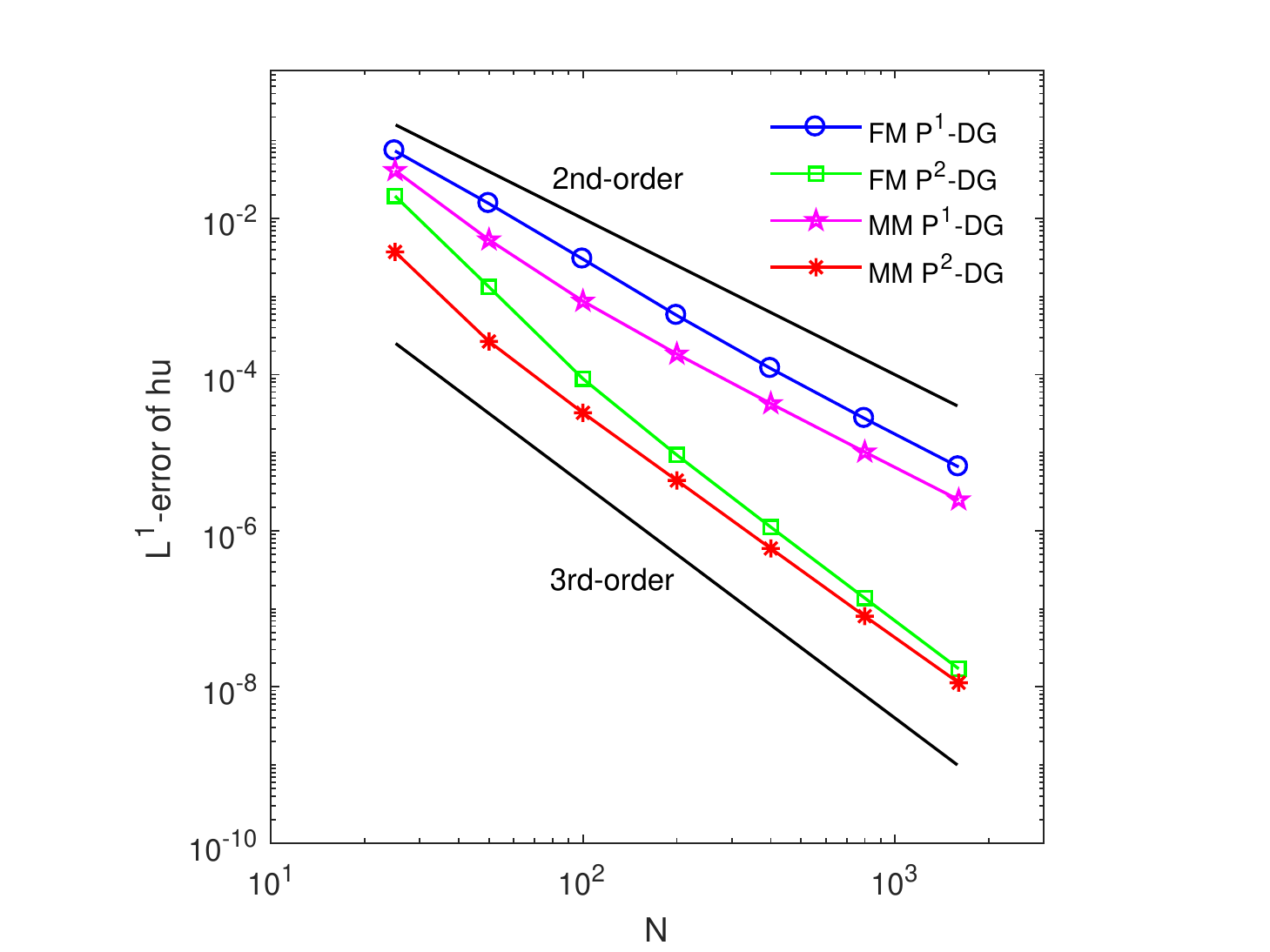}}
\subfigure[$L^\infty$-error: $hu$]{
\includegraphics[width=0.35\textwidth,trim=40 0 40 10,clip]{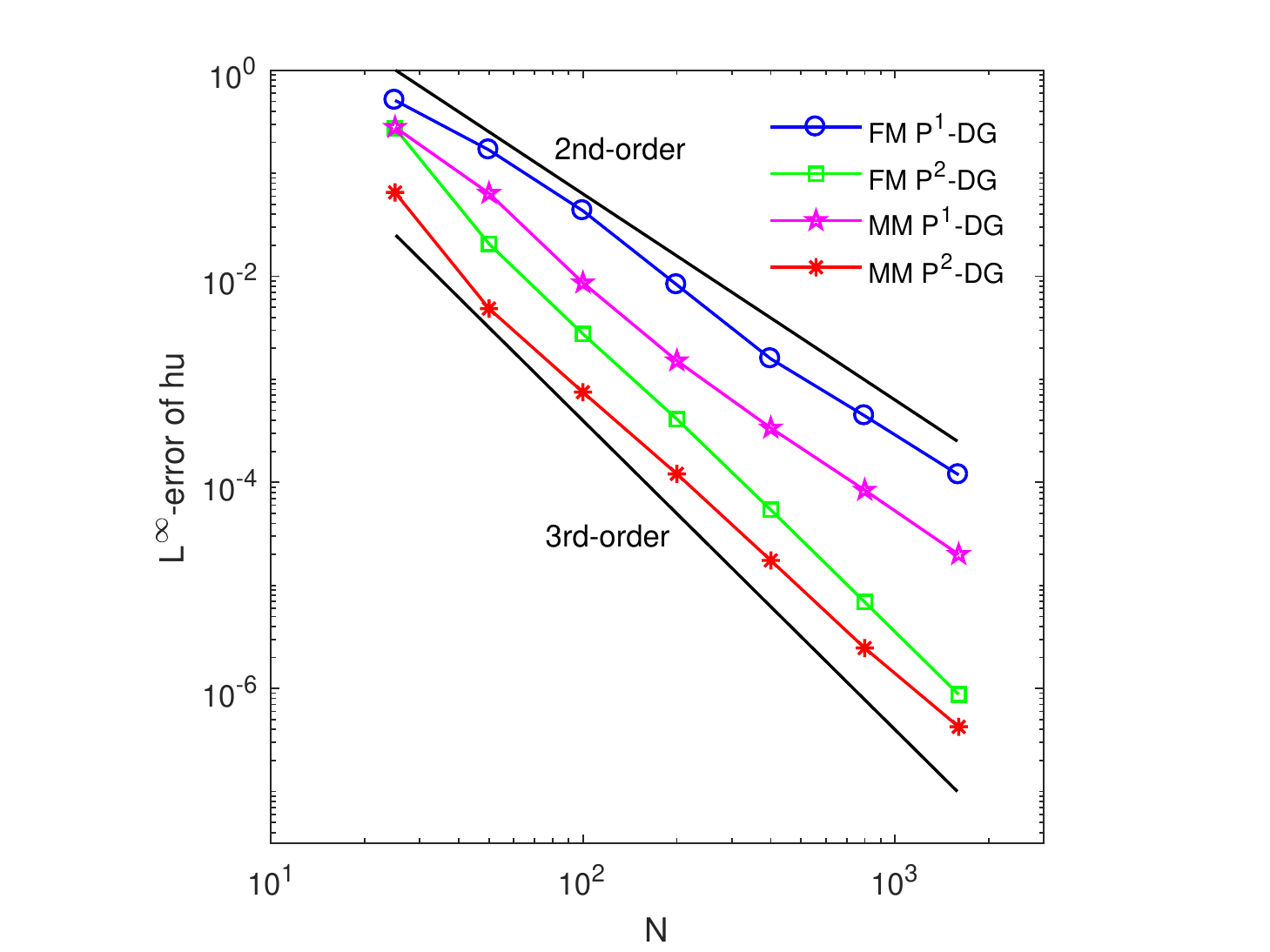}}
\caption{Example \ref{test5-1d}. The $L^1$ and $L^\infty$ norm of the error for the water depth $h=\eta-B$ and water discharge $hu$ is plotted as a function of $N$  for fixed and moving meshes.}
\label{Fig:test5-1d-h-hu}
\end{figure}


\begin{example}\label{test1+2-1d}
(The lake-at-rest steady-state flow test for the 1D SWEs over three different bottom topographies.)
\end{example}
We now test the well-balance property of the QLMM-DG method with two smooth topographies and a discontinuous topography,
\begin{align}
& B(x)=5 e^{-\frac25(x-5)^2},  \quad x \in (0,10)
\label{B-1}
\\
& B(x)=
\begin{cases}
4,& \text{for } x \in (4, 8) \\
0,& \text{for } x \in (0,4) \cup (8, 10)
\end{cases}
\label{B-2}
\\& B(x)=10 e^{-\frac25(x-5)^2},  \quad x \in (0,10).
\label{dry-B-1}
\end{align}
The initial solution is taken as the lake-at-rest steady state,
\begin{equation*}
u =0, \quad \eta = h+B=10.
\end{equation*}
We expect that this steady-state solution is preserved since the QLMM-DG method is well-balanced.
The final time is $T=0.5$.

The $L^1$ and $L^\infty$ error for $\eta$ and $hu$ is listed in Tables~\ref{tab:test1-1d-p1-error}
and~\ref{tab:test1-1d-p2-error} for smooth $B$ (\ref{B-1}) and
in Tables~\ref{tab:test2-1d-p1-error} and \ref{tab:test2-1d-p2-error} for discontinuous $B$ (\ref{B-2}).
We can observe that the error is at the level of round-off error (double precision in MATLAB),
which demonstrates that the QLMM-DG method is well-balanced.

The bottom topography \eqref{dry-B-1} (with a dry region) is used to demonstrate the well-balance and PP properties
of the QLMM-DG method. This topography has a similar shape as \eqref{B-1} but its height touches
the surface level at $x=5$ where $h=0$ initially. The computed water depth can have negative values during
the computation and the application of the PP limiter is necessary.
We computed the solution up to $t=0.5$. To ensure positivity preservation (cf. \cite{Xing-Zhang-Shu-2010}), we take
smaller CFL numbers as $0.3$ and $0.15$ for $P^1$-DG and $P^2$-DG, respectively, for this test.
The $L^1$ and $L^\infty$ error for $h+B$ and $hu$ is listed in Tables~\ref{tab:test-1d-wb-pp-P1error} and \ref{tab:test-1d-wb-pp-P2error} for $P^1$-DG and $P^2$-DG, respectively.
The results clearly show that the QLMM-DG method is well-balanced.
\begin{table}[h]
\caption{Example \ref{test1+2-1d}. Well-balance test for the $P^1$-DG method
with fixed and moving meshes for smooth $B$ defined in (\ref{B-1}).}
\vspace{3pt}
\centering
\label{tab:test1-1d-p1-error}
\begin{tabular}{ccccc}
 \toprule
  & \multicolumn{2}{c}{$\eta$}&\multicolumn{2}{c}{$hu$}\\
$N$&$L^1$-error  &$L^{\infty}$-error  &$L^1$-error  &$L^{\infty}$-error \\
\midrule
~   &  \multicolumn{4}{c}{\em{FM-DG method} }\\
\midrule
25	&	4.288E-16	&	2.262E-15	&	8.582E-15	&	2.344E-14	\\
50	&	4.143E-16	&	2.578E-15	&	1.541E-14	&	3.836E-14	\\
100	&	1.431E-15	&	4.804E-15	&	2.537E-14	&	6.756E-14	\\
\midrule
~   &  \multicolumn{4}{c}{\em{QLMM-DG method} }\\
\midrule
25	&	6.726E-15	&	1.307E-14	&	1.196E-14	&	3.719E-14	\\
50	&	1.185E-14	&	2.585E-14	&	2.153E-14	&	8.861E-14	\\
100	&	2.303E-14	&	5.630E-14	&	3.731E-14	&	1.906E-13	\\
 \bottomrule	
\end{tabular}
\end{table}
\begin{table}[h]
\caption{Example \ref{test1+2-1d}. Well-balance test for the $P^2$-DG method with fixed and moving meshes for smooth $B$ defined in (\ref{B-1})}
\vspace{3pt}
\centering
\label{tab:test1-1d-p2-error}
\begin{tabular}{ccccc}
 \toprule
  & \multicolumn{2}{c}{$\eta$}&\multicolumn{2}{c}{$hu$}\\
$N$&$L^1$-error  &$L^{\infty}$-error  &$L^1$-error  &$L^{\infty}$-error \\
\midrule
~   &  \multicolumn{4}{c}{\em{FM-DG method} }\\
\midrule
25	&	4.697E-15	&	7.849E-15	&	1.771E-14	&	3.968E-14	\\
50	&	4.813E-15	&	7.691E-15	&	1.617E-14	&	4.954E-14	\\
100	&	4.919E-15	&	8.766E-15	&	3.464E-14	&	8.394E-14	\\
\midrule
~   &  \multicolumn{4}{c}{\em{QLMM-DG method} }\\
\midrule
25	&	1.510E-14	&	2.146E-14	&	1.510E-14	&	3.684E-14	\\
50	&	2.289E-14	&	3.709E-14	&	3.102E-14	&	9.464E-14	\\
100	&	4.132E-14	&	6.846E-14	&	4.847E-14	&	1.843E-13	\\
 \bottomrule	
\end{tabular}
\end{table}
\begin{table}[h]
\caption{Example \ref{test1+2-1d}. Well-balance test for the $P^1$-DG method with fixed and moving meshes for discontinuous $B$ defined in (\ref{B-2}).}
\vspace{3pt}
\centering
\label{tab:test2-1d-p1-error}
\begin{tabular}{ccccc}
 \toprule
  & \multicolumn{2}{c}{$\eta$}&\multicolumn{2}{c}{$hu$}\\
$N$&$L^1$-error  &$L^{\infty}$-error  &$L^1$-error  &$L^{\infty}$-error \\
\midrule
~   &  \multicolumn{4}{c}{\em{FM-DG method} }\\
\midrule
25	&	3.540E-16	&	1.638E-15	&	5.512E-15	&	1.709E-14	\\
50	&	3.530E-16	&	2.171E-15	&	1.310E-14	&	3.317E-14	\\
100	&	3.499E-16	&	2.361E-15	&	2.185E-14	&	5.404E-14	\\
\midrule
~   &  \multicolumn{4}{c}{\em{QLMM-DG method} }\\
\midrule
25	&	5.445E-15	&	1.128E-14	&	1.113E-14	&	3.214E-14	\\
50	&	1.316E-14	&	2.443E-14	&	1.926E-14	&	6.575E-14	\\
100	&	2.344E-14	&	5.195E-14	&	4.425E-14	&	1.702E-13	\\
 \bottomrule	
\end{tabular}
\end{table}
\begin{table}[h]
\caption{Example \ref{test1+2-1d}. Well-balance test for the $P^2$-DG method with fixed and moving meshes for discontinuous $B$ defined in (\ref{B-2}).}
\vspace{3pt}
\centering
\label{tab:test2-1d-p2-error}
\begin{tabular}{ccccc}
 \toprule
  & \multicolumn{2}{c}{$\eta$}&\multicolumn{2}{c}{$hu$}\\
$N$&$L^1$-error  &$L^{\infty}$-error  &$L^1$-error  &$L^{\infty}$-error \\
\midrule
~   &  \multicolumn{4}{c}{\em{FM-DG method} }\\
\midrule
25	&	4.517E-15	&	6.204E-15	&	1.088E-14	&	3.416E-14	\\
50	&	4.682E-15	&	7.090E-15	&	1.876E-14	&	5.815E-14	\\
100	&	4.482E-15	&	6.680E-15	&	1.967E-14	&	5.498E-14	\\
\midrule
~   &  \multicolumn{4}{c}{\em{QLMM-DG method} }\\
\midrule
25	&	1.393E-14	&	2.126E-14	&	1.840E-14	&	4.412E-14	\\
50	&	2.177E-14	&	3.325E-14	&	3.417E-14	&	9.337E-14	\\
100	&	4.194E-14	&	6.710E-14	&	4.863E-14	&	1.867E-13	\\
 \bottomrule	
\end{tabular}
\end{table}
\begin{table}[h]
\caption{Example \ref{test1+2-1d}. Well-balance test for the $P^1$-DG method with fixed and moving meshes for the bottom topography \eqref{dry-B-1} (with a dry region).}
\vspace{3pt}
\centering
\label{tab:test-1d-wb-pp-P1error}
\begin{tabular}{ccccc}
 \toprule
  & \multicolumn{2}{c}{$\eta$}&\multicolumn{2}{c}{$hu$}\\
$N$&$L^1$-error  &$L^{\infty}$-error  &$L^1$-error  &$L^{\infty}$-error \\
\midrule
~   &  \multicolumn{4}{c}{\em{$P^1$ FM-DG method} }\\
\midrule
25	&2.463E-16	&	1.142E-15	&	1.193E-14	&	2.986E-14	\\
50	&4.796E-16	&	2.175E-15	&	1.726E-14	&	4.508E-14	\\
100	&6.640E-16	&	2.897E-15	&	1.492E-14	&	5.099E-14	\\
\midrule
~   &  \multicolumn{4}{c}{\em{$P^1$ QLMM-DG method} }\\
\midrule
25	&5.670E-15	&	1.136E-14	&	7.756E-15	&	3.501E-14	\\
50	&1.177E-14	&	2.691E-14	&	1.510E-14	&	7.347E-14	\\
100	&2.331E-14	&	5.537E-14	&	3.161E-14	&	1.801E-13	\\

 \bottomrule	
\end{tabular}
\end{table}
\begin{table}[h]
\caption{Example \ref{test1+2-1d}. Well-balance test for the $P^2$-DG method with fixed and moving meshes for the bottom topography \eqref{dry-B-1} (with a dry region).}
\vspace{3pt}
\centering
\label{tab:test-1d-wb-pp-P2error}
\begin{tabular}{ccccc}
 \toprule
  & \multicolumn{2}{c}{$\eta$}&\multicolumn{2}{c}{$hu$}\\
$N$&$L^1$-error  &$L^{\infty}$-error  &$L^1$-error  &$L^{\infty}$-error \\
\midrule
~   &  \multicolumn{4}{c}{\em{$P^2$ FM-DG method} }\\
\midrule
25	&4.513E-15	&	5.389E-15	&	2.117E-14	&	5.636E-14	\\
50	&4.579E-15	&	6.190E-15	&	3.079E-14	&	7.233E-14	\\
100	&4.644E-15	&	6.589E-15	&	3.433E-14	&	9.521E-14	\\
\midrule
~   &  \multicolumn{4}{c}{\em{$P^2$ QLMM-DG method} }\\
\midrule
25	&1.460E-14	&	2.264E-14	&	1.847E-14	&	4.411E-14	\\
50	&2.686E-14	&	4.208E-14	&	4.812E-14	&	1.182E-13	\\
100	&4.887E-14	&	7.505E-14	&	5.989E-14	&	1.909E-13	\\
 \bottomrule	
\end{tabular}
\end{table}

\begin{example}\label{test3-1d}
(The perturbed lake-at-rest steady-state flow test for the 1D SWEs.)
\end{example}
Following
\cite{Donat-etal-2014JCP,LeVeque-1998JCP,Li-etal-2018JCAM,Xing-Shu-2006JCP,Xing-Shu-2006CiCP},
we use this example to demonstrate that the QLMM-DG method is able to capture small perturbations
of the lake-at-rest steady-state flow over non-flat bottom topography. We also use it to demonstrate
the mesh adaptation ability of the method. The bottom topography in this example is taken as
\begin{equation}\label{test3-1d-B1}
B(x)=
\begin{cases}
0.25(\cos(10\pi(x-1.5))+1),& \text{for } x \in (1.4, 1.6) \\
0,& \text{for } x \in (0, 1.4) \cup (1.6, 2)
\end{cases}
\end{equation}
which has a bump in the middle of the physical interval.
The initial conditions are
\begin{equation*}
\eta(x,0)=
\begin{cases}
1+\varepsilon,& \text{for}~1.1\leq x\leq 1.2\\
1,& \text{otherwise}
\end{cases}
\quad \hbox{and}\quad u(x,0)=0,
\end{equation*}
where $\varepsilon$ is a constant for the perturbation magnitude.
We consider two cases, $\varepsilon=0.2$ (big pulse) and $\varepsilon=10^{-5}$ (small pulse).
The initial conditions for both cases are plotted in Fig.~\ref{Fig:test3-1d-initial}.
We use the transmissive boundary conditions and compute the solution up to $T=0.2$
when the right wave has passed the bottom bump.

The mesh trajectories obtained with the $P^2$-DG method and a moving mesh of $N=160$
are shown in Fig.~\ref{Fig:test3-1d-mesh}.
The obtained water surface level $\eta$ and discharge $hu$ are shown in Figs.~\ref{Fig:test3-1d-large-eta}
and \ref{Fig:test3-1d-large-hu} (for $\varepsilon=0.2$) and Figs.~\ref{Fig:test3-1d-small-eta} and
\ref{Fig:test3-1d-small-hu} (for $\varepsilon=10^{-5}$).
It is interesting to observe that the initial wave splits into two waves at about $t=0.0165$
and the two waves propagate left and right at the characteristic speeds $\pm \sqrt{gh}$,
respectively. The right-propagating wave interacts with
the bottom bump and generates a complex wave structure in the bump region.
Thus, it is beneficial to concentrate mesh points around the bump.

From the numerical results, we can see that the mesh points are concentrated
around the waves before and after the split and in the region of the bottom bump.
This is what we want as well as expect. To explain, we recall that the metric tensor for mesh adaptation
is constructed based on the equilibrium variable $\mathcal{E} = u^2/2 + g \eta$ and
the water height $h$. It is not difficult to imagine that the mesh points concentrate
around the waves because $\eta$ and thus $\mathcal{E}$ have significant changes there.
Meanwhile, $\eta$ is constant in most places of the domain. Then, the spatial variations
in $B$ are reflected in $h$, which in turn leads to the higher mesh concentration
in the region of the bottom bump.
Figs.~\ref{Fig:test3-1d-large-eta}, \ref{Fig:test3-1d-large-hu}, \ref{Fig:test3-1d-small-eta}, and
\ref{Fig:test3-1d-small-hu} show that the QLMM-DG method is able to capture
perturbations, small or large, of the lake-at-rest steady-state flow over non-flat bottom topography.
Moreover, the moving mesh solutions with $N=160$ are more accurate than those with
fixed meshes of $N=160$ and $N=480$ and contain no visible spurious numerical oscillations.

To verify the well-balance and positivity-preserving properties of the QLMM-DG method
we increase the height of the bottom topography \eqref{test3-1d-B2} to contain a dry region (near $x = 1.5$),
\begin{equation}
\label{test3-1d-B2}
\begin{split}
&B(x)=
\begin{cases}
0.5(\cos(10\pi(x-1.5))+1),& \text{for } x \in (1.4, 1.6) \\
0,& \text{for } x \in (0, 1.4) \cup (1.6, 2) .
\end{cases}
\end{split}
\end{equation}
We repeat the computation with $\varepsilon = 10^{-5}$.
The bottom topography, the initial water level, and the mesh trajectories of $N=160$
obtained with the $P^2$ QLMM-DG method are plotted in Fig.~\ref{Fig:PP-test3-1d-initial}.
The mesh has higher concentration around the shock waves and the non-flat topography region.
The mesh trajectories show that the right moving shock stops after it hits the dry region.

The water surface $\eta$ and discharge $hu$ obtained with $P^2$-DG and a moving mesh of $N=160$
and fixed meshes of $N=160$ and $N=640$ are plotted in Figs.~\ref{Fig:PP-test3-1d-eta} and
\ref{Fig:PP-test3-1d-hu}. The results show that the DG method with moving or fixed meshes
is able to capture the waves of small perturbation for situations containing dry regions.
Moreover, the moving mesh solutions with $N=160$ are more accurate than those with
fixed meshes of $N=160$ and $N=640$ and contain no visible spurious numerical oscillations.
\begin{figure}[H]
\centering
\subfigure[Big pulse $\varepsilon=0.2$]{
\includegraphics[width=0.35\textwidth,trim=40 0 40 10,clip]{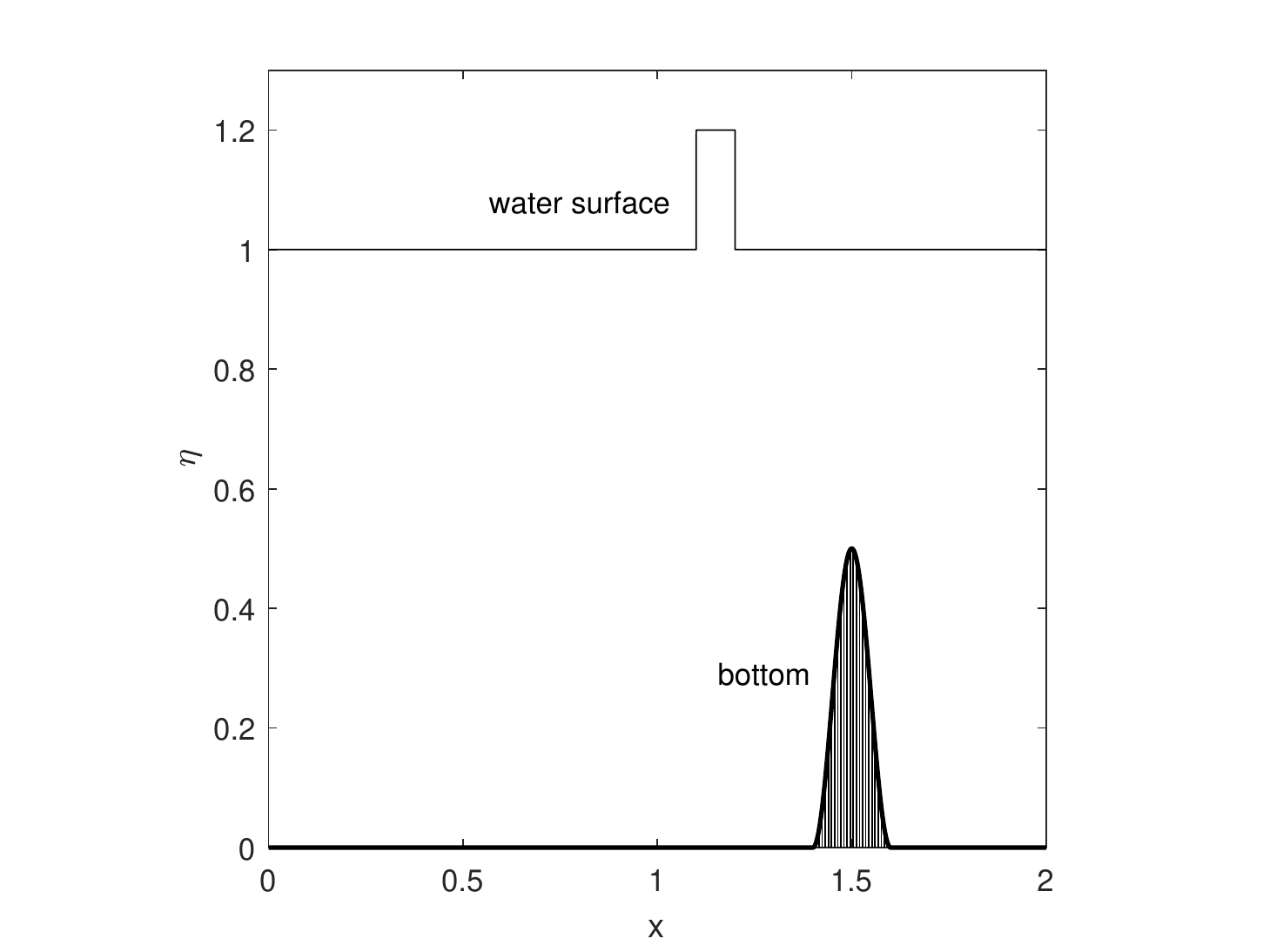}}
\subfigure[Small pulse $\varepsilon=10^{-5}$]{
\includegraphics[width=0.35\textwidth,trim=40 0 40 10,clip]{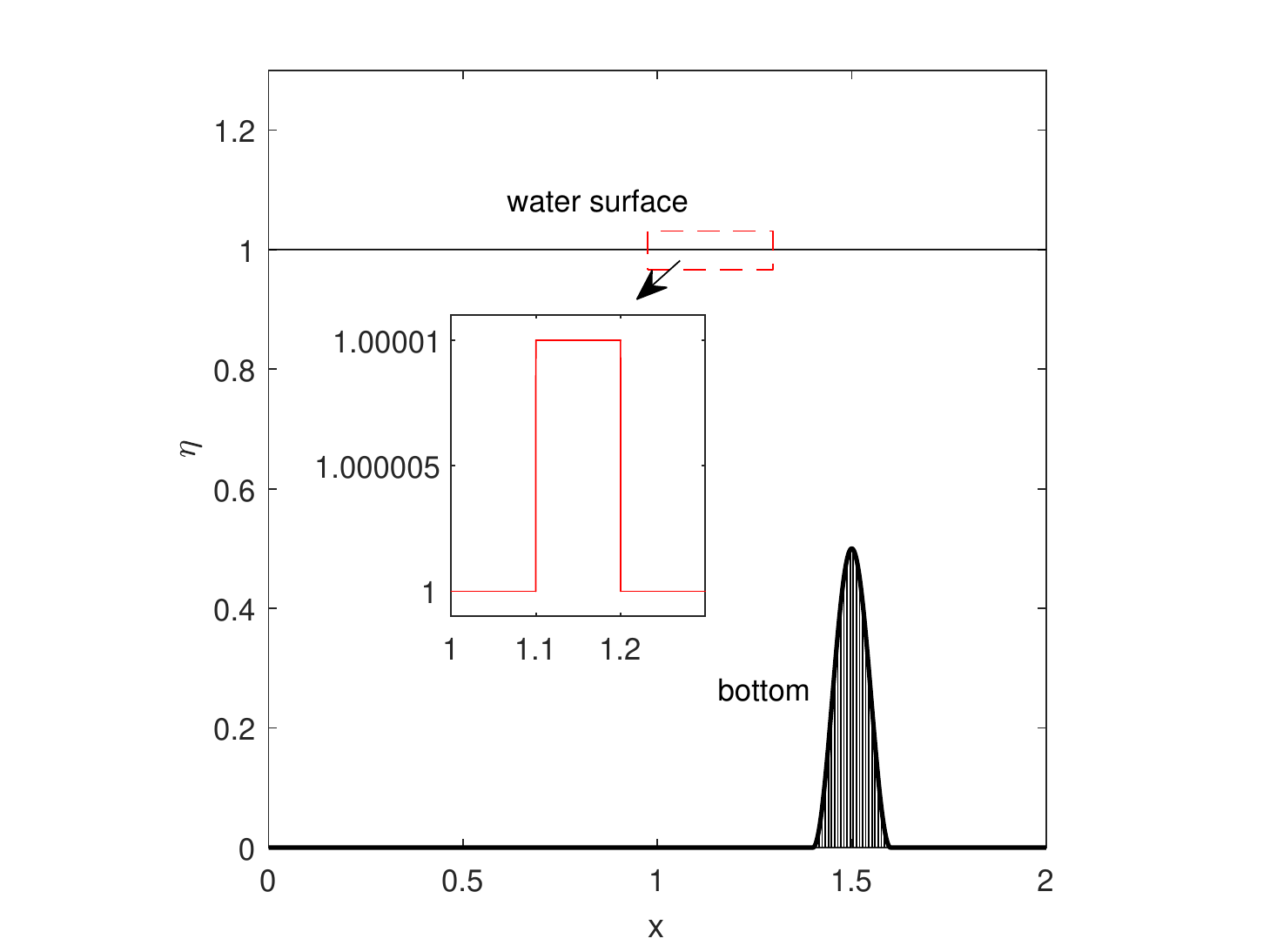}}
\caption{Example \ref{test3-1d}. The initial water surface level $\eta$ and the bottom topography $B$ are plotted for the pulse of $\varepsilon=0.2$ and $\varepsilon=10^{-5}$.}
\label{Fig:test3-1d-initial}
\end{figure}

\begin{figure}[H]
\centering
\subfigure[Big pulse $\varepsilon=0.2$]{
\includegraphics[width=0.35\textwidth,trim=40 0 40 10,clip]{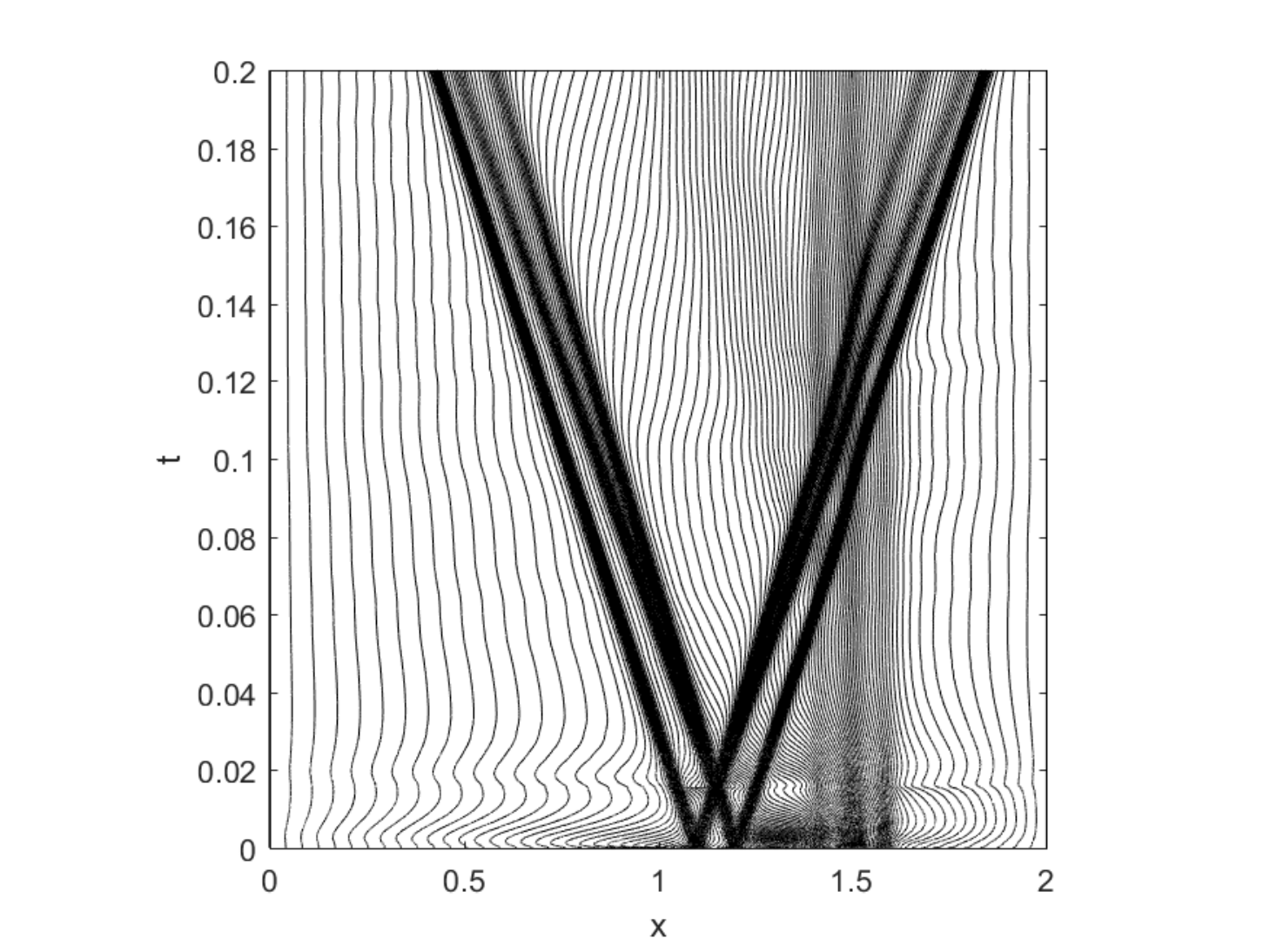}}
\subfigure[Small pulse $\varepsilon=10^{-5}$]{
\includegraphics[width=0.35\textwidth,trim=40 0 40 10,clip]{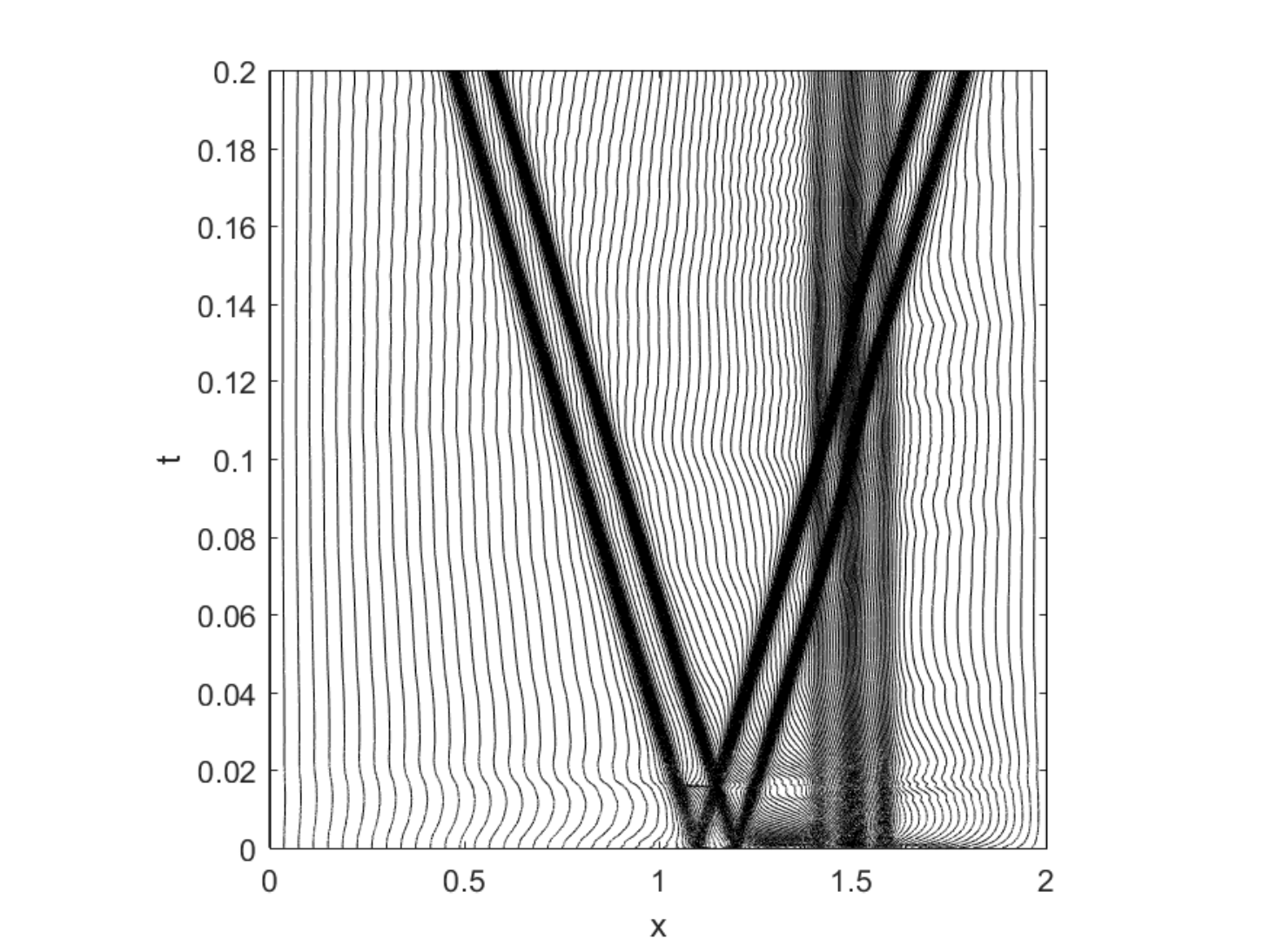}}
\caption{Example \ref{test3-1d}. The mesh trajectories are obtained with the $P^2$-DG method and a moving mesh of $N=160$ for the pulse of $\varepsilon=0.2$ and $\varepsilon=10^{-5}$.}
\label{Fig:test3-1d-mesh}
\end{figure}

\begin{figure}[H]
\centering
\subfigure[$\eta$: FM 160 vs MM 160]{
\includegraphics[width=0.35\textwidth,trim=10 0 40 10,clip]{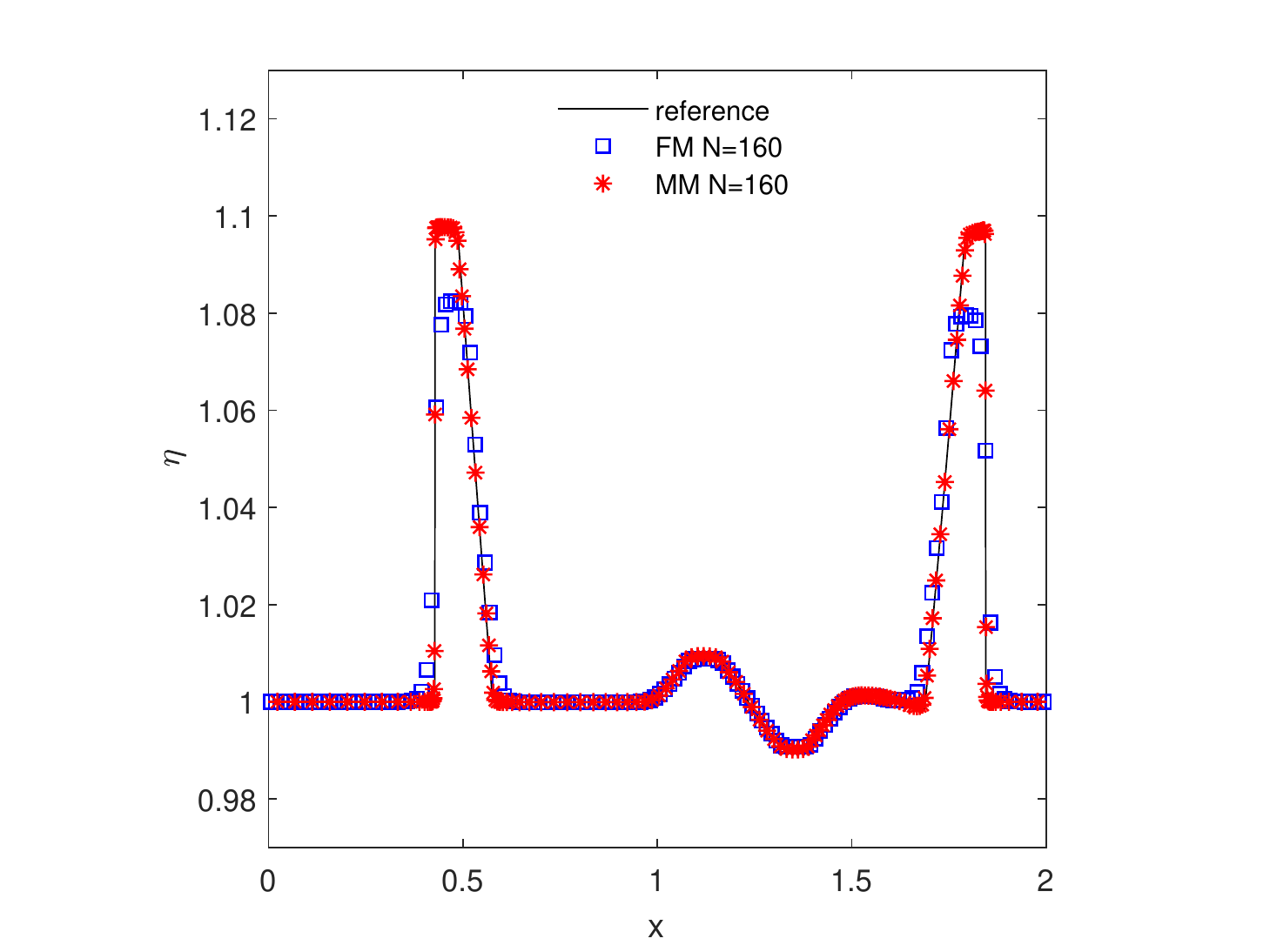}}
\subfigure[Close view of (a)]{
\includegraphics[width=0.35\textwidth,trim=10 0 39 10,clip]{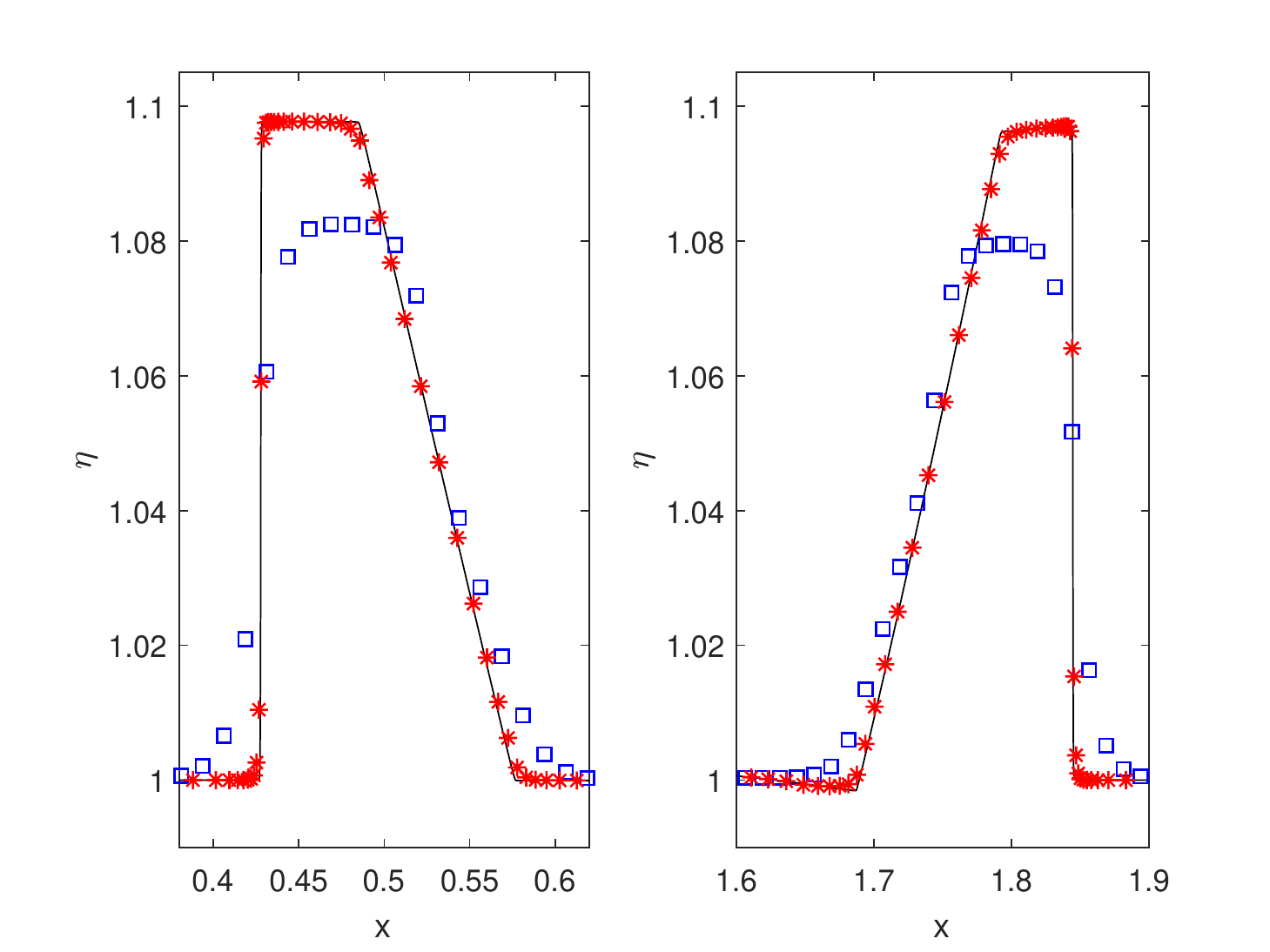}}
\subfigure[$\eta$: FM 480 vs MM 160]{
\includegraphics[width=0.35\textwidth,trim=10 0 40 10,clip]{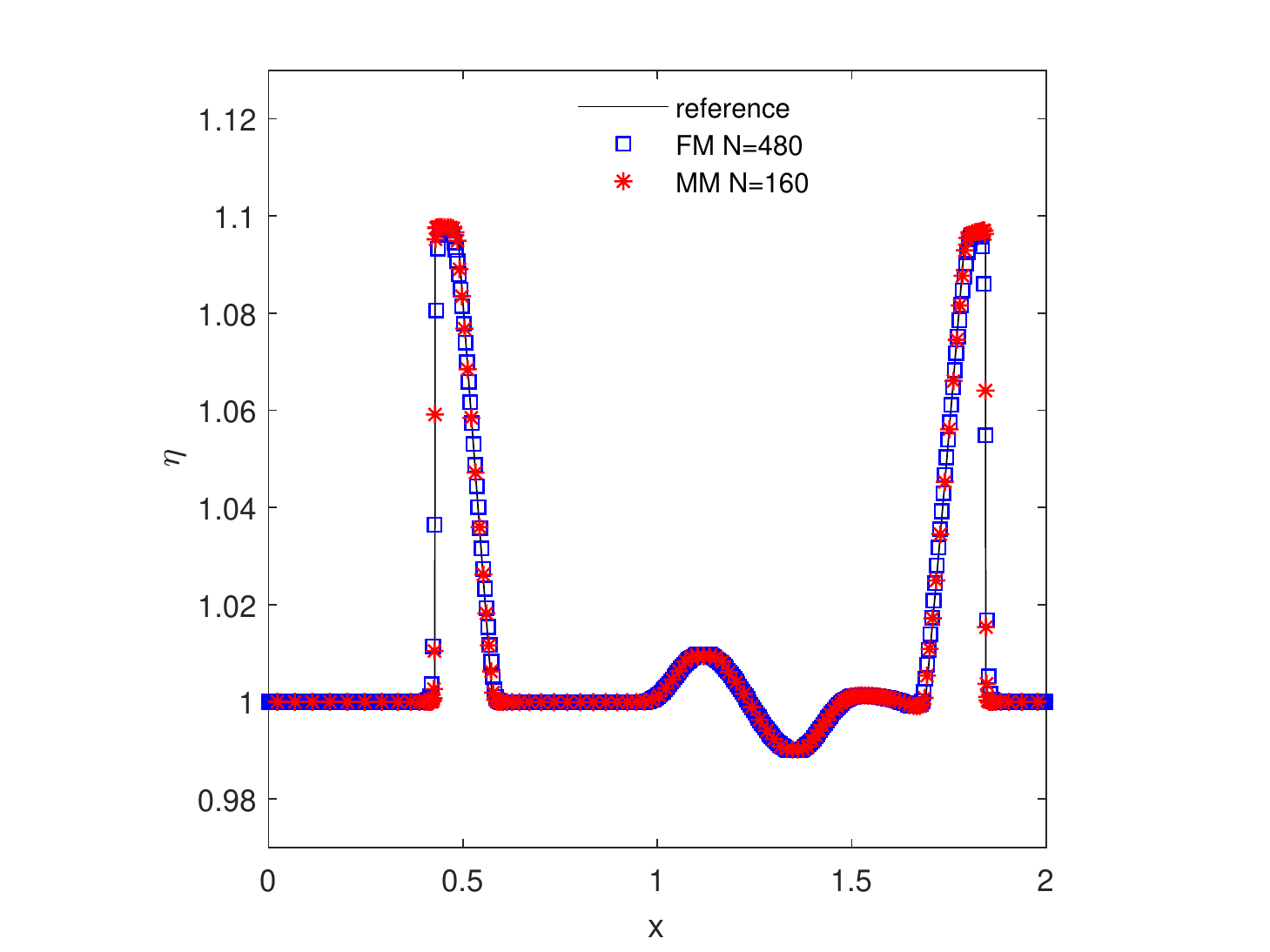}}
\subfigure[Close view of (c)]{
\includegraphics[width=0.35\textwidth,trim=10 0 39 10,clip]{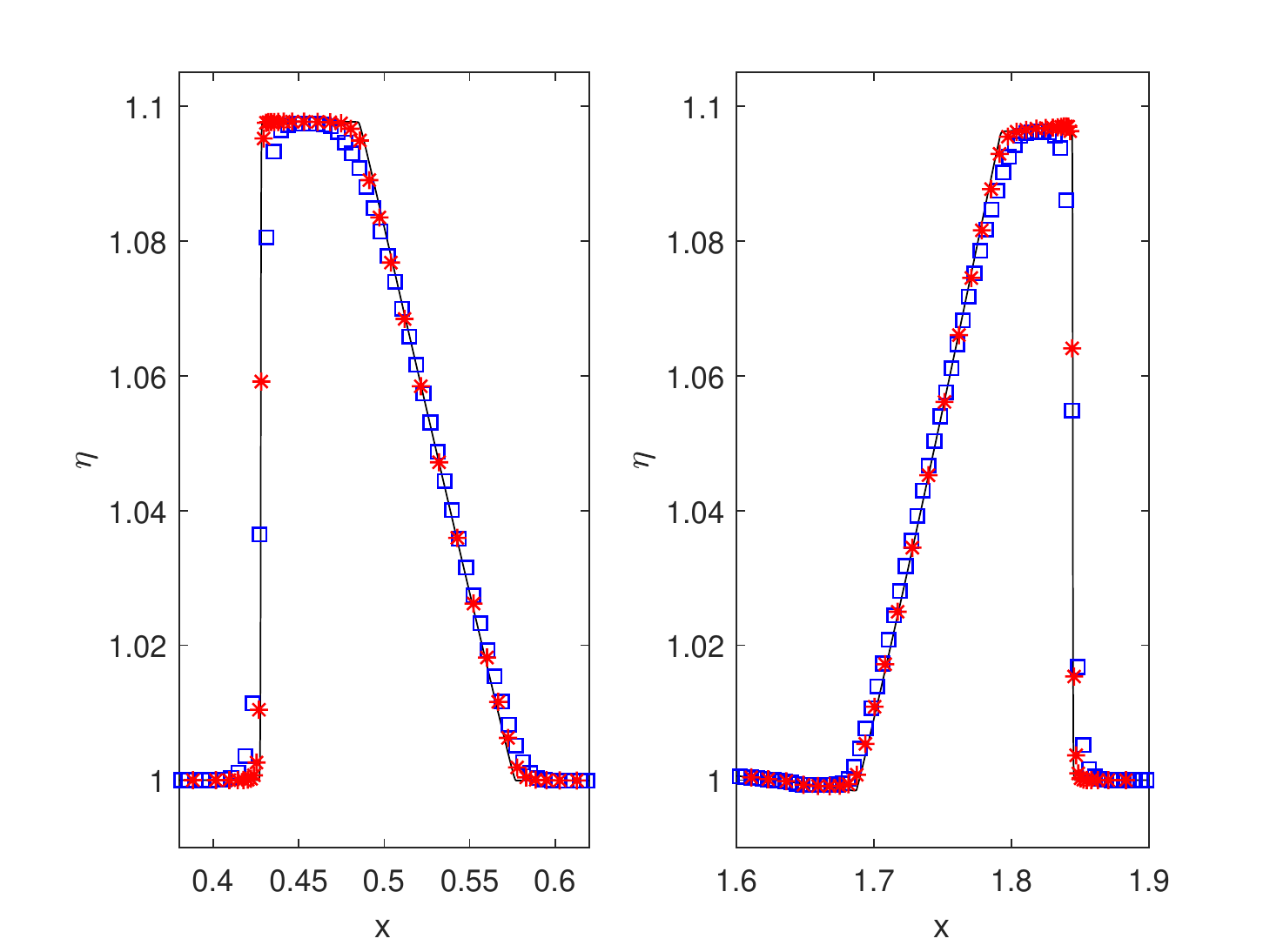}}
\caption{Example \ref{test3-1d}. The water surface level $\eta$ at $t=0.2$ obtained with the $P^2$-DG method and a moving mesh of $N=160$ is compared with those obtained with fixed meshes of $N=160$ and $N=480$ for a large pulse $\varepsilon=0.2$.}
\label{Fig:test3-1d-large-eta}
\end{figure}

\begin{figure}[H]
\centering
\subfigure[$hu$: FM 160 vs MM 160]{
\includegraphics[width=0.35\textwidth,trim=10 0 40 10,clip]{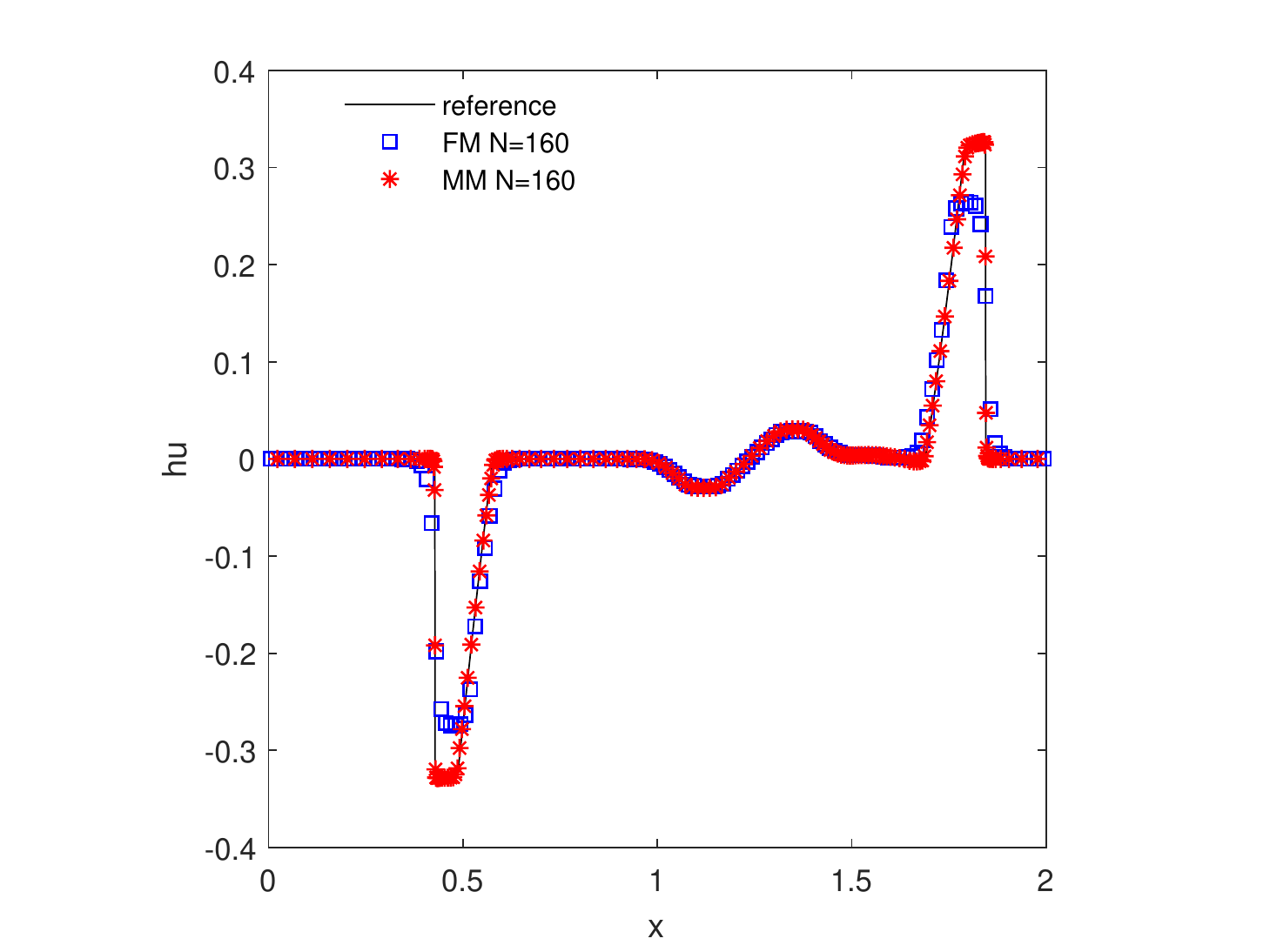}}
\subfigure[Close view of (a)]{
\includegraphics[width=0.35\textwidth,trim=10 0 39 10,clip]{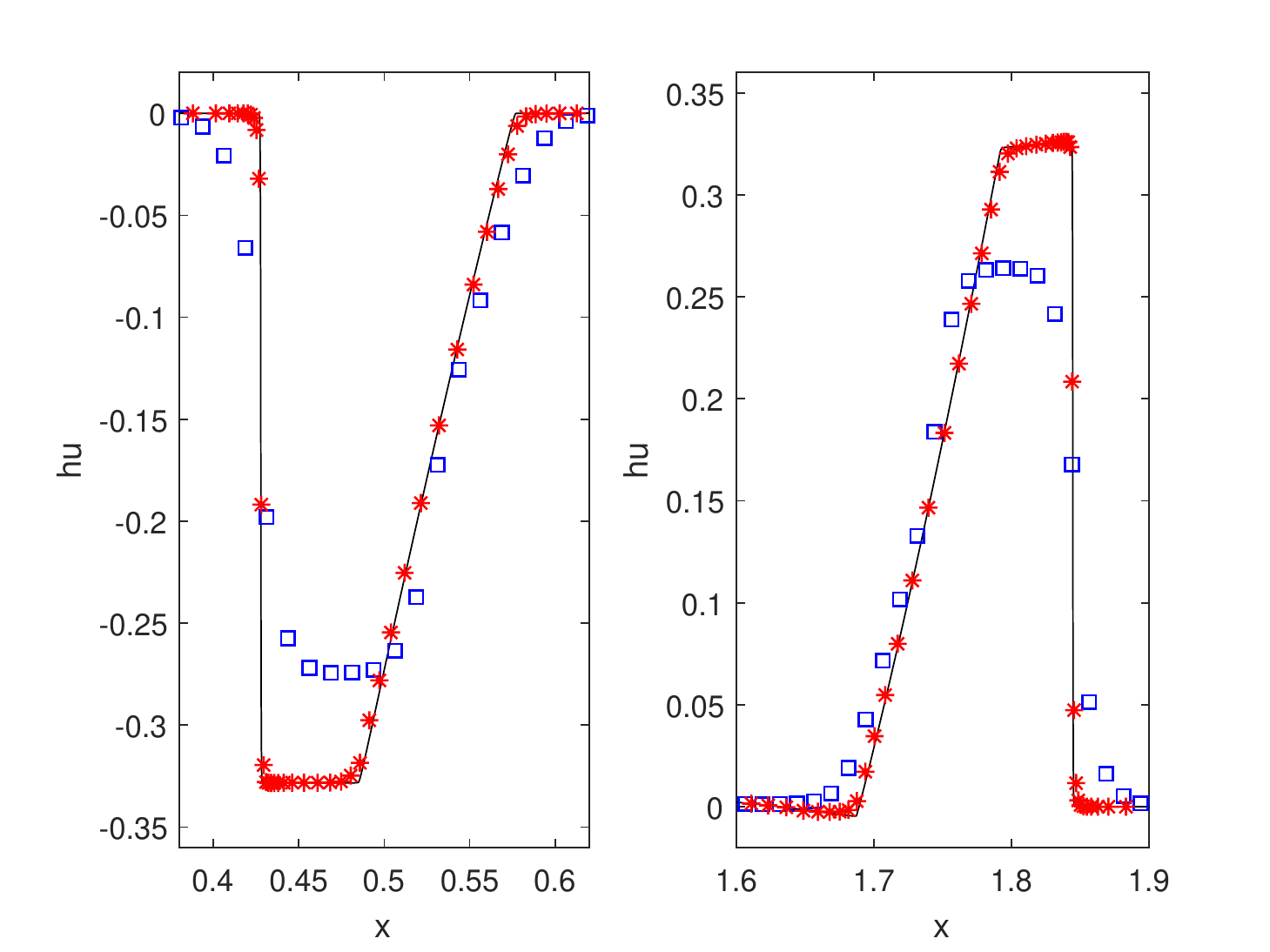}}
\subfigure[$hu$: FM 480 vs MM 160]{
\includegraphics[width=0.35\textwidth,trim=10 0 40 10,clip]{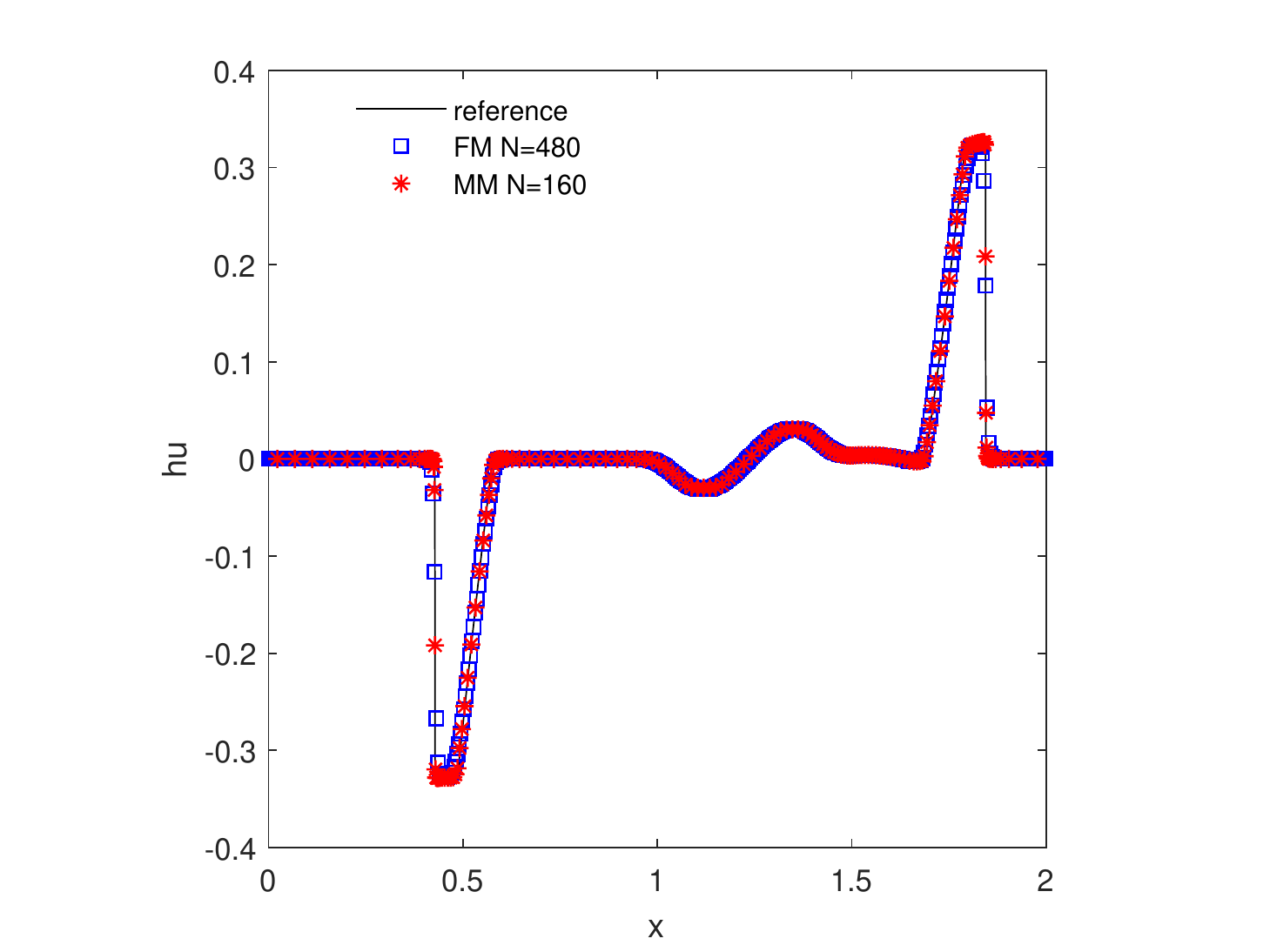}}
\subfigure[Close view of (c)]{
\includegraphics[width=0.35\textwidth,trim=10 0 39 10,clip]{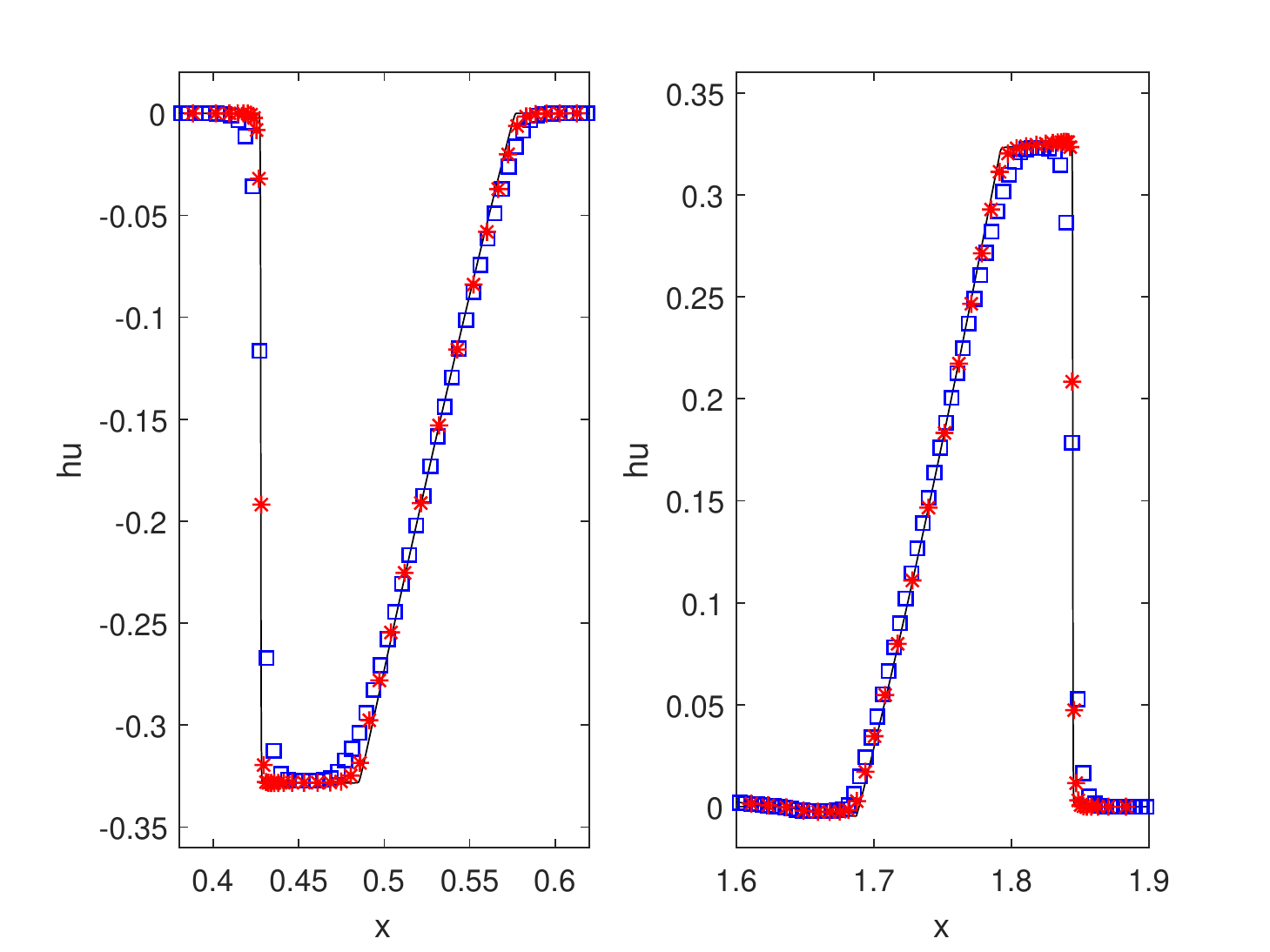}}
\caption{Example \ref{test3-1d}. The water discharge $hu$ at $t=0.2$ obtained with the $P^2$-DG method and a moving mesh of $N=160$ is compared with those obtained with fixed meshes of $N=160$ and $N=480$ for a large pulse $\varepsilon=0.2$.}
\label{Fig:test3-1d-large-hu}
\end{figure}

\begin{figure}[H]
\centering
\subfigure[$\eta$: FM 160 vs MM 160]{
\includegraphics[width=0.35\textwidth,trim=10 0 40 10,clip]{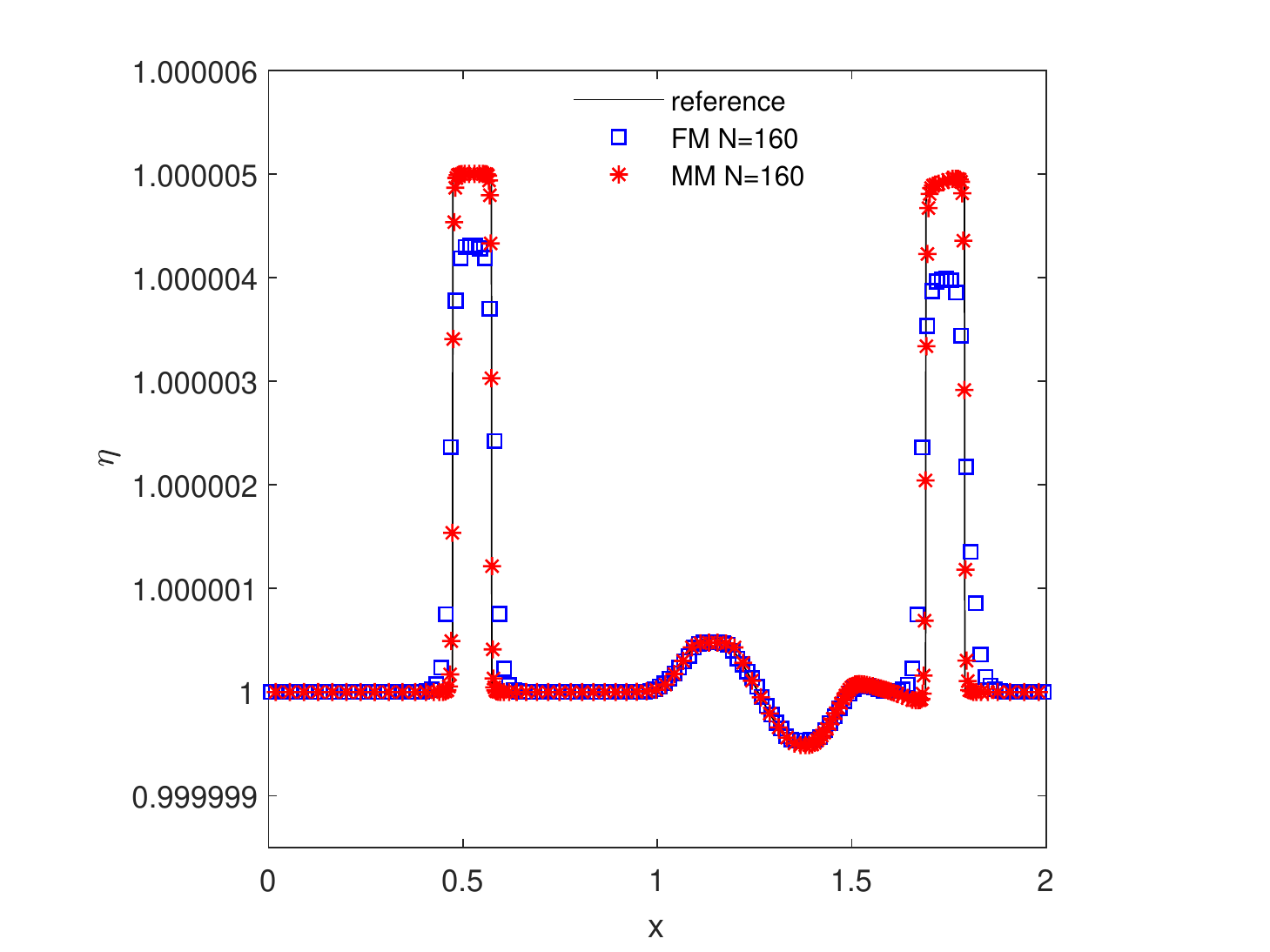}}
\subfigure[Close view of (a)]{
\includegraphics[width=0.35\textwidth,trim=10 0 39 10,clip]{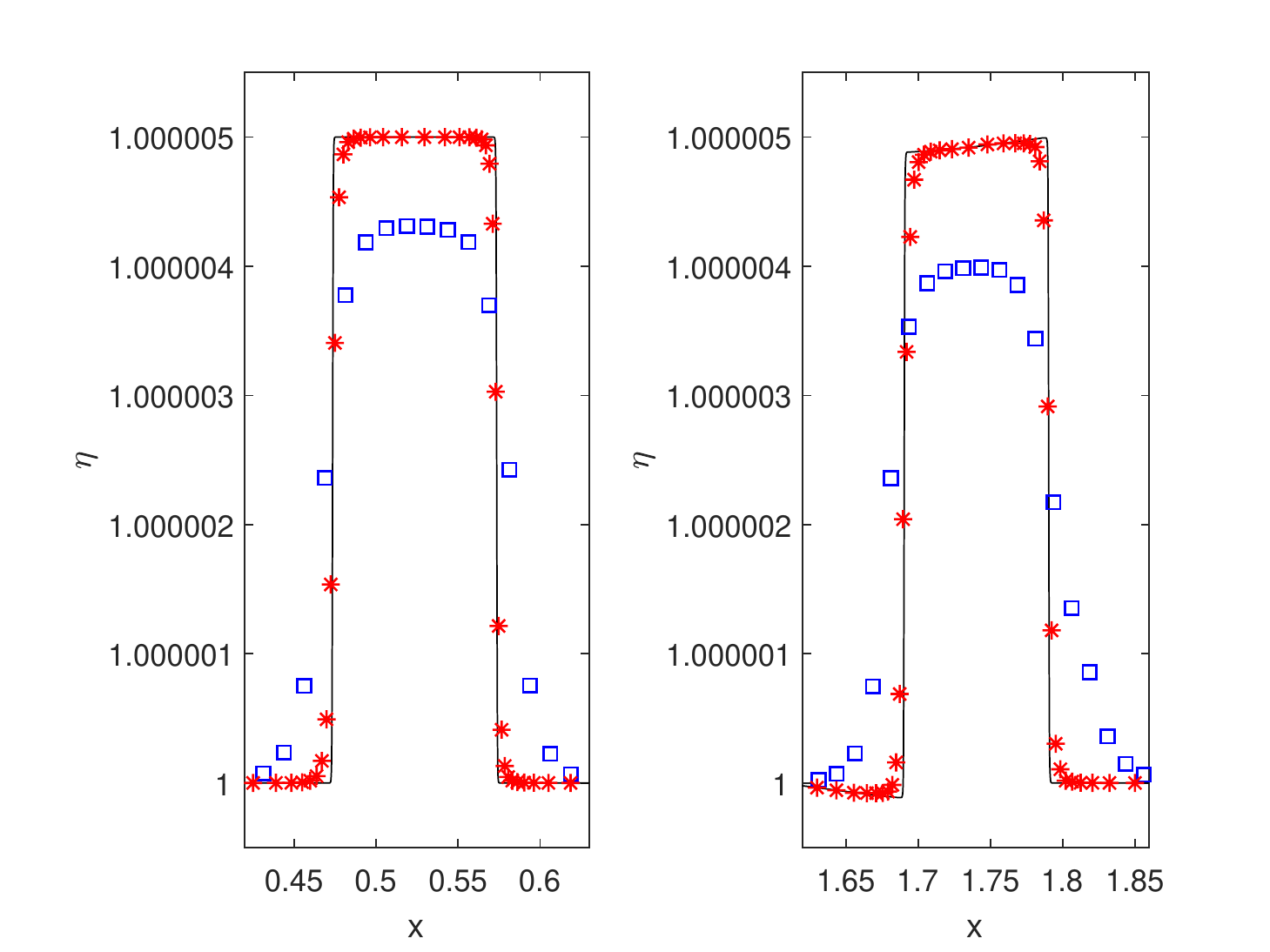}}
\subfigure[$\eta$: FM 480 vs MM 160]{
\includegraphics[width=0.35\textwidth,trim=10 0 40 10,clip]{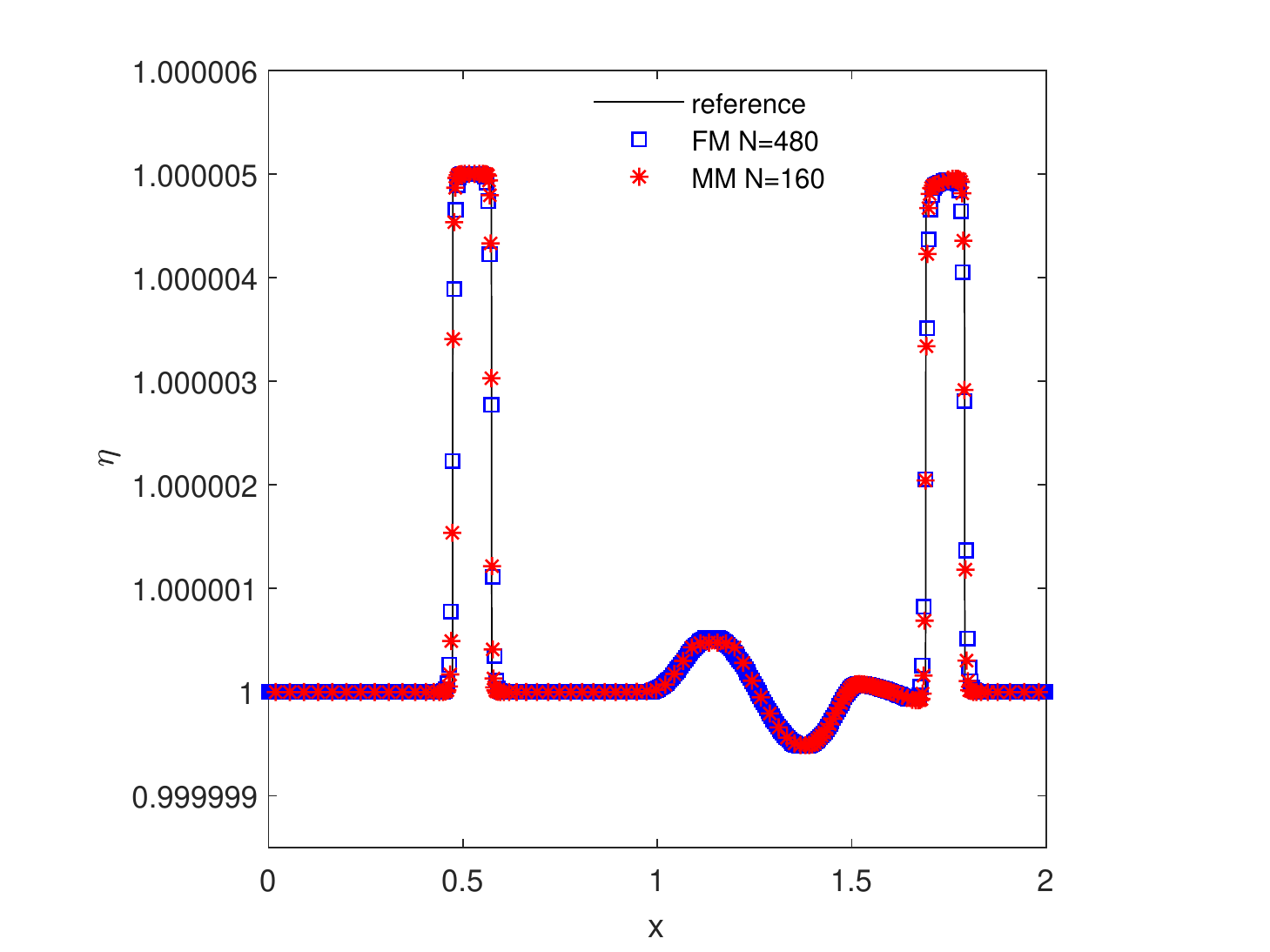}}
\subfigure[Close view of (c)]{
\includegraphics[width=0.35\textwidth,trim=10 0 39 10,clip]{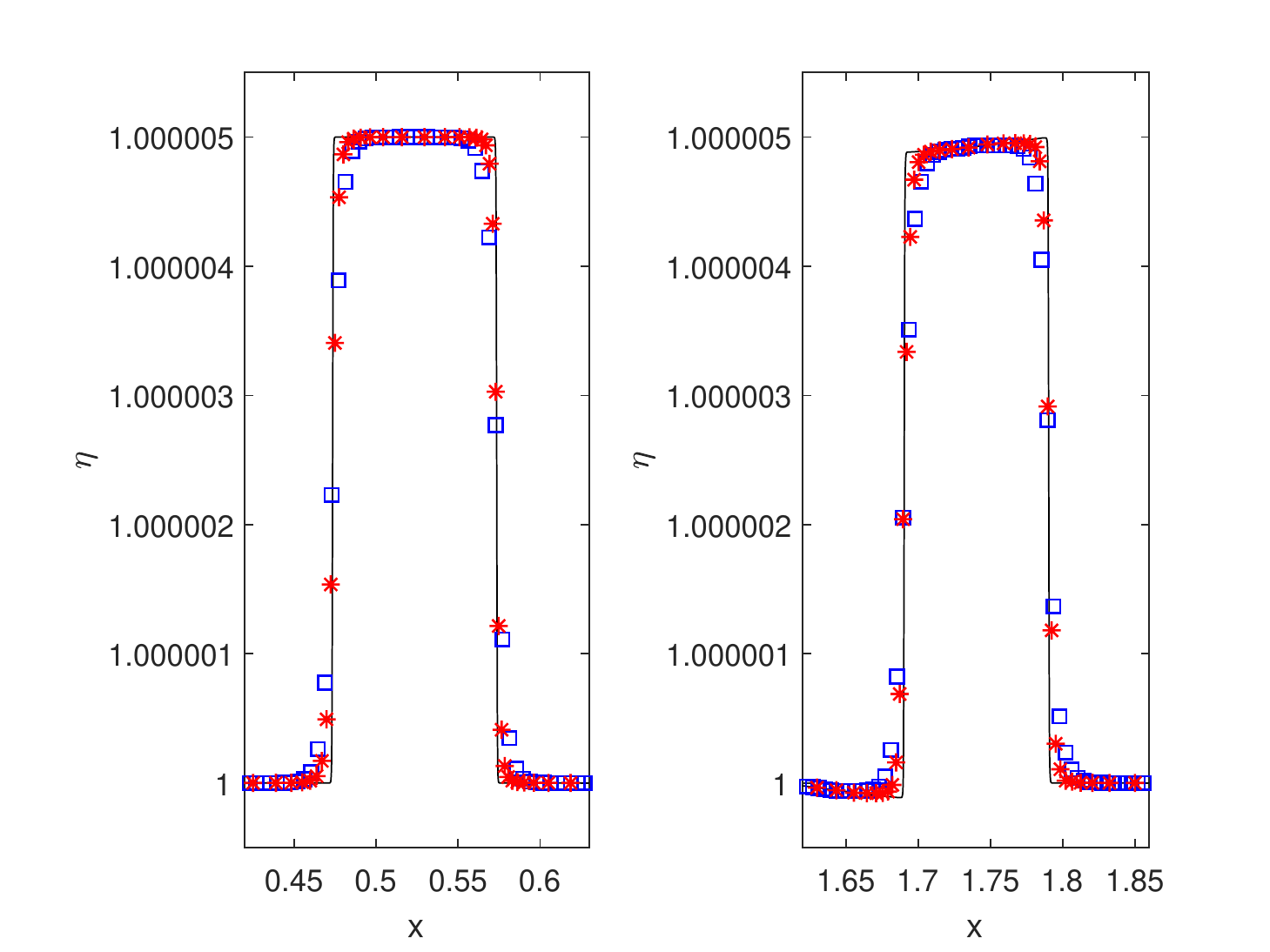}}
\caption{Example \ref{test3-1d}. The water surface level $\eta$ at $t=0.2$ obtained with the $P^2$-DG method and a moving mesh of $N=160$ is compared with those obtained with fixed meshes of $N=160$ and $N=480$ for a small pulse $\varepsilon=10^{-5}$.}
\label{Fig:test3-1d-small-eta}
\end{figure}

\begin{figure}[H]
\centering
\subfigure[$hu$: FM 160 vs MM 160]{
\includegraphics[width=0.35\textwidth,trim=10 0 40 10,clip]{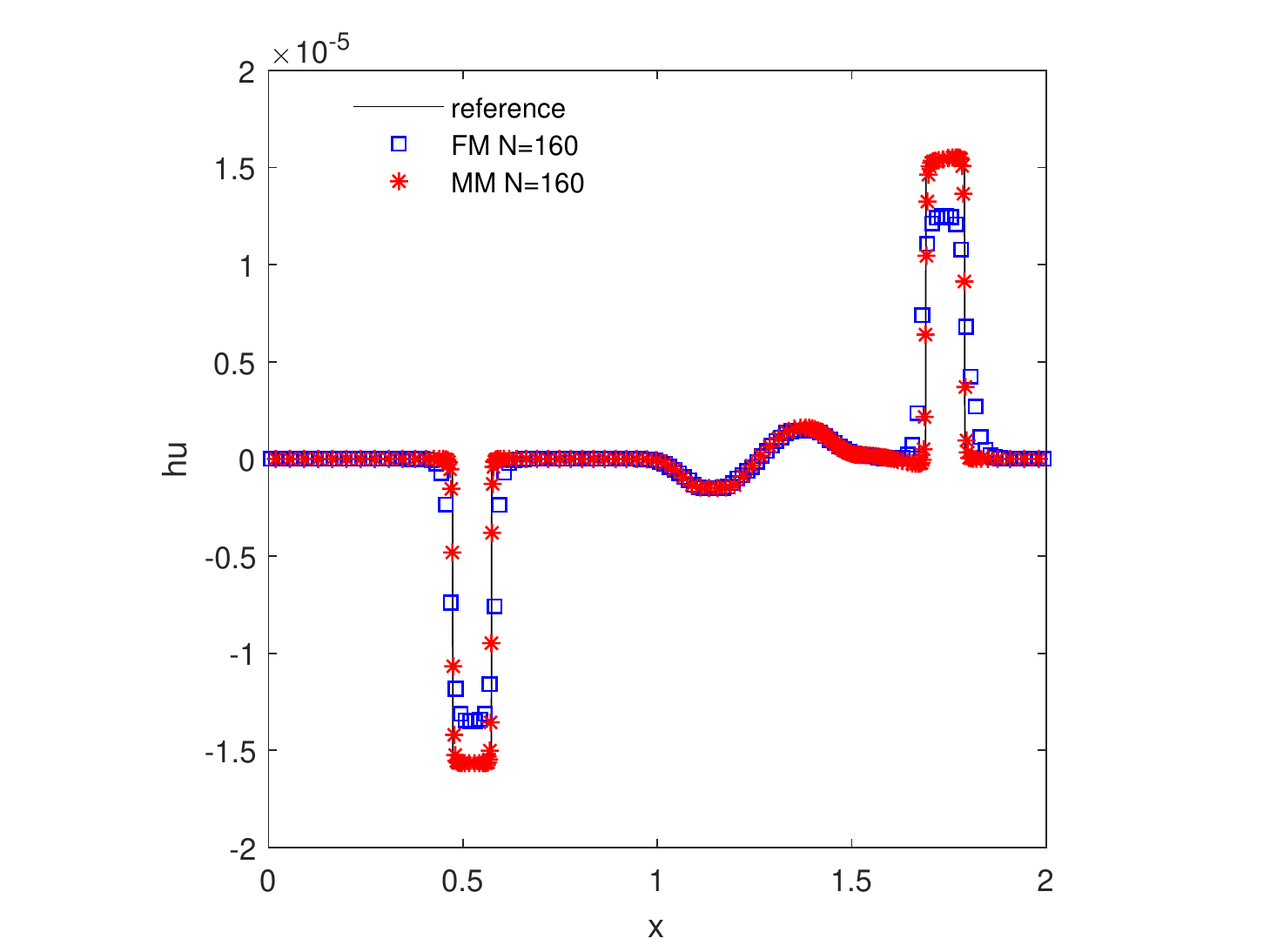}}
\subfigure[Close view of (a)]{
\includegraphics[width=0.35\textwidth,trim=10 0 39 10,clip]{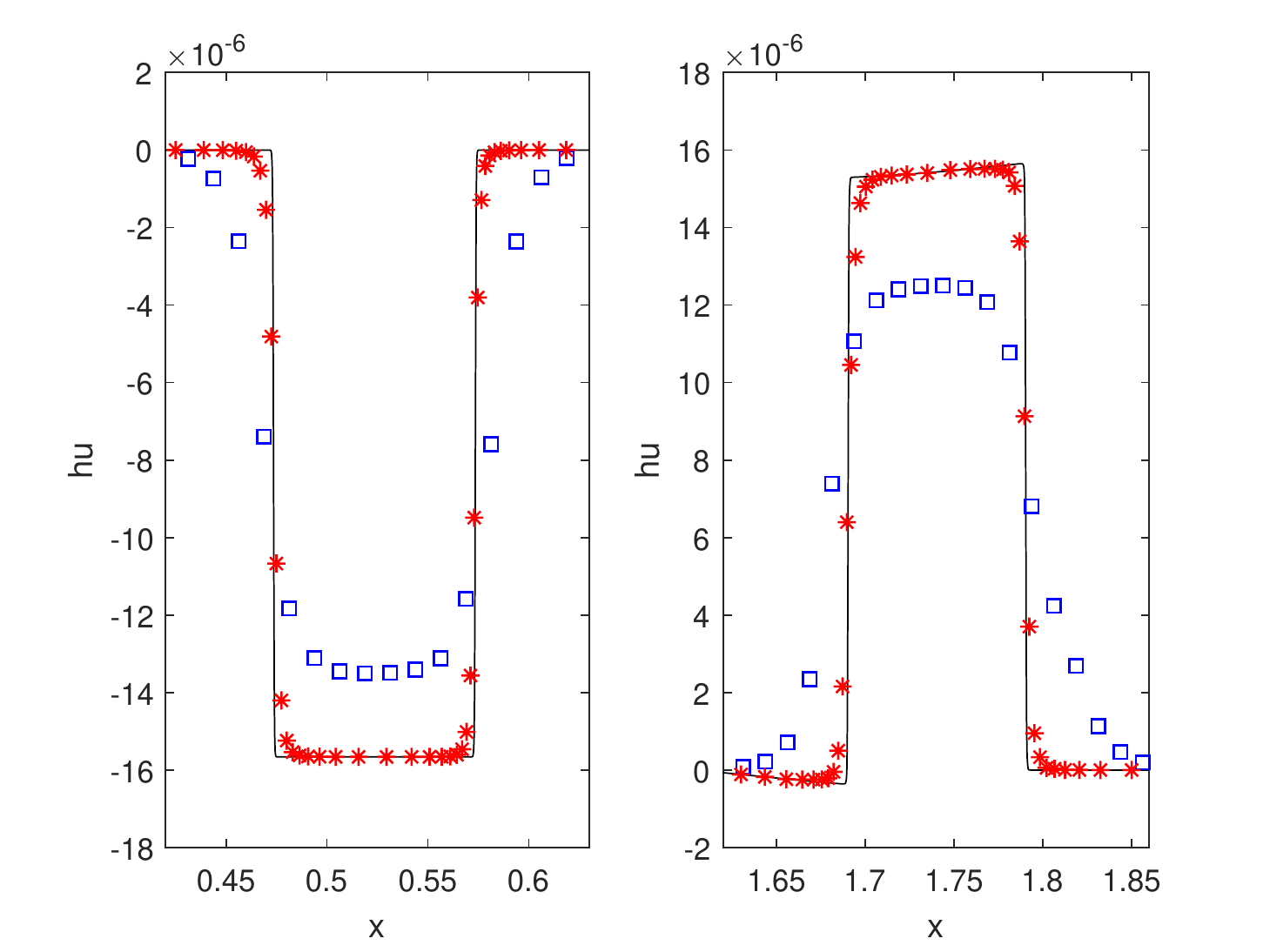}}
\subfigure[$hu$: FM 480 vs MM 160 ]{
\includegraphics[width=0.35\textwidth,trim=10 0 40 10,clip]{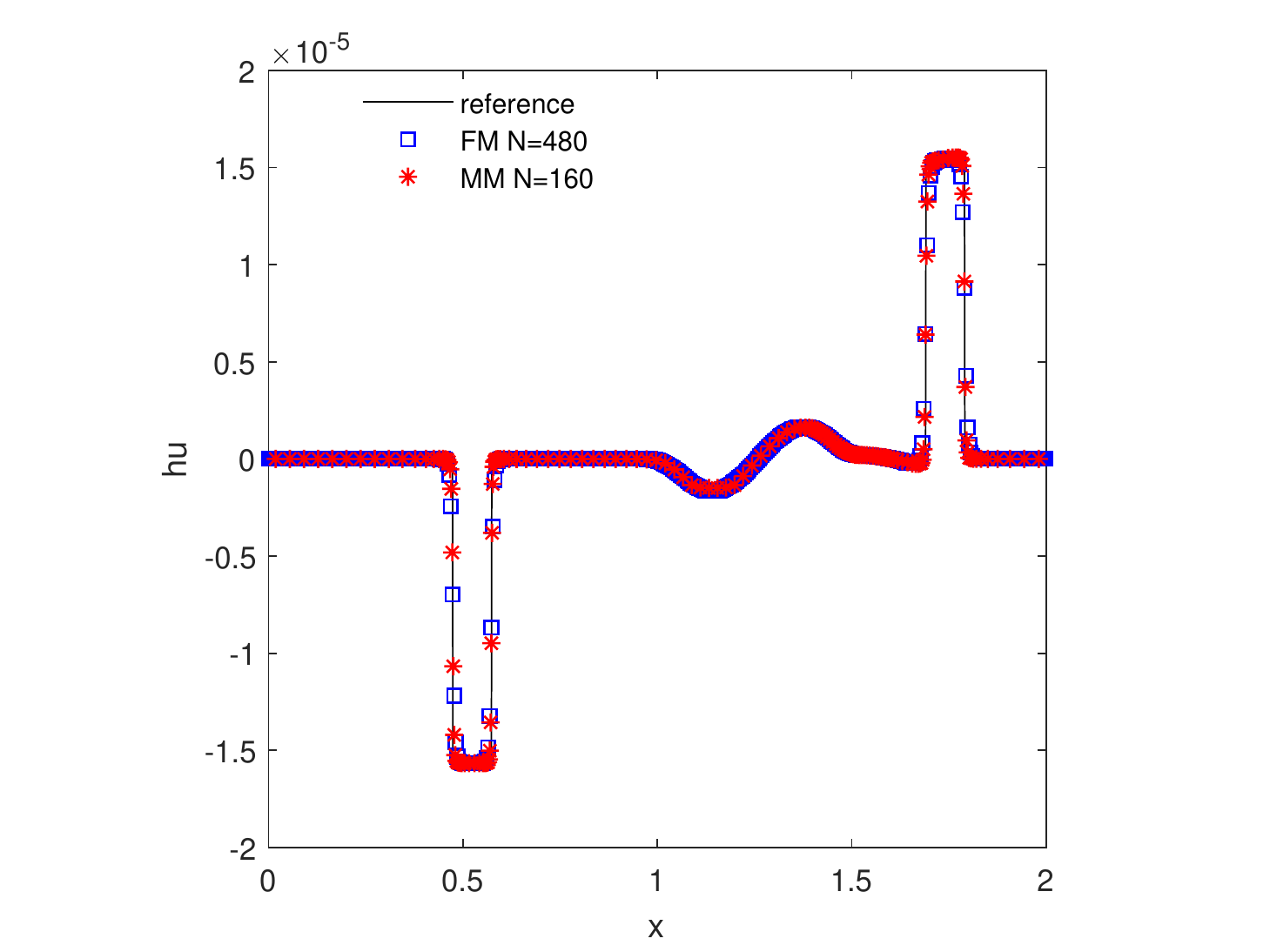}}
\subfigure[Close view of (c)]{
\includegraphics[width=0.35\textwidth,trim=10 0 39 10,clip]{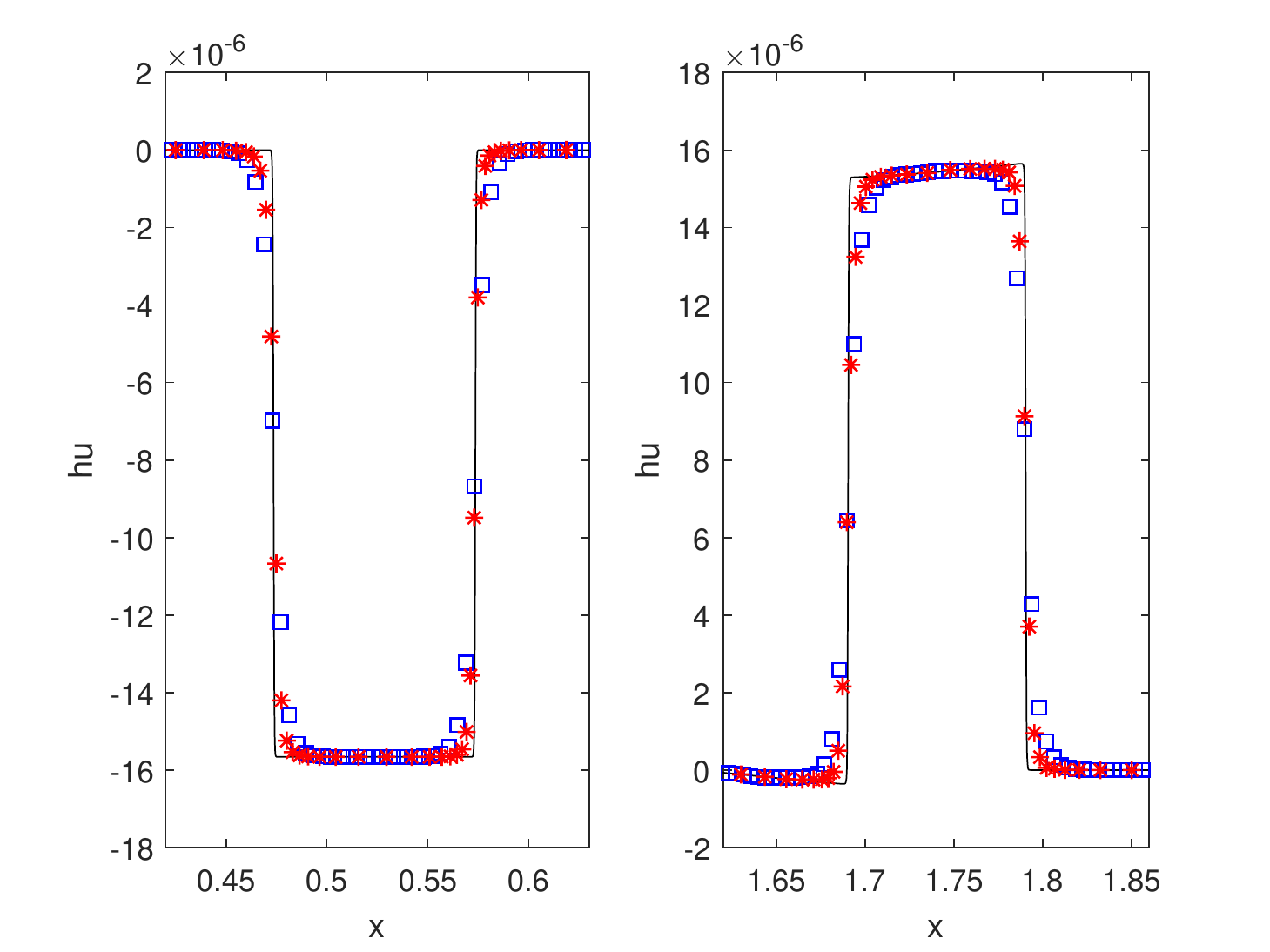}}
\caption{Example \ref{test3-1d}. The water discharge $hu$ at $t=0.2$ obtained with the $P^2$-DG method and a moving mesh of $N=160$ is compared with those obtained with fixed meshes of $N=160$ and $N=480$ for a small pulse $\varepsilon=10^{-5}$.}
\label{Fig:test3-1d-small-hu}
\end{figure}


\begin{figure}[H]
\centering
\subfigure[Initial surface and bottom]{
\includegraphics[width=0.35\textwidth,trim=20 0 35 10,clip]{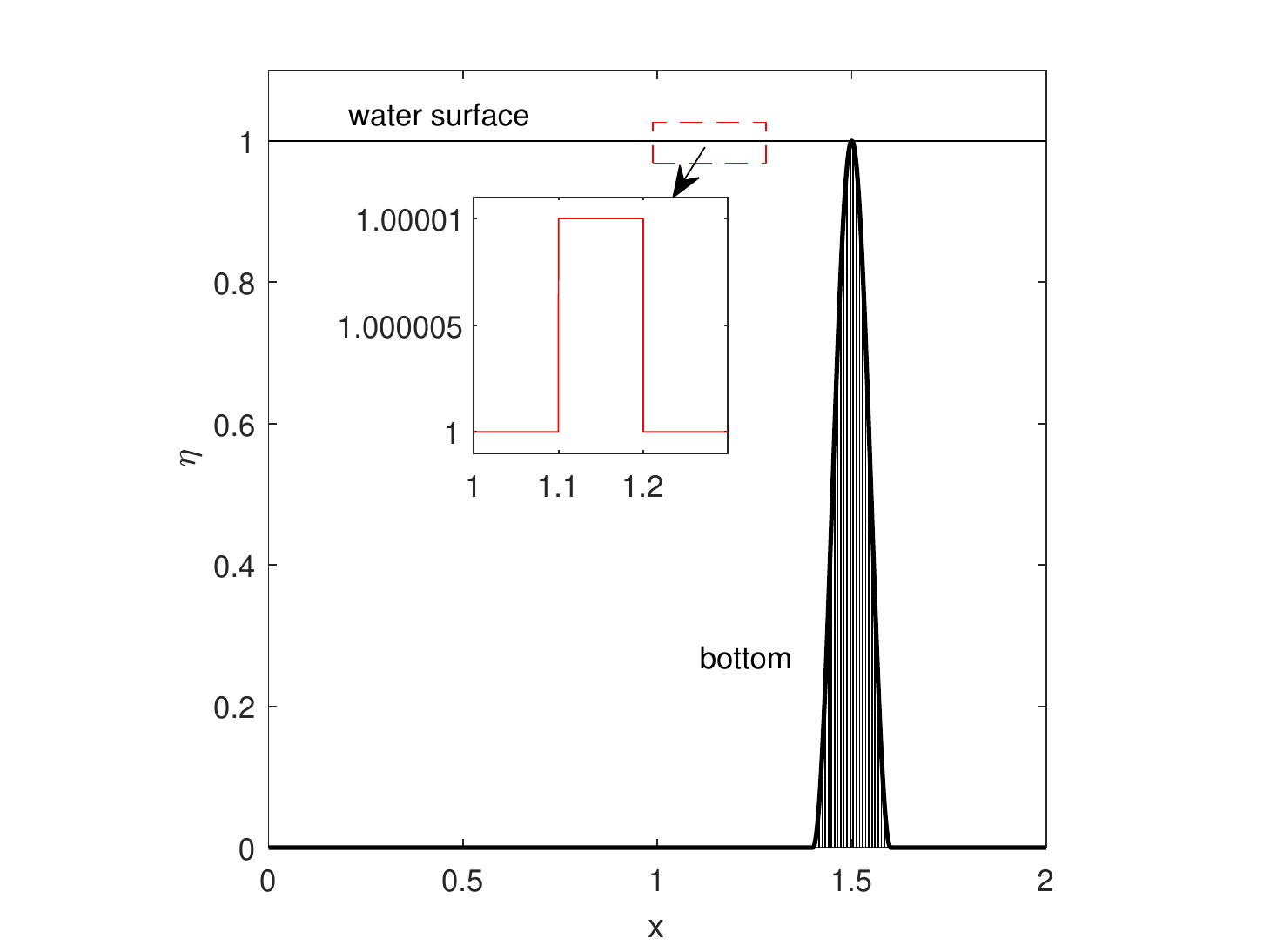}}
\subfigure[Mesh trajectories]{
\includegraphics[width=0.35\textwidth,trim=20 0 35 10,clip]{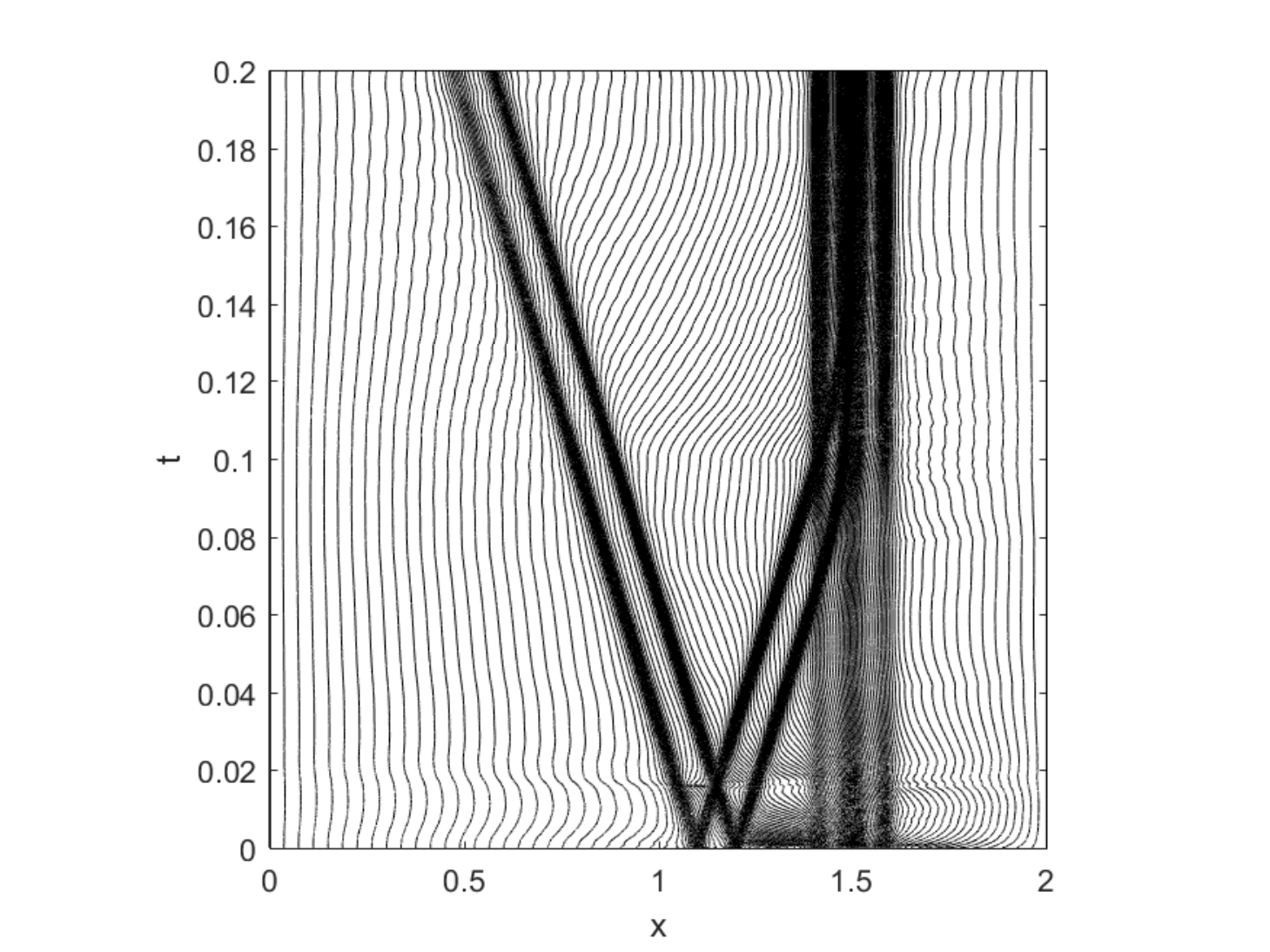}}
\caption{Example \ref{test3-1d}.
(a) The initial water surface $\eta$ and the bottom $B$ \eqref{test3-1d-B2} for the small perturbation test with a dry region.
(b) The mesh trajectories obtained with $P^2$ QLMM-DG method of $N=160$.}
\label{Fig:PP-test3-1d-initial}
\end{figure}

\begin{figure}[H]
\centering
\subfigure[$\eta$: FM 160 vs MM 160]{
\includegraphics[width=0.35\textwidth,trim=20 0 35 10,clip]
{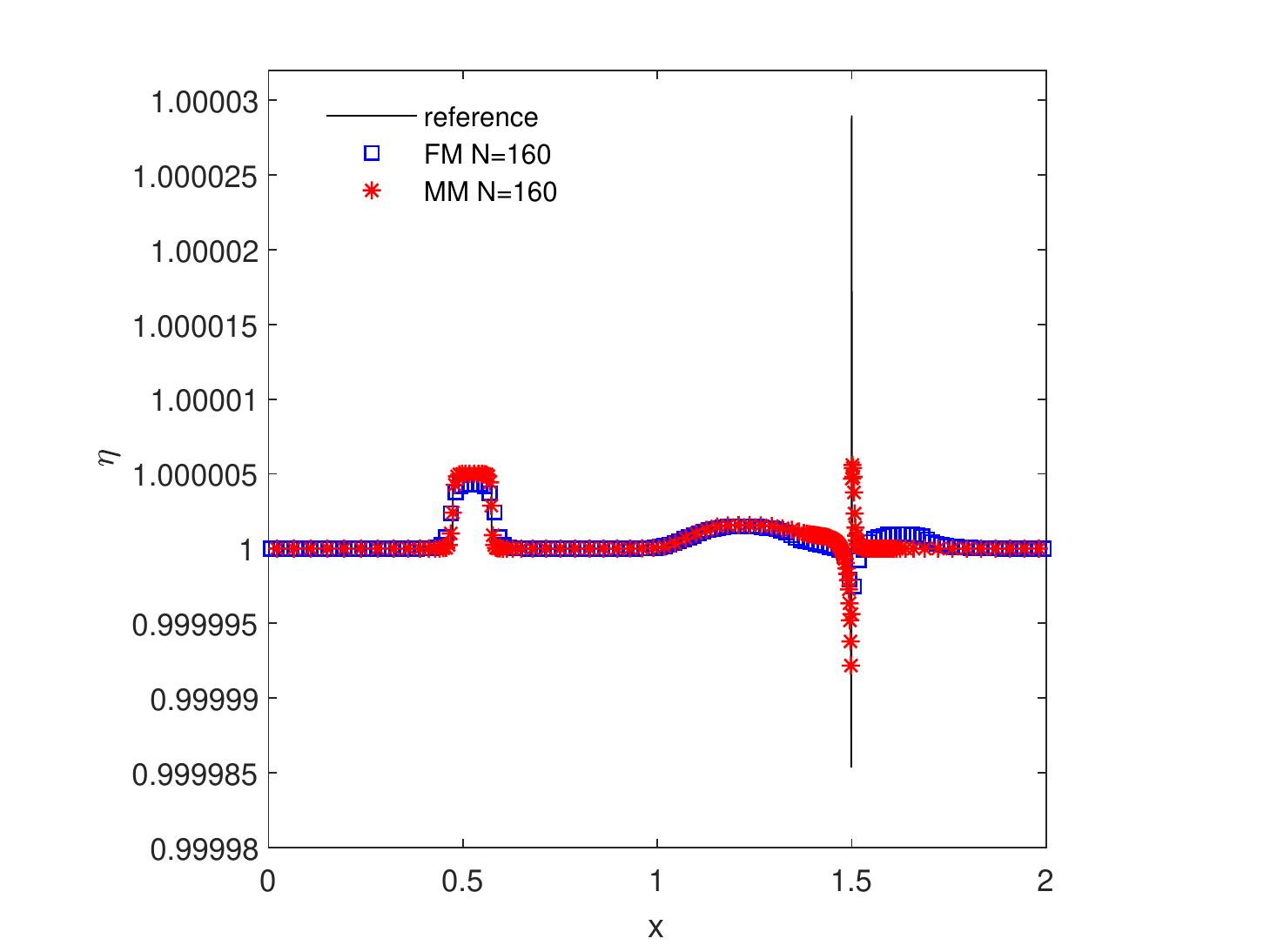}}
\subfigure[Close view of (a)]{
\includegraphics[width=0.35\textwidth,trim=20 0 35 10,clip]{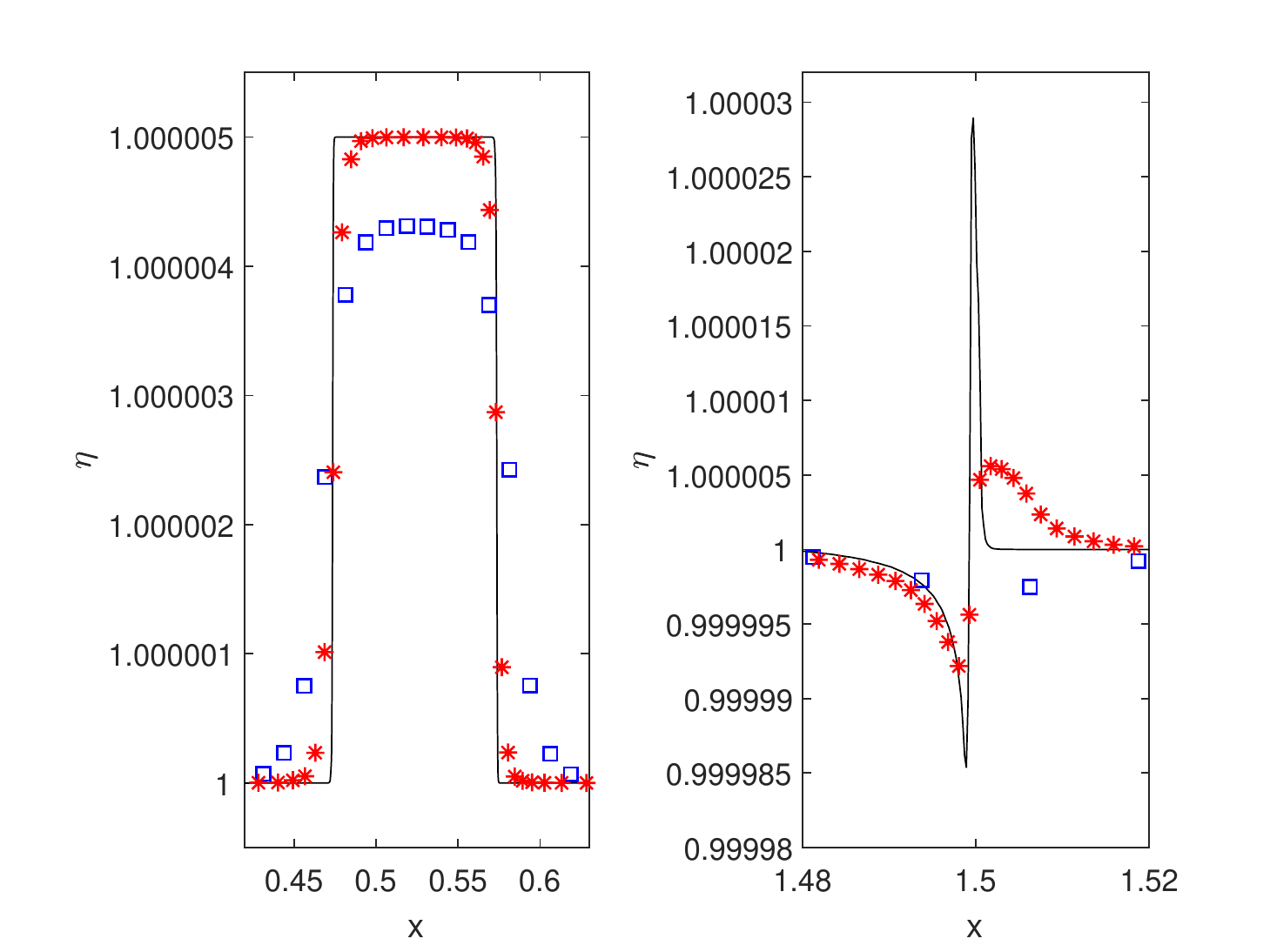}}
\subfigure[$\eta$: FM 640 vs MM 160]{
\includegraphics[width=0.35\textwidth,trim=20 0 35 10,clip]
{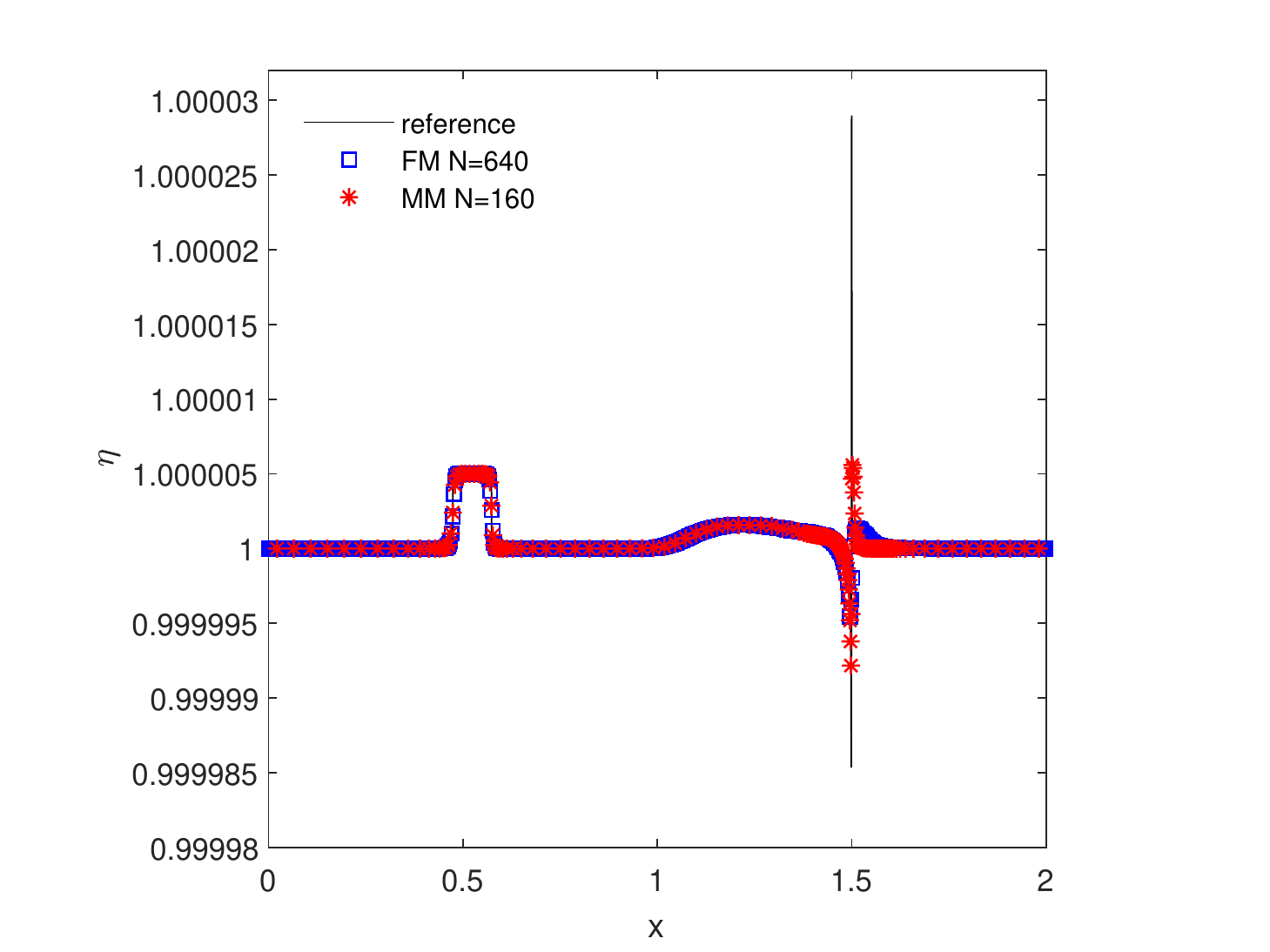}}
\subfigure[Close view of (c)]{
\includegraphics[width=0.35\textwidth,trim=20 0 35 10,clip]{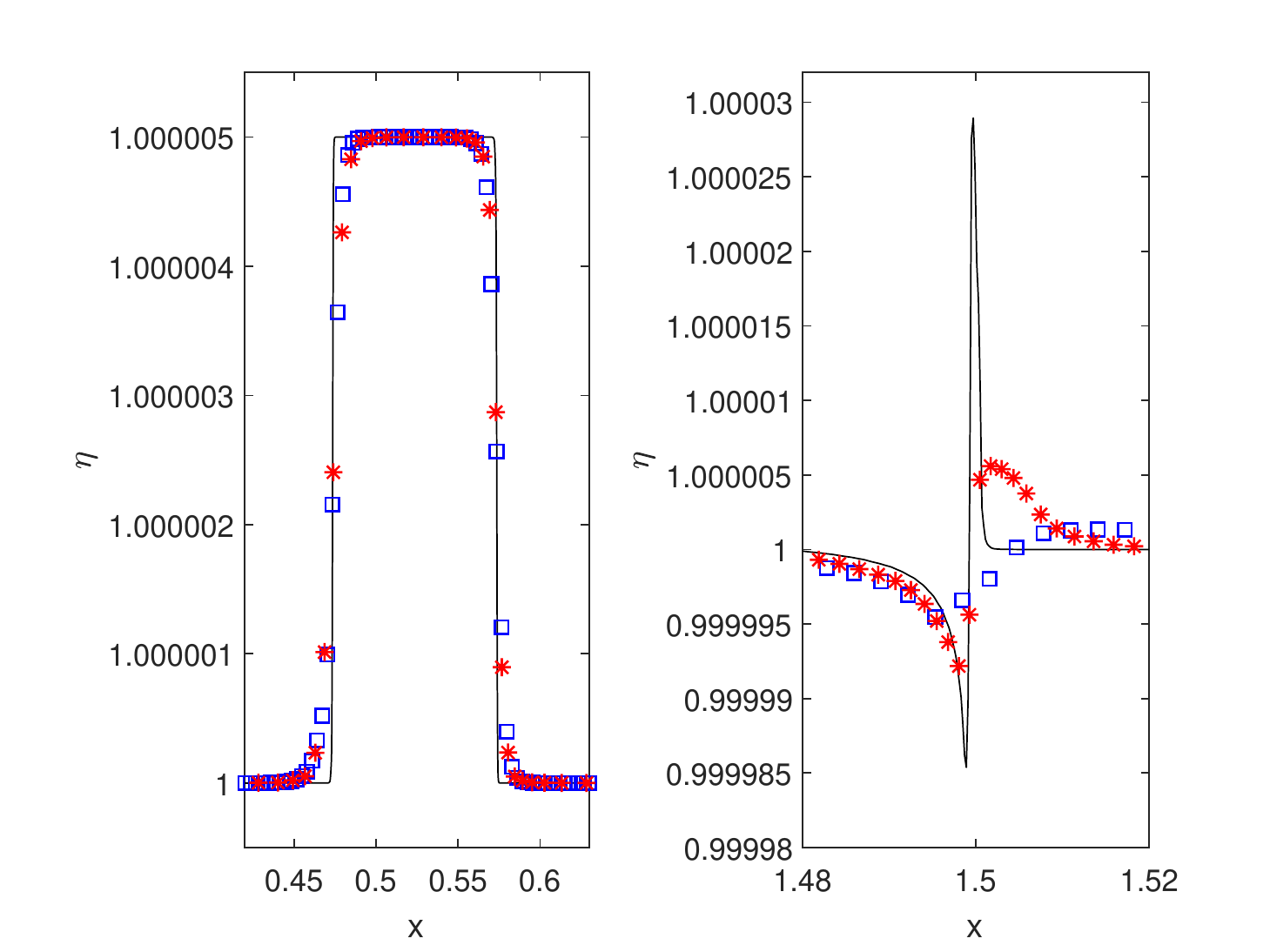}}
\caption{Example \ref{test3-1d} with the bottom topography \eqref{test3-1d-B2} with a dry region.
The water surface $\eta$ at $t=0.2$ obtained with $P^2$-DG and a moving mesh of $N=160$ are compared with those obtained with a fixed mesh of $N=160$ and $N=640$.}
\label{Fig:PP-test3-1d-eta}
\end{figure}

\begin{figure}[H]
\centering
\subfigure[$hu$: FM 160 vs MM 160]{
\includegraphics[width=0.35\textwidth,trim=20 0 35 10,clip]
{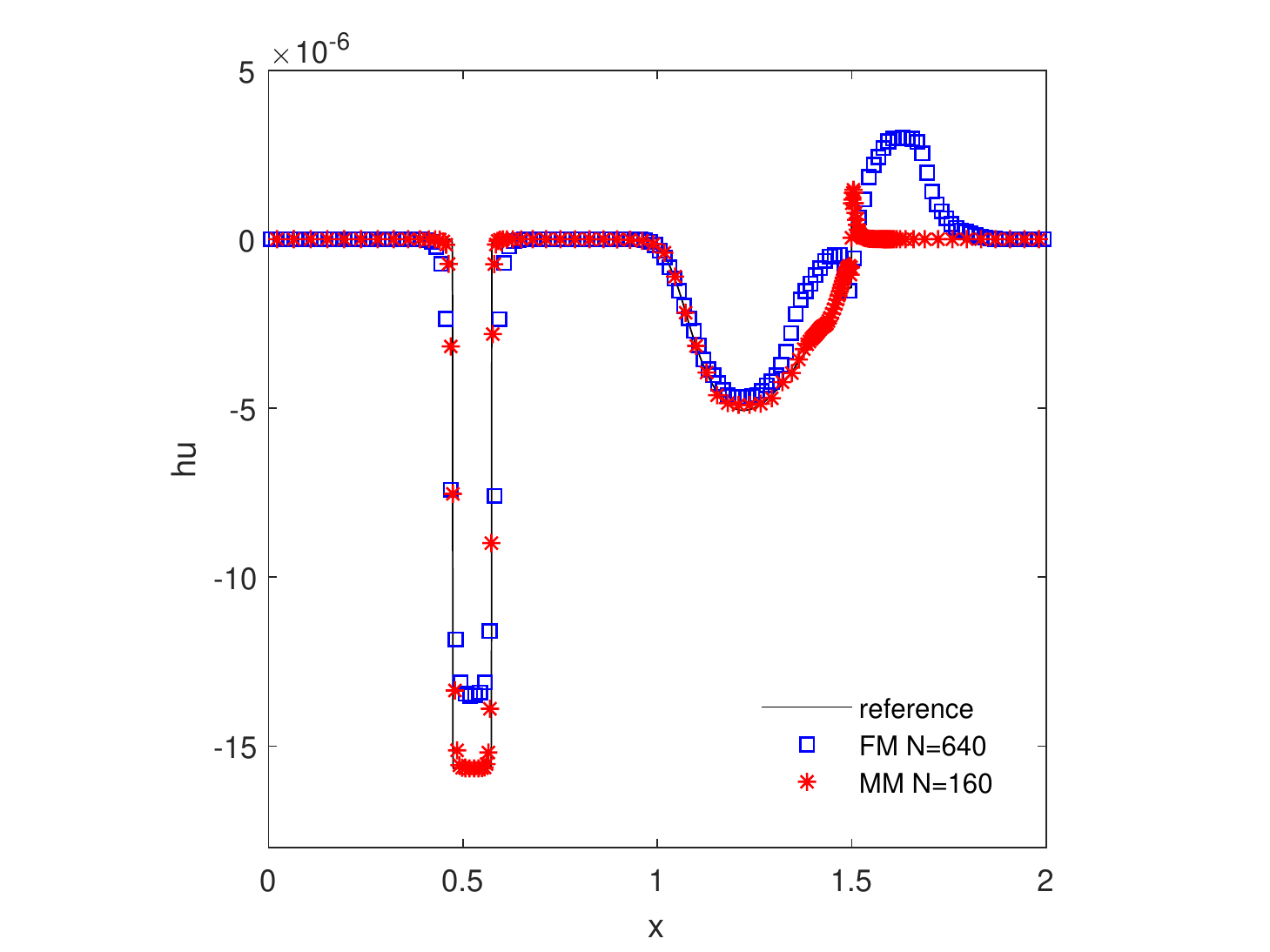}}
\subfigure[Close view of (a)]{
\includegraphics[width=0.35\textwidth,trim=20 0 35 10,clip]{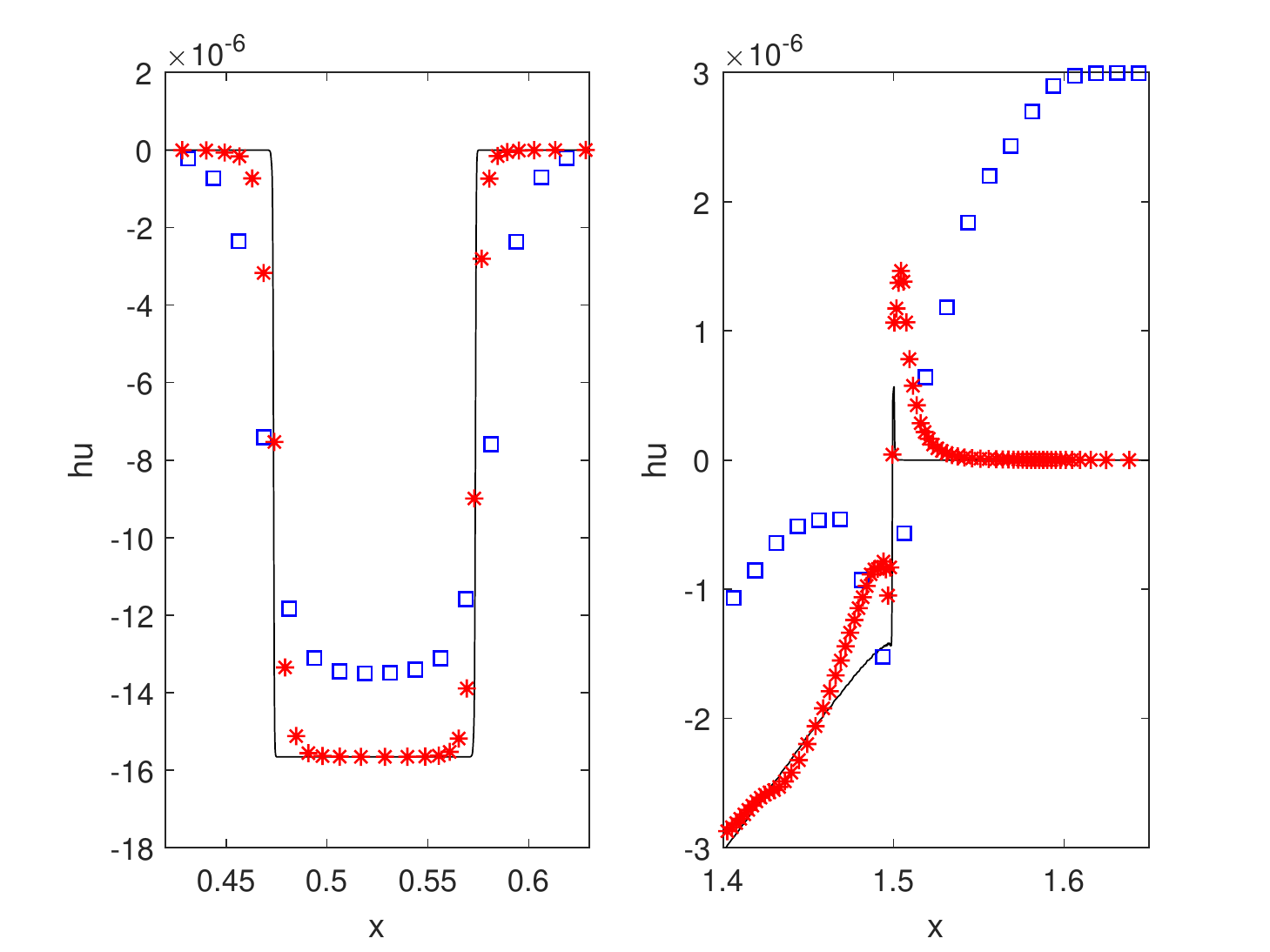}}
\subfigure[$hu$: FM 640 vs MM 160]{
\includegraphics[width=0.35\textwidth,trim=20 0 35 10,clip]
{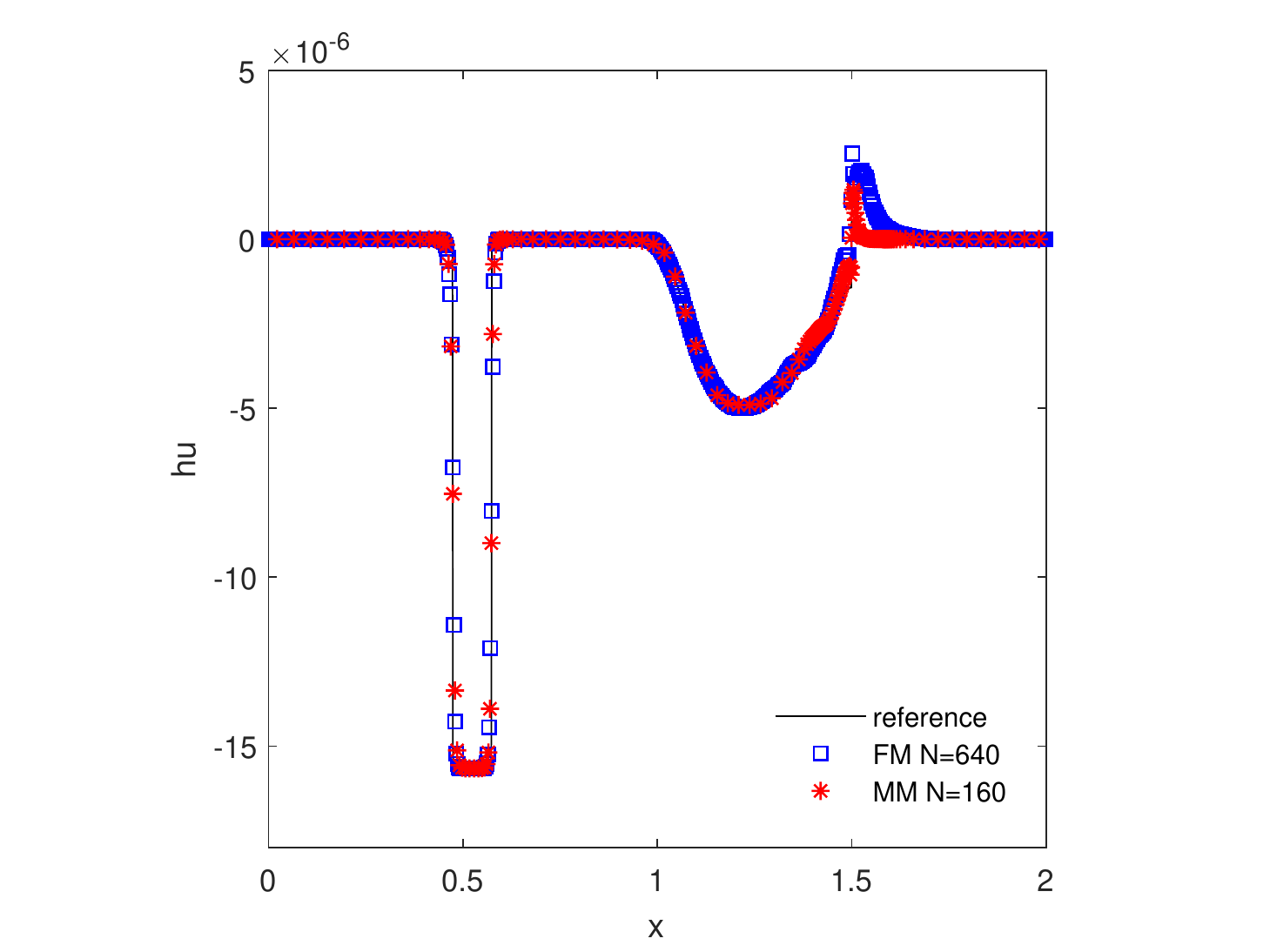}}
\subfigure[Close view of (c)]{
\includegraphics[width=0.35\textwidth,trim=20 0 35 10,clip]{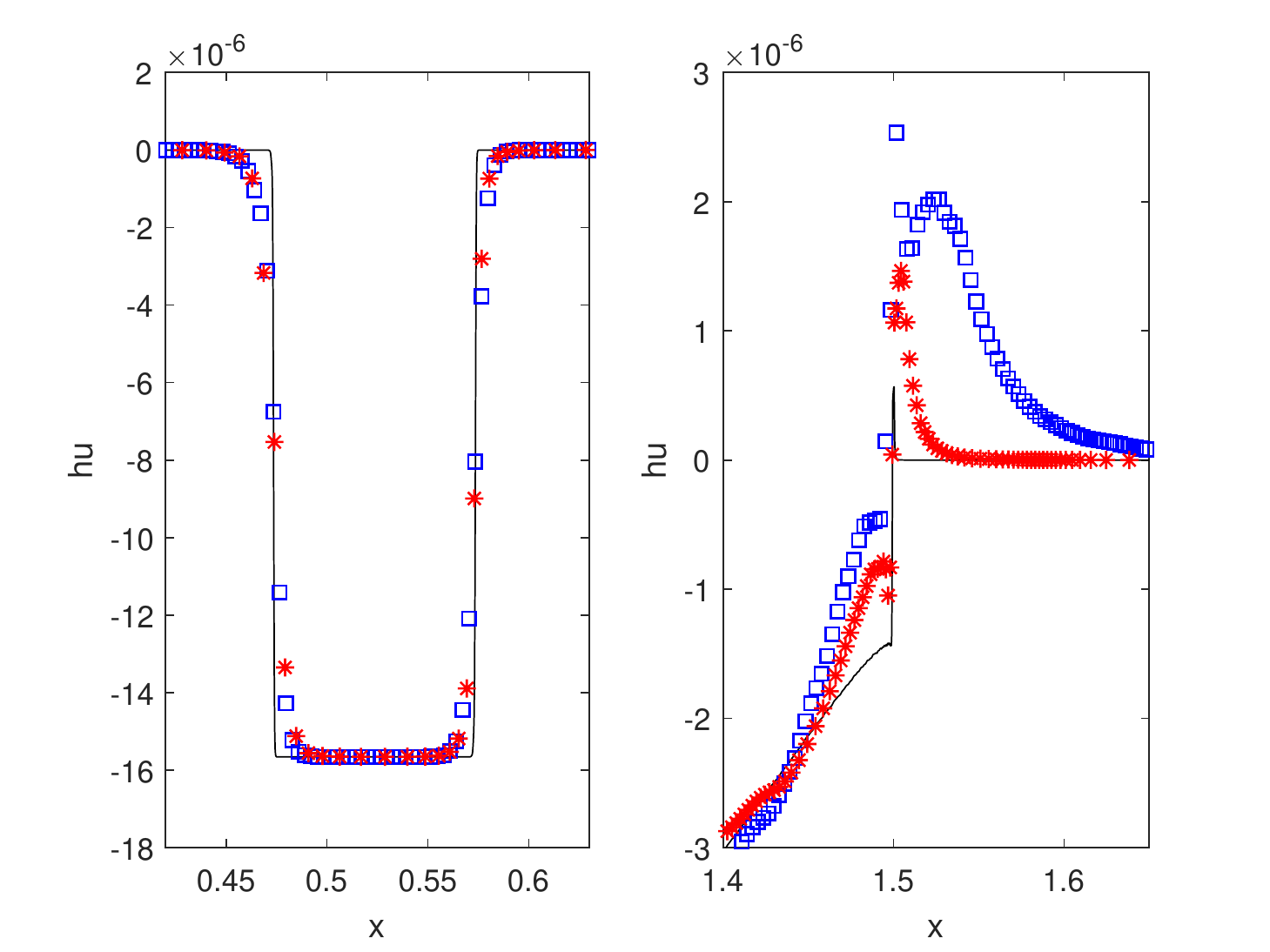}}
\caption{Example \ref{test3-1d} with the bottom topography \eqref{test3-1d-B2} with a dry region.
The water discharge $hu$ at $t=0.2$ obtained with $P^2$-DG and a moving mesh of $N=160$ are compared with those obtained with a fixed mesh of $N=160$ and $N=640$.}
\label{Fig:PP-test3-1d-hu}
\end{figure}

\begin{example}\label{test6-1d}
(The rarefaction and shock waves test for the 1D SWEs with wavy bottom topography.)
\end{example}
In this example we compute the 1D SWEs with a wavy bottom topography \cite{Tang-2004}
\begin{equation}
\label{B-3}
B(x)=
\begin{cases}
0.3 \cos^{30}(\frac{\pi}{2}(x-1)),& \text{for}~0\leq x\leq 2\\
0,& \text{otherwise}.
\end{cases}
\end{equation}
The initial conditions are
\begin{equation*}
\eta(x,0)=
\begin{cases}
2,& \text{for}~x\in (-10,1)\\
0.35,& \text{for}~x\in (1, 10) \\
\end{cases}
\quad \quad
u(x,0)=
\begin{cases}
1,& \text{for}~x\in (-10,1)\\
0,& \text{for}~x\in (1, 10).\\
\end{cases}
\end{equation*}
We choose the transmissive boundary conditions and compute the solution up to $T=1$.
The solution contains several interesting features, including a rarefaction wave traveling left
and two hydraulic jumps/shocks propagating right.

The mesh trajectories ($N=160$) are plotted in Fig.~\ref{Fig:test6-1d-mesh},
showing that the mesh points concentrate properly around the rarefaction, the hydraulic jumps/shocks,
and the region where $B$ is non-flat.
Figs.~\ref{Fig:test6-1d-eta} and \ref{Fig:test6-1d-hu} show the water surface level $\eta$ and water discharge $hu$ at $t=1$ obtained with $P^2$-DG and a moving mesh of $N=160$ and fixed meshes of $N=160$ and $N=1280$.
It can be seen that the moving mesh solutions of $N=160$ are more accurate than those
with a fixed mesh of $N=160$ and comparable with that with the fixed mesh of $N =1280$.
Moreover, the QLMM-DG method does a good job in resolving the shock near $x=2$ which
is known to be a difficult structure for a fixed-mesh method to resolve.
\begin{figure}[H]
\centering
\includegraphics[width=0.35\textwidth,trim=40 0 40 10,clip]{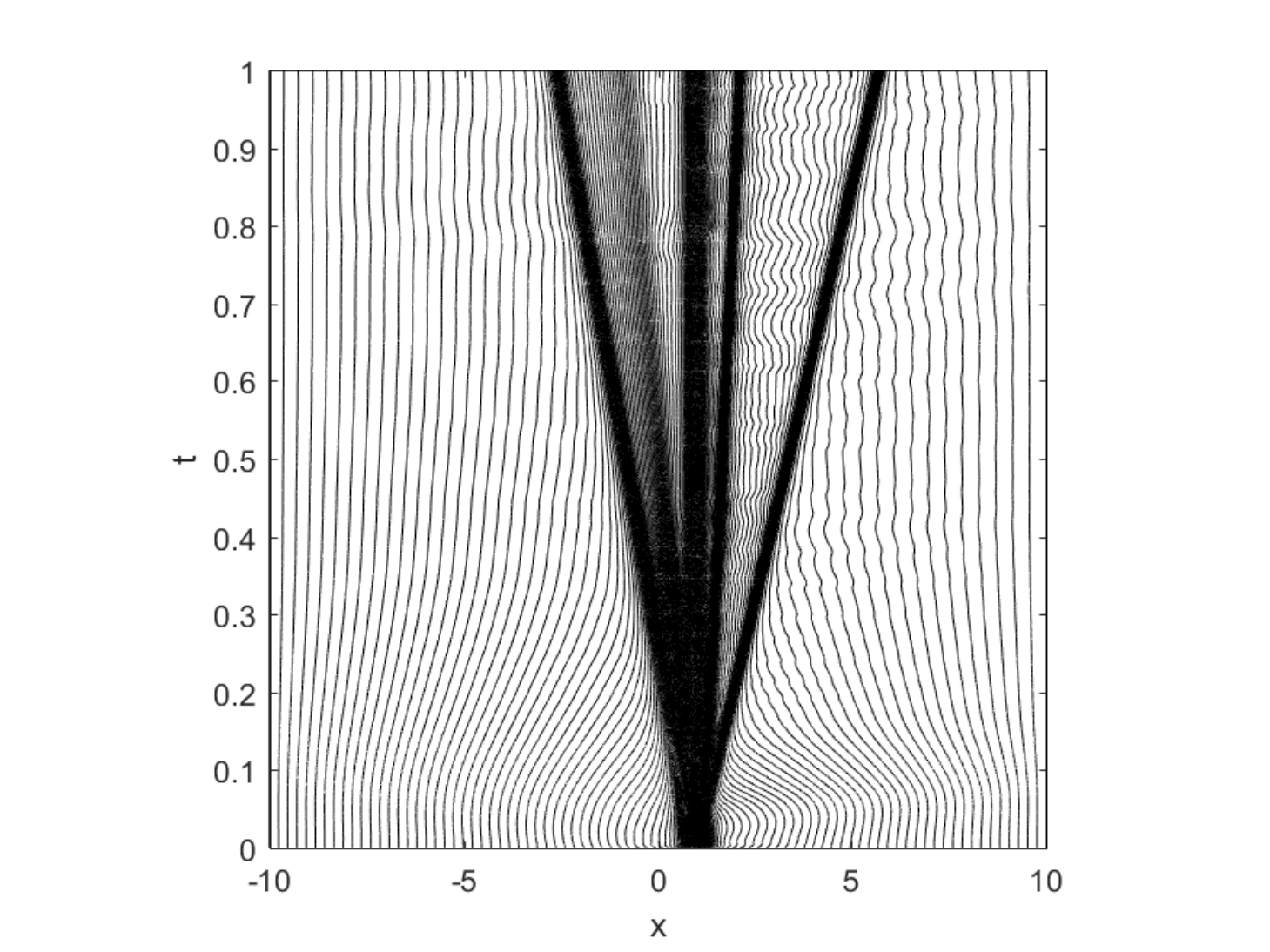}
\caption{Example \ref{test6-1d}.
The mesh trajectories are obtained with the $P^2$-DG method and a moving mesh of $N=160$.}
\label{Fig:test6-1d-mesh}
\end{figure}

\begin{figure}[H]
\centering
\subfigure[$\eta$: FM 160 vs MM 160]{
\includegraphics[width=0.35\textwidth,trim=10 0 30 10,clip]{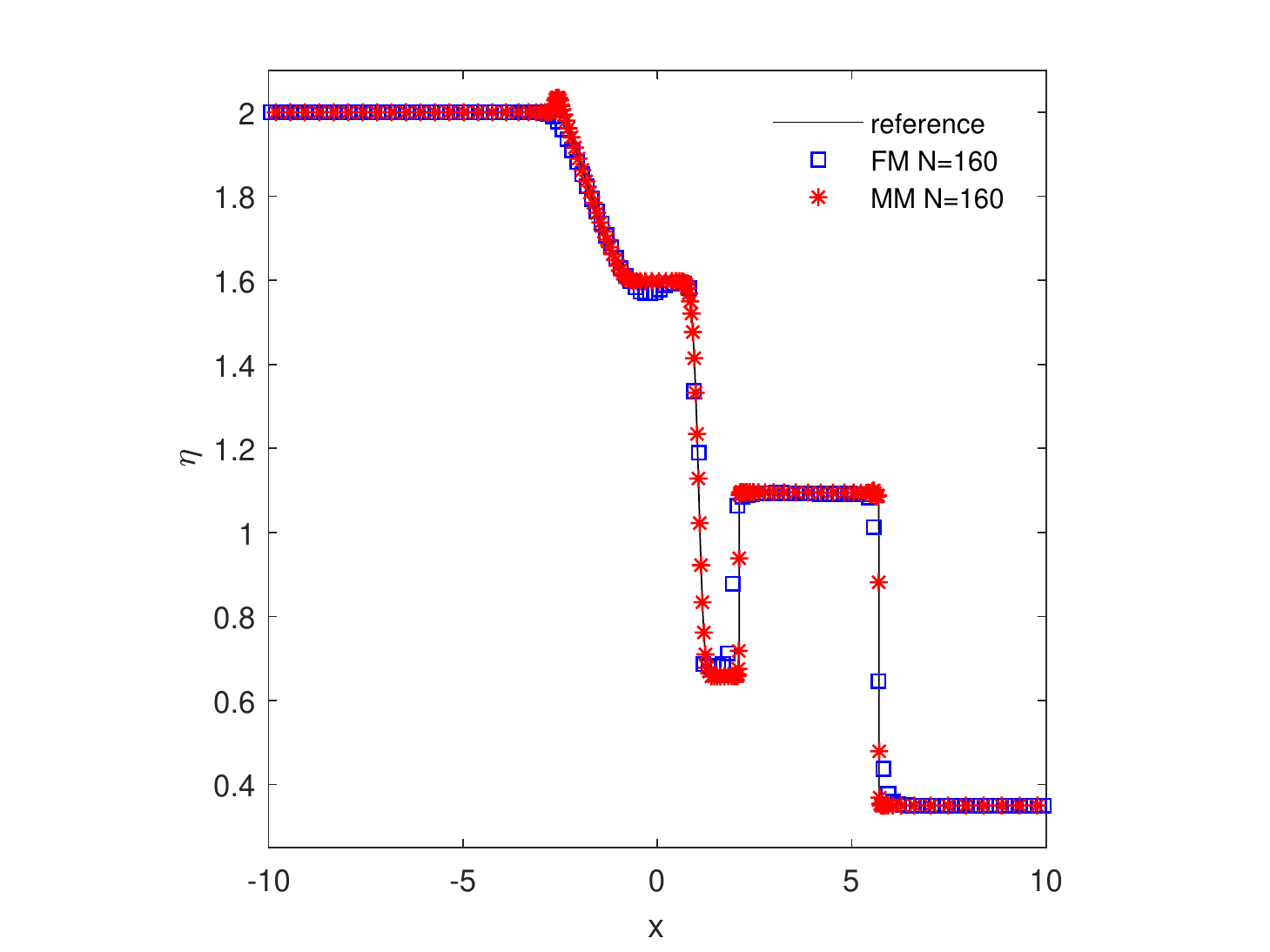}}
\subfigure[Close view of (a)]{
\includegraphics[width=0.35\textwidth,trim=10 0 30 10,clip]{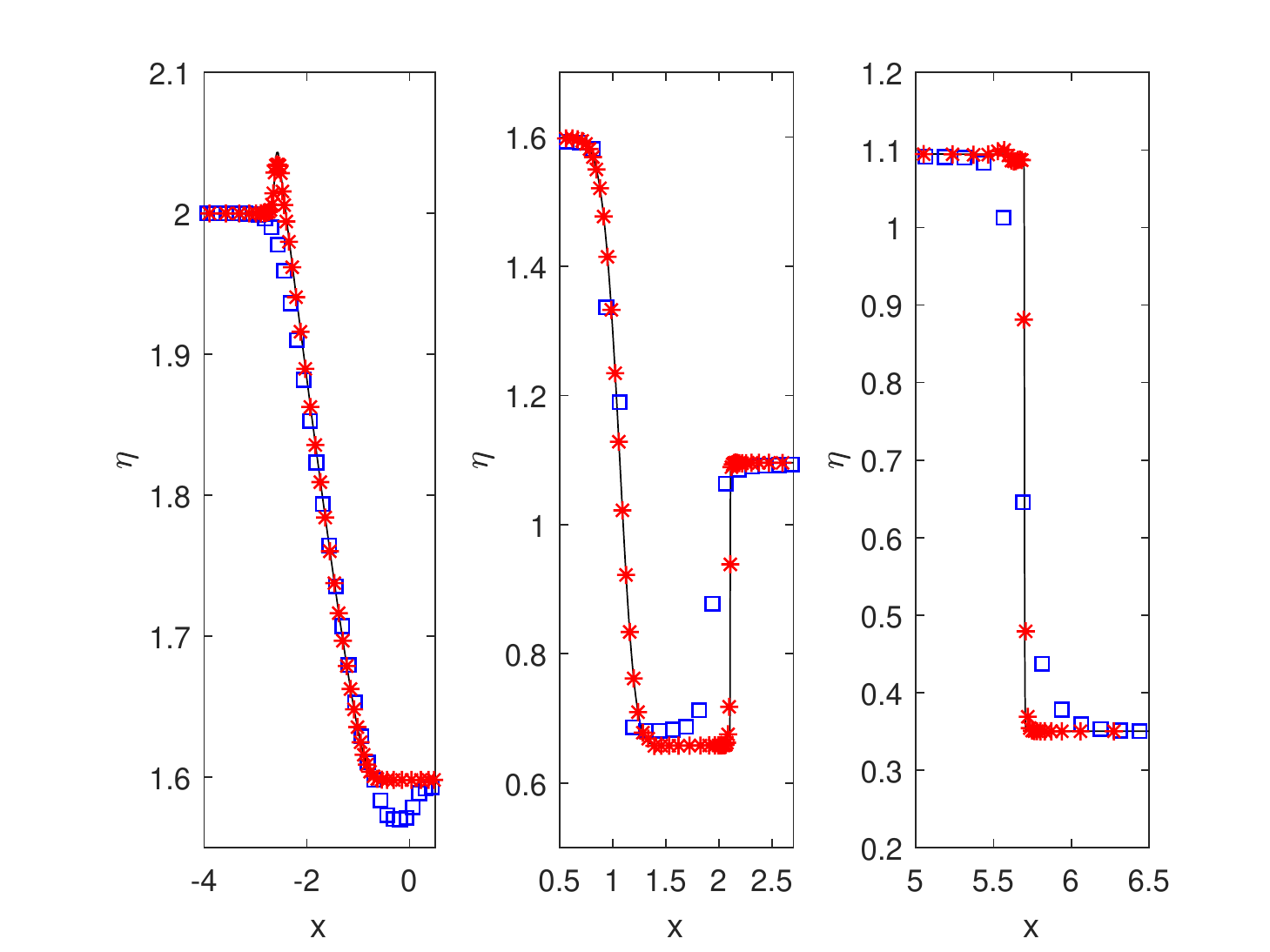}}
\subfigure[$\eta$: FM 1280 vs MM 160]{
\includegraphics[width=0.35\textwidth,trim=10 0 30 10,clip]{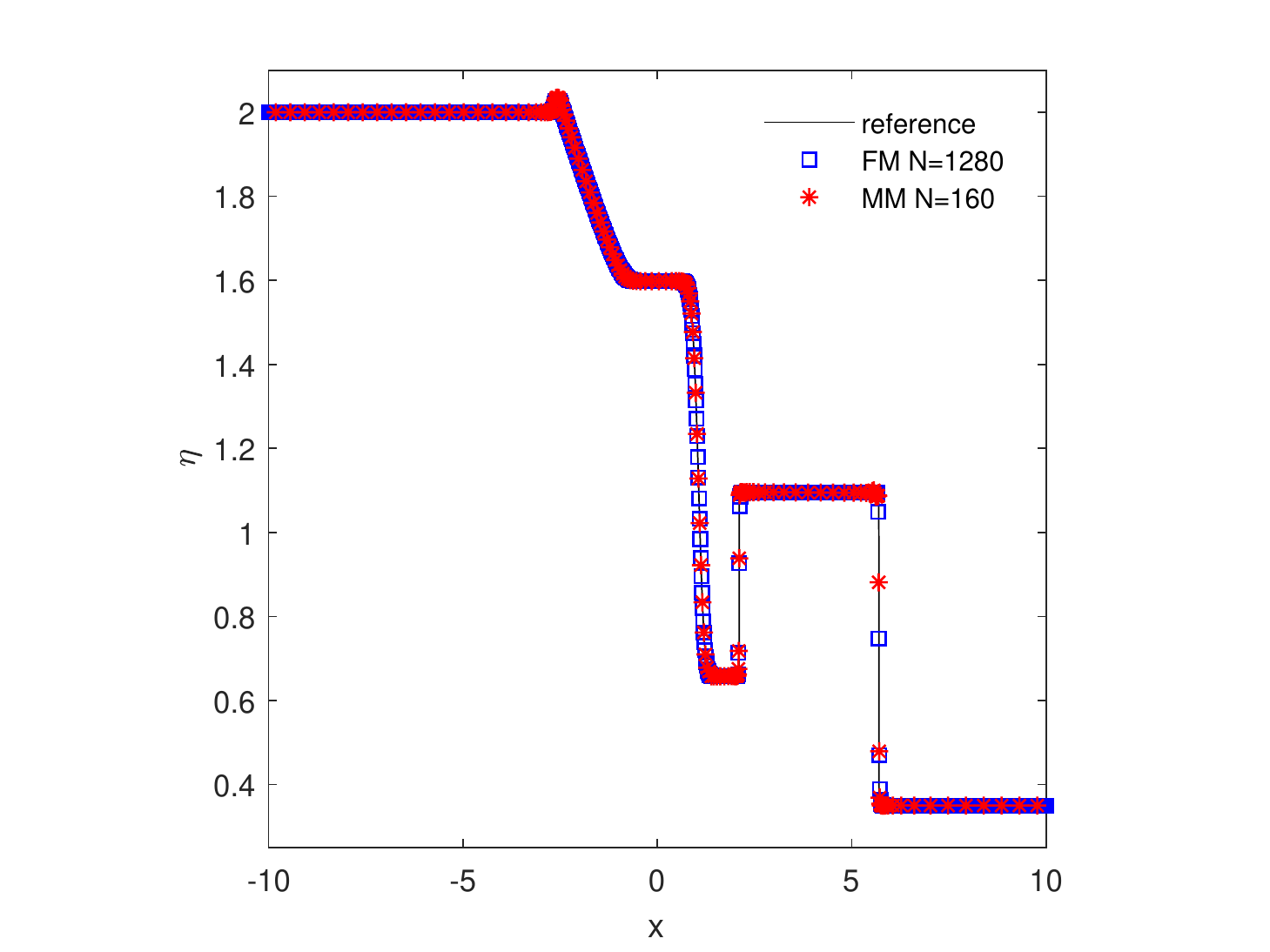}}
\subfigure[Close view of (c)]{
\includegraphics[width=0.35\textwidth,trim=10 0 30 10,clip]{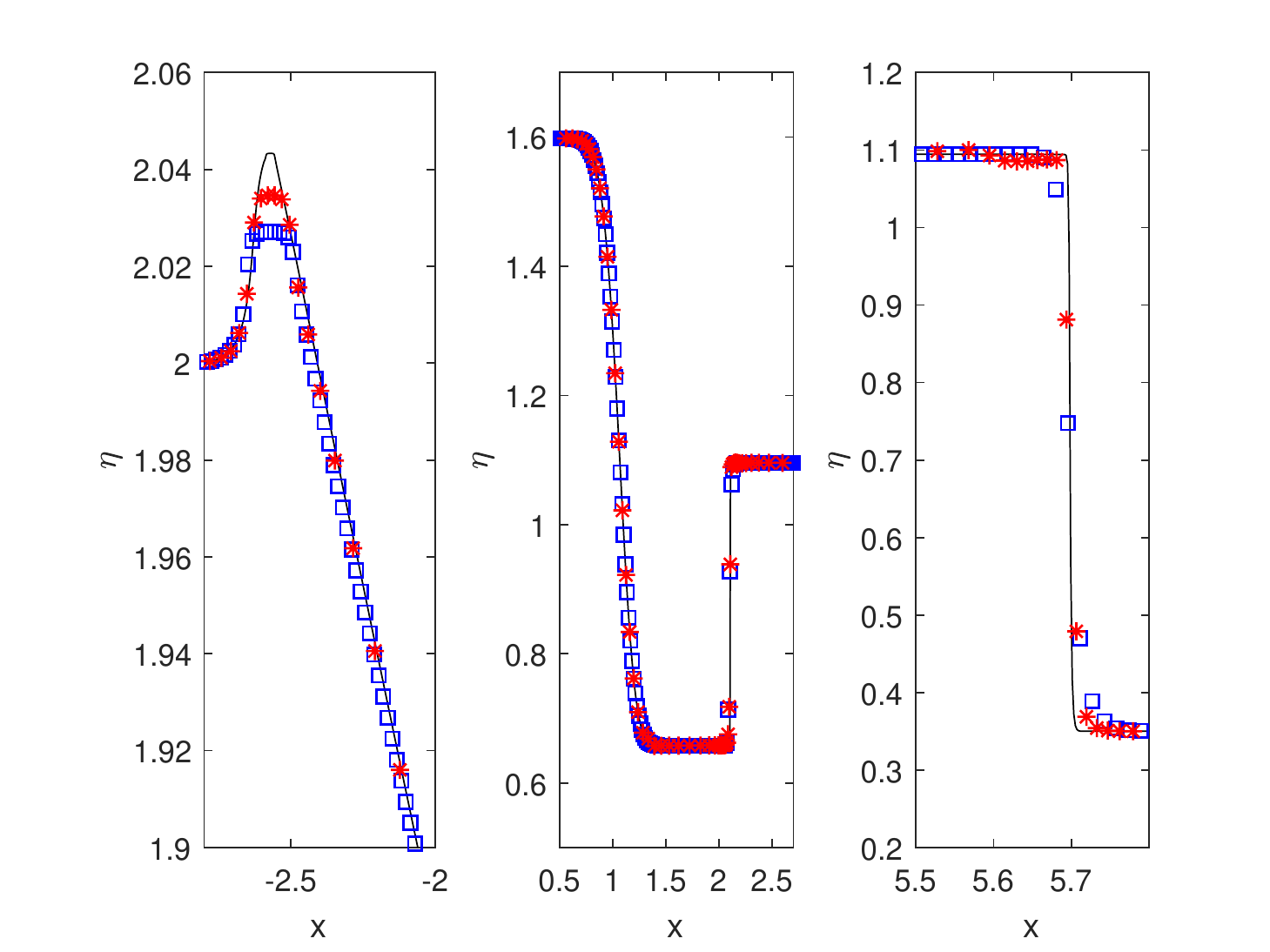}}
\caption{Example \ref{test6-1d}. The water surface level $\eta$ at $t=1$ obtained with the $P^2$-DG method and a moving mesh of $N=160$ is compared with those obtained with fixed meshes of $N=160$ and $N=1280$.}
\label{Fig:test6-1d-eta}
\end{figure}

\begin{figure}[H]
\centering
\subfigure[$hu$: FM 160 vs MM 160]{
\includegraphics[width=0.35\textwidth,trim=10 0 30 10,clip]{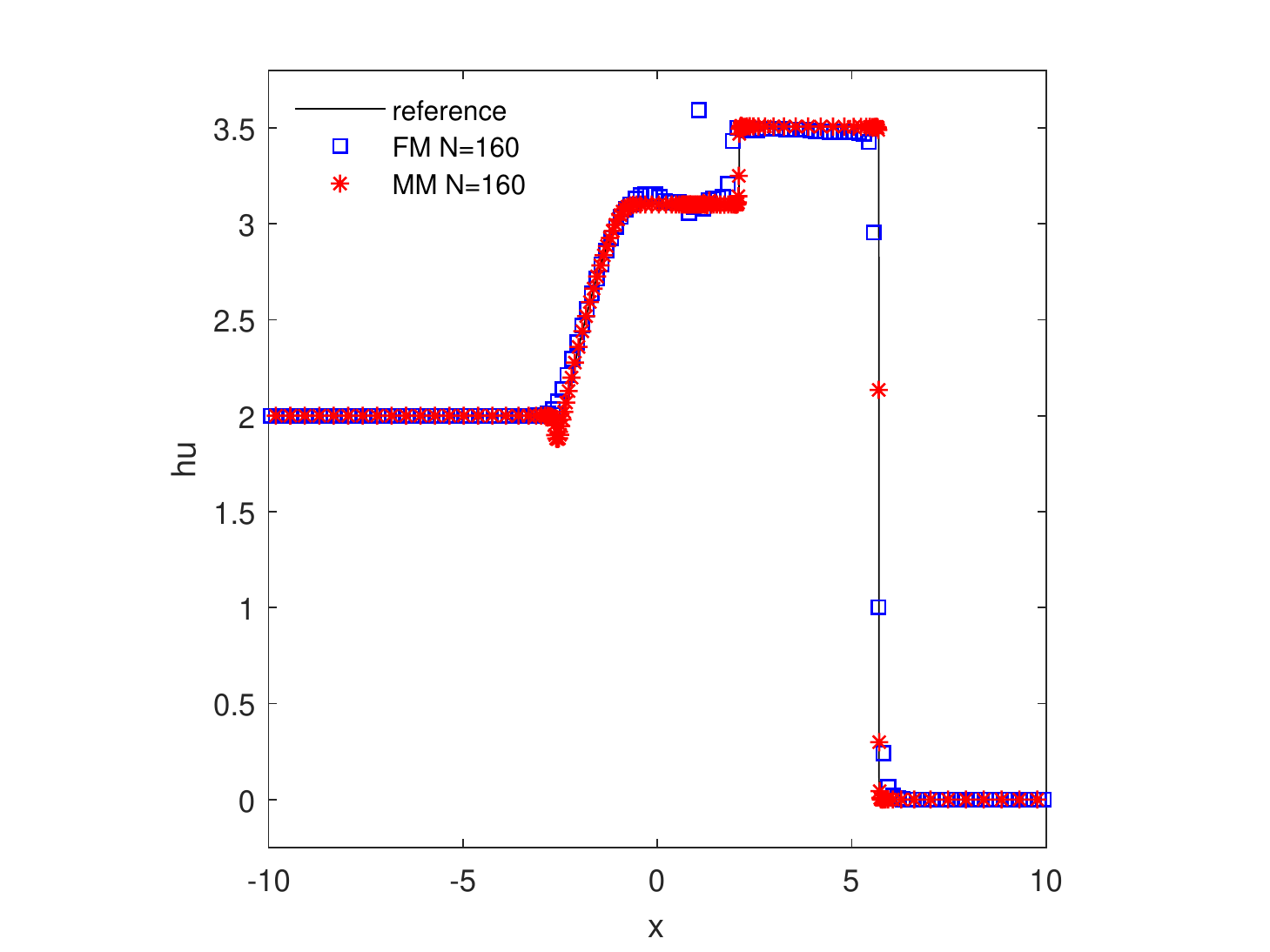}}
\subfigure[Close view of (a)]{
\includegraphics[width=0.35\textwidth,trim=10 0 30 10,clip]{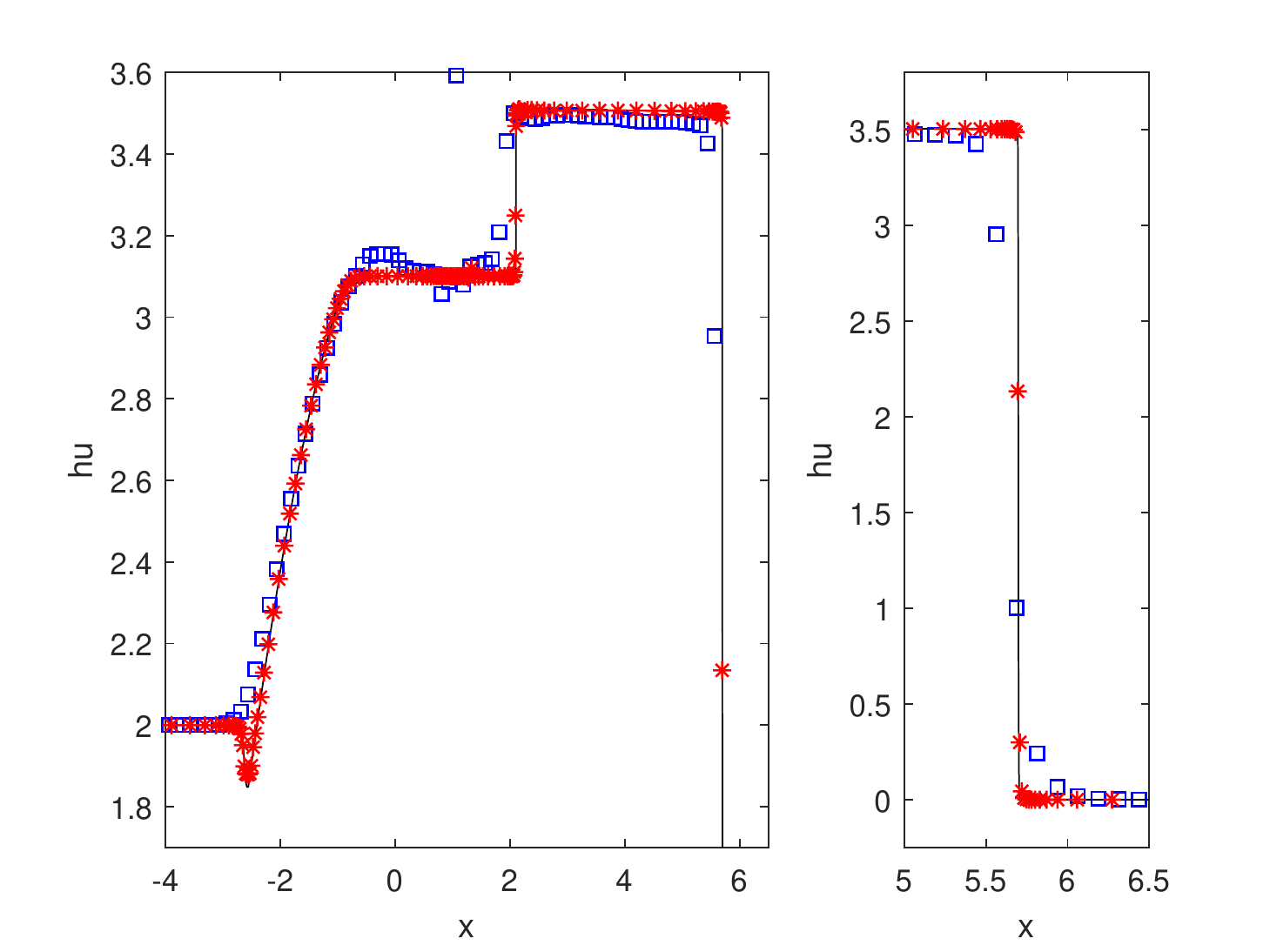}}
\subfigure[$hu$: FM 1280 vs MM 160]{
\includegraphics[width=0.35\textwidth,trim=10 0 30 10,clip]{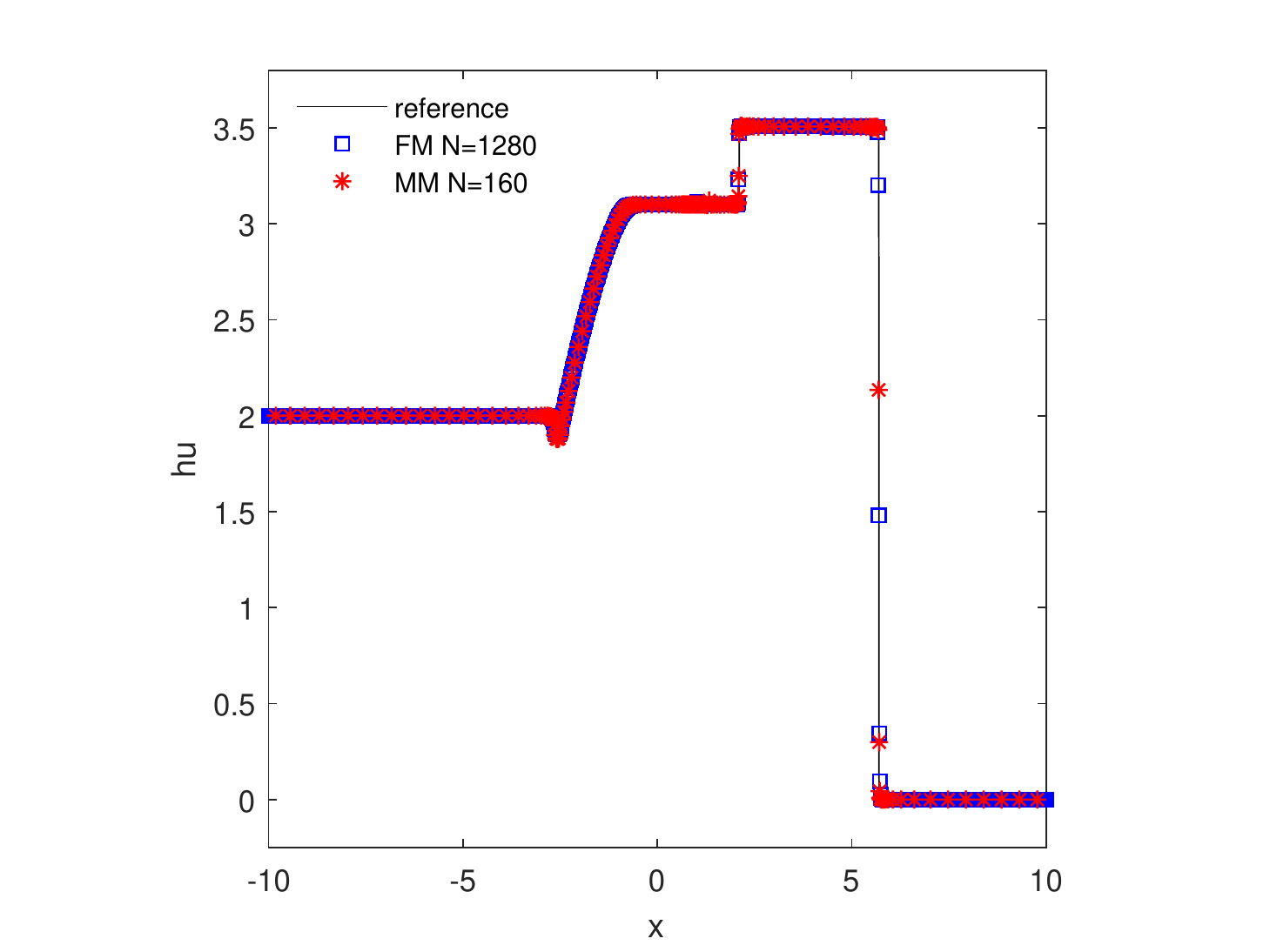}}
\subfigure[Close view of (c)]{
\includegraphics[width=0.35\textwidth,trim=10 0 30 10,clip]{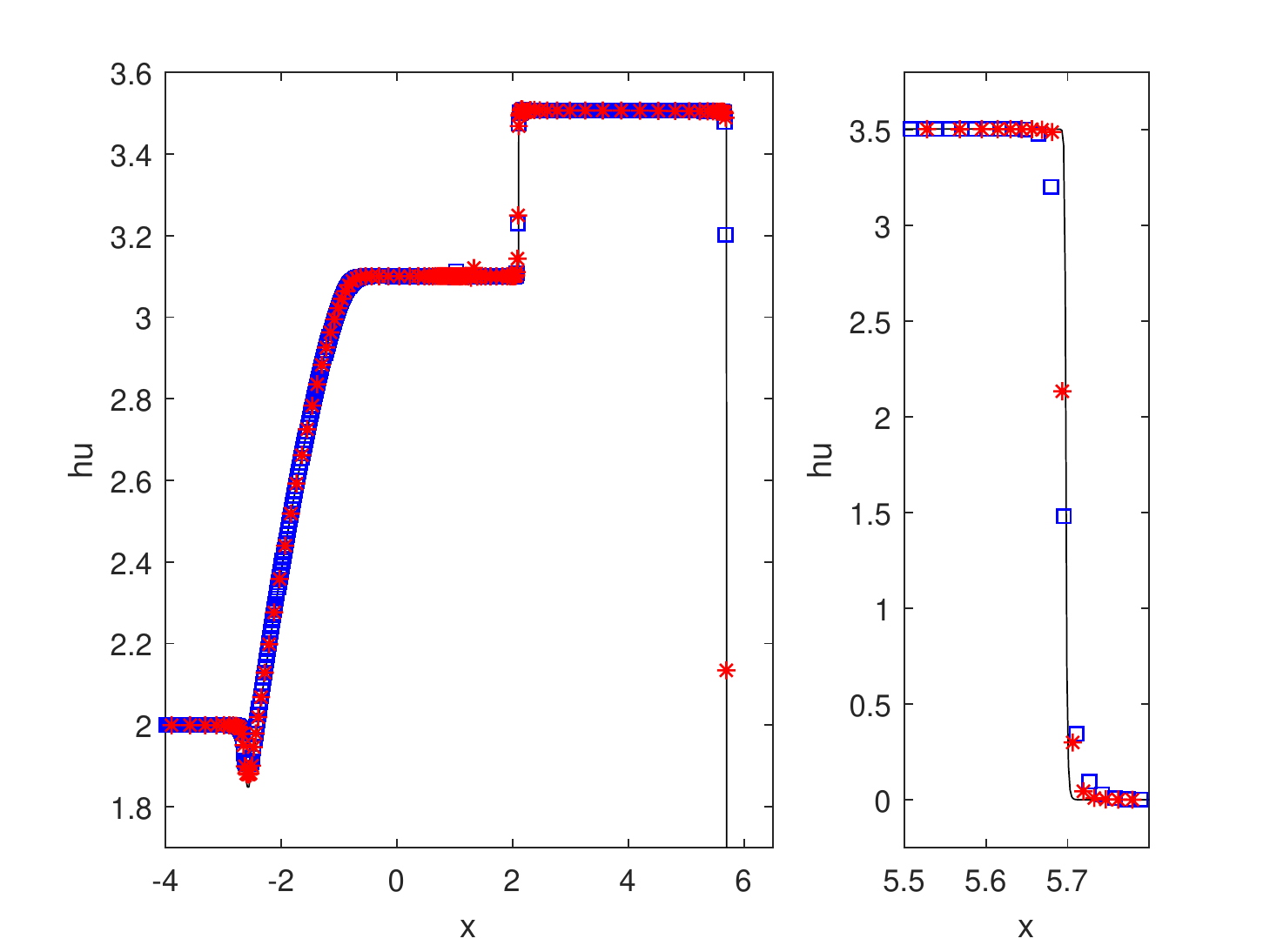}}
\caption{Example \ref{test6-1d}. The water discharge $hu$ at $t=1$ obtained with the $P^2$-DG method and a moving mesh of $N=160$ is compared with those obtained with fixed meshes of $N=160$ and $N=1280$.}
\label{Fig:test6-1d-hu}
\end{figure}

\begin{example}\label{test1-2d}
(The lake-at-rest steady-state flow test for the 2D SWEs.)
\end{example}
We choose this example to verify the well-balance property of the QLMM-DG scheme in two dimensions.
We solve the system on the domain $(x, y) \in (0,1)\times(0,1)$. The bottom topographies are the isolated elliptical-shaped bump \cite{LeVeque-1998JCP} and read as
\begin{align}
&B(x,y)=0.8e^{-50\big((x-0.5)^2+(y-0.5)^2\big)},
\label{B-4}
\\&B(x,y)=e^{-50\big((x-0.5)^2+(y-0.5)^2\big)}.
\label{dry-B-4}
\end{align}
The initial water level and velocities are given by
\begin{equation*}
\eta(x,y,0)=1,
\quad u(x,y,0)=0,\quad  v(x,y,0)=0.
\end{equation*}
The bottom topography \eqref{B-4} and \eqref{dry-B-4} have the similar shape,and the latter contain a dry region near $(x,y) = (0.5,0.5)$.
We use periodic boundary conditions for all unknown variables and compute the solution up to $t=0.1$.
The flow surface should remain steady since the method is well-balanced.

An initial triangular mesh, shown in Fig.~\ref{Fig:test-2d-tri}, is formed by dividing each cell of a rectangular mesh into four triangular elements.
The $L^1$ and $L^\infty$ error for $\eta$, $hu$, and $hv$
are listed in Tables~\ref{tab:test1-2d-p1-error} and ~\ref{tab:test1-2d-p2-error} for $P^1$-DG and $P^2$-DG,
respectively, for the bottom topography \eqref{B-4}.
They show that our DG method, with either fixed or moving meshes, maintains the lake-at-rest steady state
to the level of round-off error in both $L^1$ and $L^\infty$ norm.

To verify the well-balance and PP properties of the QLMM-DG method, we repeat the simulation for the bottom topography \eqref{dry-B-4} where the application of the PP limiter to the water depth is necessary.
The $L^1$ and $L^\infty$ error for $\eta$, $hu$, and $hv$
are listed in Tables~\ref{tab:PPtest1-2d-p1-error} and ~\ref{tab:PPtest1-2d-p2-error} for $P^1$-DG and $P^2$-DG,
respectively. One can see that the QLMM-DG method is well-balanced.

\begin{figure}[H]
\centering
\includegraphics[width=0.3\textwidth,trim=40 0 40 10,clip]{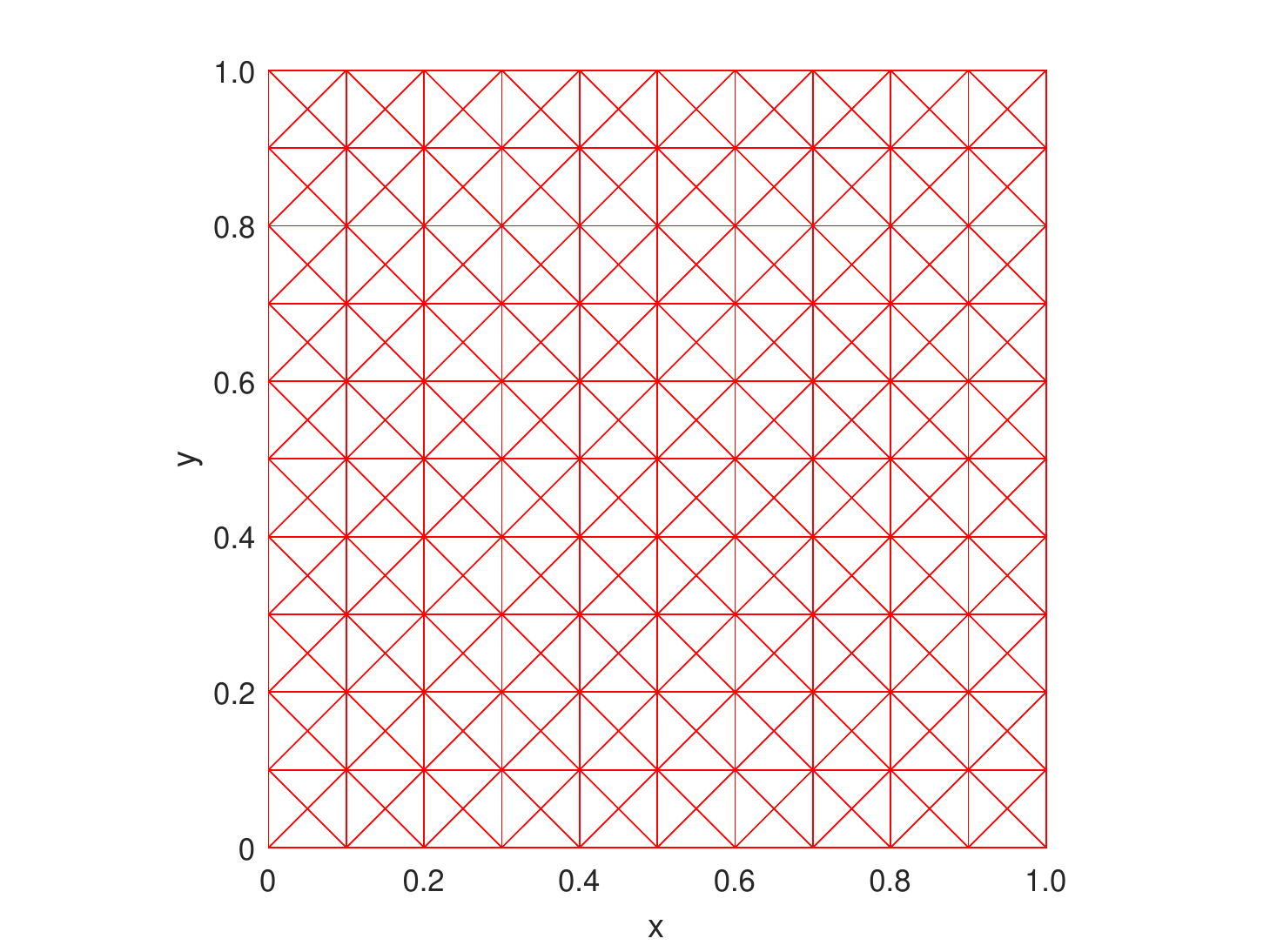}
\caption{Example \ref{test1-2d}. An initial triangular mesh used in the computation is formed by dividing each cell of a rectangular mesh into 4 triangular elements.}
\label{Fig:test-2d-tri}
\end{figure}

\begin{table}[htb]
\caption{Example \ref{test1-2d}. Well-balance test for the $P^1$-DG method with fixed and moving meshes over bottom topography \eqref{B-4}.}
\vspace{3pt}
\centering
\label{tab:test1-2d-p1-error}
\begin{tabular}{ccccccc}
 \toprule
  & \multicolumn{2}{c}{$\eta$}&\multicolumn{2}{c}{$hu$}&\multicolumn{2}{c}{$hv$}\\
$N$&$L^1$-error  &$L^{\infty}$-error  &$L^1$-error  &$L^{\infty}$-error &$L^1$-error  &$L^{\infty}$-error \\
\midrule
~   &  \multicolumn{6}{c}{\em{FM-DG method} }\\
\midrule
$10\times10\times4$	&1.539E-16	&2.837E-16	&1.257E-16	&7.696E-16	&1.343E-16	&7.879E-16	\\
$20\times20\times4$	&1.957E-16	&4.678E-16	&1.690E-16	&1.371E-15	&1.829E-16	&1.362E-15	\\
$40\times40\times4$	&3.259E-16	&1.144E-15	&2.614E-16	&1.965E-15	&2.750E-16	&2.025E-15	\\
\midrule
~   &  \multicolumn{6}{c}{\em{QLMM-DG method} }\\
\midrule
$10\times10\times4$	&1.109E-16  &2.592E-16	&1.335E-16	&9.334E-16	&1.337E-16	&8.740E-16	\\
$20\times20\times4$	&1.149E-16	&3.361E-16	&1.660E-16	&1.709E-15	&1.671E-16	&1.616E-15	\\
$40\times40\times4$	&1.146E-16	&5.533E-16	&2.567E-16	&3.126E-15	&2.580E-16	&3.087E-15	\\
 \bottomrule	
\end{tabular}
\end{table}

\begin{table}[htb]
\caption{Example \ref{test1-2d}. Well-balance test for the $P^2$-DG method with fixed and moving meshes over bottom topography \eqref{B-4}.}
\vspace{3pt}
\centering
\label{tab:test1-2d-p2-error}
\begin{tabular}{ccccccc}
 \toprule
  & \multicolumn{2}{c}{$\eta$}&\multicolumn{2}{c}{$hu$}&\multicolumn{2}{c}{$hv$}\\
$N$&$L^1$-error  &$L^{\infty}$-error  &$L^1$-error  &$L^{\infty}$-error &$L^1$-error  &$L^{\infty}$-error \\
\midrule
~   &  \multicolumn{6}{c}{\em{FM-DG method} }\\
\midrule
$10\times10\times4$	&4.287E-16	&1.841E-15	&8.615E-16	&4.685E-15	&9.310E-16	&5.907E-15	\\
$20\times20\times4$	&4.020E-16	&3.239E-15	&8.930E-16	&5.358E-15	&9.638E-16	&7.076E-15	\\
$40\times40\times4$	&4.454E-16	&4.503E-15	&1.056E-15	&9.702E-15	&1.129E-15	&1.281E-14	\\
\midrule
~   &  \multicolumn{6}{c}{\em{QLMM-DG method} }\\
\midrule
$10\times10\times4$	&4.795E-16	&1.841E-15	&7.965E-16	&4.949E-15	&8.029E-16	&4.821E-15	\\
$20\times20\times4$	&5.023E-16	&2.642E-15	&9.442E-16	&5.800E-15	&8.870E-16	&5.128E-15	\\
$40\times40\times4$	&5.155E-16	&3.598E-15	&1.428E-15	&1.063E-14	&1.167E-15	&7.697E-15	\\
 \bottomrule	
\end{tabular}
\end{table}

\begin{table}[htb]
\caption{Example \ref{test1-2d}. Well-balance test for the $P^1$-DG method with fixed and moving meshes over bottom topography \eqref{dry-B-4} (with a dry region).}
\vspace{3pt}
\centering
\label{tab:PPtest1-2d-p1-error}
\begin{tabular}{ccccccc}
 \toprule
  & \multicolumn{2}{c}{$\eta$}&\multicolumn{2}{c}{$hu$}&\multicolumn{2}{c}{$hv$}\\
$N$&$L^1$-error  &$L^{\infty}$-error  &$L^1$-error  &$L^{\infty}$-error &$L^1$-error  &$L^{\infty}$-error \\
\midrule
~   &  \multicolumn{6}{c}{\em{FM-DG method} }\\
\midrule
$10\times10\times4$	&1.256E-16	&	5.706E-16	&1.674E-16	&	1.631E-15	&1.639E-16	&	1.924E-15	\\
$20\times20\times4$	&1.314E-16	&	1.368E-15	&1.777E-16	&	2.543E-15	&1.926E-16	&	2.684E-15	\\
$40\times40\times4$	&1.434E-16	&	4.448E-15	&2.224E-16	&	3.081E-15	&2.455E-16	&	4.990E-15	\\
\midrule
~   &  \multicolumn{6}{c}{\em{QLMM-DG method} }\\
\midrule
$10\times10\times4$	&1.026E-16	&	5.661E-16	&1.497E-16	&	1.349E-15	&1.524E-16	&	1.511E-15	\\
$20\times20\times4$	&9.847E-17	&	1.555E-15	&1.930E-16	&	2.047E-15	&1.987E-16	&	2.055E-15	\\
$40\times40\times4$	&8.793E-17	&	6.290E-15	&2.351E-16	&	3.217E-15	&2.356E-16	&	3.168E-15	\\
 \bottomrule	
\end{tabular}
\end{table}

\begin{table}[htb]
\caption{Example \ref{test1-2d}. Well-balance test for the $P^2$-DG method with fixed and moving meshes over bottom topography \eqref{dry-B-4} (with a dry region).}
\vspace{3pt}
\centering
\label{tab:PPtest1-2d-p2-error}
\begin{tabular}{ccccccc}
 \toprule
  & \multicolumn{2}{c}{$\eta$}&\multicolumn{2}{c}{$hu$}&\multicolumn{2}{c}{$hv$}\\
$N$&$L^1$-error  &$L^{\infty}$-error  &$L^1$-error  &$L^{\infty}$-error &$L^1$-error  &$L^{\infty}$-error \\
\midrule
~   &  \multicolumn{6}{c}{\em{FM-DG method} }\\
\midrule
$10\times10\times4$	&4.830E-16	&	1.673E-14	&9.355E-16	&	1.432E-14	&9.978E-16	&	1.315E-14	\\
$20\times20\times4$	&5.039E-16	&	6.368E-14	&9.708E-16	&	2.008E-14	&1.053E-15	&	2.865E-14	\\
$40\times40\times4$	&5.421E-16	&	1.280E-13	&1.199E-15	&	2.173E-14	&1.274E-15	&	2.015E-14	\\
\midrule
~   &  \multicolumn{6}{c}{\em{QLMM-DG method} }\\
\midrule
$10\times10\times4$	&5.735E-16	&	1.623E-14	&8.498E-16	&	1.366E-14	&8.691E-16	&	1.461E-14	\\
$20\times20\times4$	&6.264E-16	&	5.730E-14	&9.876E-16	&	1.663E-14	&9.285E-16	&	2.035E-14	\\
$40\times40\times4$	&7.067E-16	&	1.021E-13	&1.411E-15	&	1.231E-14	&1.163E-15	&	1.243E-14	\\
 \bottomrule	
\end{tabular}
\end{table}

\begin{example}\label{test2-2d}
(The perturbed lake-at-rest steady-state flow test for the 2D SWEs.)
\end{example}
We choose this example first used by LeVeque \cite{LeVeque-1998JCP}
to demonstrate the ability of the QLMM-DG method to simulate small perturbations of the water surface.
The bottom topography is an isolated elliptical shaped hump,
\begin{equation}\label{test2-2d-B1}
B(x,y)=0.8e^{\big(-5(x-0.9)^2-50(y-0.5)^2\big)},\quad (x, y) \in (-1,2)\times(0,1).
\end{equation}
The initial conditions are given by
\begin{equation*}
\begin{split}
&\eta (x,y,0)=\begin{cases}
1+0.01,& \hbox{for~$x\in(0.05 , 0.15)$}\\
1,&\hbox{otherwise}\\
\end{cases}
\\&u(x,y,0)=0, \quad \hbox{and}\quad v(x,y,0)=0.
\end{split}
\end{equation*}
Reflection boundary conditions \cite{Tumolo-2013JCP} are used for all domain boundary.
Theoretically, this perturbation splits into two waves, propagating left and right
at the characteristic speeds $\pm \sqrt{gh}$.
One of these waves is moving towards the bump in the bottom topography, interacting with it, and generating
a complex wave structure. The difficulty of this test case is to resolve the waves that are very small
in magnitude in comparison to the average values of the quantities.

The moving mesh of $N=150\times 50\times4$ at $t=0.12,0.24,0.36,0.48$ obtained with $P^2$ QLMM-DG method are shown in Fig.~\ref{Fig:test2-2d-P2-mesh}. The contours of the obtained solutions $\eta$, $hu$, and $hv$
are shown in Figs.~\ref{Fig:test2-2d-P2-h-hu-hv-t12} -- \ref{Fig:test2-2d-P2-h-hu-hv-t48}.
For comparison purpose, the numerical solutions obtained with fixed meshes of $N=150\times 50\times4$
and $N=600\times 200\times4$ are also shown.
We can see that the mesh points concentrate correctly around the waves and the point $(x,y) = (0.9, 0.5)$,
the center of the non-flat region of the bottom topography.
Moreover, the QLMM-DG method resolves well the complex small-scale features of the water flow.
The moving mesh solutions with $N=150\times 50\times4$ do not contain visibly spurious oscillations
and is more accurate than that with a fixed mesh of $N=150\times 50\times4$ and comparable with
that with a fixed mesh of $N=600\times 200\times4$.
\begin{figure}[H]
\centering
\subfigure[Mesh at $t=0.12$]{
\includegraphics[width=0.4\textwidth,trim=20 50 0 60,clip]{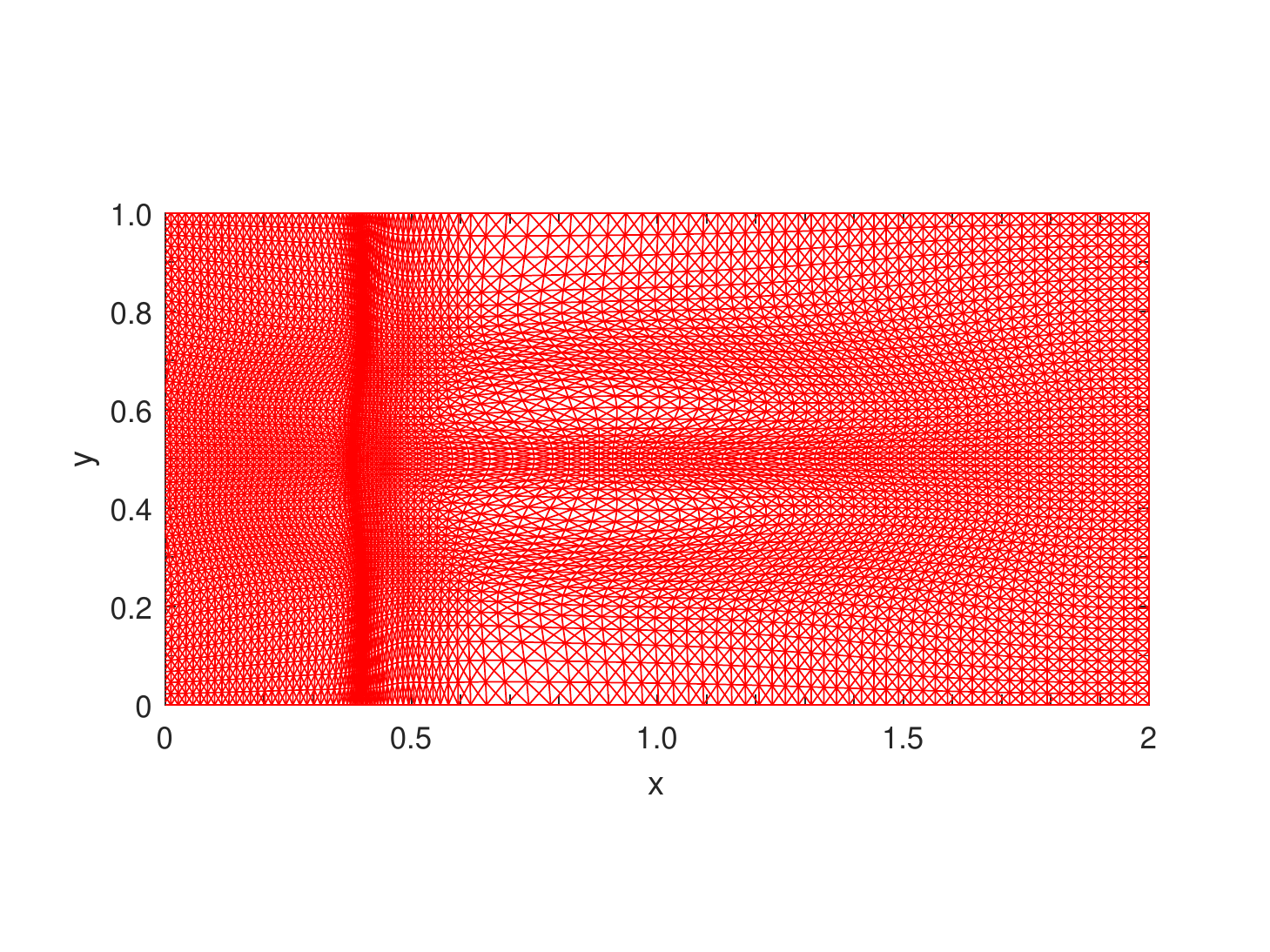}}
\subfigure[Mesh at $t=0.24$]{
\includegraphics[width=0.4\textwidth,trim=20 50 0 60,clip]{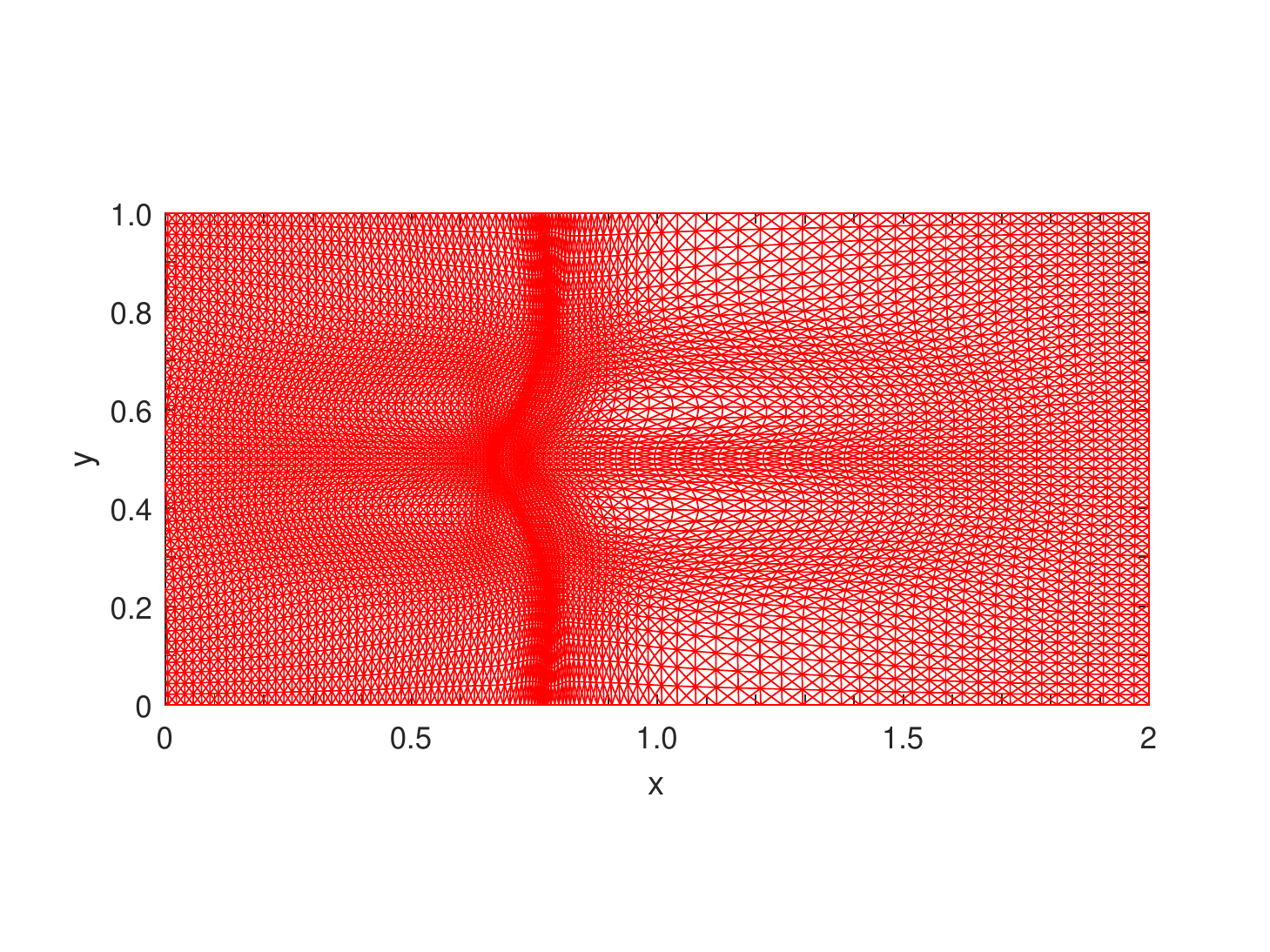}}
\subfigure[Mesh at $t=0.36$]{
\includegraphics[width=0.4\textwidth,trim=20 50 0 60,clip]{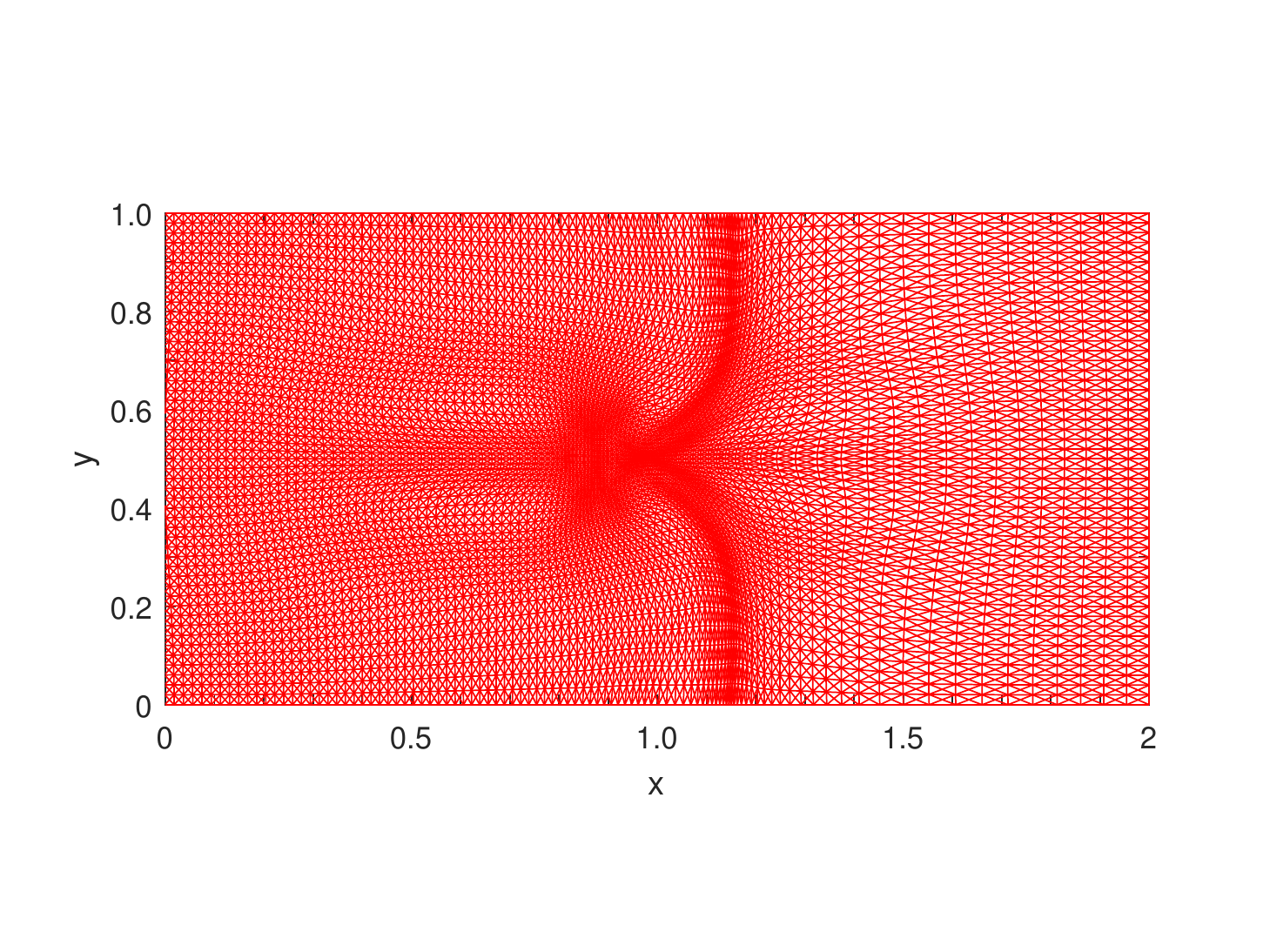}}
\subfigure[Mesh at $t=0.48$]{
\includegraphics[width=0.4\textwidth,trim=20 50 0 60,clip]{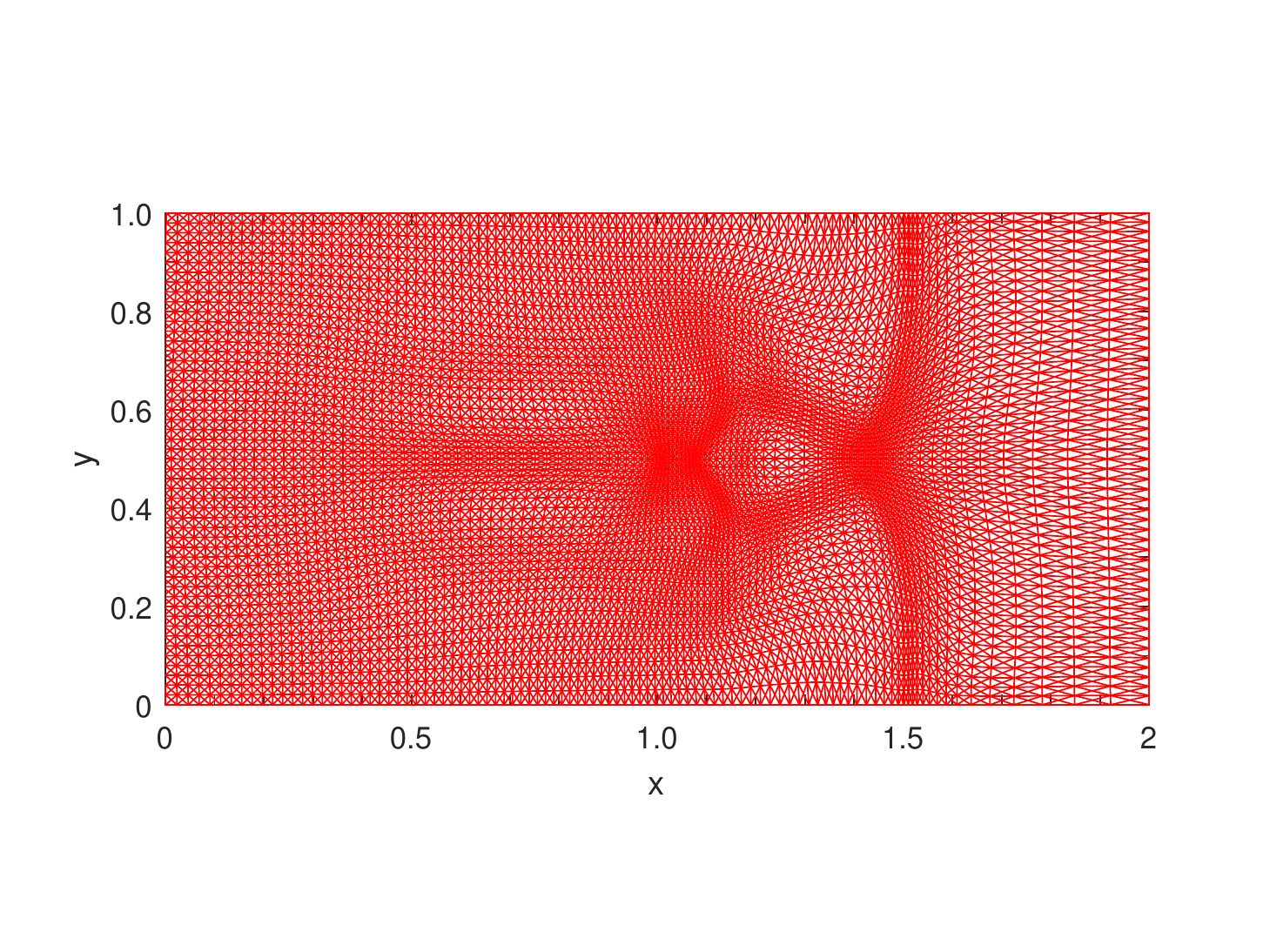}}
\caption{Example \ref{test2-2d}. The moving mesh of $N=150\times 50\times4$ at $t=0.12,~0.24,~0.36,~0.48$ is obtained with the $P^2$ QLMM-DG method.}
\label{Fig:test2-2d-P2-mesh}
\end{figure}

\begin{figure}[H]
\centering
\subfigure[$\eta$: MM $N=150\times 50\times4$]{
\includegraphics[width=0.30\textwidth, trim=15 60 15 60, clip]{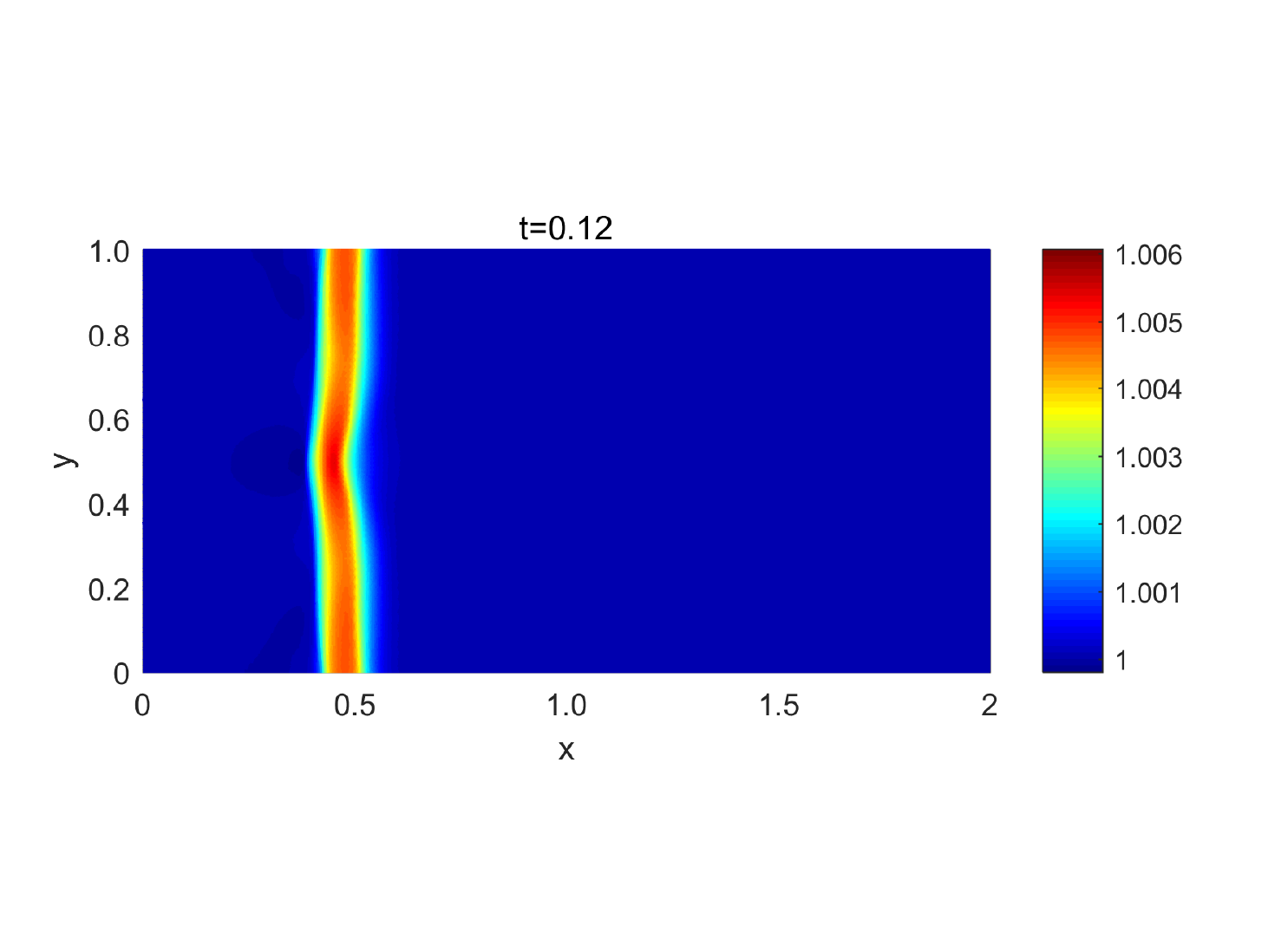}}
\subfigure[$hu$: MM $N=150\times 50\times4$]{
\includegraphics[width=0.30\textwidth, trim=15 60 15 60, clip]{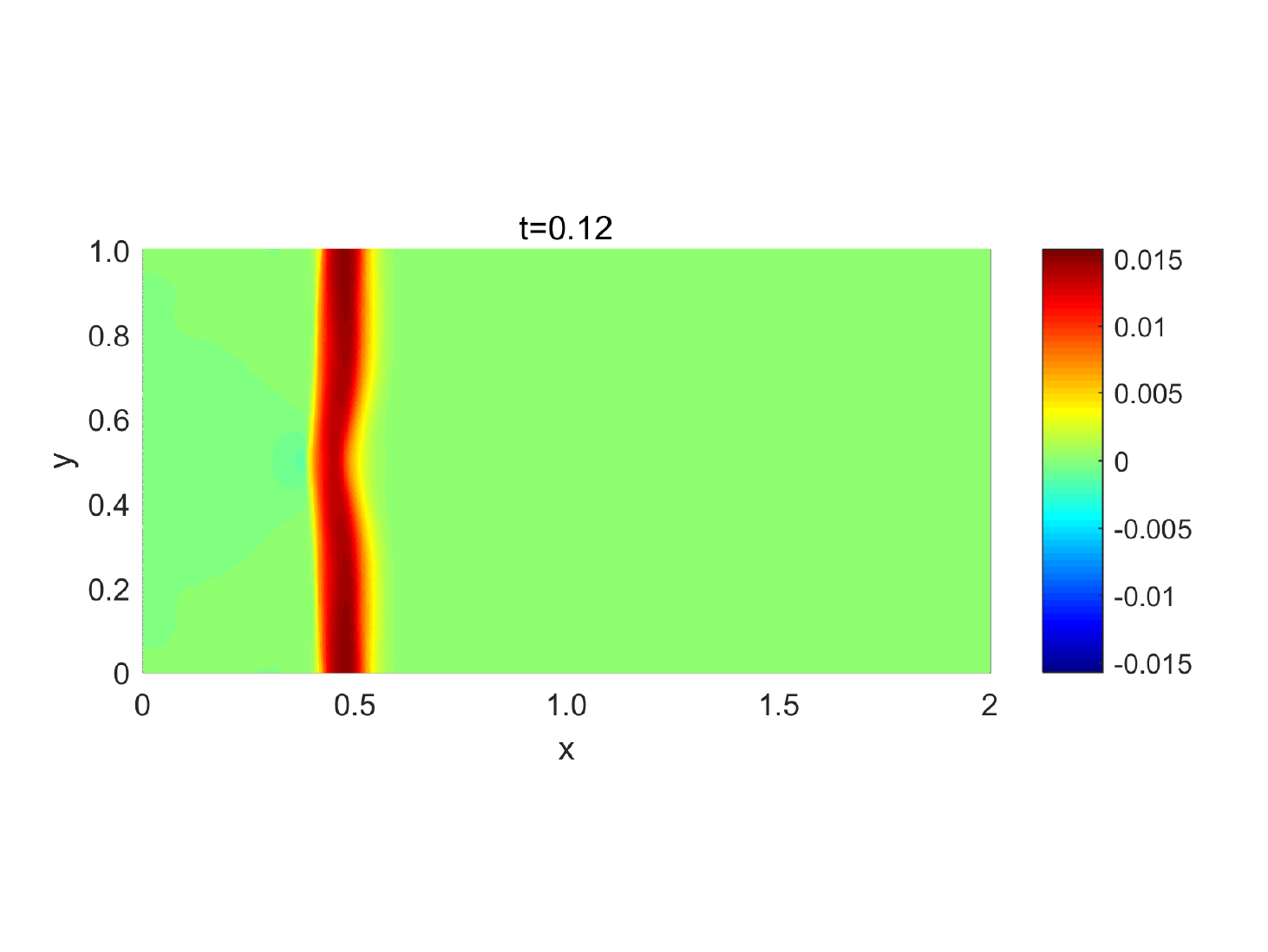}}
\subfigure[$hv$: MM $N=150\times 50\times4$]{
\includegraphics[width=0.30\textwidth, trim=15 60 15 60, clip]{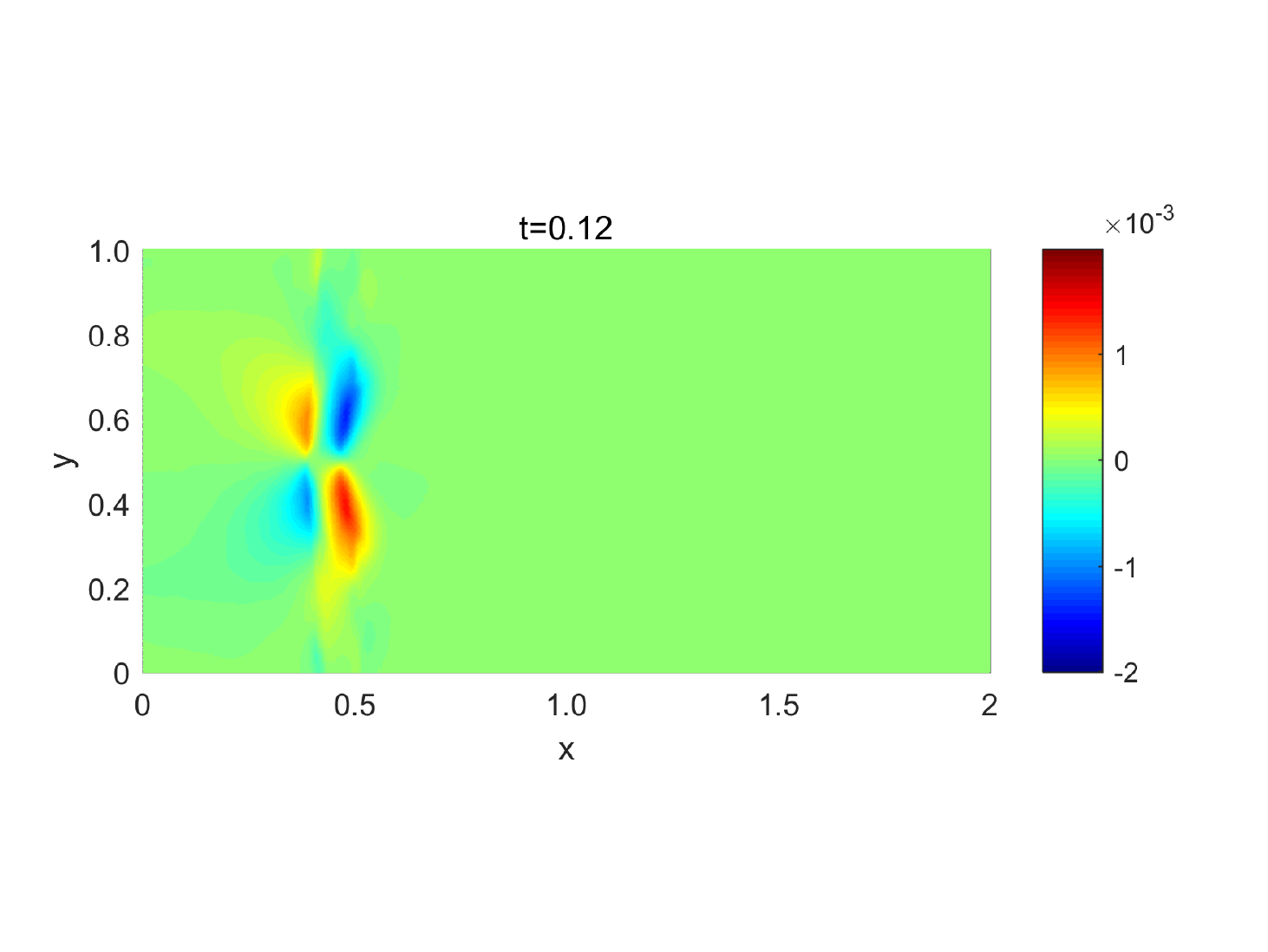}}
\subfigure[$\eta$: FM $N=150\times 50\times4$]{
\includegraphics[width=0.30\textwidth, trim=15 60 15 60, clip]{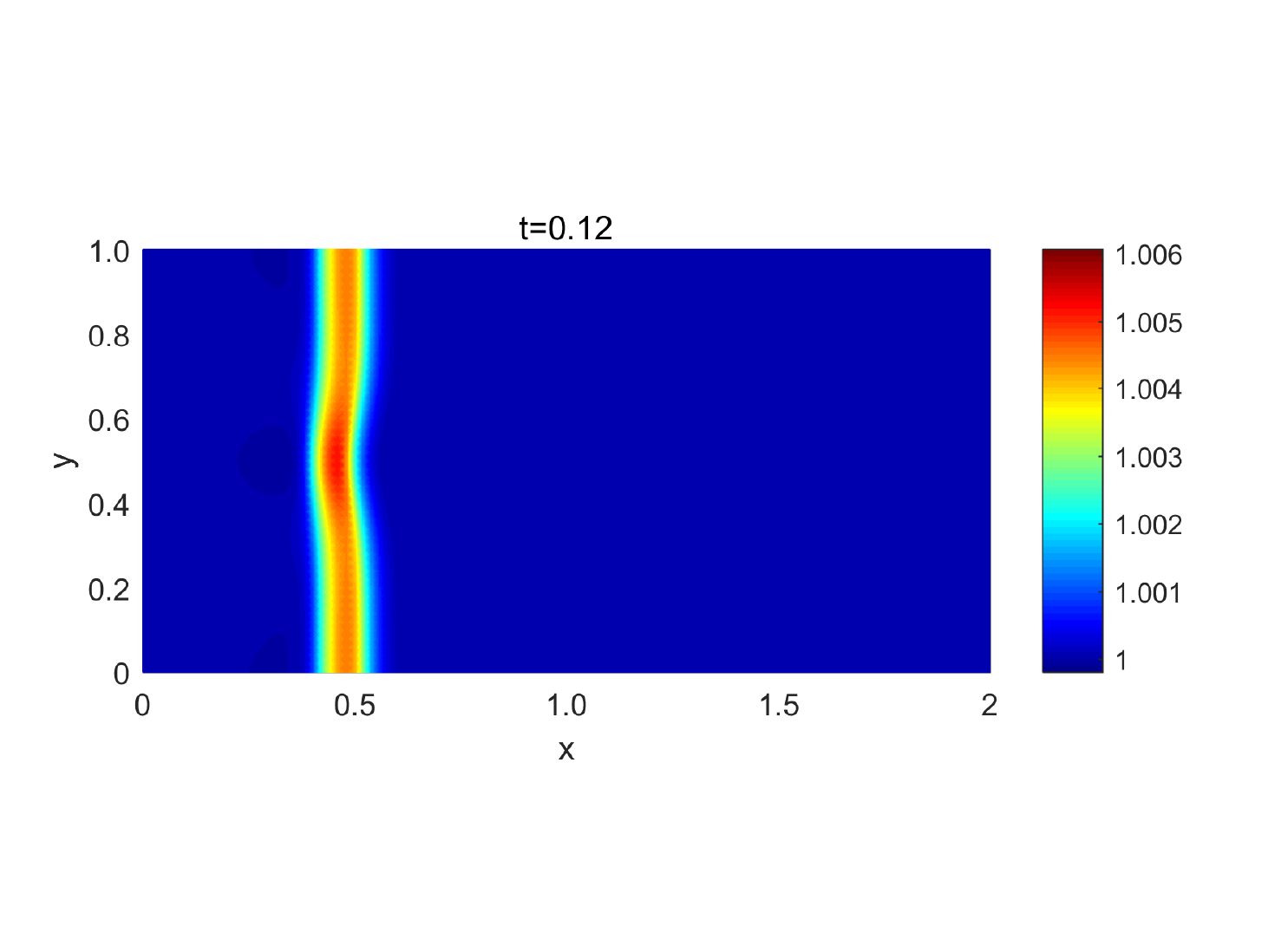}}
\subfigure[$hu$: FM $N=150\times 50\times4$]{
\includegraphics[width=0.30\textwidth, trim=15 60 15 60, clip]{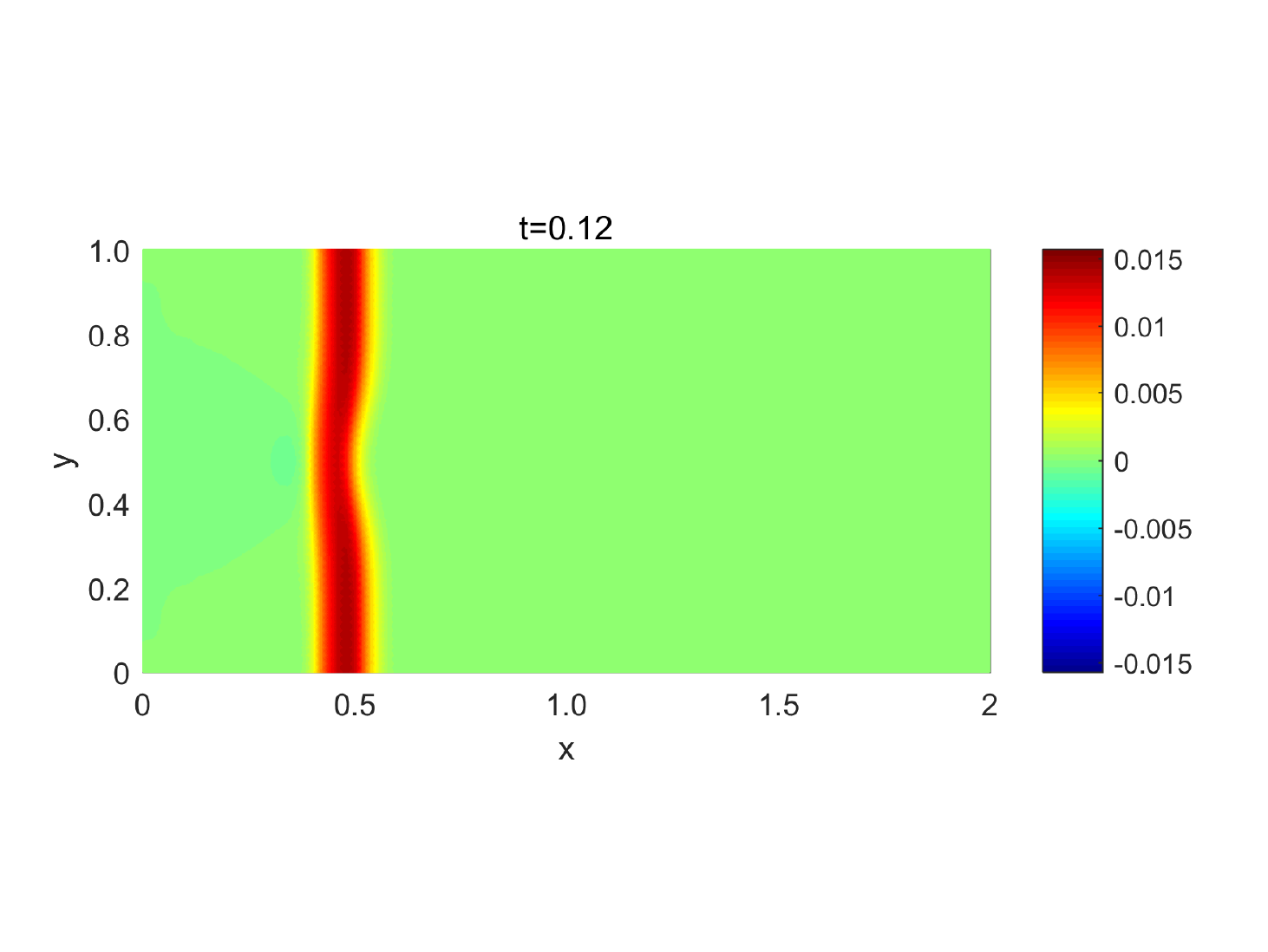}}
\subfigure[$hv$: FM $N=150\times 50\times4$]{
\includegraphics[width=0.30\textwidth, trim=15 60 15 60, clip]{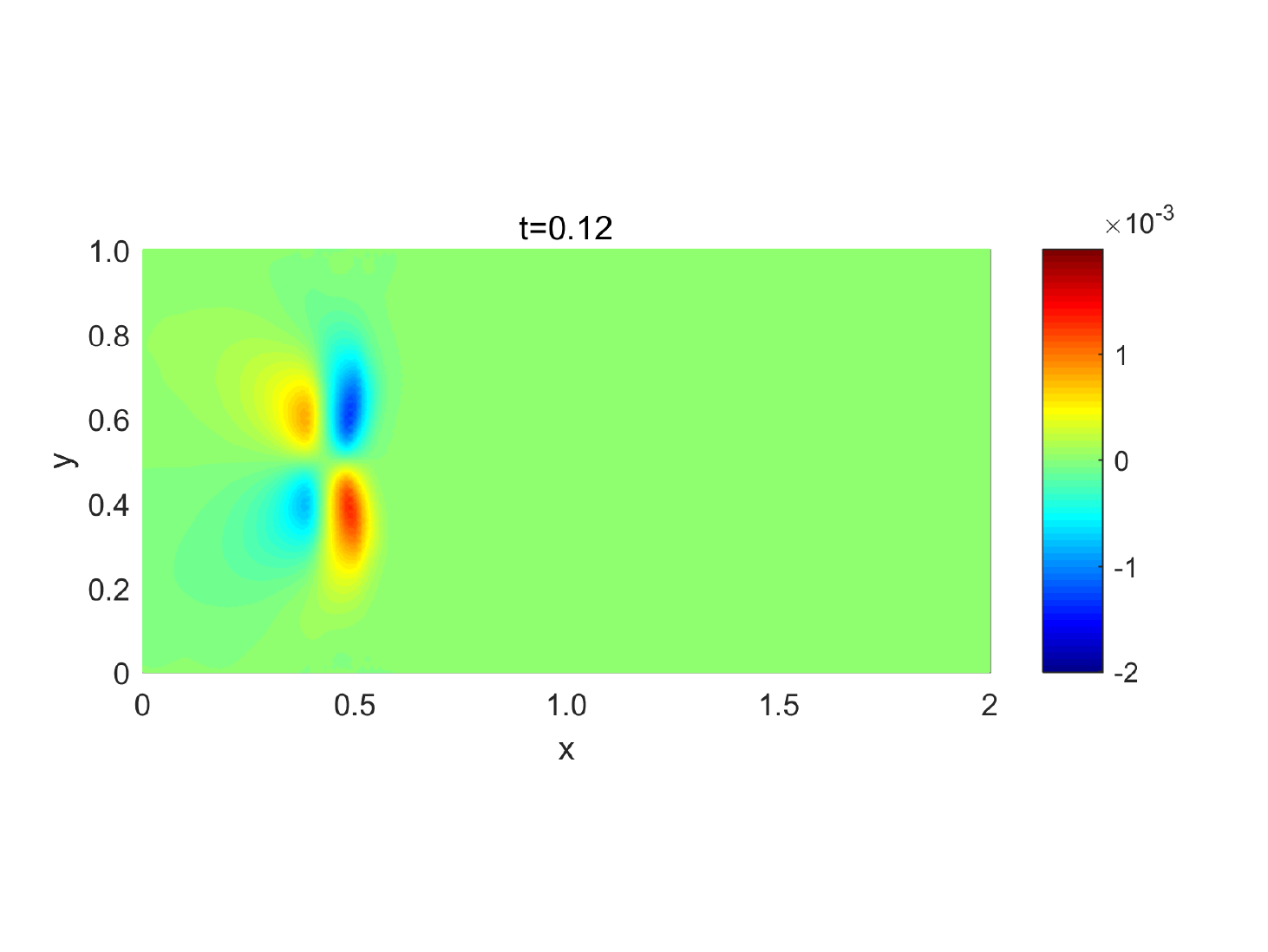}}
\subfigure[$\eta$: FM $600\times 200\times4$]{
\includegraphics[width=0.30\textwidth, trim=15 60 15 60, clip]{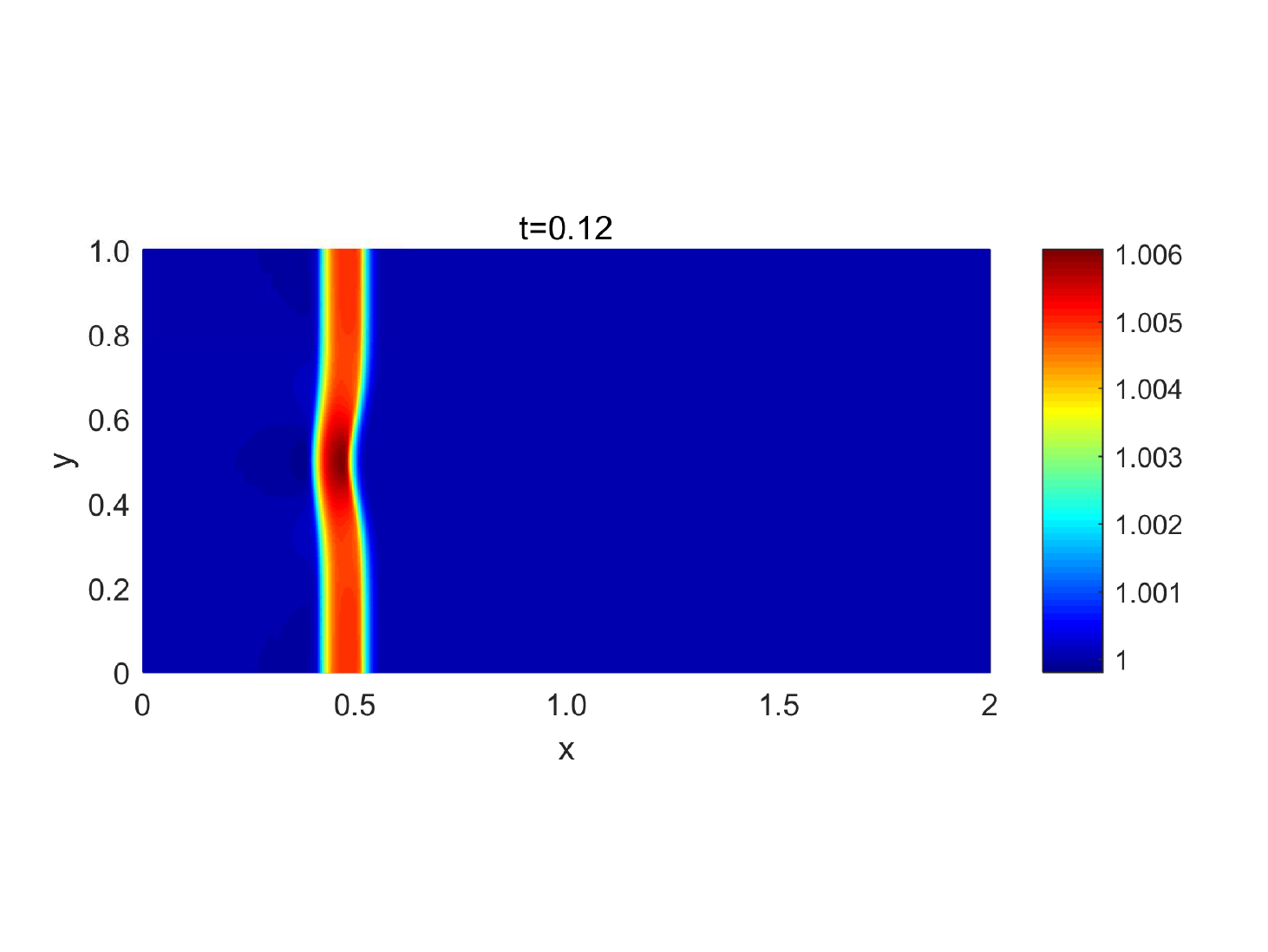}}
\subfigure[$hu$: FM $N=600\times 200\times4$]{
\includegraphics[width=0.30\textwidth, trim=15 60 15 60, clip]{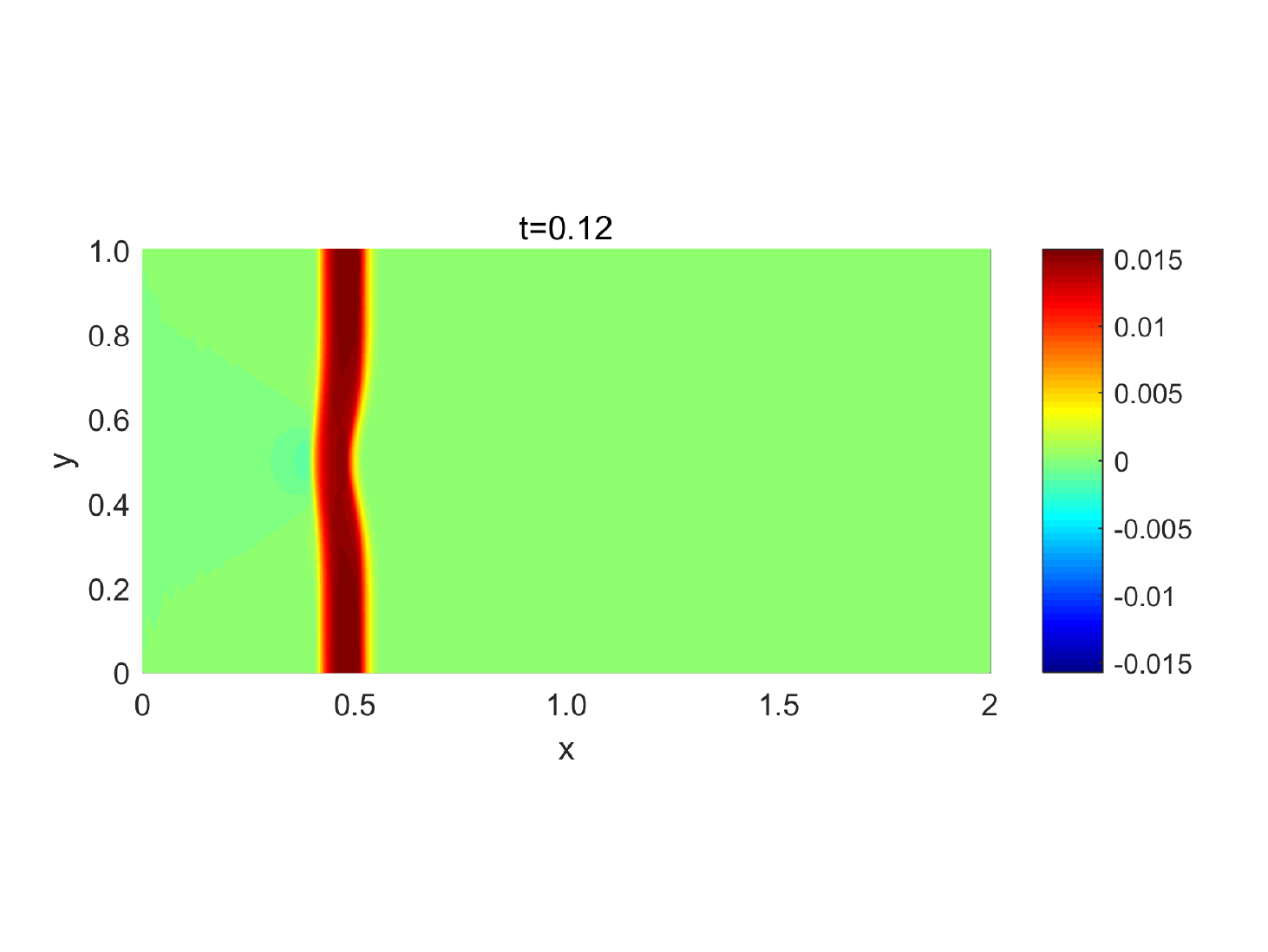}}
\subfigure[$hv$: FM $N=600\times 200\times4$]{
\includegraphics[width=0.30\textwidth, trim=15 60 15 60, clip]{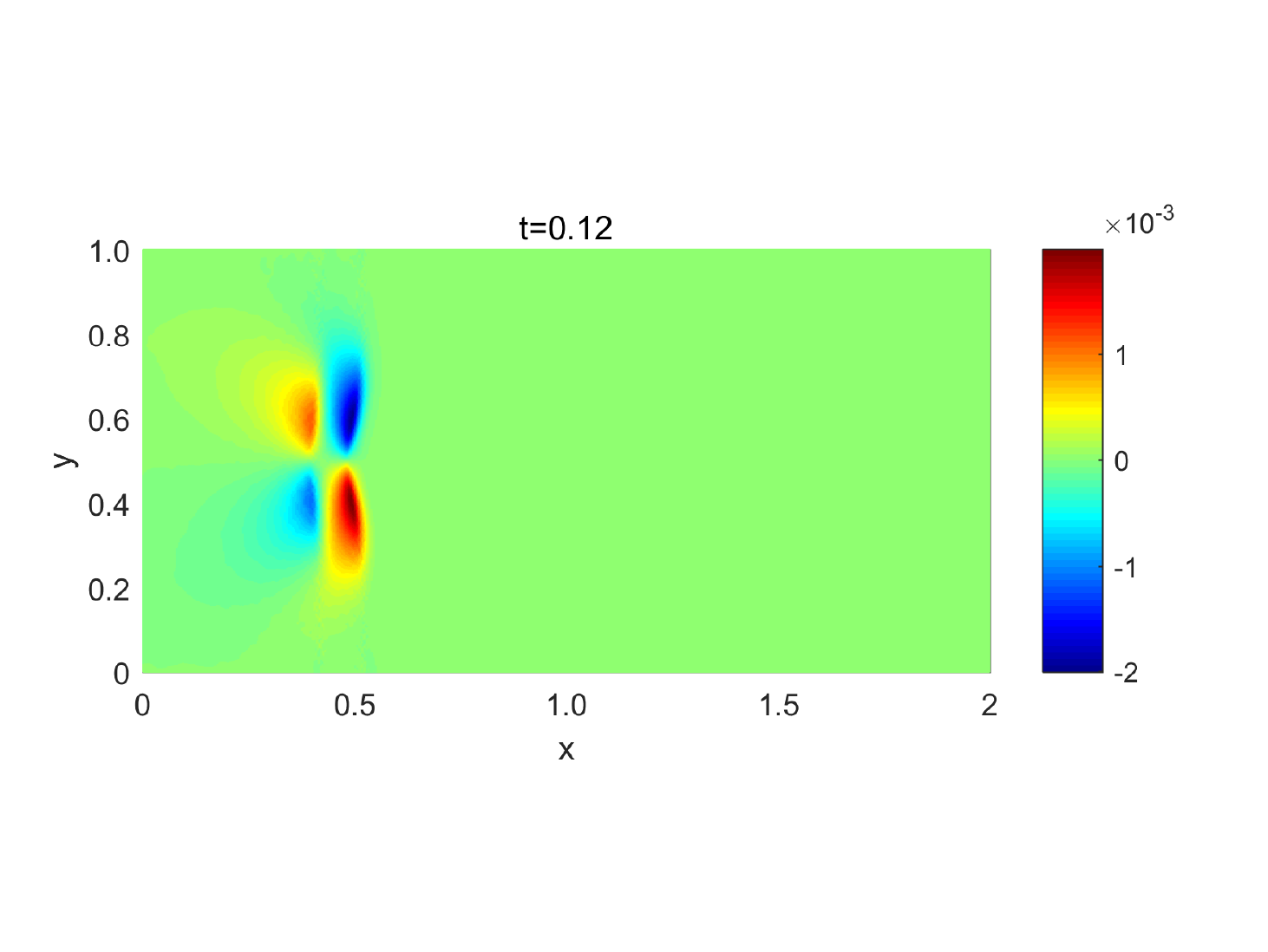}}
\caption{Example \ref{test2-2d}. The contours at $t = 0.12$ of $\eta$, $hu$, and $hv$ at $t=0.12$ are obtained with the $P^2$ QLMM-DG method and a moving mesh of $N=150\times 50\times4$ and fixed meshes of $N=150\times 50\times4$ and $N=600\times 200\times4$.}
\label{Fig:test2-2d-P2-h-hu-hv-t12}
\end{figure}

\begin{figure}[H]
\centering
\subfigure[$\eta$: MM $N=150\times 50\times4$]{
\includegraphics[width=0.30\textwidth, trim=15 60 15 60, clip]{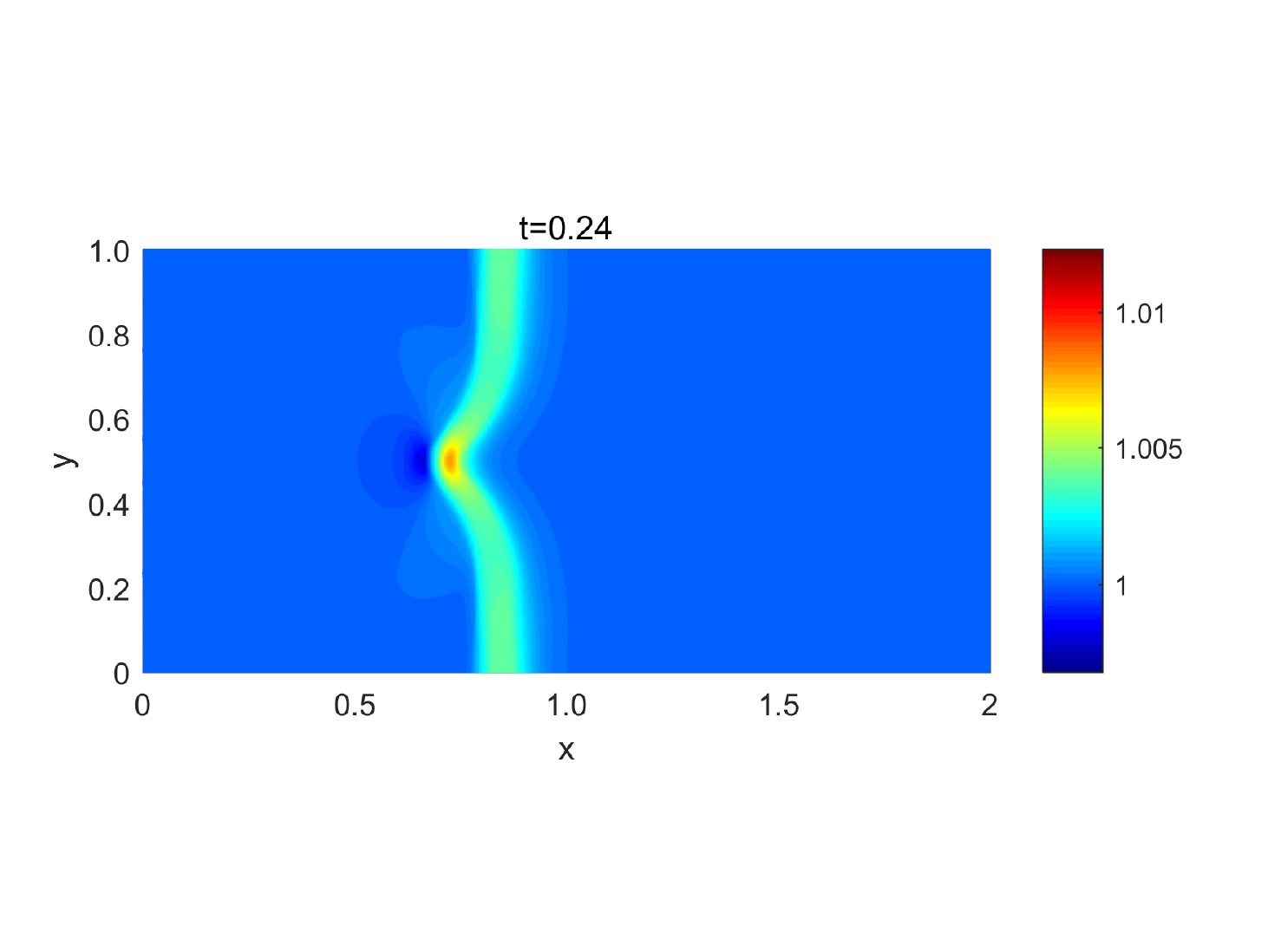}}
\subfigure[$hu$: MM $N=150\times 50\times4$]{
\includegraphics[width=0.30\textwidth, trim=15 60 15 60, clip]{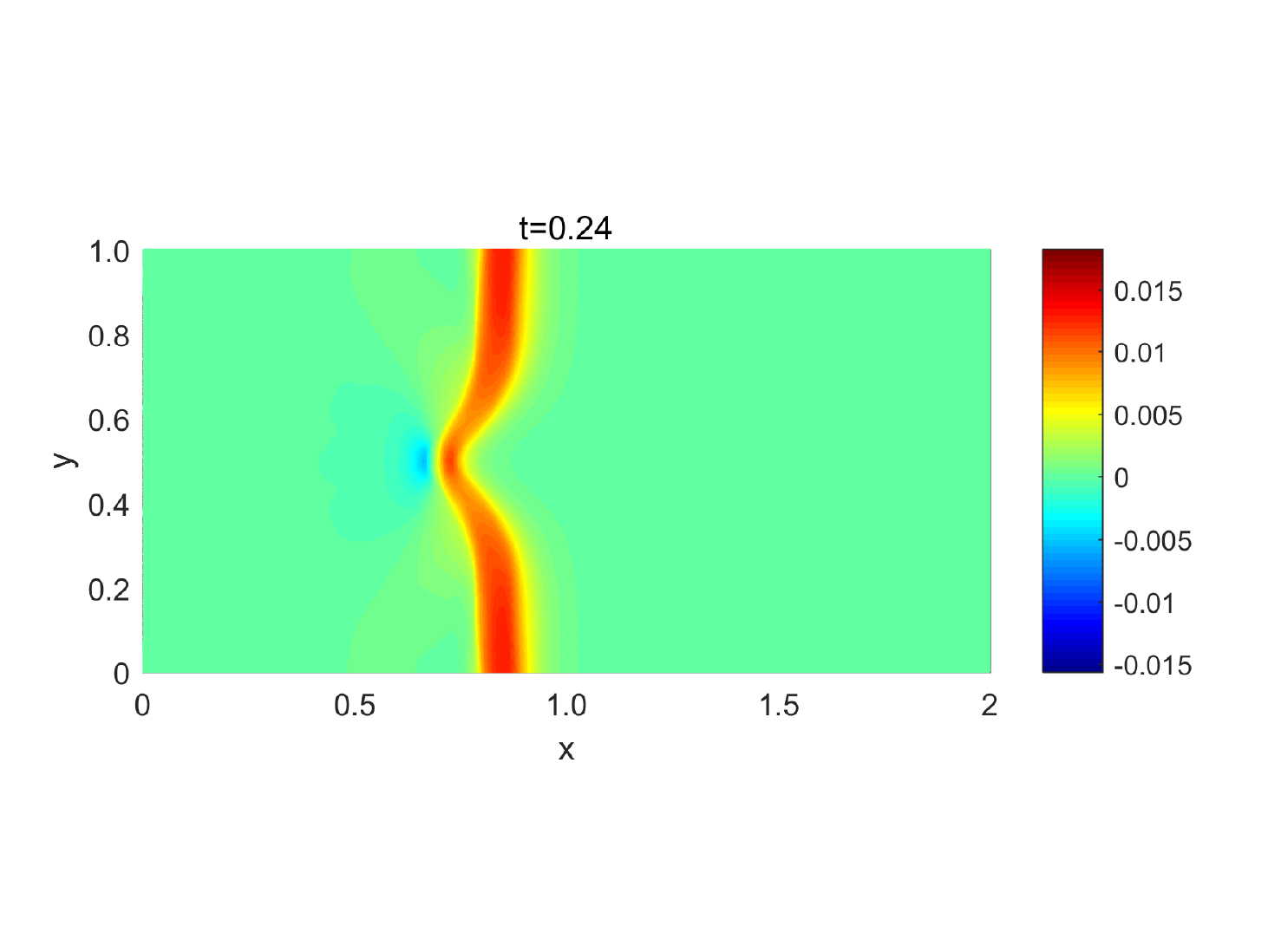}}
\subfigure[$hv$: MM $N=150\times 50\times4$]{
\includegraphics[width=0.30\textwidth, trim=15 60 15 60, clip]{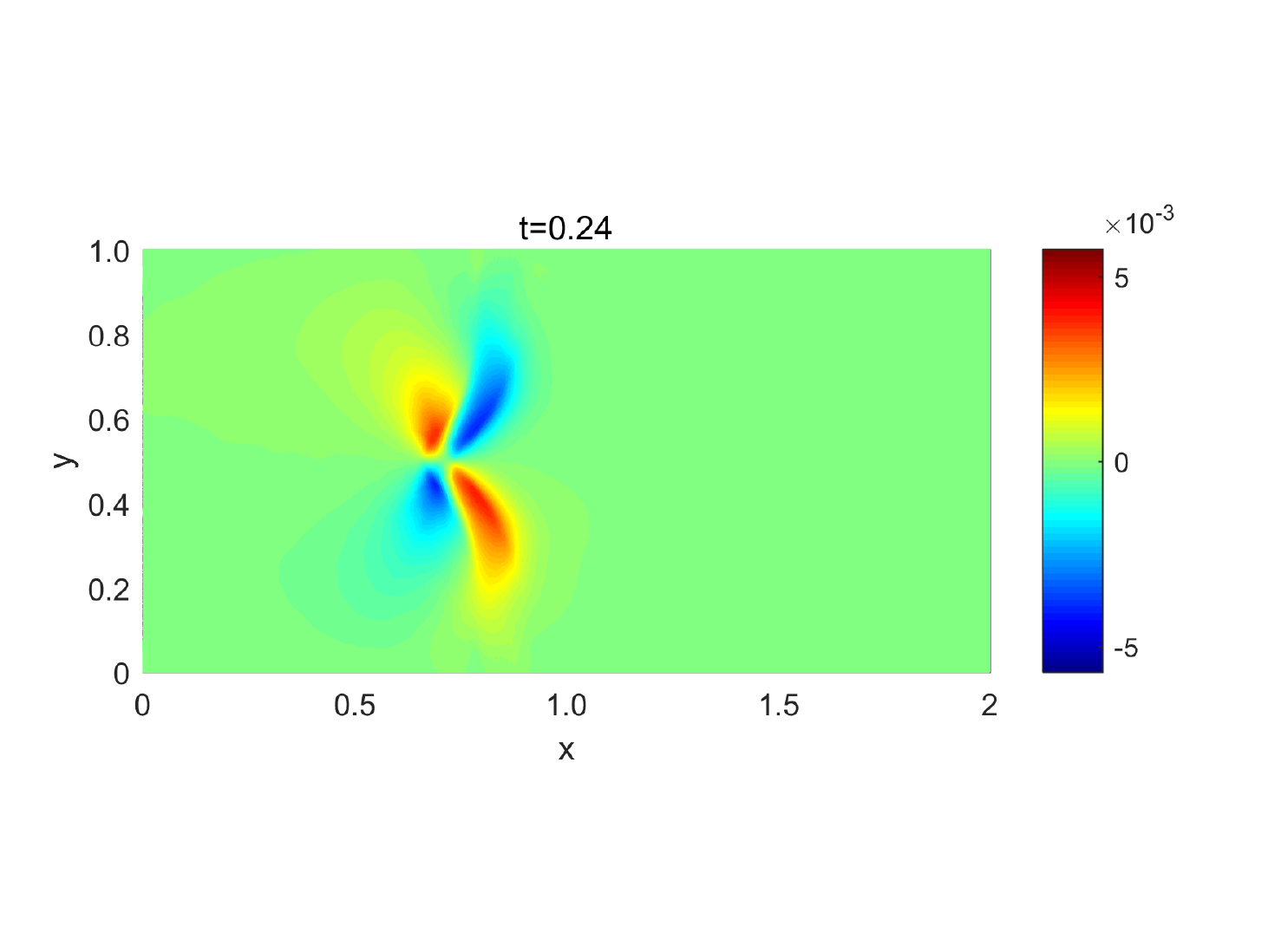}}
\subfigure[$\eta$: FM $N=150\times 50\times4$]{
\includegraphics[width=0.30\textwidth, trim=15 60 15 60, clip]{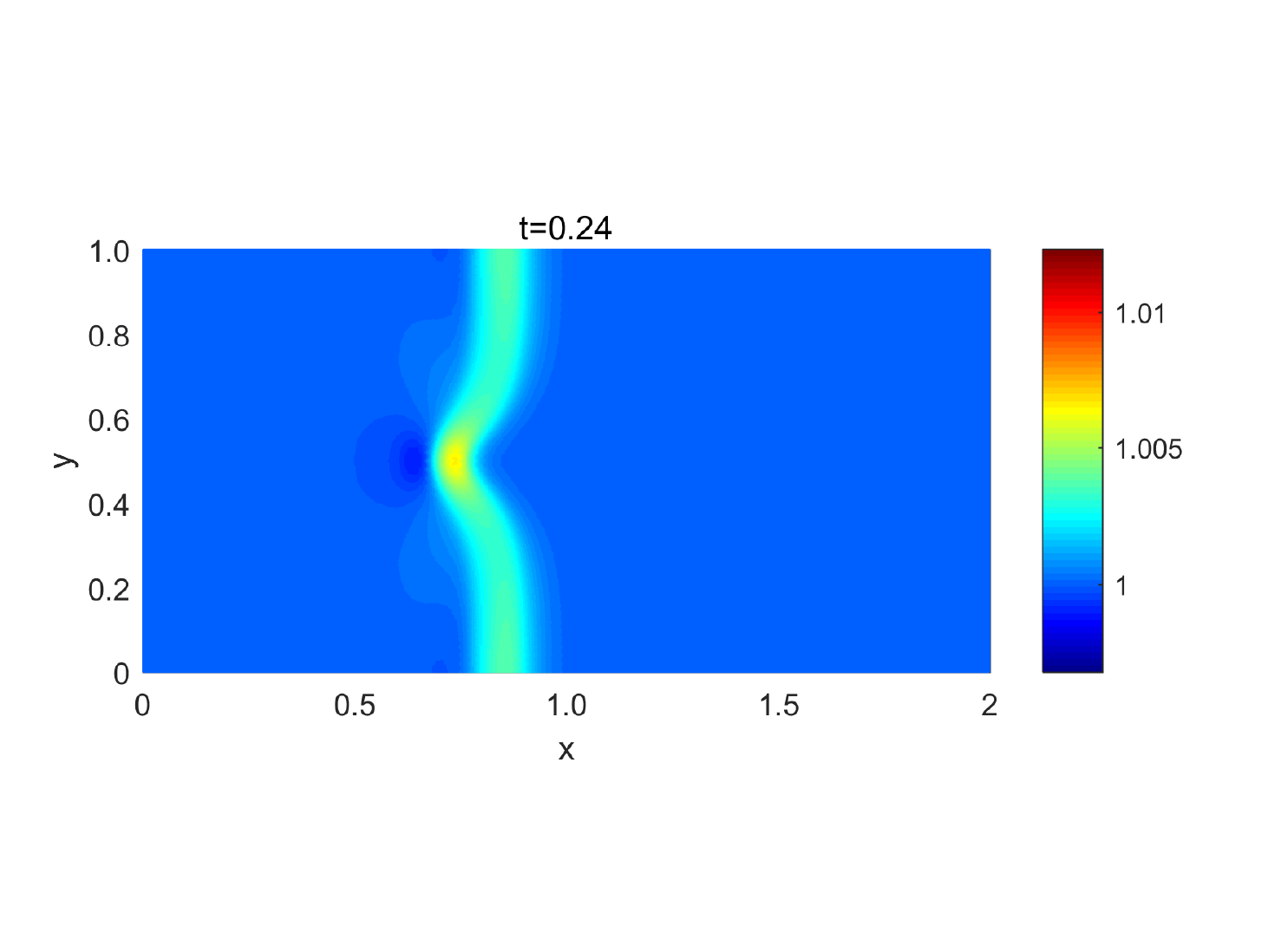}}
\subfigure[$hu$: FM $N=150\times 50\times4$]{
\includegraphics[width=0.30\textwidth, trim=15 60 15 60, clip]{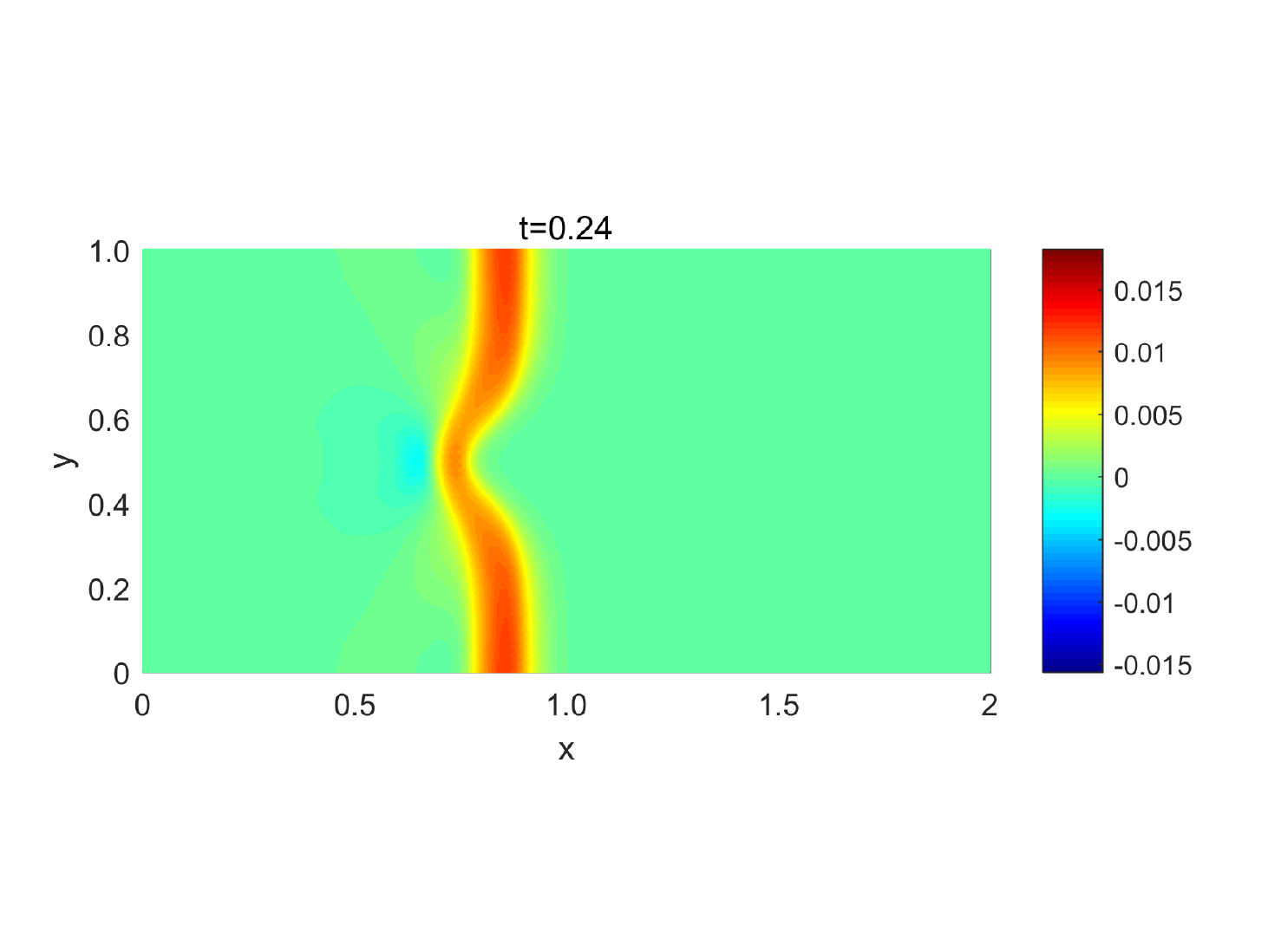}}
\subfigure[$hv$: FM $N=150\times 50\times4$]{
\includegraphics[width=0.30\textwidth, trim=15 60 15 60, clip]{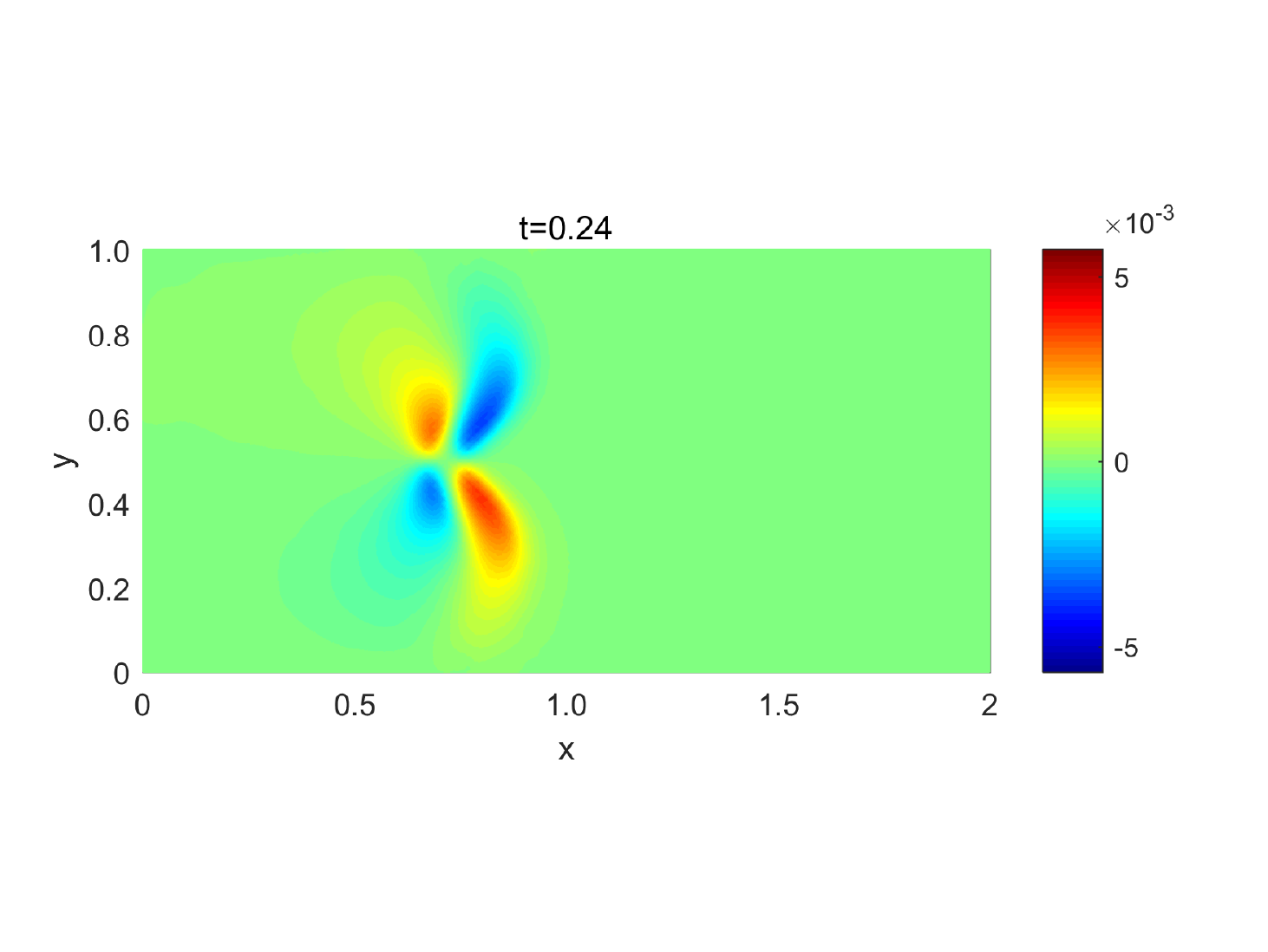}}
\subfigure[$\eta$: FM $600\times 200\times4$]{
\includegraphics[width=0.30\textwidth, trim=15 60 15 60, clip]{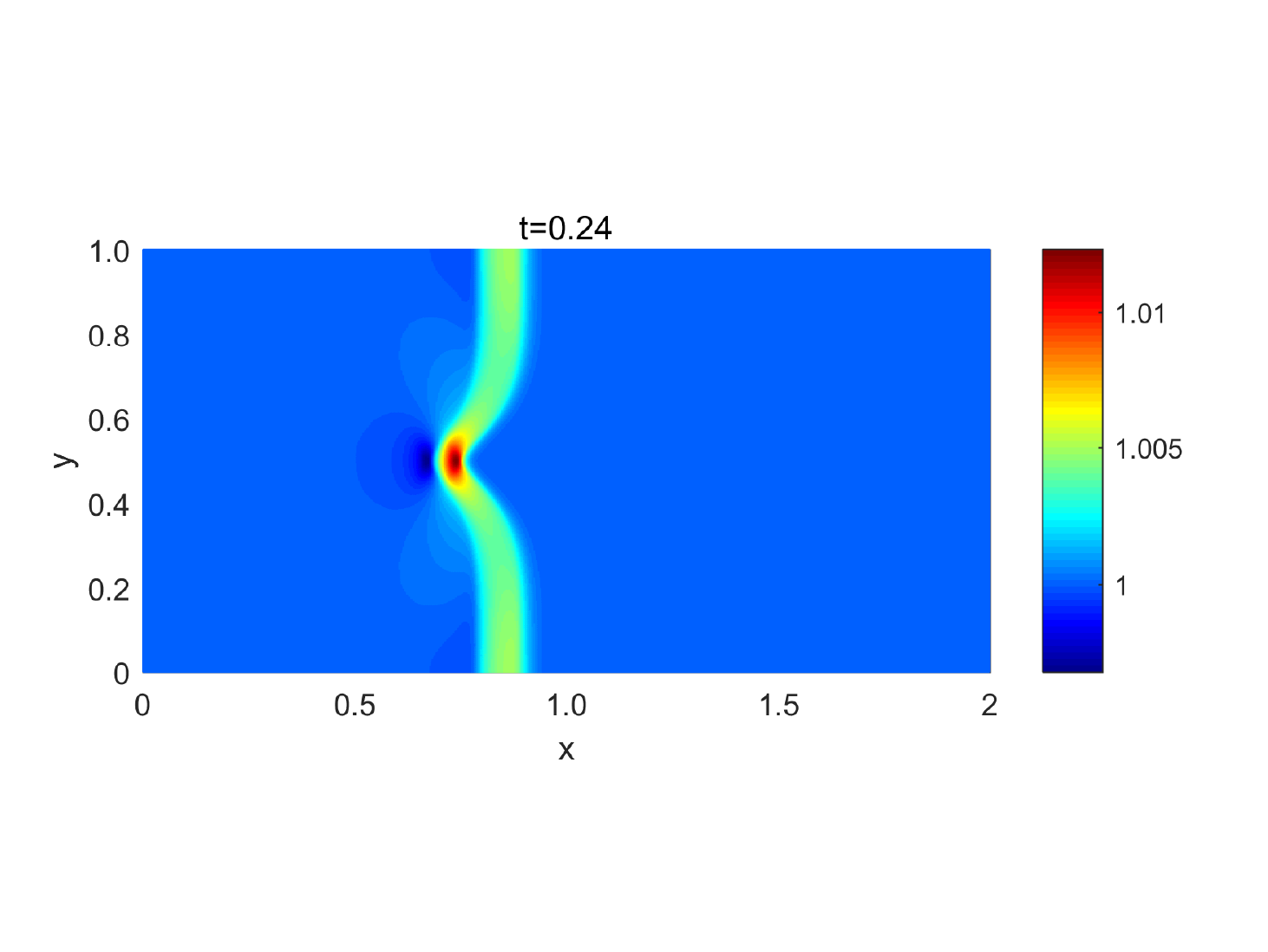}}
\subfigure[$hu$: FM $N=600\times 200\times4$]{
\includegraphics[width=0.30\textwidth, trim=15 60 15 60, clip]{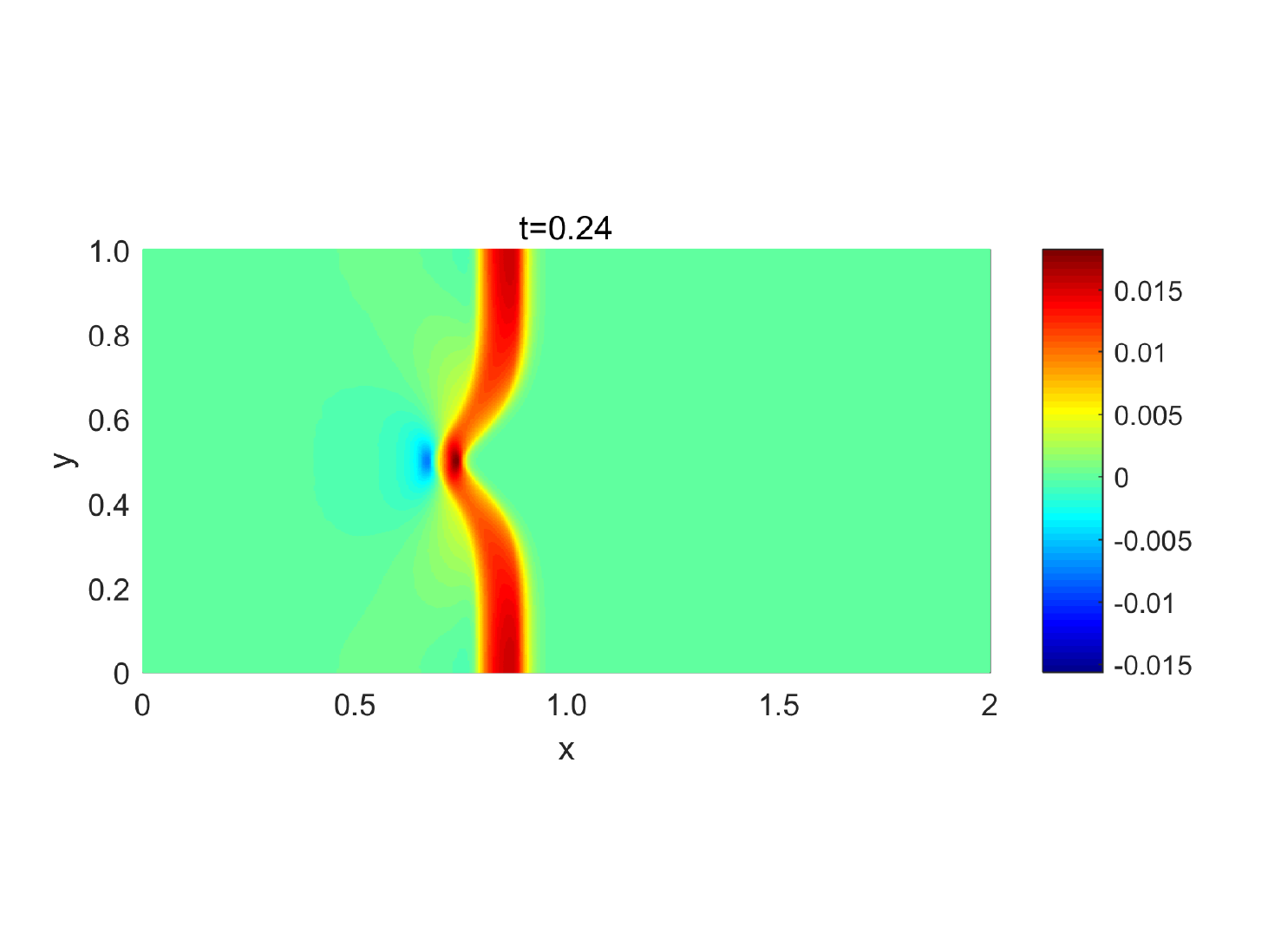}}
\subfigure[$hv$: FM $N=600\times 200\times4$]{
\includegraphics[width=0.30\textwidth, trim=15 60 15 60, clip]{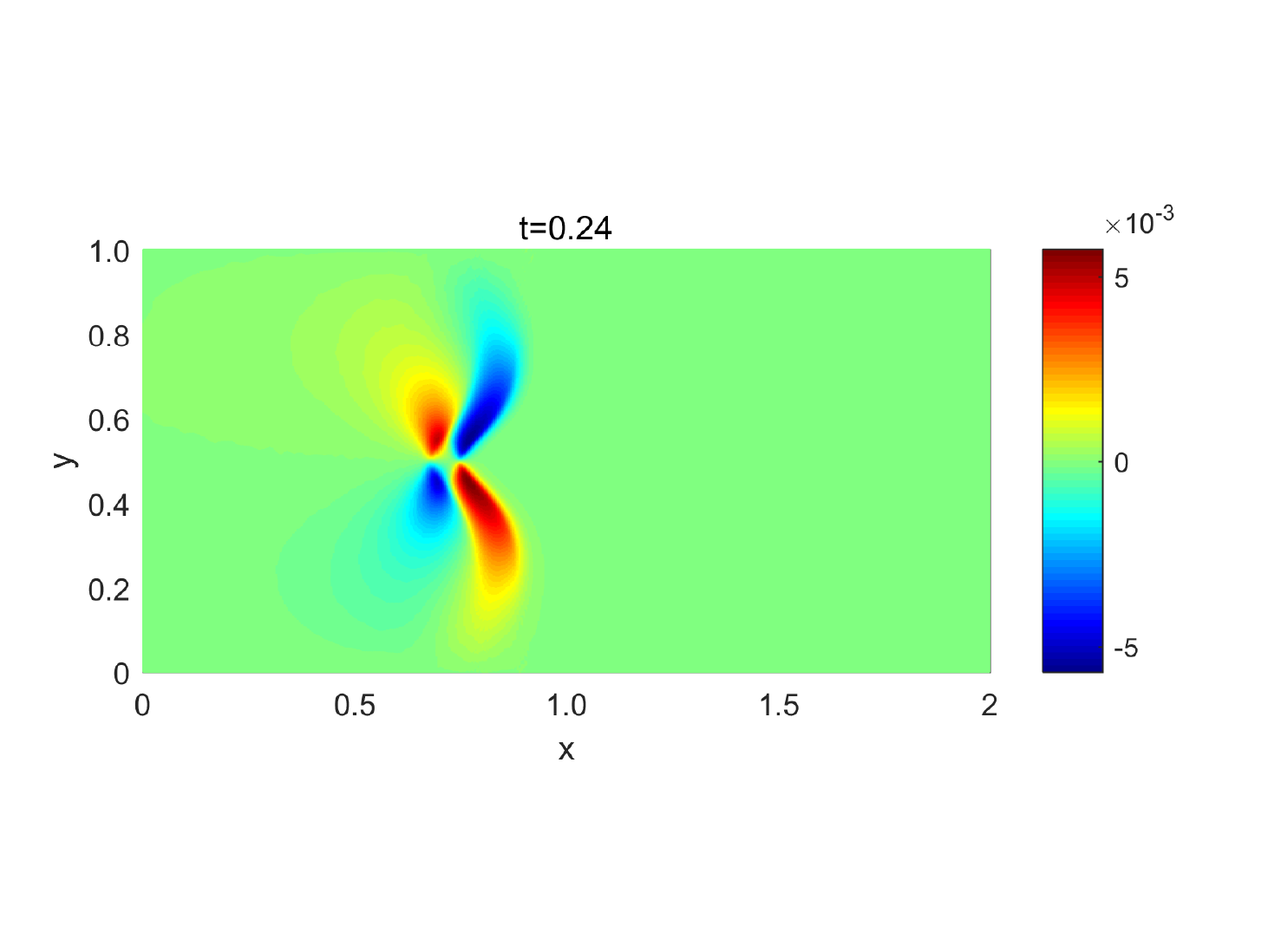}}
\caption{Continuation of Fig.~\ref{Fig:test2-2d-P2-h-hu-hv-t12}: $t = 0.24$.}
\label{Fig:test2-2d-P2-h-hu-hv-t24}
\end{figure}

\begin{figure}[H]
\centering
\subfigure[$\eta$: MM $N=150\times 50\times4$]{
\includegraphics[width=0.30\textwidth, trim=15 60 15 60, clip]{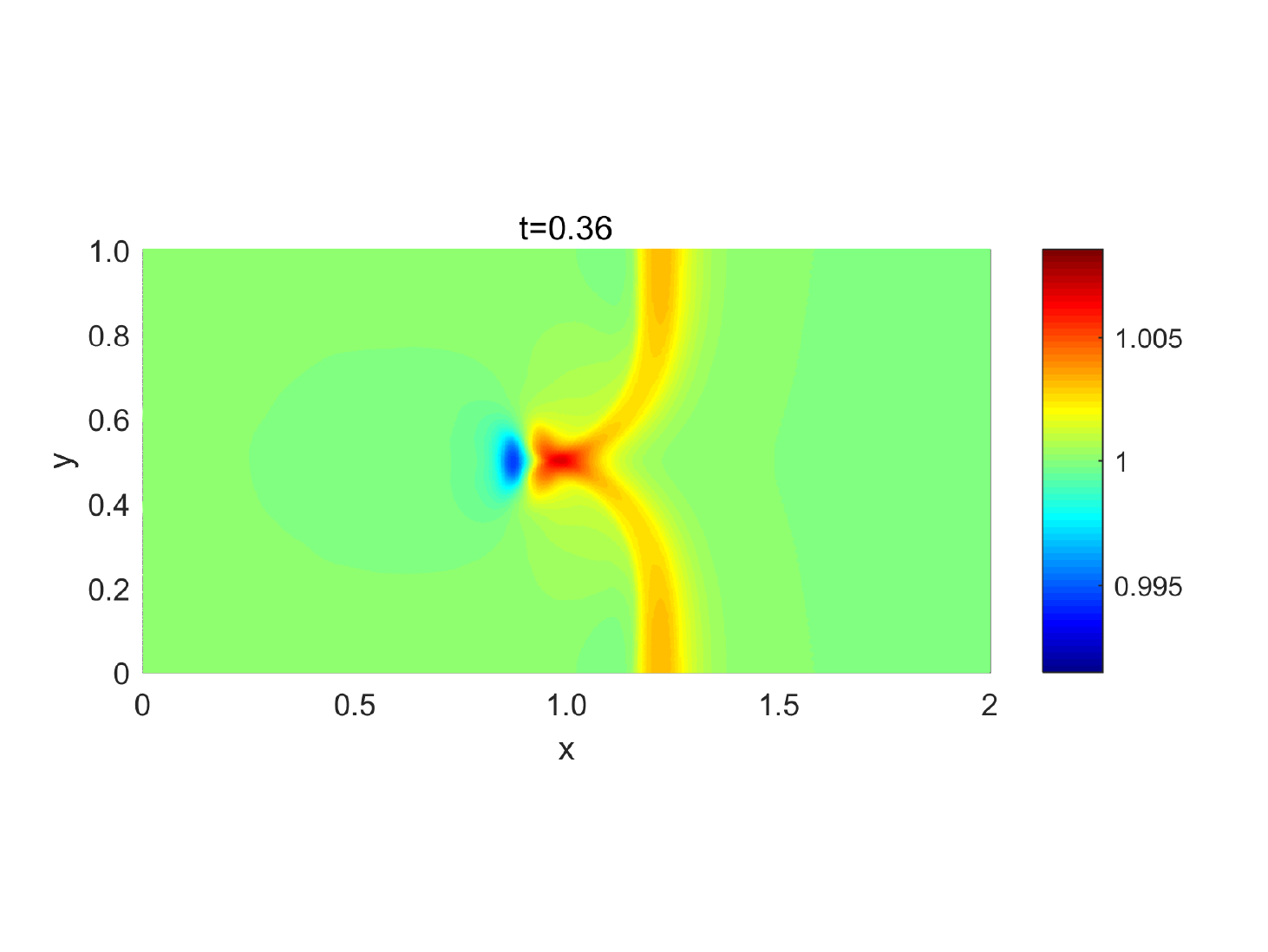}}
\subfigure[$hu$: MM $N=150\times 50\times4$]{
\includegraphics[width=0.30\textwidth, trim=15 60 15 60, clip]{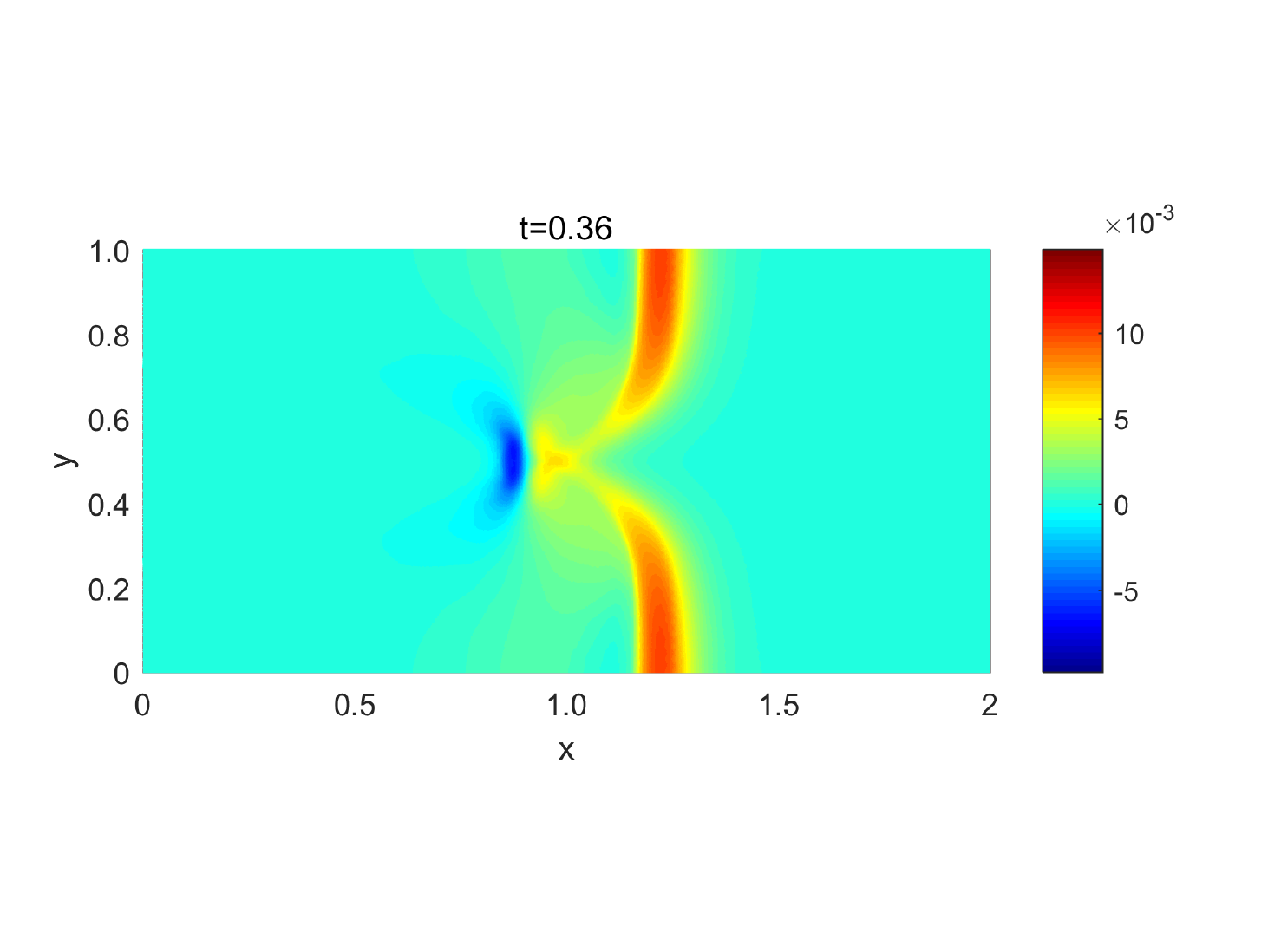}}
\subfigure[$hv$: MM $N=150\times 50\times4$]{
\includegraphics[width=0.30\textwidth, trim=15 60 15 60, clip]{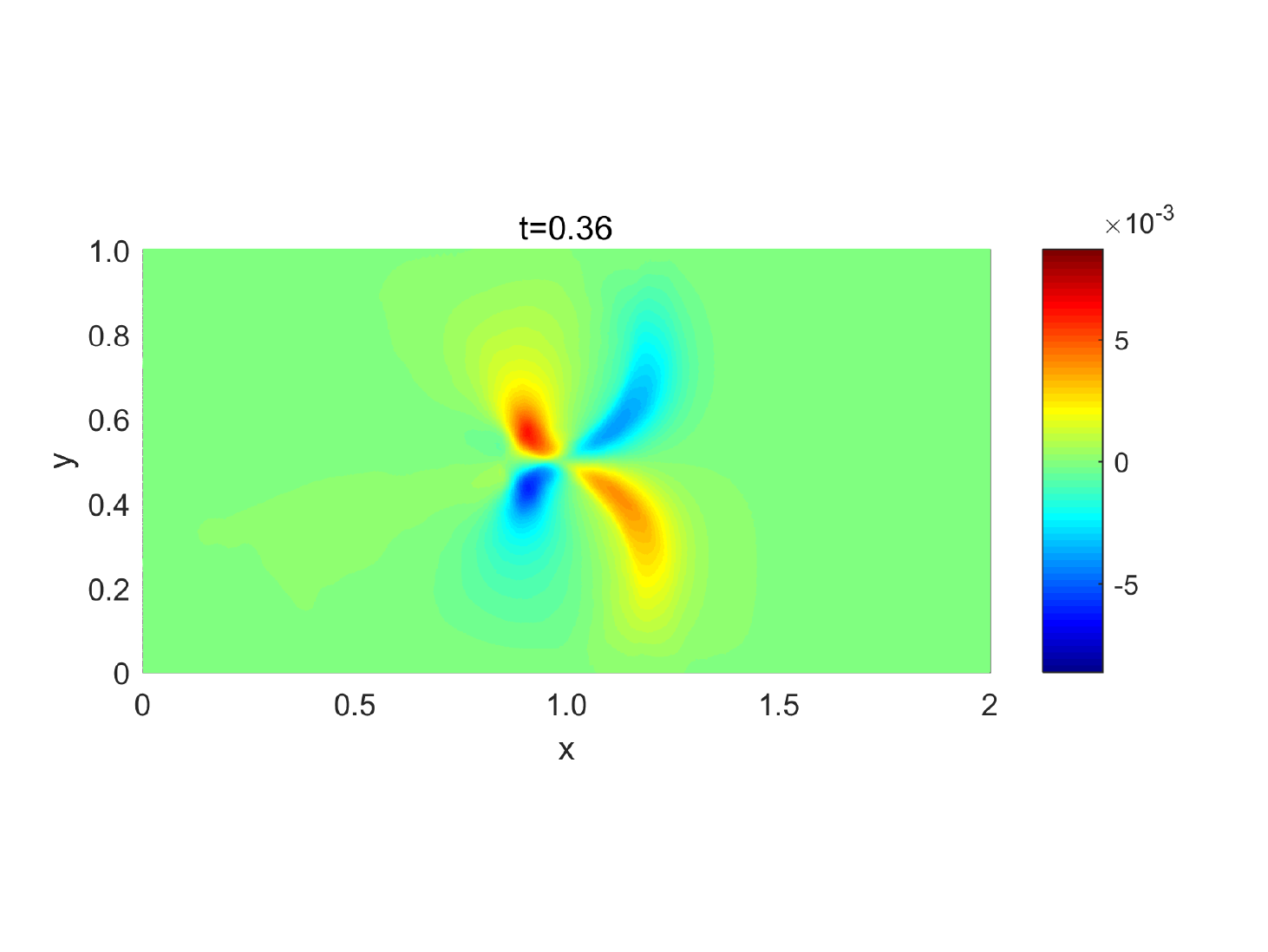}}
\subfigure[$\eta$: FM $N=150\times 50\times4$]{
\includegraphics[width=0.30\textwidth, trim=15 60 15 60, clip]{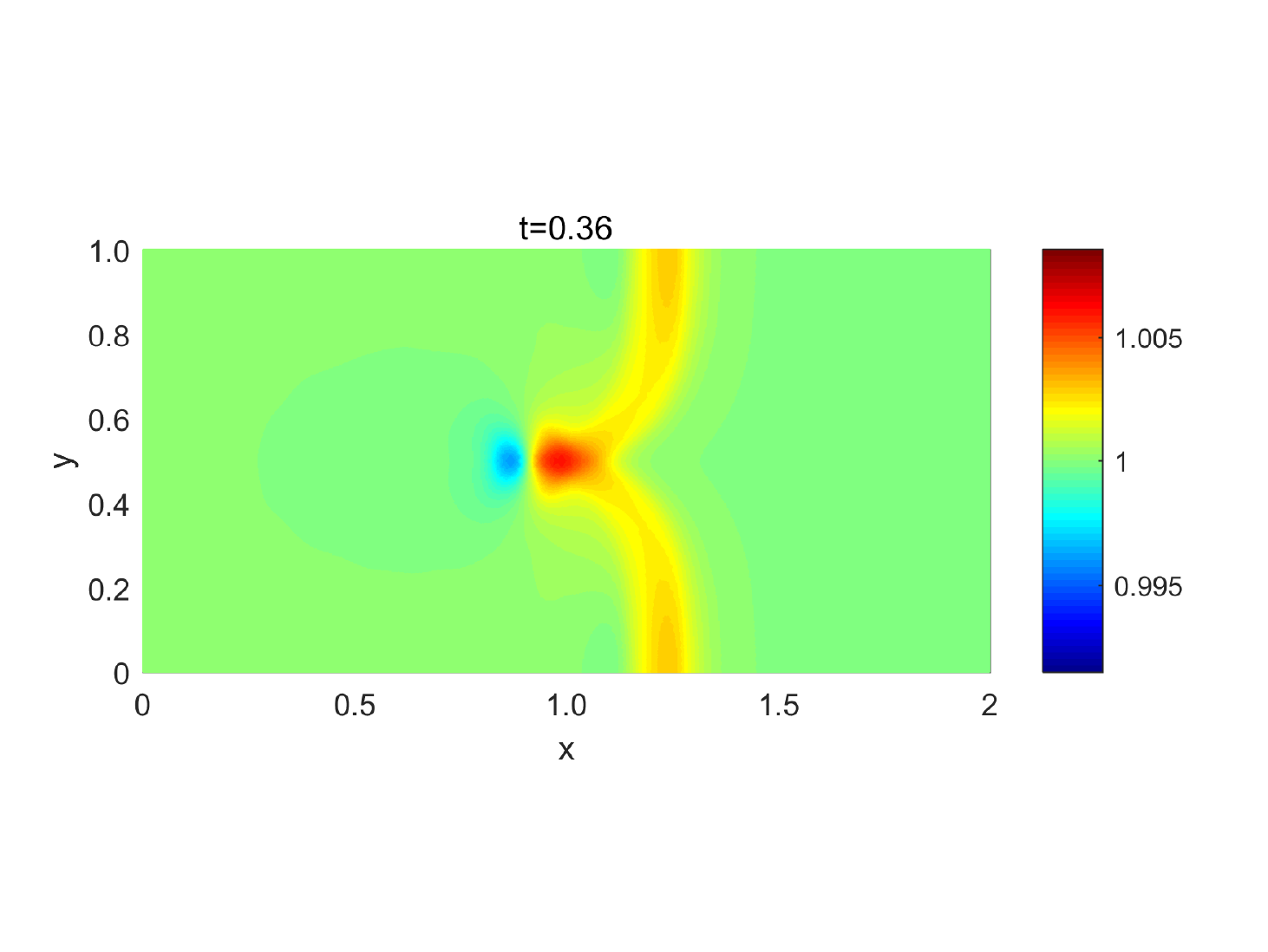}}
\subfigure[$hu$: FM $N=150\times 50\times4$]{
\includegraphics[width=0.30\textwidth, trim=15 60 15 60, clip]{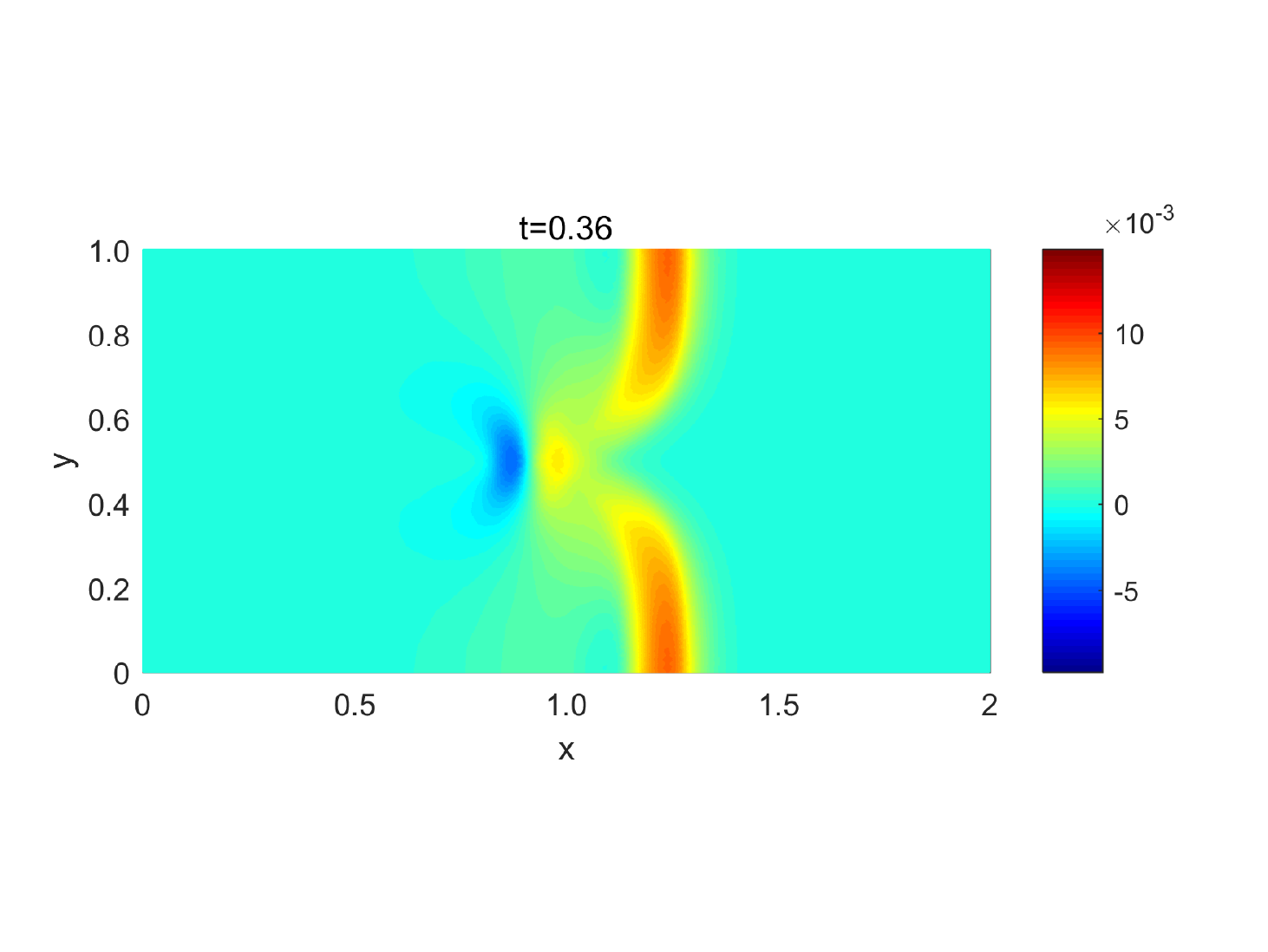}}
\subfigure[$hv$: FM $N=150\times 50\times4$]{
\includegraphics[width=0.30\textwidth, trim=15 60 15 60, clip]{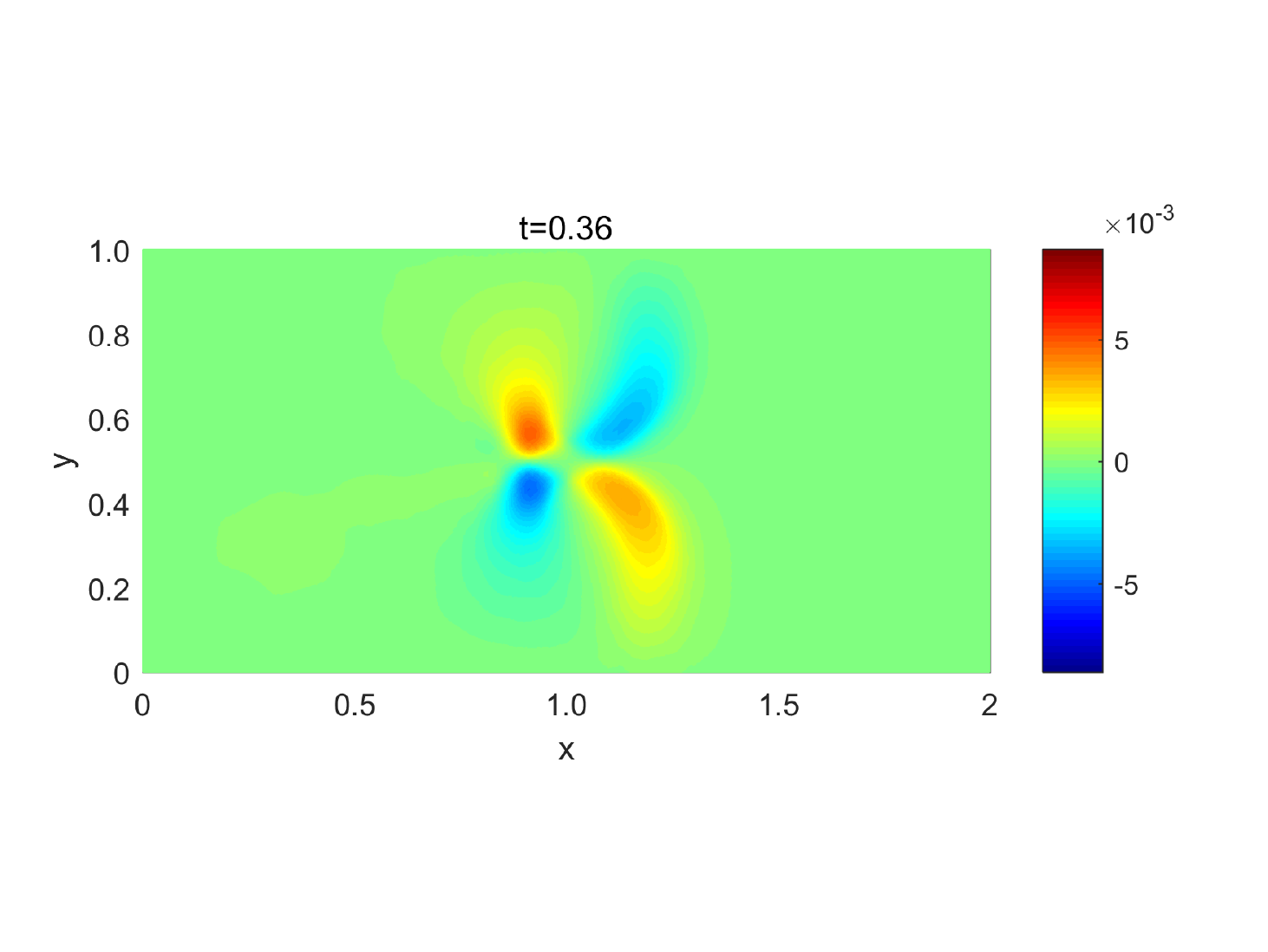}}
\subfigure[$\eta$: FM $600\times 200\times4$]{
\includegraphics[width=0.30\textwidth, trim=15 60 15 60, clip]{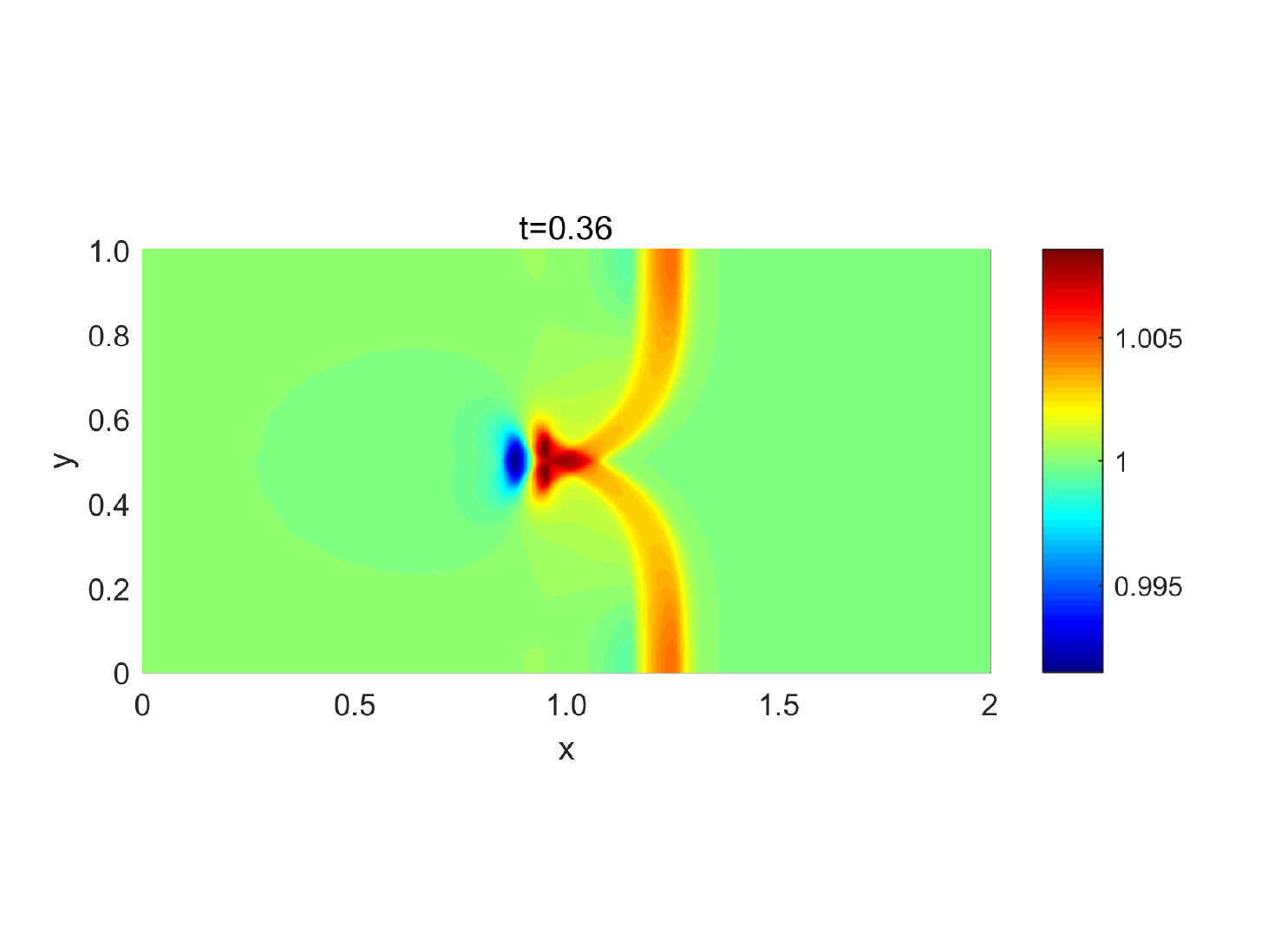}}
\subfigure[$hu$: FM $N=600\times 200\times4$]{
\includegraphics[width=0.30\textwidth, trim=15 60 15 60, clip]{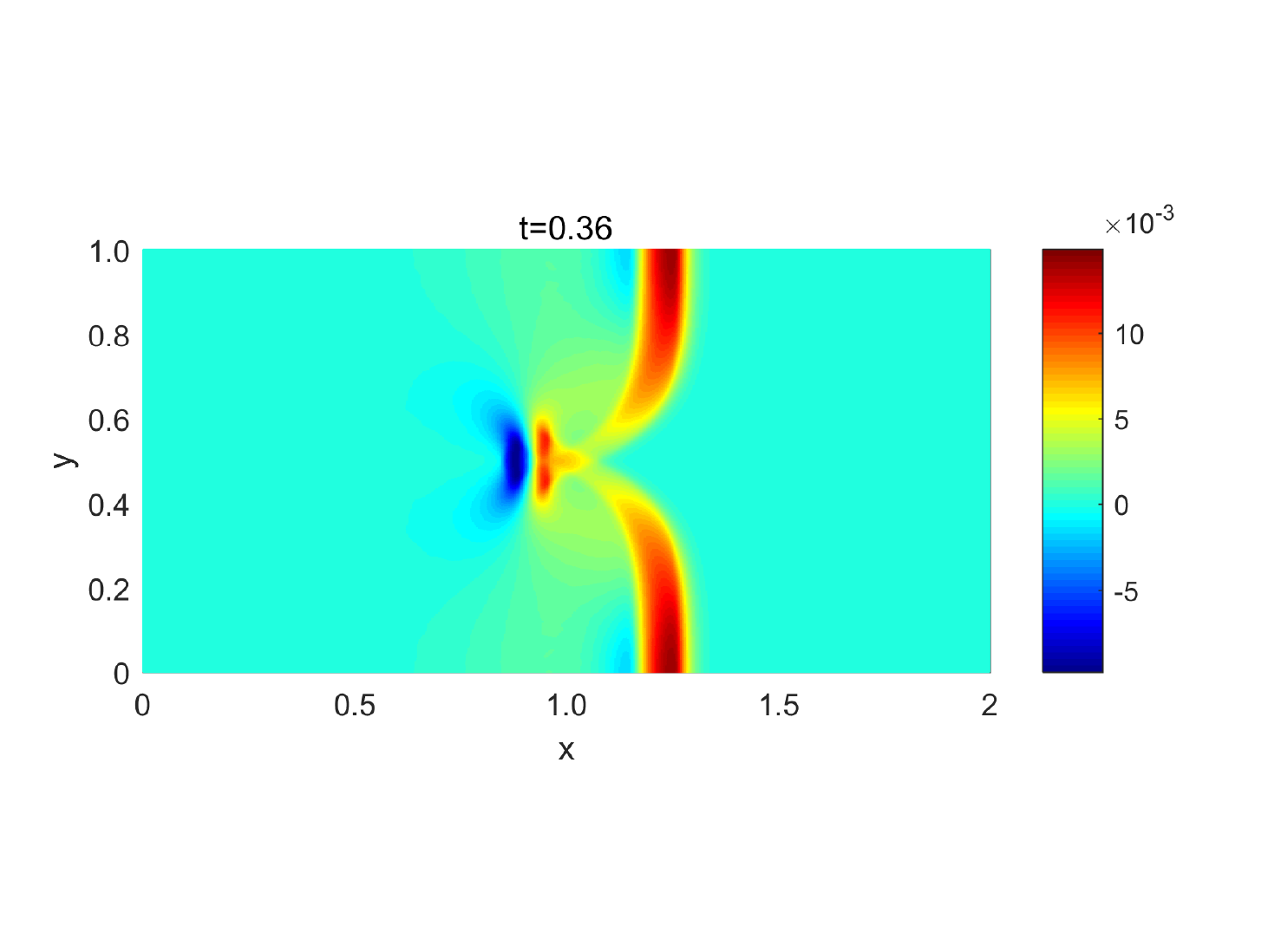}}
\subfigure[$hv$: FM $N=600\times 200\times4$]{
\includegraphics[width=0.30\textwidth, trim=15 60 15 60, clip]{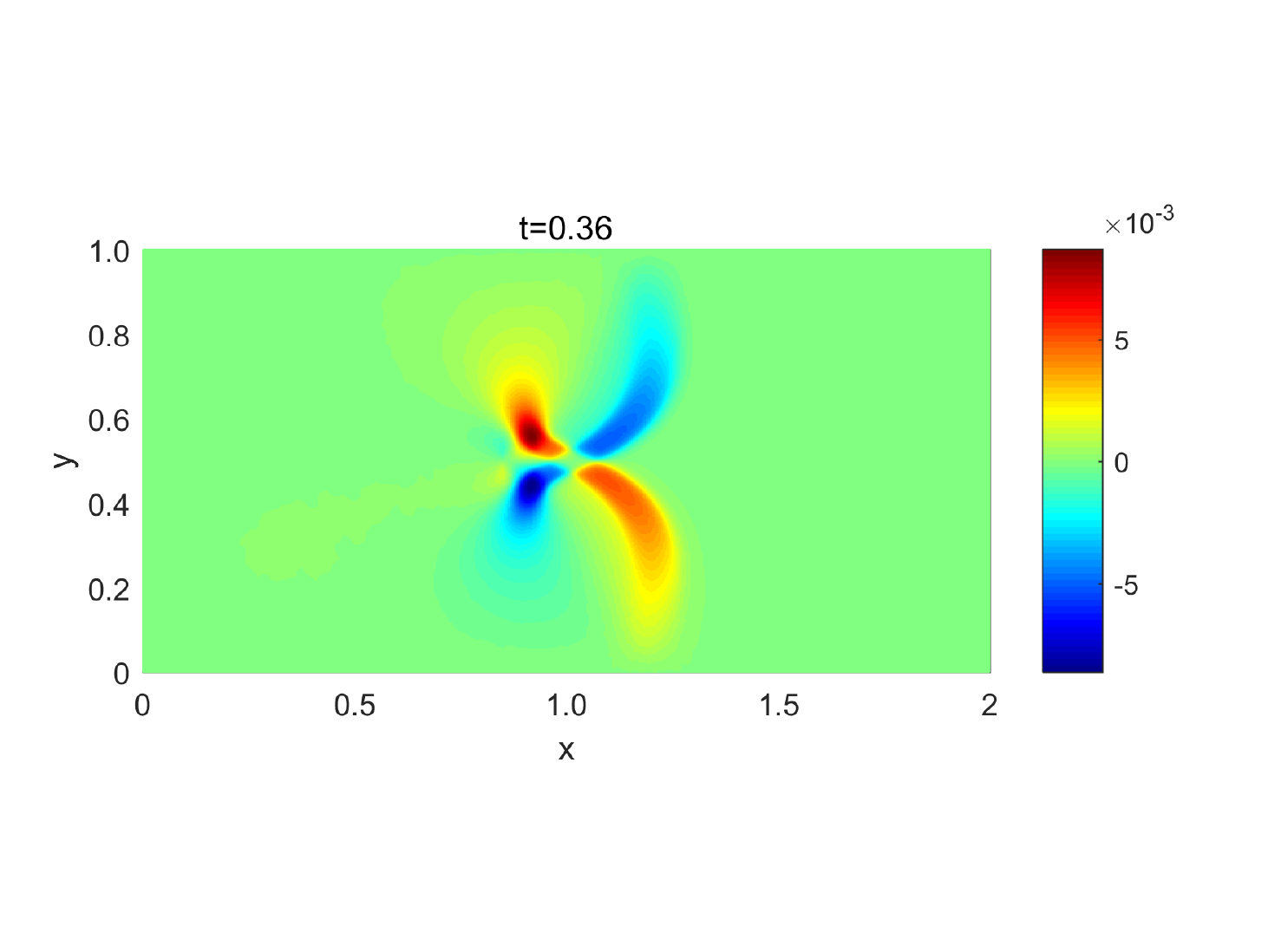}}
\caption{Continuation of Fig.~\ref{Fig:test2-2d-P2-h-hu-hv-t12}: $t = 0.36$.}
\label{Fig:test2-2d-P2-h-hu-hv-t36}
\end{figure}

\begin{figure}[H]
\centering
\subfigure[$\eta$: MM $N=150\times 50\times4$]{
\includegraphics[width=0.30\textwidth, trim=15 60 15 60, clip]{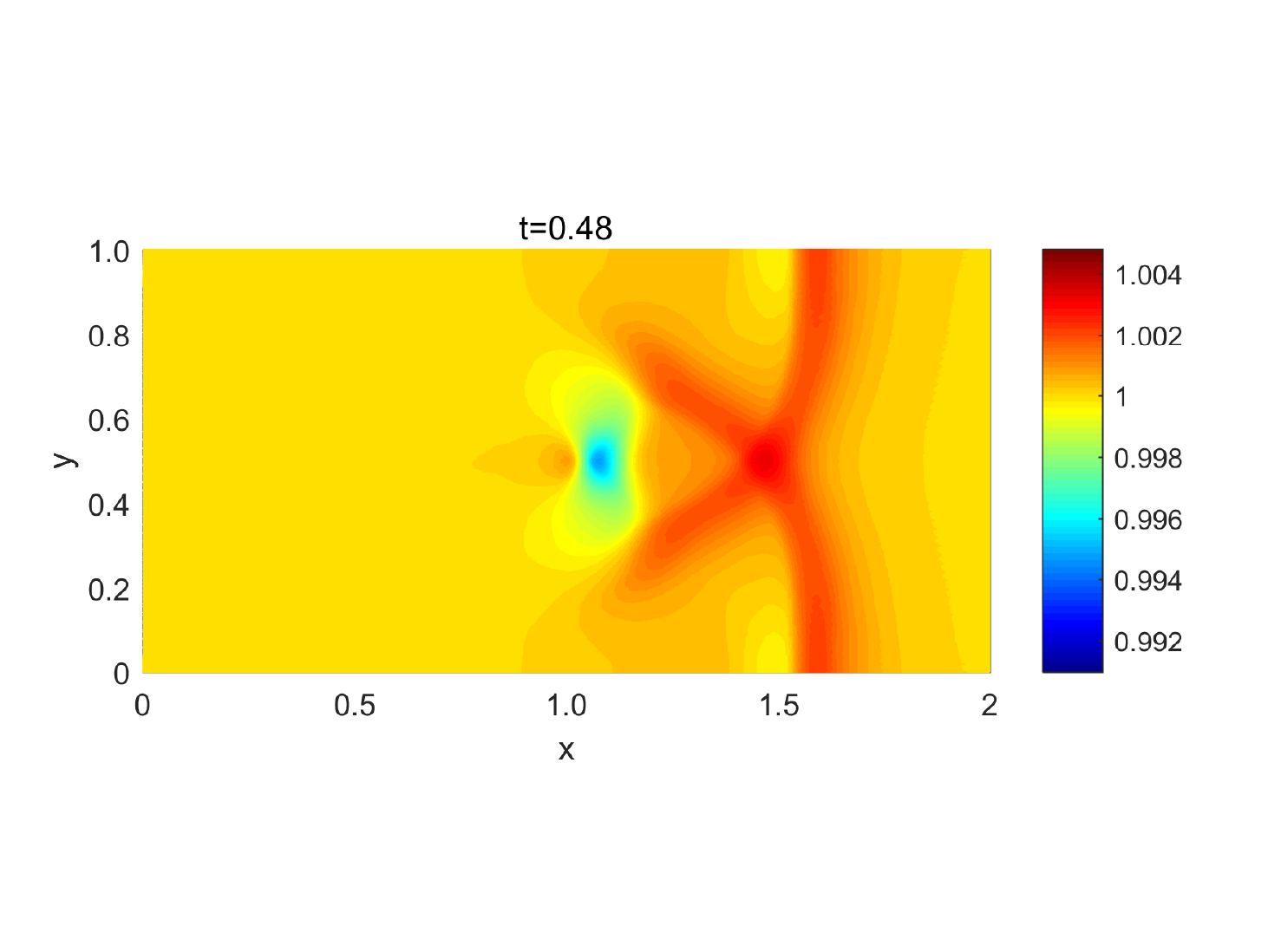}}
\subfigure[$hu$: MM $N=150\times 50\times4$]{
\includegraphics[width=0.30\textwidth, trim=15 60 15 60, clip]{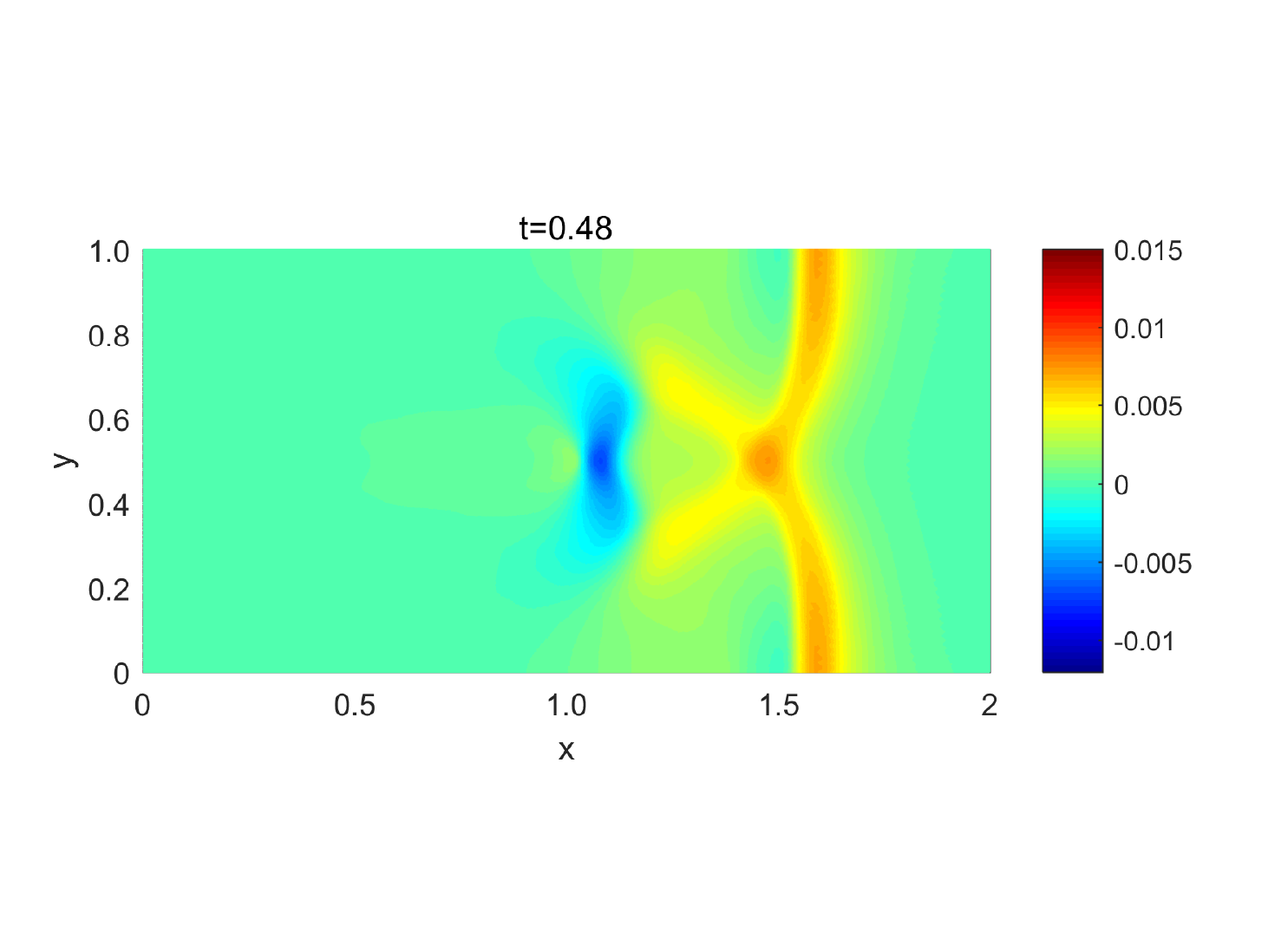}}
\subfigure[$hv$: MM $N=150\times 50\times4$]{
\includegraphics[width=0.30\textwidth, trim=15 60 15 60, clip]{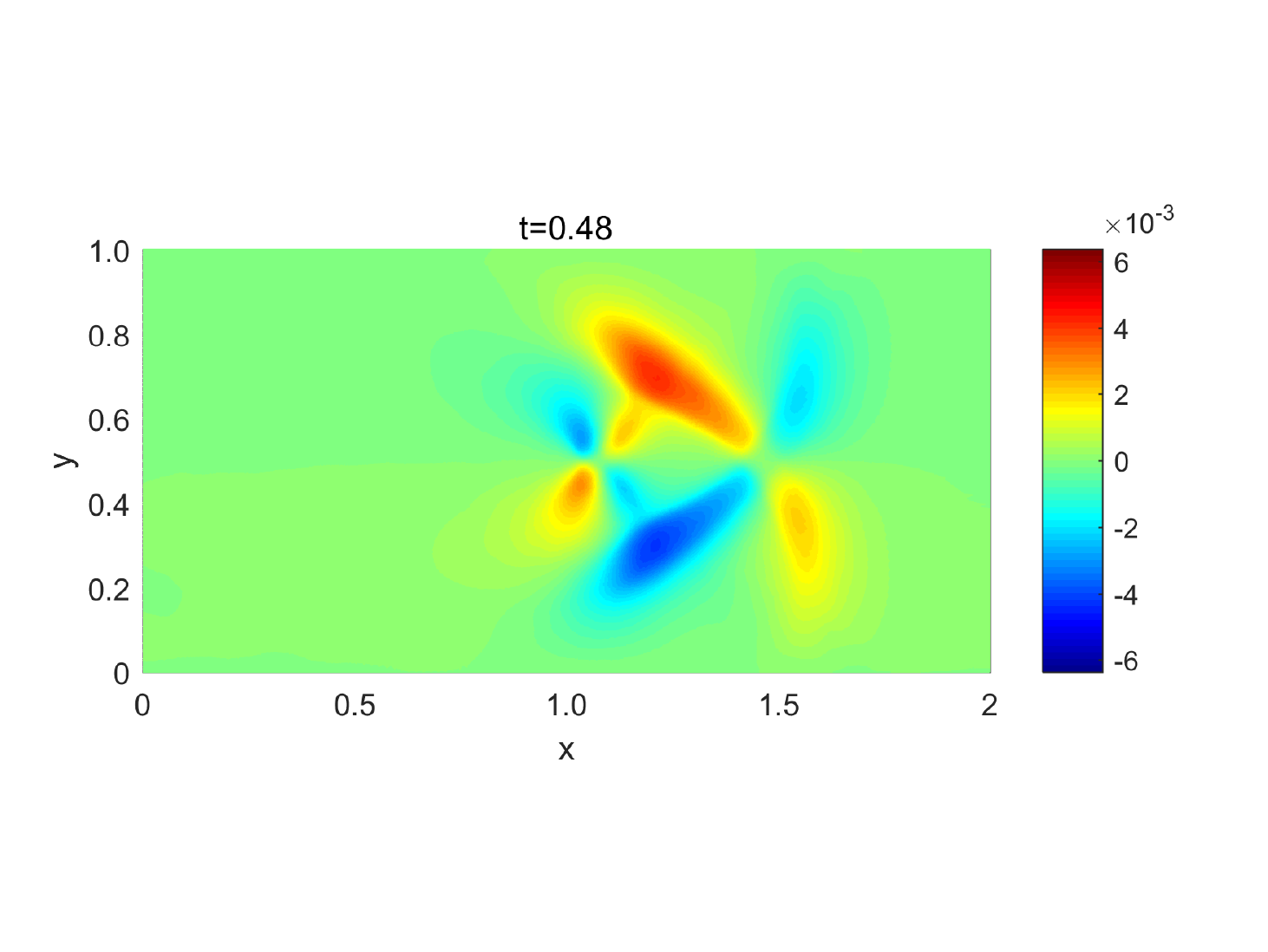}}
\subfigure[$\eta$: FM $N=150\times 50\times4$]{
\includegraphics[width=0.30\textwidth, trim=15 60 15 60, clip]{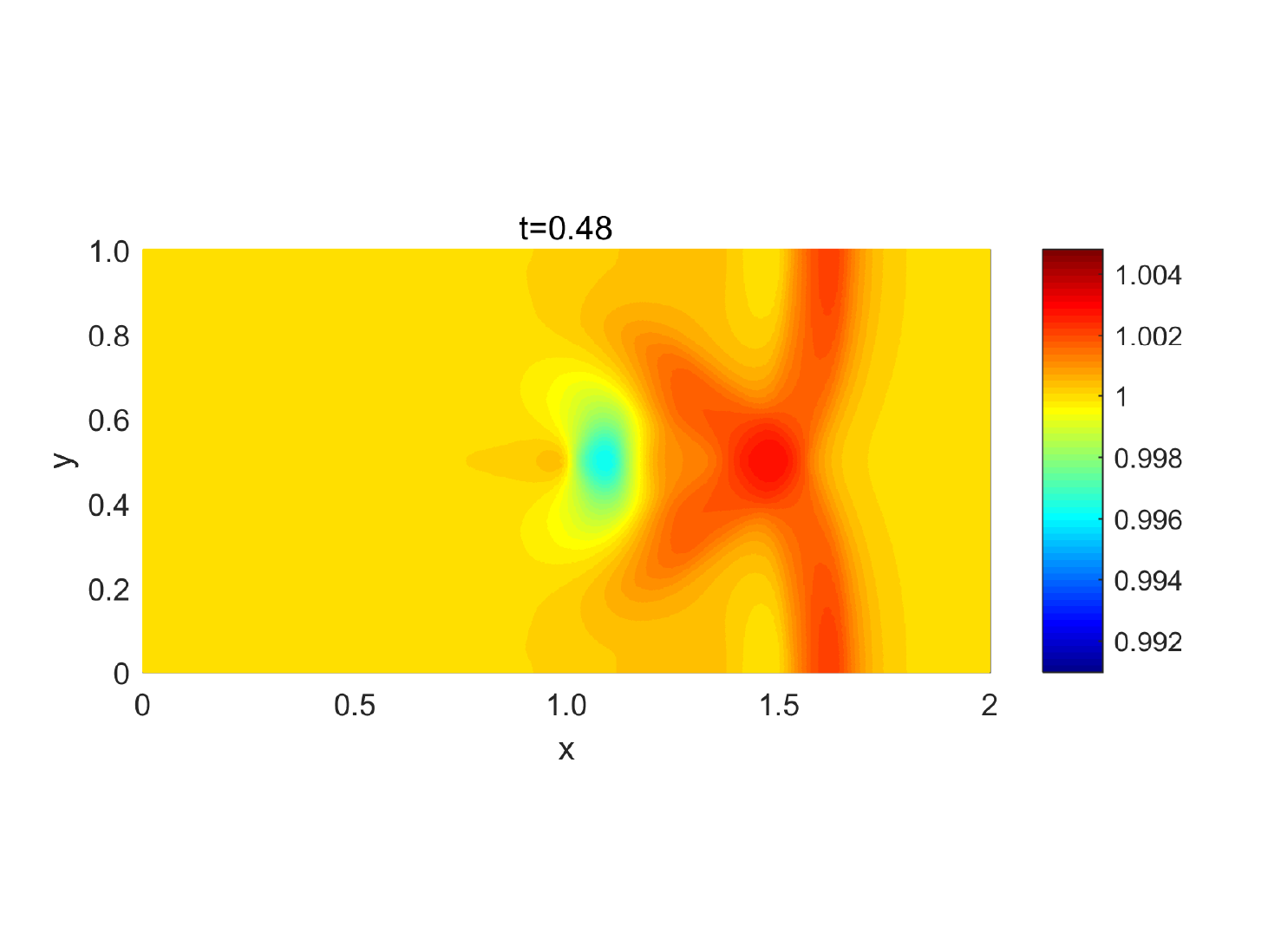}}
\subfigure[$hu$: FM $N=150\times 50\times4$]{
\includegraphics[width=0.30\textwidth, trim=15 60 15 60, clip]{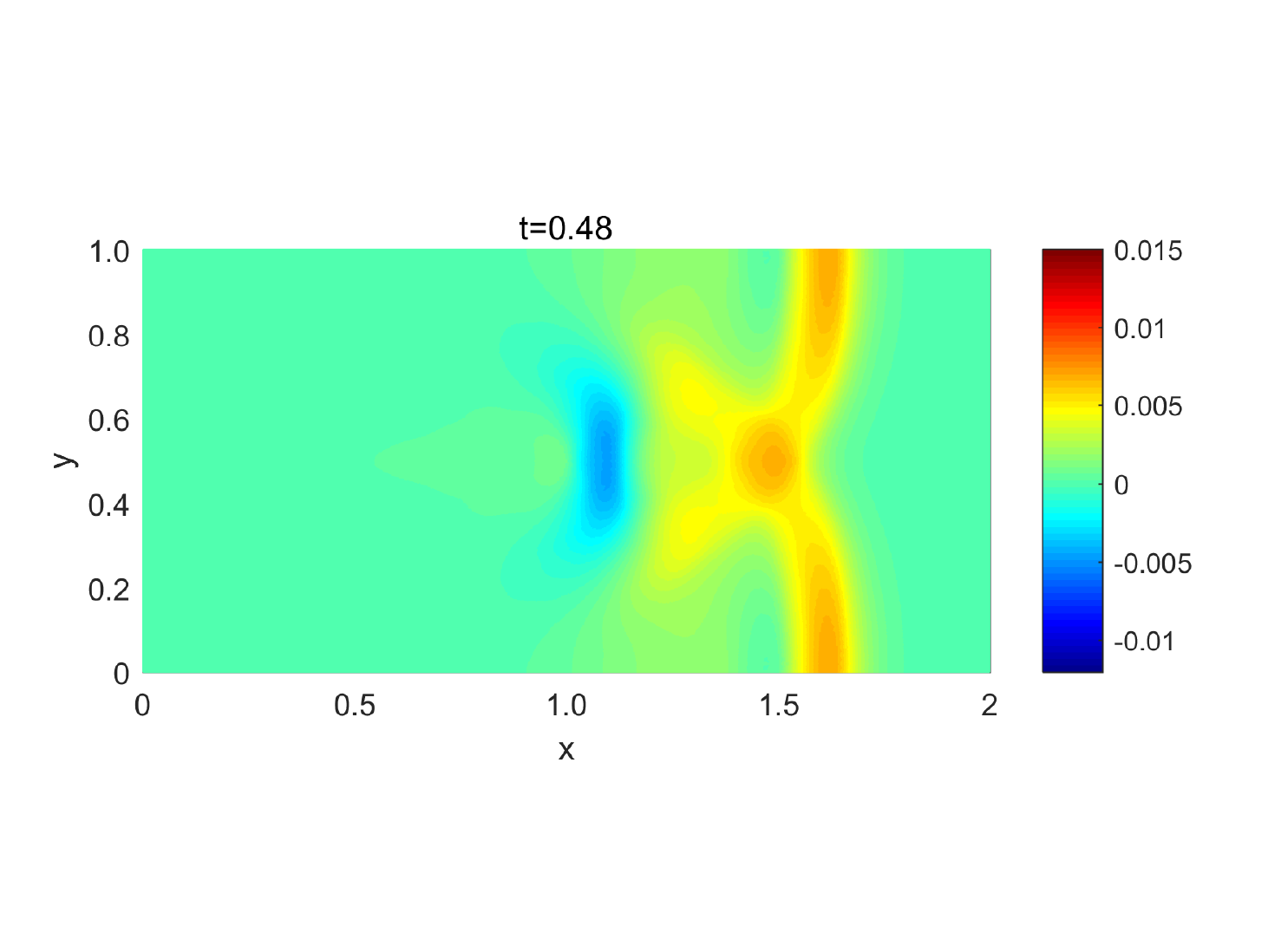}}
\subfigure[$hv$: FM $N=150\times 50\times4$]{
\includegraphics[width=0.30\textwidth, trim=15 60 15 60, clip]{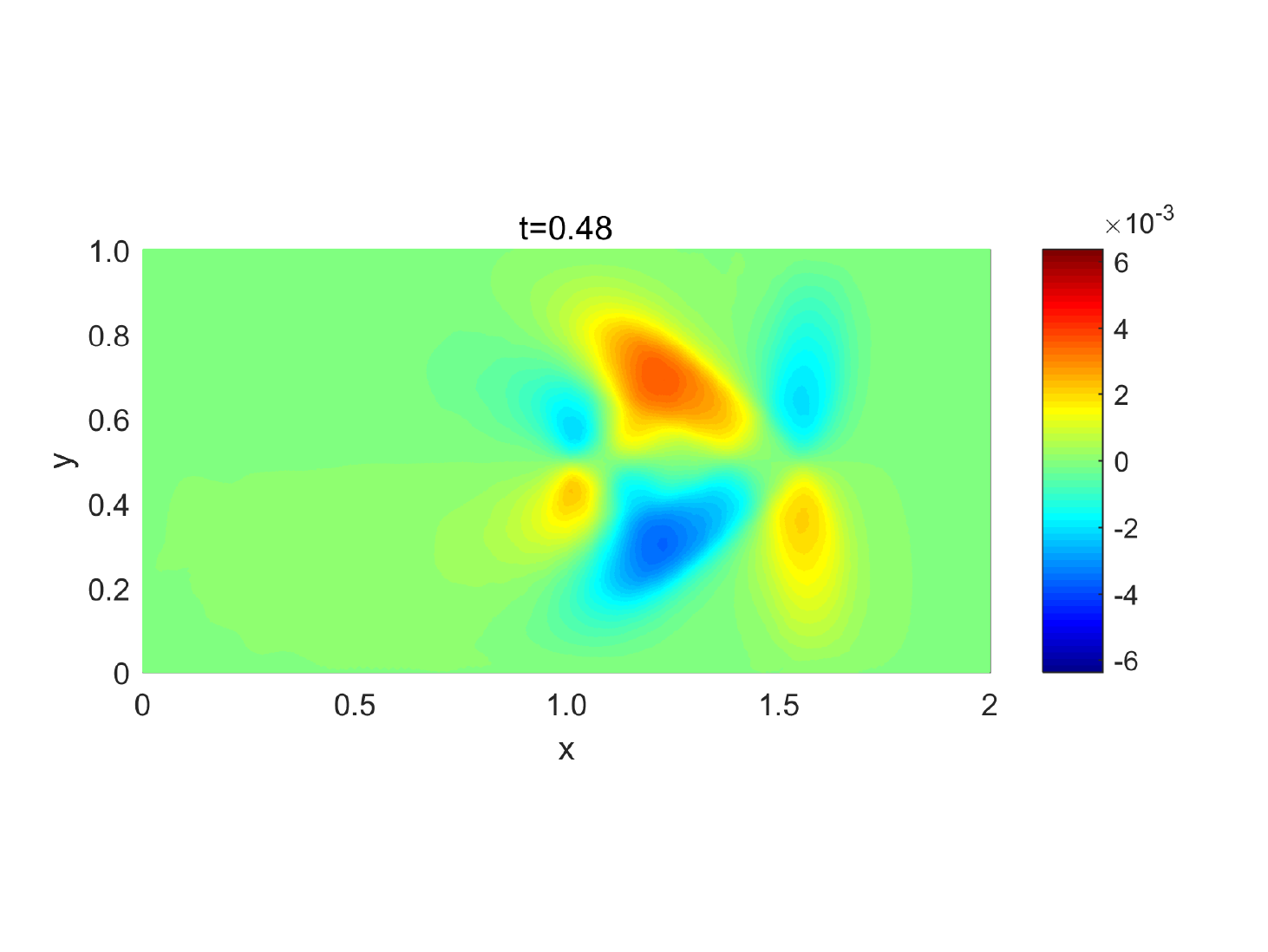}}
\subfigure[$\eta$: FM $600\times 200\times4$]{
\includegraphics[width=0.30\textwidth, trim=15 60 15 60, clip]{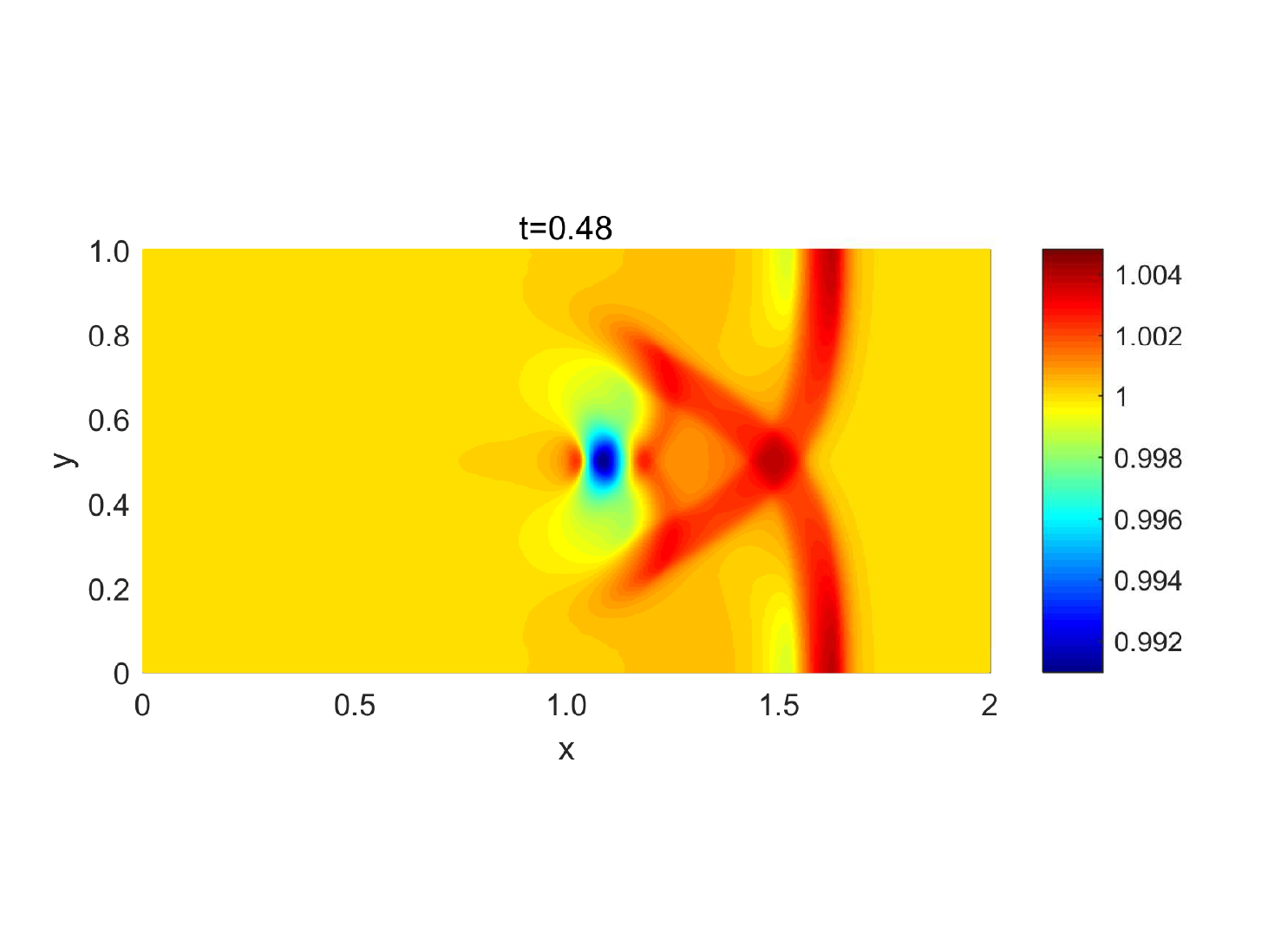}}
\subfigure[$hu$: FM $N=600\times 200\times4$]{
\includegraphics[width=0.30\textwidth, trim=15 60 15 60, clip]{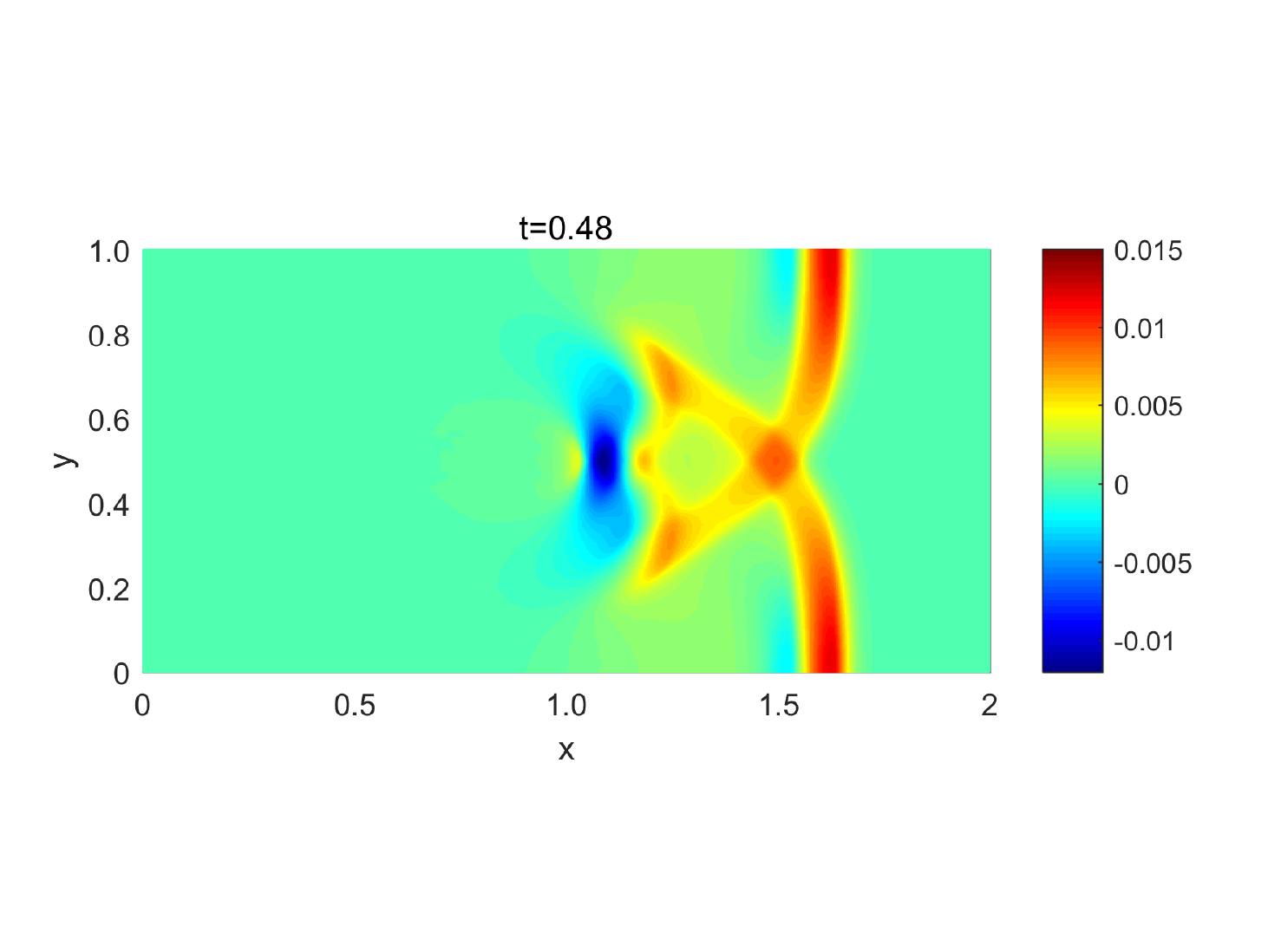}}
\subfigure[$hv$: FM $N=600\times 200\times4$]{
\includegraphics[width=0.30\textwidth, trim=15 60 15 60, clip]{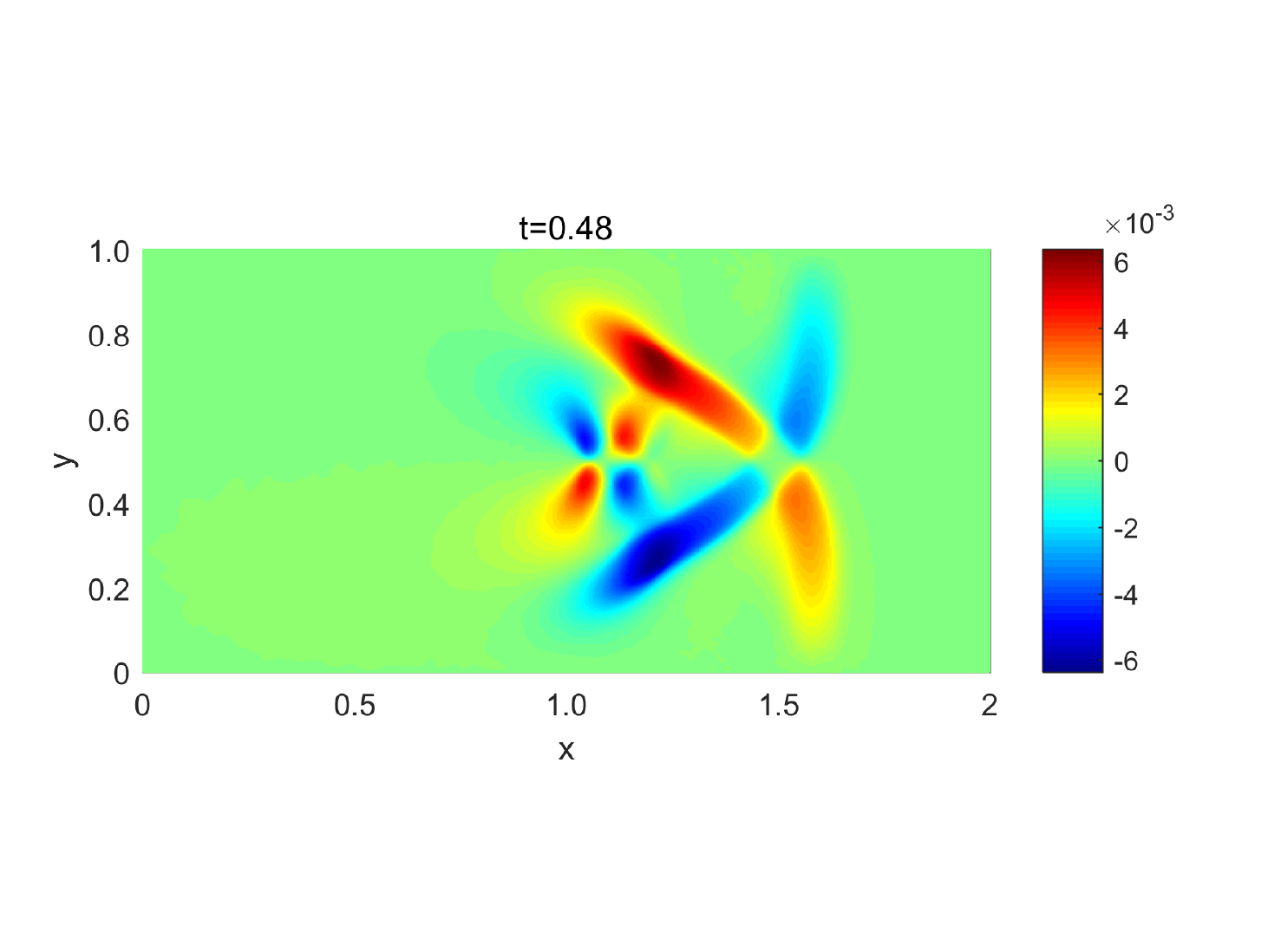}}
\caption{Continuation of Fig.~\ref{Fig:test2-2d-P2-h-hu-hv-t12}: $t = 0.48$.}
\label{Fig:test2-2d-P2-h-hu-hv-t48}
\end{figure}

\section{Conclusions}
\label{sec:conclusions}

We have developed a high-order, well-balanced, positivity-preserving quasi-Lagrange moving mesh DG (QLMM-DG)
method for the SWEs with non-flat bottom topography in the previous sections.
The method combines the quasi-Lagrange moving mesh DG method
\cite{Luo-Huang-Qiu-2019JCP,Zhang-Cheng-Huang-Qiu-2020CiCP} with
the hydrostatic reconstruction technique \cite{Audusse-etal-2004Siam,Xing-Shu-2006CiCP,Xing-Zhang-2013JSC}
and a change of unknown variables to achieve the well-balance property.
Specifically, we use the new variables $(\eta= h + B,\, hu,\, hv)$ instead of
the original ones $(h,\, hu,\, hv)$ and rewrite the flux in a special form (\ref{swe-2d-FUh}) where
some $h$ are replaced by $\eta$ and the others remain the same.
In the construction of the DG numerical flux, the value of $h$ is modified using
the hydrostatic reconstruction technique whereas $\eta$ stays unmodified.
It has been shown that the method, in both semi-discrete and fully discrete forms, preserves
the lake-at-rest steady-state solutions while maintaining the high-order accuracy of DG methods.
It has also been shown that a QLMM-DG scheme can be developed based on the SWEs
in the original variables $(h,\, hu,\, hv)$ but it is well-balanced only in semi-discrete form.

It is worth pointing out that the bottom topography $B$ needs to be updated on the new mesh
at each time step. In the rezoning moving mesh DG method recently developed
in \cite{Zhang-Huang-Qiu-2020arXiv}, it is required that $B$ be updated using the same scheme
as that for the flow variables to attain the well-balance property. This makes the choice
of the scheme for updating $B$ limited. A DG-interpolation scheme \cite{Zhang-Huang-Qiu-2019arXiv}
has been used in \cite{Zhang-Huang-Qiu-2020arXiv} for the purpose.
In contrast, there is no constraint on the choice of the scheme for updating $B$
in the current QLMM-DG method.
We have used $L^2$-projection for updating $B$ in our computation
since it is straightforward and economic to implement.

It should be emphasized that the water depth should be kept nonnegative in the computation.
Following \cite{Xing-Zhang-Shu-2010,Xing-Zhang-2013JSC}, we use a linear scaling positivity-preserving limiter \cite{Liu-Osher1996,ZhangShu2010,ZhangXiaShu2012} to ensure the nonnegativity of the water depth.
To recover the well-balance property violated by the PP limiter, a high-order correction
is made to the approximation of the bottom topography according to the modifications
in the water depth due to the PP limiting; see (\ref{B-update-2}).

The numerical results for a selection of one- and two-dimensional examples have been presented
to demonstrate the well-balance and positive-preserving properties and high-order accuracy of the QLMM-DG method.
They have also shown that the method works well for the lake-at-rest steady state and its perturbations
and is able to adapt the mesh according to structures in the flow and bottom topography.

\appendix
\section{The MMPDE moving mesh method}
\label{sec:mmpde}

In this appendix we describe the generation of the new physical mesh $\mathcal{T}^{n+1}_{h}$
from the old one $\mathcal{T}^{n}_{h}$ using the MMPDE moving mesh method \cite{Huang-etal-1994Siam,Huang-Sun-2003JCP,Huang-Russell-2011}.
We use the $\bm{\xi}$-formulation of the method and its new implementation
proposed in \cite{Huang-Kamenski-2015JCP}.

For mesh generation purpose, we introduce a computational mesh
$\mathcal{T}_c$ with the vertices $\bm{\xi}_1, \;...,\;\bm{\xi}_{N_v}$ and a physical mesh
$\mathcal{T}_h$ with the vertices $\bm{x}_1, \;...,\;\bm{x}_{N_v}$ which can be viewed
as deformations of the mesh $\mathcal{T}^{n}_{h}$. These two meshes serve as intermediate variables.
We also assume that a reference computational mesh
$\hat{\mathcal{T}}_c = \{\hat{ \bm{\xi}}_1,\; ...,\; \hat{ \bm{\xi}}_{N_v} \}$ has been given.
This mesh is kept fixed in the computation and should be chosen as uniform as possible.
Often it can be taken as the initial physical mesh.

A key idea of the MMPDE method is to view any nonuniform mesh as a uniform one
in some Riemannian metric \cite{Huang-Russell-2011}.
The metric tensor $\mathbb{M}$,  a symmetric and uniformly positive definite
matrix-valued function defined on $\mathcal{D}$, provides the information
needed for determining the size, shape, and orientation of the mesh elements throughout the domain.
A choice of $\mathbb{M}$ has been given in \S\ref{sec:numerical-results}.

Since $\mathcal{T}_h$ and $\mathcal{T}_c$ can be viewed
as deformations of $\mathcal{T}^{n}_{h}$,
for any element $K\in \mathcal{T}_h$, there exits an element $K_c\in \mathcal{T}_c$ corresponding to $K$.
Denote the affine mapping from $K_c$ to $K$ as $F_K$ and its Jacobian matrix as $F'_K$.
It is known \cite{Huang-Russell-2011}
that any mesh $\mathcal{T}_h$ which is uniform
in the metric $\mathbb{M}$ in reference to $\mathcal{T}_c$, satisfies
\begin{align}
\label{ec}
& |K|\sqrt{\det(\mathbb{M}_K)}=\frac{\sigma_h|K_c|}{|\mathcal{D}_c|},&\quad \forall K\in \mathcal{T}_h
\\
\label{al}
& \frac{1}{d}\hbox{tr}\big{(} (F'_K)^{-1}\mathbb{M}_K^{-1}(F'_K)^{-T}\big{)}
= \hbox{det}\big{(}
(F'_K)^{-1}\mathbb{M}_K^{-1}(F'_K)^{-T}\big{)}^{\frac{1}{d}},&\quad \forall K \in \mathcal{T}_h
\end{align}
where $d$ is the dimension of the domain ($d=2$ in two dimensions), $\hbox{tr}(\cdot)$ and $\hbox{det}(\cdot)$
denote the trace and determinant of a matrix, respectively,
$\mathbb{M}_K$ is the average of $\mathbb{M}$ over $K$, and
\[
|\mathcal{D}_c|=\sum\limits_{K_c\in\mathcal{T}_c}|K_c|,\quad \sigma_h
=\sum\limits_{K\in\mathcal{T}_h}|K|\hbox{det}(\mathbb{M}_K)^{\frac{1}{2}} .
\]
Condition \eqref{ec} is called the equidistribution condition that determines the size of elements through
the metric tensor $\mathbb{M}$ and requires that all the elements have the same size in the metric tensor $\mathbb{M}$.
Condition \eqref{al} is called the alignment condition that determines the shape and orientation of $K$ through $\mathbb{M}$ and the shape of $K_c$ and requires that $K$, when measured in the metric $\mathbb{M}_K$, be similar to $K_c$ measured in the Euclidean metric.
A mesh energy function, with its minimization resulting in a mesh satisfying the equidistribution and
alignment conditions as closely as possible, is given by
\begin{equation}
\label{energy}
\begin{split}
I_h(\mathcal{T}_h;\mathcal{T}_c)
=&\frac{1}{3}\sum_{K\in\mathcal{T}_h}|K|\hbox{det}(\mathbb{M}_K)^{\frac{1}{2}}\big{(}
\hbox{tr}((F'_K)^{-1}\mathbb{M}^{-1}_K(F'_K)^{-T})\big{)}^{\frac{3 d}{4}}
\\&+\frac{1}{3} d^{\frac{3 d}{4}}\sum_{K\in\mathcal{T}_h}|K|\hbox{det}(\mathbb{M}_K)^{\frac{1}{2}}
\left (\hbox{det}(F'_K) \hbox{det}(\mathbb{M}_K)^{\frac{1}{2}} \right )^{-\frac{3}{2}},
\end{split}
\end{equation}
which is a Riemann sum of a continuous functional developed in \cite{Huang-Russell-2011}.

For a given mesh $\mathcal{T}_h^{n}$, we want to find a new mesh $\mathcal{T}_h^{n+1}$ by minmimizing $I_h$.
Notice that $I_h(\mathcal{T}_h,\mathcal{T}_c)$ is a function of $\mathcal{T}_c$ and $\mathcal{T}_h$
or a function of their vertices.
We take $\mathcal{T}_h$ as $\mathcal{T}_h^{n}$,
minimize $I_h(\mathcal{T}_h^{n},\mathcal{T}_c)$ with respect to $\mathcal{T}_c$
(and denote the final mesh as $\mathcal{T}_c^{n+1}$),
and obtain $\mathcal{T}_h^{n+1}$ through the relation between
$\mathcal{T}_h^{n}$ and $\mathcal{T}_c^{n+1}$.
The minimization of $I_h(\mathcal{T}_h^{n},\mathcal{T}_c)$ is carried out by solving
the mesh equation defined as the gradient system of the energy function (the MMPDE approach), i.e.,
\begin{equation}\label{MM}
\begin{split}
\frac{d \bm{\xi}_i }{dt}
=-\frac{\hbox{det}(\mathbb{M}(\bm{x_i}))^{\frac{1}{2}} }{\tau}
\Big{(}\frac{\partial I_h}{\partial \pmb{\xi}_i}\Big{)}^T,
\quad i=1,...,N_v
\end{split}
\end{equation}
where ${\partial I_h }/{\partial \bm{\xi}_i}$ is considered as a row vector and
$\tau>0$ is a parameter used to adjust the response time of mesh movement to
the changes in $\mathbb{M}$.
We take $\tau=0.1 N^{-\frac{1}{d}}$ in our computation.
Define a function $G$ associated with the energy \eqref{energy} as
\begin{equation}\label{G}
G(\mathbb{J},\hbox{det}(\mathbb{J}))=
\frac{1}{3}\hbox{det}(\mathbb{M}_{K})^{\frac{1}{2}}
(\hbox{tr}(\mathbb{J}\mathbb{M}_{K}^{-1}\mathbb{J}^T))^{\frac{3 d}{4}}
+\frac{1}{3} d^{\frac{3 d}{4}}\hbox{det}(\mathbb{M}_{K})^{\frac{1}{2}}
\left (\frac{\hbox{det}(\mathbb{J})}{\hbox{det}(\mathbb{M}_{K})^{\frac{1}{2}}} \right )^{\frac{3}{2}},
\end{equation}
where $\mathbb{J}=(F'_K)^{-1} = E_{K_c}E_K^{-1}$
and $E_K=[\bm{x}_1^K-\bm{x}_0^K,\;...,\; \bm{x}_d^K-\bm{x}_0^K]$ and
$E_{K_c}=[\bm{\xi}_1^K - \bm{\xi}_0^K,\;...,\; \bm{\xi}_d^K -\bm{\xi}_0^K ]$ are the edge matrices of $K$ and $K_c$, respectively.
By scalar-by-matrix differentiation, we can obtain the derivatives of $G$ with respect
to $\mathbb{J}$ and $\hbox{det}(\mathbb{J})$ \cite{Huang-Kamenski-2015JCP} as
\begin{align}
&\frac{\partial G}{\partial\mathbb{J}}=
\frac{d}{2}\hbox{det}(\mathbb{M}_{K})^{\frac{1}{2}}
(\hbox{tr}(\mathbb{J}\mathbb{M}_K^{-1}\mathbb{J}^T))^{\frac{3 d}{4}-1}\mathbb{M}_K^{-1}\mathbb{J}^T,
\label{partial-J}
\\
&\frac{\partial G}{\partial \hbox{det}(\mathbb{J})}=
\frac{1}{2} d^\frac{3 d}{4}\hbox{det}(\mathbb{M}_{K})^{-\frac{1}{4}}\hbox{det}(\mathbb{J})^{\frac{1}{2}}
\label{partial-detJ}.
\end{align}
Using these formulas, we can rewrite \eqref{MM} as
\begin{equation}\label{xim}
\begin{split}
\frac{d\bm{\xi}_i}{dt}=
\frac{\hbox{det}(\mathbb{M}(\bm{x_i}))^{\frac{1}{2}} }{\tau}
\sum_{K\in\omega_i}|K|\bm{v}^K_{i_K},
\quad i=1,...,N_v
\end{split}
\end{equation}
where $\omega_i$ is the element patch associated with the vertex $\bm{x}_i$,
$i_K$ is the local index of $\bm{x}_i$ on $K$,
and $\bm{v}^K_{i_K}$ is the local velocity contributed by the element $K$ to the vertex $i_K$.
The local velocities associated with $K$ are given by
\begin{equation}\label{vjk}
\begin{split}
\left[
  \begin{array}{c}
    ( \bm{v}_1^K  )^T  \\
    ( \bm{v}_2^K  )^T   \\
    \vdots\\
    ( \bm{v}_d^K  )^T   \\
   \end{array}
 \right]
=
- E_K^{-1}\frac{\partial G }{\partial \mathbb{J} }
- \frac{\partial G}{\partial \hbox{det}(\mathbb{J})}\frac{ \hbox{det}( E_{K_c} )}
{\hbox{det}(E_K)} E_{K_c}^{-1},
\quad
\bm{v}^K_0=-\sum_{i=1}^d\bm{v}^K_{i} .
\end{split}
\end{equation}
We emphasize that the velocities for the boundary nodes must be modified properly.
For example, they should be set to be zero for the corner vertices.
For other boundary vertices, the velocities should be modified such that their normal components along the domain boundary are zeros so they slide only along the boundary and do not move out of the domain.

Starting with the reference computational mesh $\hat{\mathcal{T}}_c$ as the initial mesh,
the mesh equation \eqref{xim} is integrated over a physical time step.
The obtained new mesh is denoted by $\mathcal{T}_c^{n+1}$.
Notice that $\mathcal{T}_h^{n}$ is kept fixed during the integration
and forms a correspondence with $\mathcal{T}_c^{n+1}$. We denote this correspondence
formally by $\mathcal{T}_h^{n}=\Phi_h(\mathcal{T}_c^{n+1})$ with $\bm{x}_i^n = \Phi_h(\bm{\xi}_i^{n+1})$,
$i = 1, ..., N_v$.
Then the new physical mesh $\mathcal{T}_h^{n+1}$ is defined
 as $\mathcal{T}_h^{n+1}=\Phi_h(\hat{\mathcal{T}}_c)$ with $\bm{x}_i^{n+1} = \Phi_h(\hat{\bm{\xi}}_i)$,
$i = 1, ..., N_v$. Notice that $\mathcal{T}_h^{n+1}$ can be computed using linear interpolation.


\end{document}